\documentclass[12pt,a4paper]{amsart}
\usepackage{a4wide}
\usepackage{amsmath, amssymb, amsfonts,enumerate}
\usepackage[all]{xy}
\usepackage{amscd}
\usepackage{comment}
\usepackage{mathtools}

\newtheorem{theorem}{Theorem}[section]
\newtheorem{lemma}[theorem]{Lemma}
\newtheorem{corollary}[theorem]{Corollary}
\newtheorem{proposition}[theorem]{Proposition}

 \theoremstyle{definition}
 \newtheorem{definition}[theorem]{Definition}
 \newtheorem{remark}[theorem]{Remark}

 \newtheorem{example}[theorem]{Example}

\newtheorem{problem}[theorem]{Problem}

\numberwithin{equation}{section}

\newcommand{\diim}{\mbox{\rm dim}}

\newcommand{\Ind}{\mbox{\rm Ind}}
\newcommand{\Hom}{\mbox{\rm Hom}}
\newcommand{\End}{\mbox{\rm End}}

\newcommand{\ran}{{\rm ran}}

\newcommand{\Res}{\mbox{\rm Res}}

\def\FF{{\mathbb{F}}}

\def\CC{{\mathbb{C}}}

\def\TT{{{\mathbb T}}}

\def\C{{\mathbb{C}}}
\newcommand{\GL}{\mbox{\rm GL}} 

\newcommand{\Aff}{{\mbox{\rm Aff}}}

\newcommand{\tr}{\mbox{tr}}
\newcommand{\spann}{\mbox{span}}

 \newcommand{\supp}{\mbox{supp}}

\newcommand{\Id}{\mbox{\rm Id}}

\begin{document}
\title[Multiplicity-free induced representations of finite groups]{Harmonic
analysis and spherical functions for multiplicity-free induced representations of finite groups}

\author{Tullio Ceccherini-Silberstein}
\address{Dipartimento di Ingegneria, Universit\`a del Sannio,
C.so Garibaldi 107, 82100 Benevento (Italy)}
\email{tullio.cs@sbai.uniroma1.it}

\author{Fabio Scarabotti} 
\address{Dipartimento SBAI, Universit\`a degli Studi di Roma ``La Sapienza'', via A. Scarpa 8, 00161 Roma (Italy)}
\email{fabio.scarabotti@sbai.uniroma1.it}

\author{Filippo Tolli}
\address{Dipartimento di Matematica e Fisica, Universit\`a degli Studi Roma Tre, L.go S. Leonardo Murialdo 1, 00146 Roma (Italy)}
\email{tolli@mat.uniroma3.it}

\subjclass[2010]{20C15, 20C08, 20G05, 20G40, 43A65, 43A90, 43A35}
\keywords{Finite group, representation, character, Gelfand pair, Hecke algebra, spherical function, general linear group over a finite field, positive-definite function.}
\begin{abstract}
In this work, we study multiplicity-free induced representations of finite groups. 
We analyze in great detail the structure of the Hecke algebra corresponding to the
commutant of an induced representation and then specialize to the multiplicity-free case. 
We then develop a suitable theory of spherical functions that, in the case of induction 
of the trivial representation of the subgroup, reduces to the classical theory of 
spherical functions for finite Gelfand pairs. 
\par
We also examine in detail the case when we induce from a normal subgroup, 
showing that the corresponding harmonic analysis can be reduced to that on a 
suitable Abelian group. 
\par
The second part of the work is devoted to a comprehensive study of two examples 
constructed by means of the general linear group $\GL(2,\mathbb{F}_q)$, 
where $\mathbb{F}_q$ is the finite field on $q$ elements.
In the first example we induce an indecomposable character of the Cartan subgroup. 
In the second example we induce to $\GL(2,\mathbb{F}_{\!q^2})$ a cuspidal representation of $\GL(2,\mathbb{F}_q)$.
\end{abstract}
\date{\today}
\maketitle 

\tableofcontents

\section{Introduction}

Finite Gelfand pairs play an important role in mathematics and have been studied from several points of view: 
in algebra (we refer, for instance, to the work of Bump and Ginzburg \cite{Bump, BumpGinz} and Saxl \cite{Saxl}; 
see also \cite{AM}), in representation theory (as witnessed by the new approach to the
representation theory of the symmetric groups by Okounkov and Vershik \cite{OV}, see also \cite{book2}), in analysis (with relevant contributions to the theory of special functions by Dunkl \cite{Dunkl} and Stanton \cite{Stanton}), 
in number theory (we refer to the book by Terras \cite{Terras} for a comprehensive introduction; see also \cite{book, book4}), in combinatorics (in the language of association schemes as developed by Bannai and Ito \cite{BannaiIto}), and in probability theory (with the remarkable applications to the study of diffusion processes by Diaconis \cite{Diaconis}; see also \cite{GelGeorg, book}). 
Indeed, Gelfand pairs arise in the study of algebraic, geometrical, or combinatorial structures with a large group of symmetries such that the corresponding permutation representations decompose without multiplicities: it is then possible to develop a useful theory of spherical functions with an associated spherical Fourier transform.

In our preceding work, we have shown that the theory of spherical functions may be studied in a more general setting, namely, for permutation representations that decompose with multiplicities \cite{book3, st1}, for subgroup conjugacy invariant functions 
\cite{CSTTohoku, st3},  and for general induced representations \cite{st5}. Indeed, a finite Gelfand pair may be considered as the simplest example of a multiplicity-free induced representation (the induction of the trivial representation of the subgroup), and this is the motivation of the present paper. 

The most famous of these multiplicity-free representations is the Gelfand-Graev representation of a reductive group over a finite field
\cite{GG} (see also Bump \cite{Bump}). In this direction, we have started our investigations in Part IV of our monograph \cite{book4}, where we have developed a theory of spherical functions and spherical representations for multiplicity-free induced representations of the form $\Ind_K^G\chi$, where $\chi$ is a one-dimensional representation of subgroup $K$. This case was previously investigated by Stembridge \cite{Stembridge}, Macdonald \cite[Exercise 10, Chapter VII]{Macdonald}, and Mizukawa \cite{Mizukawa1, Mizukawa2}. 
We have applied this theory to the Gelfand-Graev representation of $\GL(2,\mathbb{F}_q)$, following the beautiful expository paper of Piatetski-Shapiro \cite{PS}, where the author did not use the terminology/theory of spherical functions but, actually, computed them. 
In such a way, we have shed light on the results and the calculations in \cite{PS} by framing them in a more comprehensive theory.

In the present paper, we face the more general case: we study multiplicity-free induced representations of the form $\Ind_K^G\theta$, 
where $\theta$ is an irreducible $K$-representation, not necessarily one-dimensional. In this case, borrowing a terminology used by 
Bump in \cite[Section 47]{Bump}, we call $(G,K,\theta)$ a {\it multiplicity-free triple}. 
  
Our first target (cf.\ Section \ref{s:MFIR}) is a deep analysis of Mackey's formula for invariants. We show that the commutant $\End_G(\Ind_K^G\theta)$ of an arbitrary induced representation $\Ind_K^G\theta$, with $\theta$ an irreducible $K$-representation, is isomorphic to both a suitable convolution algebra of operator-valued functions defined on $G$ and to a subalgebra of the group algebra of $G$. We call it the \emph{Hecke algebra} of the triple $(G,K,\theta)$ (cf.\ Bump \cite[Section 47]{Bump}, Curtis and Reiner \cite[Section 11D]{CR2}, and Stembridge \cite{Stembridge}; see also \cite[Chapter 13]{book4} and \cite{st4}). Note that this study does not assume multiplicity-freeness. In fact, we shall see (cf.\ Theorem \ref{GKtwisted}) that the triple $(G,K,\theta)$ is multiplicity-free exactly when the associated Hecke algebra is commutative.

We then focus on our main subject of study, namely, multiplicity-free induced representations (cf.\ Section 4); we extend to higher dimensions a criterion of Bump and Ginzburg from \cite{BumpGinz}: this constitutes an analogue of the so-called
weakly-symmetric Gelfand pairs (cf.\ \cite[Example 4.3.2 and Exercise 4.3.3]{book}); we develop the theory of spherical functions in an intrinsic way, that is, by regarding them as eigenfunctions of convolution operators (without using the decomposition of $\Ind_K^G\theta$ into irreducible representations) and obtain a characterization of spherical functions by means of a functional equation. This approach is suitable to more general settings, such as compact or locally compact groups: here we limit ourselves to the finite
case since the main examples that we have discovered (and that we have fully analyzed) fall into this setting. 
Later (cf.\ Section \ref{section 7}), we express spherical functions as matrix coefficients of irreducible (spherical) representations. 
In Section \ref{s:FS-mft} we prove a Frobenius-Schur type theorem for multiplicity-free triples (it provides a criterion for determining the type of a given irreducible spherical representation, namely, being real, quaternionic, or complex).

As mentioned before, the case when $\theta$ is a one-dimensional representation and the example of the Gelfand-Graev representation of $\GL(2,\FF_q)$ were developed, in full details, in \cite[Chapters 13 and 14]{book4} (the last chapter is based on \cite{PS}; see also the pioneering work by Green \cite{Green}). 
Here (cf.\ Section \ref{s:case-dim-1}) we recover the analysis of the one-dimensional case from the general theory we have developed so far
and we briefly sketch the Gelfand-Graev example (cf.\ Section \ref{ss:GGrep})
in order to provide some of the necessary tools for our main new examples of multiplicity-free triples to which the second part of the paper 
(Sections \ref{s:HAMFT1} and \ref{s:IItrippa}) is entirely devoted.

A particular case of interest is when the subgroup $K=N \leq G$ is normal (cf.\ Section \ref{s:normal}). In the classical framework, 
$(G,N)$ is a Gelfand pair if and only if the quotient group $G/N$ is Abelian and, in this case, the spherical Fourier analysis simply reduces to the commutative harmonic analysis on $G/N$. In Section \ref{s:normal} we face the corresponding analysis for multiplicity-free triples 
of the form $(G,N,\theta)$, where $\theta$ is an irreducible $N$-representation. Now, $G$ acts by conjugation on the dual of $N$ and we denote by $I_G(\theta)$ the stabilizer of $\theta$ (this is the {\em inertia group} in Clifford theory; cf.\ the monographs by Berkovich and Zhmudʹ \cite{BZ}, Huppert \cite{Hu}, and Isaacs \cite{Isaacs}; see also \cite{CSTCli, book3}). 
First of all, we study the commutant of $\Ind_N^G\theta$ -- we show that it is isomorphic to a modified convolution algebra on the quotient group $I_G(\theta)/N$ -- and we describe the associated Hecke algebra: all of this theory is developed without assuming 
multiplicity-freeness. 
We then prove that $(G,N,\theta)$ is a multiplicity-free triple if and only if $I_G(\theta)/N$ is Abelian and the multiplicity of $\theta$ in each irreducible representation of $I_G(\theta)$ is at most one. Moreover, if this is the case, the associated Hecke algebra is isomorphic to $L(I_G(\theta)/N)$, the (commutative) group algebra of $I_G(\theta)/N$ with its ordinary convolution product. Thus, as for Gelfand pairs, normality of the subgroup somehow trivializes the analysis of multiplicity-free triples.

As mentioned above, the last two sections of the paper are devoted to two examples of multiplicity-free triples constructed by means of the group 
$\GL(2,\mathbb{F}_q)$. Section \ref{s:HAMFT1} is devoted to the multiplicity-free triple $(\GL(2,\mathbb{F}_q),C,\nu_0)$, where $C$ is the Cartan subgroup of $\GL(2,\mathbb{F}_q)$, which is isomorphic to the quadratic extension $\mathbb{F}_{q^2}$ of $\mathbb{F}_q$, and $\nu_0$ is an \emph{indecomposable} multiplicative character of $\mathbb{F}_{q^2}$ (that is, $\nu_0$ is a character of the multiplicative group $\mathbb{F}_{q^2}^*$ such that $\nu_0(z)$ is \emph{not} of the form $\psi(z\overline{z}), z\in\mathbb{F}_{q^2}^*$, where $\psi$ is a multiplicative character of $\mathbb{F}_q$ and $\overline{z}$ is the conjugate of $z$). 
Actually, $C$ is a \emph{multiplicity-free subgroup}, that is, $(\GL(2,\mathbb{F}_q),C,\nu_0)$ is multiplicity-free for \emph{every} multiplicative character $\nu_0$. We remark that the case $\nu_0=\iota_C$ (the trivial character of $C$) has been extensively studied by Terras under the name of \emph{finite upper half plane} \cite[Chapters 19, 20]{Terras} and corresponds to the Gelfand pair $(\GL(2,\mathbb{F}_q),C)$. We have chosen to study, in full details, the indecomposable case because it is quite different from the Gelfand pair case analyzed by Terras, and constitutes a new example, though much more difficult. 
We begin with a brief description of the representation theory of $\GL(2,\mathbb{F}_q)$, including the Kloosterman sums used for the cuspidal representations. We then compute the decomposition of $\Ind_C^{\tiny \GL(2,\mathbb{F}_q)}\nu_0$ into irreducible representations 
(cf.\ Section \ref{s:deco-prima-trippa}) and the corresponding spherical functions 
(cf.\ Sections \ref{s:spher-prima-trippa-par} and \ref{ss:cuspidal case}). 
We have developed new methods: in particular, in the study of the cuspidal representations, in order to circumvent some technical difficulties, we use, in a smart way, a projection formula onto a one-dimensional subspace. 

In Section \ref{s:IItrippa} we face the most important multiplicity-free triple of this paper, namely, 
$(\GL(2,\mathbb{F}_{q^2}), \GL(2,\mathbb{F}_q), \rho_\nu)$, where $\rho_\nu$ is a cuspidal representation. 
Now the representation that is induced is no more one-dimensional nor is itself an induced representation (as in the parabolic case). 
We have found an intriguing phenomenon: in the computations of the spherical functions associated with the corresponding parabolic spherical representations, we must use the results of Section \ref{s:HAMFT1}, in particular the decomposition of an induced representation of the form $\Ind_C^{\tiny\GL(2,\mathbb{F}_q)}\xi$, with $\xi$ a character of $C$. In other words, the methods developed in Section \ref{s:HAMFT1}
(for the triple $(\GL(2,\mathbb{F}_{q}, C, \nu_0)$) turned out to be essential in the much more involved analysis of the second 
triple $(\GL(2,\mathbb{F}_{q^2}), \GL(2,\mathbb{F}_q), \rho_\nu)$.

We finally remark that it is not so difficult to find other examples of multiplcity-free induced representations within the framework of finite classical groups: for instance, as a consequence of the \emph{branching rule} in the representation theory of the symmetric groups 
(see, e.g., \cite{book2}), $S_{n}$ is a multiplicity-free subgroup of $S_{n+1}$ for all $n \geq 1$. So, although the two examples that we have presented and fully analyzed here are new and highly nontrivial, we believe that we have only scraped the surface and that the subject deserves a wider investigation. For instance, in \cite{BT1, BT2} several Gelfand pairs
constructed by means of $\GL(n,\mathbb{F}_q)$ and other finite linear groups are described. 
It would be interesting to analyze if some of these pairs gives rise to multiplicity free-triples by induction of a nontrivial 
representation.

We also mention that, very recently, a similar theory for locally compact groups has been developed by Ricci and Samanta in \cite{Ricci-Samanta}. In particular, their condition (0.1) corresponds exactly to our condition \eqref{e:F-def}. 
In Section \ref{s:ricci} we show that the Gelfand-Graev representation yields a solution to a problem raised in the Introduction of their paper.\\


\noindent
{\bf Acknowledgments.} We express our deep gratitude to Charles F.\ Dunkl, David Ginzburg, Pierre de la Harpe, Hiroshi Mizukawa,  Akihiro Munemasa, Jean-Pierre Serre, Hajime Tanaka, and Alain Valette, for useful remarks and suggestions.

\section{Preliminaries}\label{Secprel}

In this section, we fix notation and recall some basic facts on linear algebra and representation theory of finite groups
that will be used in the proofs of several results in subsequent sections.

\subsection{Representations of finite groups}
\label{s:rfg}
All vector spaces considered here are complex. Moreover, we shall equip every finite dimensional vector space $V$ with a scalar product denoted by $\langle \cdot, \cdot\rangle_V$ and associated norm $\lVert\cdot \rVert_V$; we usually  omit the subscript if the vector space we are referring to is clear from the context. 
Given two finite dimensional vector spaces $W$ and $U$, we denote by $\Hom(W,U)$ the vector space of all linear maps  
from $W$ to $U$.  When $U = W$ we write $\End(W) = \Hom(W,W)$ and denote by $\GL(W) \subseteq \End(W)$ the general linear group of $W$ consisting of all all bijective linear self-maps of $W$. Also, for $T \in \Hom(W,U)$ we denote by $T^* \in \Hom(U,W)$ the adjoint of $T$.

We define a (normalized {\it Hilbert-Schmidt}) scalar product on  $\Hom(W,U)$ by setting 
\begin{equation}
\label{e:HS-WU}
\langle T_1, T_2 \rangle_{{\tiny \Hom}(W,U)} = \frac{1}{\dim W}\tr(T^*_2T_1)
\end{equation}
for all $T_1, T_2 \in \Hom(W,U)$, where $\tr(\cdot)$ denotes the trace of linear operators; note that this scalar product (as well
as all other scalar products which we shall introduce thereafter) is conjugate-linear in the {\it second} argument. Moreover, by
centrality of the trace (so that $\tr(T_2^*T_1) = \tr(T_1T_2^*)$), we have
\begin{equation}\label{PS}
\langle T_1, T_2 \rangle_{{\tiny \Hom}(W,U)}=\frac{\dim U}{\dim W}\langle T_2^*, T_1^* \rangle_{{\tiny \Hom}(U,W)}. 
\end{equation}
In particular, the map $T \mapsto\sqrt{\dim U/\dim W} T^*$ is an isometry from $\Hom(W,U)$ onto $\Hom(U,W)$.
Finally, note that denoting by $I_W \colon W\rightarrow W$ the identity operator, we have $\lVert I_W\rVert_{{\tiny \End}(W)}=1$.
\par

We now recall some basic facts on the representation theory of finite groups. For more details we refer to our monographs \cite{book, book2, book4}.
Let $G$ be a finite group. A \emph{unitary representation} of $G$ is a pair $(\sigma,W)$ where $W$ is a finite dimensional vector space and 
$\sigma \colon G \to \GL(W)$ is a group homomorphism such that $\sigma(g)$ is unitary (that is, $\sigma(g)^*\sigma(g) = I_W$) for all 
$g \in G$. In the sequel, the term ``unitary'' will be omitted. We denote by $d_\sigma = \dim(W)$ the dimension of the representation 
$(\sigma,W)$. We denote by $(\iota_G,\CC)$ the \emph{trivial representation} of $G$, that is, the one-dimensional $G$-representation defined by $\iota_G(g) = \Id_\CC$ for all
$g \in G$.

Let $(\sigma, W)$ be a $G$-representation. A subspace $V \leq W$ is said to be \emph{G-invariant} provided
$\sigma(g)V \subseteq V$ for all $g \in G$. Writing $\sigma\vert_V(g) = \sigma(g)\vert_V$ for all $g \in G$, we
have that $(\sigma\vert_V,V)$ is a $G$-representation, called a \emph{subrepresentation} of $\sigma$. We then write
$\sigma\vert_V \leq \sigma$.
One says that $\sigma$ is irreducible provided the only $G$-invariant subspaces are trivial (equivalently,
$\sigma$ admits no proper subrepresentations).

Let $(\sigma,W)$ and $(\rho, U)$ be two $G$-representations. 
We denote by 
\[
\Hom_G(W,U) = \{T \in \Hom(W,U):  T\sigma(g) = \rho(g)T, \mbox{ for all }  g \in G\},
\] 
the space of all {\it intertwining operators}. When $U = W$ we write $\End_G(W) = \Hom_G(W,W)$. We equip $\Hom_G(W,U)$ with a 
scalar product by restricting the Hilbert-Schmidt scalar product \eqref{e:HS-WU}.

Observe that if $T\in\Hom_G(W,U)$ then $T^* \in\Hom_G(U,W)$. Indeed,
for all $g \in G$,
\begin{equation}\label{adjoint}
T^*\rho(g) = T^*\rho(g^{-1})^* = (\rho(g^{-1})T)^* = (T\sigma(g^{-1}))^* =
\sigma(g^{-1})^*T^* = \sigma(g)T^*.
\end{equation}
One says that $(\sigma,W)$ and $(\rho, U)$ are \emph{equivalent}, and we shall write $(\sigma,W) \sim (\rho, U)$
(or simply $\sigma \sim \rho$), if there exists a bijective intertwining operator $T \in \Hom_G(W,U)$.

The vector space $\End_G(W)$ of all intertwining operators of $(\sigma,W)$ with itself, when equipped with the multiplication given by the composition of maps and the adjoint operation is a $*$-algebra (see \cite[Chapter 7]{book2}, \cite[Sections 10.3 and 10.6]{book4}), called the \emph{commutant} of $(\sigma,W)$. 
We can thus express the well known Schur's lemma as follows:
$(\sigma,W)$ is irreducible if and only if its commutant is one-dimensional (as a vector space), that is,
it reduces to the scalars (the scalar multiples of the identity $I_W$).

We denote by $\widehat{G}$ a (fixed, once and for all) complete set of pairwise-inequivalent irreducible representations of $G$. It is well known (cf.\ \cite[Theorem 3.9.10]{book} or \cite[Theorem 10.3.13.(ii)]{book4}) that the cardinality of $\widehat{G}$ equals the number of conjugacy classes in $G$ so that, in particular, $\widehat{G}$ is finite. Moreover, if $\sigma, \rho\in \widehat{G}$ we set $\delta_{\sigma,\rho}=1$ (resp.\ $=0$) if $\sigma = \rho$ (resp.\ otherwise).

Let $(\sigma, W)$ and $(\rho, U)$ be two $G$-representations.

The \emph{direct sum} of $\sigma$ and $\rho$ is the representation $(\sigma \oplus \rho, W \oplus U)$ 
defined by $[(\sigma \oplus \rho)(g)](w,u) = (\sigma(g)w,\rho(g)u)$
for all $g \in G$, $w \in W$ and $u \in U$. 

Moreover, if $\sigma$ is a subrepresentation of $\rho$, then denoting by $W^\perp = \{u \in U: \langle u,w\rangle_U=0 \mbox{ for all } w \in W\}$ the orthogonal complement of $W$ in $U$, we have that $W^\perp$ is a $G$-invariant subspace and
$\rho = \sigma \oplus \rho\vert_{W^\perp}$. From this, one deduces that every representation $\rho$ decomposes as a
(finite) direct sum of irreducible subrepresentations. 
More generally, when $\sigma$ is equivalent to a subrepresentation of $\rho$, we say that $\sigma$ is \emph{contained} in $\rho$
and we write $\sigma \preceq \rho$ (clearly, if $\sigma \leq \rho$ then $\sigma \preceq \rho$). 

Suppose that $(\sigma, W)$ is irreducible. Then the number $m = \dim \Hom_G(W, U)$ denotes the \emph{multiplicity}
of $\sigma$ in $\rho$. This means that one may decompose $U = U_1 \oplus U_2 \oplus \cdots \oplus U_m \oplus U_{m+1}$ 
with $(\rho\vert_{U_i},U_i) \sim (\sigma, W)$ for all $i=1,2,\ldots,m$ and $\sigma$ is not contained in $\rho\vert_{U_{m+1}}$. 
The $G$-invariant subspace $U_1 \oplus U_2 \oplus \cdots \oplus U_m \leq U$ is called the $W$-{\it isotypic component} of $U$ and is denoted by $mW$. One also says that $\rho$ (or, equivalently, $U$) contains  $m$ copies of $\sigma$ (resp. of $W$). If this is the case, we say that  $T_1, T_2, \ldots, T_m \in \Hom_G(W,U)$ yield an {\it isometric orthogonal decomposition} of $mW$
if $T_i \in \Hom_G(W,U)$, $T_iW \leq U \ominus U_{m+1}$, and, in addition,
\begin{equation}
\label{component}
\langle T_iw_1, T_jw_2\rangle_U = \langle w_1, w_2\rangle_W\delta_{i,j}
\end{equation}
for all $w_1, w_2 \in W$ and $i,j=1,2,\ldots,m$.
This implies that the subrepresentation $mW = U_1 \oplus U_2 \oplus \cdots \oplus U_m$ is equal to the \emph{orthogonal} direct sum 
$T_1W\oplus T_2W \oplus \cdots \oplus T_mW$, and each operator $T_j$ is a isometry from $W$ onto $U_j \equiv TW_j$. For a quite detailed
analysis of this decomposition, we refer to \cite[Section 10.6]{book4}.

Finally, a representation $(\rho, U)$ is \emph{multiplicity-free} if every $(\sigma, W) \in \widehat{G}$ has multiplicity at most
one in $\rho$, that is, $\dim \Hom_G(W,U) \leq 1$. In other words, given a decomposition of $\rho = \rho_1 \oplus \rho_2 \oplus \cdots \oplus \rho_n$ into irreducible subrepresentations, the $\rho_i$'s are pairwise inequivalent.
Alternatively, as suggested by de la Harpe \cite{Harpe}, one has that $(\rho, U)$ is multiplicity-free if for any nontrivial decomposition $\rho = \rho_1 \oplus \rho_2$ (with $(\rho_1, U_1)$ and $(\rho_2, U_2)$ not necessarily irreducible) there is no
$(\sigma,W) \in \widehat{G}$ such that $\sigma \preceq \rho_i$ (i.e., $\dim \Hom_G(W,U_i) \geq 1$) for $i=1,2$. 
The equivalence between the two definitions is an immediate consequence of the isomorphism 
$\Hom_G(W,U_1 \oplus U_2) \cong \Hom_G(W,U_1) \oplus \Hom_G(W,U_1)$.

\subsection{The group algebra, the left-regular and the permutation representations, and Gelfand pairs}
We denote by $L(G)$ the group algebra of $G$. This is the vector space of all functions $f \colon G \to \CC$
equipped with the \emph{convolution product} $*$ defined by setting $[f_1 * f_2](g) = \sum_{h \in G} f_1(h)f_2(h^{-1}
g) = \sum_{h \in G} f_1(gh)f_2(h^{-1})$, for all $f_1,f_2 \in L(G)$ and $g \in G$. We shall endow $L(G)$ with the
scalar product $\langle \cdot, \cdot \rangle_{L(G)}$ defined by setting
\begin{equation}
\label{e:scalar-l-g}
\langle f_1, f_2 \rangle_{L(G)} = \sum_{g \in G} f_1(g)\overline{f_2(g)}
\end{equation}
for all $f_1,f_2 \in L(G)$. The {\em Dirac functions} $\delta_g$, defined by $\delta_g(g)=1$ and $\delta_g(h)=0$ if $h\neq g$, for all $g,h \in G$, constitute a natural orthonormal basis for $L(G)$.
We shall also equip $L(G)$ with the \emph{involution} $f \mapsto f^*$, where $f^*(g)=\overline{f(g^{-1})}$, for all
$f \in L(G)$ and $g \in G$. It is straightforward to check that $(f_1 * f_2)^* = f_2^* * f_1^*$, for all
$f_1,f_2 \in L(G)$. We shall thus regard $L(G)$ as a $*$-algebra.

The \emph{left-regular representation} of $G$ is the $G$-representation $(\lambda_G, L(G))$ defined by setting
$[\lambda_G(h)f](g) = f(h^{-1}g)$, for all $f \in L(G)$ and $h, g \in G$. Similarly, the \emph{right-regular representation} of $G$ is the $G$-representation $(\rho_G, L(G))$ defined by setting $[\rho_G(h)f](g) = f(gh)$, for all $f \in L(G)$ and $h, g \in G$.
Note that the left-regular and right-regular representations commute, that is, 
\begin{equation}
\label{e:l-r-commutano}
\lambda_G(g_1)\rho_G(g_2) = \rho_G(g_2)\lambda_G(g_1)
\end{equation}
for all $g_1, g_2 \in G$.

Given a subgroup $K \leq G$ we denote by
\[
L(G)^K = \{f \in L(G): f(gk) = f(g), \mbox{ for all } g \in G, k \in K\}
\]
and
\[
^K\!L(G)^K =\{f \in L(G): f(k_1gk_2) = f(g), \mbox{ for all } g \in G, k_1,k_2 \in K\}
\] 
the $L(G)$-subalgebra of $K$-\emph{right-invariant} and \emph{bi}-$K$-\emph{invariant} functions on $G$, respectively.
Note that the subspace $L(G)^K \leq L(G)$ is $G$-invariant with respect to the left-regular representation.
The $G$-representation $(\lambda, L(G)^K)$, where $\lambda(g)f = \lambda_G(g)f$ for all $g \in G$ and
$f \in L(G)^K$ (equivalently, $\lambda=\lambda_G\vert_{L(G)^K}$) is called the \emph{permutation representation}
of $G$ with respect to the subgroup $K$.

More generally, given a representation $(\sigma,W)$ we denote by 
\[
W^K = \{w \in W: \sigma(k)w = w, \mbox{ for all }  k \in K\} \leq W
\] 
the subspace of $K$-\emph{invariant vectors} of $W$. This way, if $(\sigma,W) = (\rho_G, L(G))$ we have $(L(G))^K = L(G)^K$ while,
if $(\sigma,W) = (\lambda, L(G)^K)$ we have $\left(L(G)^K\right)^K = \ ^K\!L(G)^K$.

For the following result we refer to \cite{GelGeorg} and/or to the monographs \cite[Chapter 4]{book} and \cite{Diaconis}.

\begin{theorem} 
\label{t:GP}
The following conditions are equivalent:
\begin{enumerate}[{\rm (a)}]
\item The algebra $^K\!L(G)^K$ is commutative;
\item the permutation representation $(\lambda,L(G)^K)$ is multiplicity-free;
\item the algebra $\End_G(L(G)^K)$ is commutative;
\item for every $(\sigma,W) \in \widehat{G}$ one has $\dim(W^K) \leq 1$;
\item for every $(\sigma,W) \in \widehat{G}$ one has $\dim \Hom_G(W,L(G)^K) \leq 1$.
\end{enumerate}
\end{theorem}

Note that the equivalence (a) $\Leftrightarrow$ (c) follows from the anti-isomorphism \eqref{e:KGK} below.

\begin{definition}
\label{d:GP}{\rm 
If one of the equivalent conditions in Theorem \ref{t:GP} is satisfied, one says that $(G,K)$ is a
\emph{Gelfand pair}.}
\end{definition}

\subsection{The commutant of the left-regular and permutation representations}
Given $f\in L(G)$, the (right) {\em convolution operator} with {\em kernel} $f$ is the linear map $T_{f} \colon L(G) \to L(G)$ defined by
\begin{equation}\label{convoper}
T_{f}f'=f'*f
\end{equation}
for all $f'\in L(G)$. We have 
\begin{equation}
\label{e:conv-anti}
T_{f_1*f_2}=T_{f_2}T_{f_1}
\end{equation} 
 and 
\begin{equation}
\label{e:conv-*}
T_{f^*}=(T_f)^*,
\end{equation} 
for all $f_1, f_2$ and $f$ in $L(G)$. Moreover, $T_{f} \in \End_G(L(G))$ (this is  a consequence of \eqref{e:l-r-commutano}) and the map 
\begin{equation}\label{antiTf}
\begin{array}{ccc}
L(G)&\longrightarrow  &\End_G(L(G))\\
f&\longmapsto &T_f\\
\end{array}
\end{equation} 
is a $*$-anti-isomorphism of $*$-algebras (see \cite[Proposition 1.5.2]{book2} or \cite[Proposition 10.3.5]{book4}). 
Note that the restriction of the map \eqref{antiTf} to the subalgebra $^K\!L(G)^K$
of bi-$K$-invariant functions on $G$ yields a $*$-anti-isomorphism 
\begin{equation}
\label{e:KGK}
^K\!L(G)^K \to \End_G(L(G)^K).
\end{equation}

It is easy to check that $T_f\delta_g=\lambda_G(g)f$ and $\text{tr}(T_f)=\lvert G\rvert f(1_G)$. We deduce that
\begin{equation}\label{Tf1Tf2}
\langle T_{f_1},T_{f_2}\rangle_{{\tiny \End}(L(G))}=\frac{1}{\lvert G\rvert}\text{tr}\left[(T_{f_2})^*T_{f_1}\right]=\frac{1}{\lvert G\rvert}\text{tr}\left[T_{f_1*f_2^*}\right]= [f_1*f_2^*](1_G)=\langle f_1,f_2\rangle_{L(G)}
\end{equation}
for all $f_1,f_2\in L(G)$. This shows that the map \eqref{antiTf} is an isometry.

Let $(\sigma,W)$ be a representation of $G$ and let $\{w_1,w_2, \ldots, w_{d_\sigma}\}$ be an orthonormal basis of $W$. The corresponding \emph{matrix coefficients} $u_{j,i}^\sigma\in L(G)$  are defined by setting 
\begin{equation}\label{stellap2}
u_{j,i}^\sigma(g) = \langle \sigma(g)w_i, w_j\rangle
\end{equation}
for all $i,j = 1,2, \ldots, d_\sigma$ and $g \in G$.

\begin{proposition}
Let $\sigma, \rho \in \widehat{G}$. Then 
\begin{equation}\label{ORT}
\langle u_{i,j}^\sigma, u_{h,k}^\rho\rangle = \frac{|G|}{d_\sigma} \delta_{\sigma,
\rho}\delta_{i,h}\delta_{j,k} \ \ \ \ 
\mbox{\rm (orthogonality relations),}
\end{equation}
\begin{equation}\label{CON} 
u_{i,j}^\sigma* u_{h,k}^\rho = \frac{|G|}{d_\sigma} \delta_{\sigma,
\rho}\delta_{j,h} u^\sigma_{i,k} \ \ \ \
\mbox{\rm    (convolution properties),}
\end{equation}
and
\begin{equation}\label{CONproduct} 
u_{i,j}^\sigma (g_1g_2) = \sum_{\ell=1}^{d_\sigma} u_{i,\ell}^\sigma (g_1)u_{\ell,j}^\sigma (g_2)
\end{equation}
for all $i,j = 1,2, \ldots, d_\sigma$, $h,k = 1,2, \ldots, d_\rho$, and $g_1,g_2 \in G$.
\end{proposition}
\begin{proof}
See  \cite[Lemma 3.6.3 and Lemma 3.9.14]{book} or \cite[Lemma 10.2.10, Lemma 10.2.13, and Proposition 10.3.6]{book4}.
\end{proof}

The sum $\chi^\sigma = \sum_{i=1}^{d_\sigma} u_{i,i}^\sigma\in L(G)$ of the diagonal entries of the matrix coefficients is called the \emph{character} of $\sigma$. Note that $\chi^\sigma(g) = \tr(\sigma(g))$ for all $g \in G$.
The following elementary formula is a generalization of \cite[Exercise 9.5.8.(2)]{book} (see also \cite[Proposition 10.2.26)]{book4}).

\begin{proposition}
Suppose $(\sigma, W)$ is irreducible and let $w \in W$ be a vector of norm $1$. 
Consider the associated diagonal matrix coefficient $\phi_w \in L(G)$ defined by
$\phi_w(g) = \langle\sigma(g) w, w\rangle$ for all $g \in G$. 
Then
\begin{equation}\label{eD21}
\chi^\sigma(g) = \frac{d_\sigma}{|G|}\sum_{h \in G}\phi_w(h^{-1}gh)
\end{equation}
for all $g\in G$.
\end{proposition}
\begin{proof}
Let $\{w_1, w_2, \ldots, w_{d_\sigma}\}$ be an orthonormal basis of $W$ with $w_1 = w$ and let
$u^\sigma_{j,i}$ as in \eqref{stellap2}. Then, for all $g \in G$ and $i=1,2,\ldots, d_\sigma$, we have
\[
\sigma(g)w_i = \sum_{j = 1}^{d_\sigma} u^\sigma_{j,i}(g)w_j
\]
so that
\[
\begin{split}
\sum_{h \in G}\phi_w(h^{-1}gh) & = \sum_{h \in G}\langle\sigma(g)\sigma(h) w_1,
\sigma(h) w_1\rangle\\
& = \sum_{j, \ell = 1}^{d_\sigma}\sum_{h \in G} 
u^\sigma_{j,1}(h)\overline{u^\sigma_{\ell,1}(h)}\langle\sigma(g) w_j, w_\ell\rangle\\
& = \sum_{j, \ell = 1}^{d_\sigma} \langle u^\sigma_{j,1} u^\sigma_{\ell,1} \rangle \langle\sigma(g) w_j, 
w_\ell\rangle\\
\mbox{(by \eqref{ORT})}\ \ \ \ & = \frac{|G|}{d_\sigma}\chi^\sigma(g),
\end{split}
\]
and \eqref{eD21} follows.
\end{proof}

From \cite[Corollary 1.3.15]{book2} we recall the following fact. Let $(\sigma, W)$ and $(\rho, V)$ be two $G$-representations
and suppose that $\rho$ is irreducible and contained in $\sigma$. Then
\begin{equation}
\label{estar:PTPT1}
E_\rho = \frac{d_\rho}{|G|} \sum_{g \in G} \overline{\chi^\rho(g)} \sigma(g)
\end{equation}
is the orthogonal projection onto the $\rho$-isotypic component of $W$.

\subsection{Induced representations}
Let now $K \leq G$ be a subgroup. We denote by $(\Res^G_K\sigma,W)$ the
\emph{restriction} of the $G$-representation $(\sigma,W)$ to $K$, that is, 
the $K$-representation defined by $[\Res^G_K\sigma](k) =
\sigma(k)$ for all $k \in K$.

Given a $K$-representation  $(\theta, V)$ of $K$, denote by
$\lambda = \Ind_K^G\theta$ the {\it induced representation} (see, for
instance, \cite{Bump, CSTind, book2, book3, book4, FellDoran,  NS, Simon, Sternberg, Terras}). 
We recall that the representation space of $\lambda$ is given by
\begin{equation}\label{H31}
\Ind_K^GV = \{f\colon G \to V \mbox{\rm \ such that } f(gk) = \theta(k^{-1})f(g),  \mbox{ for all } g \in G, k \in K\}
\end{equation} 
and that 
\begin{equation}\label{H32}
[\lambda(g) f](g') = f(g^{-1}g'),
\end{equation}
for all $f \in \Ind_K^GV$ and $g,g' \in G$.

As an example, one checks that if $(\iota_K,\CC)$ is the trivial representation of
$K$, then $(\Ind_K^G\iota_K,\Ind_K^G\CC)$ equals the permutation representation $(\lambda, L(G)^K)$
of $G$ with respect to the subgroup $K$ (see \cite[Proposition 1.1.7]{book3} or \cite[Example 11.1.6]{book4}).

Let $\mathcal{T} \subseteq G$ be a left-transversal for $K$, that is, a complete
set of representatives for the left-cosets $gK$ of $K$ in $G$. Then we have the decomposition
\begin{equation}\label{rightcosets}
G  = \bigsqcup_{t \in \mathcal{T}}tK,
\end{equation} 
where, from now on, $\bigsqcup$ denotes a disjoint union.
For $v \in V$ we define  $f_v \in \Ind_K^GV$  by setting 
\begin{equation}\label{definizione stellata1}
f_v(g) = \begin{cases}
\theta(g^{-1})v & \mbox{if $g \in K$}\\
0 & \mbox{otherwise.}
\end{cases}
\end{equation}
Then, for every  $f \in \Ind_K^GV$, we have
\begin{equation}\label{definizione stellata2}
f = \sum_{t\in \mathcal{T}}\lambda(t)f_{v_t}
\end{equation}
where $v_t = f(t)$ for all $t\in \mathcal{T}$. 
The induced representation $\Ind_K^G\theta$ is unitary with respect to the scalar
product $\langle \cdot, \cdot \rangle_{\Ind_K^GV}$ defined by
\begin{equation}\label{H33}
\langle f_1, f_2\rangle_{\tiny\Ind_K^GV}  = \frac{1}{|K|} \sum_{g \in G}\langle f_1(g),
f_2(g) \rangle_V = \sum_{t \in  \mathcal{T}} \langle f_1(t), f_2(t) \rangle_V
\end{equation}
for all $f_1, f_2 \in \Ind_K^GV$. Moreover, if $\{v_j:j=1,2,\dotsc,d_\theta\}$ is an orthonormal basis in $V$ 
then the set
\begin{equation}\label{orthbasisind}
\{\lambda(t)f_{v_j}:t\in\mathcal{T}, j=1,2,\dotsc,d_\theta\}
\end{equation}
is an orthonormal basis in $\Ind_K^G V$ (see \cite[Theorem 2.1]{CSTind} and \cite[Theorem 11.1.11]{book4}).

A well known relation between the induction of a $K$-representation $(\theta,V)$ and a $G$-representation
$(\sigma,W)$ is expressed by the so called \emph{Frobenius reciprocity} (cf.\ \cite[Theorem 1.6.11]{book2},
\cite[Theorem 1.1.19]{book3}, or \cite[Theorem 11.2.1]{book4}):
\begin{equation}
\label{e:FR}
\Hom_G(W,\Ind_K^G V) \cong \Hom_K(\Res^K_G W,V).
\end{equation}

Let $J = \{\sigma \in \widehat{G}: \sigma \preceq \Ind_K^G\theta\}$ denote a complete set of pairwise inequivalent
irreducible $G$-representations contained in $\Ind_K^G\theta$.
For $\sigma\in J$ we denote by $W_\sigma$ its representation space and by $m_\sigma=\dim \Hom_G(\sigma,\Ind_K^G\theta)
\geq 1$ its multiplicity in $\Ind_K^G\theta$. Then
\begin{equation}\label{isomIndmW}
\Ind_K^GV\cong \bigoplus_{\sigma\in J}m_\sigma W_\sigma
\end{equation}
is the decomposition of $\Ind_K^GV$ into irreducible $G$-representations and we have the $*$-isomorphism of $*$-algebras
\begin{equation}\label{isomHomIndMat}
\End_G(\Ind_K^GV)\cong\bigoplus_{\sigma\in J}M_{m_\sigma}(\mathbb{C}),
\end{equation}
where $M_{m}(\mathbb{C})$ denotes the $*$-algebra of all $m\times m$ complex matrices (cf.\ \cite[Theorem 10.6.3]{book4}). 
In particular:

\begin{proposition}
\label{p:equiv-MF-pre}
The following conditions are equivalent:
\begin{enumerate}[{\rm (a)}]
\item $\Ind_K^G\theta$ is multiplicity-free (that is, $m_\sigma=1$ for all $\sigma\in J$);
\item the algebra $\End_G(\Ind_K^GV)$ (i.e. the commutant of $\Ind^K_G V$) is commutative;
\item $\End_G(\Ind_K^GV)$ is isomorphic to the $*$-algebra $\mathbb{C}^J = \{f \colon J \to \CC\}$ equipped with pointwise multiplication and complex conjugation.
\end{enumerate}
\end{proposition}

\begin{remark}
\label{r:1111}
{\rm In \eqref{isomIndmW} and \eqref{isomHomIndMat} we have used the symbol $\cong$ to denote an {\em isomorphism} (with respect to the corresponding algebraic structure). We will use the equality symbol $=$ to denote an {\em explicit decomposition}. For instance, in the multiplicity-free case this corresponds to a choice of an {\em isometric immersion} of $W_\sigma$ into $\Ind_K^GV$, that is, to a map $T_\sigma\in\Hom_G(W_\sigma,\Ind_K^GV)$ which is also an isometry. Clearly, in this case, $\Hom_G(W_\sigma,\Ind_K^GV)=\CC T_\sigma$ and
\begin{equation}\label{decIndmW}
\Ind_K^GV=\bigoplus_{\sigma\in J}T_\sigma W_\sigma
\end{equation}
is the explicit decomposition. If multiplicities arise, then we decompose explicitly each isotypic component as in Section \ref{s:rfg} (cf.\ \eqref{component}).}
\end{remark}

\section{Hecke algebras}

Let $G$ be a finite group and $K \leq G$ a subgroup. Recalling the equality between the induced representation
$(\Ind^G_K\iota_K,\Ind^G_K\CC)$ and the permutation representation $(\lambda,L(G)^K)$,
\eqref{e:KGK} yields a $*$-algebra isomorphism between the algebra of bi-$K$-invariant 
functions on $G$ and the commutant of the representation obtained by inducing to $G$ the 
trivial representation of $K$.

In Section \ref{s:MFIR}, expanding the ideas in \cite[Theorem 34.1]{Bump} and in \cite[Section 3]{PS}, 
we generalize this fact by showing that for a generic representation $(\theta, V)$ of $K$, the commutant
of $\Ind^K_G V$, that is, $\End_G(\Ind_K^G V)$, is isomorphic to a suitable convolution algebra of operator-valued maps on $G$. 
This may be considered as a detailed formulation of Mackey's formula for invariants (see \cite[Section 6]{CSTind} or \cite[Corollary 11.4.4]{book4}).

Later, in Section \ref{s:HaR}, we show that, when $\theta$ is irreducible, the algebra $\End_G(\Ind_K^G V)$ is isomorphic to a
suitable subalgebra of the group algebra $L(G)$ of $G$.

\subsection{Mackey's formula for invariants revisited}
\label{s:MFIR}
In this section, we study isomorphisms (or antiisomorphisms) between three $*$-algebras. We explicitly use the terminology of a Hecke algebra only for the first one (cf.\ Definition \ref{defHtilde}), although one may carry it also for the other two. 
\par
Let $(\theta, V)$ be a $K$-representation.
We denote by $\widetilde{\mathcal{H}}(G,K,\theta)$ the set of all maps $F\colon G \to \End(V)$ such that
\begin{equation}
\label{e:F-def}
F(k_1gk_2) = \theta(k_2^{-1})F(g)\theta(k_1^{-1}), \text{ for all } g \in G \text{ and } k_1,k_2 \in K.
\end{equation}
Given $F_1,F_2\in\widetilde{\mathcal{H}}(G,K,\theta)$ we define
their \emph{convolution product} $F_1*F_2 \colon G \to \End(V)$ by setting
\begin{equation}
\label{e:star-pag8-}
[F_1*F_2](g)=\sum_{h\in G}F_1(h^{-1}g)F_2(h)
\end{equation}
for all $g\in G$, and their scalar product as
\begin{equation}
\label{e:star-pag8}
\langle F_1, F_2\rangle_{\widetilde{\mathcal{H}}(G,K,\theta)}=\sum_{g\in G}\langle F_1(g),F_2(g)\rangle_{{\tiny \End}(V)}.
\end{equation}
Finally, for $F\in\widetilde{\mathcal{H}}(G,K,\theta)$ we define the \emph{adjoint} $F^* \colon G \to \End(V)$ by setting 
\begin{equation}
\label{e:star-pag8+}
F^*(g)=[F(g^{-1})]^*
\end{equation}
for all $g\in G$, where $[F(g^{-1})]^*$ is the adjoint of the operator $F(g^{-1})\in\End(V)$.

It is easy to check that $\widetilde{\mathcal{H}}(G,K,\theta)$ is an associative unital algebra with respect to this convolution.
The identity is the function $F_0$ defined by setting $F_0(k) = \frac{1}{|K|} \theta(k^{-1})$ for all $k \in K$ and $F_0(g) = 0$
for $g \in G$ not in $K$; see also \eqref{LTg} below. 
Moreover, $F^*$ still belongs to $\widetilde{\mathcal{H}}(G,K,\theta)$, the map $F\mapsto F^*$ is an involution, that is, $(F^*)^*=F$, and $(F_1*F_2)^*=F_2^**F_1^*$, for all $F,F_1,F_2\in \widetilde{\mathcal{H}}(G,K,\theta)$.

\begin{definition}
\label{defHtilde}
{\rm The unital $*$-algebra $\widetilde{\mathcal{H}}(G,K,\theta)$ is called the \emph{Hecke algebra} associated
with the group $G$ and the $K$-representation $(\theta, V)$.}
\end{definition}

Let $\mathcal{S} \subseteq G$ be a complete set of representatives for the double $K$-cosets in $G$ so that
\begin{equation}\label{doublecosets}
G = \bigsqcup_{s \in \mathcal{S}}KsK.
\end{equation}

We assume that $1_G \in \mathcal{S}$, that is, $1_G$ is the representative of $K$.
For $s \in \mathcal{S}$ we set 
\begin{equation}\label{defKs}
K_s =  K\cap sKs^{-1}
\end{equation} 
and observe that given $g \in KsK$ we have $\vert\{(k_1,k_2)\in K^2: k_1sk_2 = g\}\vert = \vert K_s\vert$.
Indeed, suppose that $k_1sk_2 = g = h_1sh_2$, where $k_1,k_2,h_1,h_2 \in K$.
Then we have ${h_1}^{-1}k_1 = sh_2{k_2}^{-1}s^{-1}$ which gives, in particular, ${h_1}^{-1}k_1 \in K_s$.
Thus there are $\vert K_s \vert = \vert k_1K_s \vert$ different choices for $h_1$, and since 
$h_2 = s^{-1}{h_1}^{-1}g$ is determined by $h_1$, the observation follows.

As a consequence, given an Abelian group $A$ (e.g.\ $\mathbb{C}$, a vector space, etc.), for any map $\Phi \colon G \to A$ and 
$s\in \mathcal{S}$ we have
\begin{equation}\label{sumPhi}
\sum_{g\in KsK}\Phi(g)=\frac{1}{\lvert K_s\rvert}\sum_{k_1,k_2\in K}\Phi(k_1sk_2).
\end{equation}
 
For $s \in \mathcal{S}$ we denote by $(\theta^s, V_s)$ the $K_s$-representation defined by setting 
$V_s = V$ and
\begin{equation}
\label{e:theta-s}
\theta^s(x) = \theta(s^{-1}xs)
\end{equation}
for all $x \in K_s$. 

For $T \in \Hom_{K_s}(\Res^K_{K_s} \theta, \theta^s)$ define $\mathcal{L}_T \colon G \to \End(V)$ by setting
\begin{equation}\label{LTg}
\mathcal{L}_T(g) = \begin{cases}
\theta(k_2^{-1})T\theta(k_1^{-1}) & \mbox{ if } g = k_1sk_2 \mbox{ for some } k_1, k_2 \in K\\
0 & \mbox{ if } g \notin KsK.
\end{cases}
\end{equation}
Let us show that \eqref{LTg} is well defined and that $\mathcal{L}_T \in \widetilde{\mathcal{H}}(G,K,\theta)$. 
Let $k_1,k_2,h_1,h_2 \in K$. 

Suppose again that $k_1sk_2 = g = h_1sh_2$, so that, as before, we can find $k_s \in K_s$ such that $k_1 = h_1k_s$ and therefore
\begin{equation*}
\label{e:ben-def}
k_1^{-1} = {k_s}^{-1}{h_1}^{-1} \ \mbox{ and } \ k_2^{-1} = {h_2}^{-1}s^{-1}k_ss.
\end{equation*}
We then have
\[
\begin{split}
\theta(k_2^{-1})T\theta(k_1^{-1}) & = \theta({h_2}^{-1}s^{-1}k_ss)T\theta({k_s}^{-1}{h_1}^{-1})\\
& =  \theta({h_2}^{-1})\theta(s^{-1}k_ss)T\theta({k_s}^{-1}) \theta({h_1}^{-1})\\
& =   \theta({h_2}^{-1})\theta^s(k_s)T\theta({k_s}^{-1}) \theta({h_1}^{-1})\\
\mbox{(since $T\in \Hom_{K_s}(\Res^K_{K_s}\theta,\theta^s)$) \ } & =   \theta({h_2}^{-1})T\theta(k_s)\theta({k_s}^{-1}) \theta({h_1}^{-1})\\
& =   \theta({h_2}^{-1})T \theta({h_1}^{-1}).
\end{split}
\]
It follows that \eqref{LTg} is well defined. 

Suppose now that $g = k_1sk_2$ so that $h_1gh_2 = h_1k_1sk_2h_2$. Then, by \eqref{LTg}, we have
\[
\begin{split}
\mathcal{L}_T(h_1gh_2) & = \mathcal{L}_T(h_1k_1sk_2h_2)\\ 
& = \theta((k_2h_2)^{-1})T\theta((h_1k_1)^{-1})\\ 
& = \theta({h_2}^{-1})\left(\theta(k_2^{-1})T\theta(k_1^{-1})\right)\theta({h_1}^{-1})\\
& = \theta(h_2^{-1})\mathcal{L}_T(g)\theta(h_1^{-1}).
\end{split}
\]
This shows that $\mathcal{L}_T \in \widetilde{\mathcal{H}}(G,K,\theta)$.

We set
\begin{equation}
\label{e:s-0}
\mathcal{S}_0 = \{s \in \mathcal{S}: \Hom_{K_s}(\Res^K_{K_s}\theta,\theta^s) \mbox{ is nontrivial}\}.
\end{equation}

\begin{lemma} \label{lemma6}
\begin{enumerate}[{\rm (1)}]
\item
If $F\in \widetilde{\mathcal{H}}(G,K,\theta)$ then 
\begin{equation}
\label{e:FS}
F(s) \in \Hom_{K_s}(\Res^K_{K_s} \theta, \theta^s) \mbox{ for all $s \in \mathcal{S}$.}
\end{equation}
\item\label{lemma6b1}
If $F\in \widetilde{\mathcal{H}}(G,K,\theta)$ then
\begin{equation}
\label{e:FS1111111}
F = \sum_{s \in \mathcal{S}_0}\mathcal{L}_{F(s)}
\end{equation} 
and the nontrivial elements in this sum are linearly independent.
\item\label{lemma6c2}
If $F_1,F_2 \in \widetilde{\mathcal{H}}(G,K,\theta)$ then 
\[
\langle F_1, F_2\rangle_{\widetilde{\mathcal{H}}(G,K,\theta)} = |K|^2\sum_{s \in \mathcal{S}_0}\frac{1}{|K_s|}\langle F_1(s), F_2(s)\rangle_{{\tiny \End}(V)}.
\]
\end{enumerate}
\end{lemma}
\begin{proof} 
(1) Let $s \in \mathcal{S}$. For all $x \in K_s$, by \eqref{e:F-def}, we have 
\[
\begin{split}
F(s) \theta(x) & = F(x^{-1}s)\\
& = F(s\cdot s^{-1}x^{-1}s)\\
& = \theta(s^{-1}xs)F(s)\\
&  = \theta^s(x)F(s),
\end{split}
\]
that is, $F(s) \in \Hom_{K_s}(\Res^K_{K_s} \theta, \theta^s)$. In particular, $F(s) = 0$ if $s \notin \mathcal{S}_0$.

(2)  From \eqref{e:F-def} we deduce $F(k_1sk_2) = \theta(k_2^{-1})F(s)\theta(k_1^{-1}) = \mathcal{L}_{F(s)}(k_1sk_2)$ for all $s \in S$ and $k_1,k_2 \in K$. As a consequence, $F$ is determined by its values on $\mathcal{S}_0$.
Moreover, for distinct $s,s' \in \mathcal{S}_0$ the maps $\mathcal{L}_{F(s)}$ and $\mathcal{L}_{F(s')}$ have disjoint supports (namely the double cosets $KsK$ and $Ks'K$, respectively).
From \eqref{doublecosets} we then deduce \eqref{e:FS1111111} and that the nontrivial elements in the sum are linearly independent.

(3) We have
\[
\begin{split}
\langle F_1, F_2\rangle_{\widetilde{\mathcal{H}}(G,K,\theta)} & =
\frac{1}{\dim V}\sum_{g \in G}\tr[F_2(g)^*F_1(g)]\\
(\text{by } \eqref{doublecosets} \text{ and } \eqref{sumPhi}) \  & =   \frac{1}{\dim V}\sum_{s \in {\mathcal S}_0}\frac{1}{|K_s|}\sum_{k_1,k_2 \in K}\tr[F_2(k_1sk_2)^*F_1(k_1sk_2)]\\
& = \frac{1}{\dim V}\sum_{s \in {\mathcal S}_0}\frac{1}{|K_s|}\sum_{k_1,k_2 \in K}\tr[\theta(k_1)F_2(s)^*\theta(k_2)\theta(k_2^{-1})F_1(s)\theta(k_1^{-1})]\\
& = |K|^2\sum_{s \in {\mathcal S}_0} \frac{1}{|K_s|\cdot \dim V}\tr[F_2(s)^*F_1(s)]\\
&=|K|^2 \sum_{s \in {\mathcal S}_0} \frac{1}{|K_s|}\langle F_1(s), F_2(s)\rangle_{{\tiny \End}(V)}.
\end{split}
\]
\end{proof}

\begin{proposition}\label{CorbasisL}
For each $s\in\mathcal{S}_0$, select an orthonormal basis $\{T_{s,1},T_{s,2},\dotsc,T_{s,m_s}\}$ in $\Hom_{K_s}(\Res^K_{K_s} \theta, \theta^s)$. Then the set $\{\mathcal{L}_{T_{s,i}}: s\in\mathcal{S}_0,1\leq i\leq m_s\}$ is an orthogonal basis of $\widetilde{\mathcal{H}}(G,K,\theta)$ and, for $s,t\in\mathcal{S}_0$, $1\leq i\leq m_s$, $1\leq j\leq m_t$, we have:
\[
\langle \mathcal{L}_{T_{s,i}},\mathcal{L}_{T_{t,j}}\rangle_{\widetilde{\mathcal{H}}(G,K,\theta)}=\delta_{s,t}\delta_{i,j}\frac{|K|^2}{|K_s|}.
\]
\end{proposition}
\begin{proof}
From Lemma \ref{lemma6}.(2) it follows that $\{\mathcal{L}_{T_{s,i}}: s\in\mathcal{S}_0,1\leq i\leq m_s\}$ is a basis of $\widetilde{\mathcal{H}}(G,K,\theta)$. The orthogonality relations follow easily from Lemma \ref{lemma6}.(3).
\end{proof}

We now define a map
\begin{equation}
\label{e:csi}
\xi \colon \widetilde{\mathcal{H}}(G,K,\theta)\to \End(\Ind_K^G V),
\end{equation}
by setting
\begin{equation}\label{e;indtau4}
[\xi(F)f](g) =\sum_{h \in G} F(h^{-1} g) f(h)
\end{equation} 
for all $F \in \widetilde{\mathcal{H}}(G,K,\theta)$, $f \in \Ind_K^G V$ and $g \in G$. Note that $F(h^{-1} g) f(h)$ indicates
the action of the operator $F(h^{-1}g)$ on the vector $f(h)$.

\begin{lemma}\label{lemma9}
If $F \in \widetilde{\mathcal{H}}(G,K,\theta)$ then 
\[
\mbox{\rm \tr}[\xi(F)] = |G| \cdot \mbox{\rm \tr}[F(1_G)].
\]
\end{lemma}
\begin{proof}
Suppose that $f_{v}$ is as in \eqref{definizione stellata1} and $t\in\mathcal{T}$ (see \eqref{rightcosets}).
We then have:
\begin{equation}\label{exclaim}
\begin{split}
[\xi(F)\lambda(t)f_{v}](g) &= \sum_{h \in G}F(h^{-1}g)f_v(t^{-1}h)\\
\mbox{(setting $t^{-1}h = k$)} \ \ \ \ \ & = \sum_{k \in K}F(k^{-1}t^{-1}g)\theta(k^{-1})v\\
\mbox{(by \eqref{e:F-def})}\ \ \ \ \ & =|K| F(t^{-1}g)v.
\end{split}
\end{equation}
Then, computing the trace of $\xi(F)$ by means of the orthonormal basis \eqref{orthbasisind} we get:
\[
\begin{split}
\tr[\xi(F)]& = \sum_{t \in \mathcal{T}}\sum_{j = 1}^{d_\theta}\langle \xi(F)\lambda(t)f_{v_j},\lambda(t)f_{v_j}\rangle_{\tiny\Ind_K^GV}\\
\mbox{(by \eqref{H33})} \  & = \sum_{t \in \mathcal{T}}\sum_{j = 1}^{d_\theta}\frac{1}{|K|}\sum_{g \in G}\langle[\xi(F)\lambda(t)f_{v_j}](g), [\lambda(t)f_{v_j}](g)\rangle_V\\
\mbox{(by \eqref{exclaim})} \ & = \sum_{t \in \mathcal{T}}\sum_{j = 1}^{d_\theta}\sum_{g \in G}\langle F(t^{-1}g)v_j, f_{v_j}(t^{-1}g)\rangle_V\\
\mbox{(setting $t^{-1}g = k$)} \  & =  \sum_{t \in \mathcal{T}}\sum_{j = 1}^{d_\theta}\sum_{k \in K}\langle F(k)v_j, \theta(k^{-1})v_j\rangle_V\\
\mbox{(by \eqref{e:F-def})} \ &  =  \sum_{t \in \mathcal{T}}\sum_{j = 1}^{d_\theta}|K| \langle F(1_G)v_j, v_j\rangle_V\\
& = |G|\tr[F(1_G)].
\end{split}
\]
\end{proof}

We are now in a position to state and prove the main result of this section.

\begin{theorem}\label{isomtildeH}
$\xi(F)\in\End_G(\Ind_K^G V)$ for all $F\in\widetilde{\mathcal{H}}(G,K,\theta)$
and $\xi$ is a $\ast$-isomorphism between $\widetilde{\mathcal{H}}(G,K,\theta)$ and $\End_G(\Ind_K^G V)$. Moreover,
\begin{equation}\label{xiisometry}
\langle\xi(F_1),\xi(F_2)\rangle_{{\tiny \End}_G({\tiny \Ind}_K^G V)}=\lvert K\rvert\langle F_1,F_2\rangle_{\widetilde{\mathcal{H}}(G,K,\theta)},
\end{equation}
for all $F_1,F_2 \in \widetilde{\mathcal{H}}(G,K,\theta)$,
that is, the normalized map $F\mapsto \frac{1}{\sqrt{\lvert K\rvert}}\xi(F)$ is an isometry.
\end{theorem}

\begin{proof} 
Let $F\in \widetilde{\mathcal{H}}(G,K,\theta)$, $f\in \text{Ind}_K^G V$ and $g,h\in G$. Then, if $\lambda = \Ind^G_K \theta$ as before, we have
\[
\begin{split}
[\lambda(h)\xi(F)f](g)&=[\xi(F)f](h^{-1}g)\\
\mbox{(by \eqref{e;indtau4})} \ \ &=\sum_{r\in G}F(r^{-1}h^{-1}g)f(r)\\
\mbox{(setting $q=hr$)} \  &=\sum_{q\in G}F(q^{-1}g)f(h^{-1}q)\\
&=\sum_{q\in G}F(q^{-1}g)[\lambda(h)f](q)\\
&=[\xi(F)\lambda(h)f](g),
\end{split}
\]
that is, $\lambda(h)\xi(F)=\xi(F)\lambda(h)$. This shows that $\xi(F)\in\End_G(\Ind_K^G V)$. 
Moreover it is also easy to check that
\begin{equation}\label{isomor1}
\xi(F_1*F_2) = \xi(F_1)\xi(F_2).
\end{equation}
and
\begin{equation}\label{isomor2}
\xi(F)^* = \xi(F^*).
\end{equation}
Just note that $\xi(F)^*$ is the adjoint of an operator in $\End(\Ind_K^G V)$ with respect to the scalar product in \eqref{H33}.
By \eqref{isomor1} and  \eqref{isomor2} we have $\xi(F_2)^*\xi(F_1) = \xi(F_2^**F_1)$ and therefore 
\[
\begin{split}
\langle \xi(F_1), \xi(F_2) \rangle_{{\tiny \End}_G({\tiny \Ind}_K^G V)}   & = \frac{1}{\dim V\cdot |G/K|}\tr[\xi(F_2^**F_1)]\\
\mbox{(by Lemma \ref{lemma9})}\  & =\frac{|K|}{\dim V} \tr[(F_2^**F_1)(1_G)]\\
&= \frac{|K|}{\dim V}\sum_{g \in G}\tr[F_2^*(g^{-1})F_1(g)]\\
& =  \frac{|K|}{\dim V}\sum_{g \in G}\tr[F_2(g)^*F_1(g)]\\
\mbox{(by \eqref{e:star-pag8})} \ \ & = \lvert K\rvert\langle F_1,F_2\rangle_{\widetilde{\mathcal{H}}(G,K,\theta)},
\end{split}
\]
and \eqref{xiisometry} follows. In particular, $\xi$ is injective. It only remains to prove that $\xi$ is surjective.

Let $T\in \End_G(\Ind_K^G V)$ and define $\Xi(T) \colon G\rightarrow\End(V)$ by setting, for all $v \in V$ and $g\in G$,  
\begin{equation}\label{e;indtau2} 
\Xi(T)(g)v = \frac{1}{|K|} [Tf_v] (g), 
\end{equation}
where $f_v$ is as in \eqref{definizione stellata1}. 
For $k_1,k_2\in K$ we then have
\[\begin{split}
\Xi(T)(k_1gk_2)v & = \frac{1}{|K|} [Tf_v] (k_1gk_2)  \\
\mbox{(since $Tf_v \in \Ind_K^G V$)} \  & = 
\theta(k_2^{-1})  \frac{1}{|K|} [Tf_v] (k_1g) \\
& = \theta(k_2^{-1}) \left\{\frac{1}{|K|} \lambda(k_1^{-1})[T f_v] (g)\right\}  \\
\mbox{(since $T\in \End_G(\Ind_K^G V)$)} \ & = \theta(k_2^{-1}) \left\{\frac{1}{|K|} [T \lambda(k_1^{-1})f_v] (g)\right\}\\
\mbox{(since $\lambda(k_1^{-1})f_v = f_{\theta(k_1^{-1})v}$)} \ & = \theta(k_2^{-1}) \left\{\frac{1}{|K|} [T f_{\theta(k_1^{-1})v}] (g)\right\}  \\
\ & =  \theta(k_2^{-1}) [\Xi(T)(g)] \theta(k_1^{-1})v.
\end{split}
\]
This shows that $\Xi(T) \in \widetilde{\mathcal{H}}(G,K,\theta)$.

Let us show that the map $\Xi \colon \End_G(\Ind_K^G V) \to \widetilde{\mathcal{H}}(G,K,\theta)$ is a right-inverse of $\xi$.
First observe that for all $f \in \Ind_K^G V$
\begin{equation}
\label{e:lambda-f}
\frac{1}{|K|}  \sum_{h \in G}[\lambda(h)f_{f(h)}] = f,
\end{equation}
which is an equivalent form of \eqref{definizione stellata2}.
Indeed, for all $g \in G$ we have
\[
\begin{split}
\frac{1}{|K|} \sum_{h \in G}[\lambda(h)f_{f(h)}](g) & = \frac{1}{|K|} \sum_{h \in G}[f_{f(h)}](h^{-1}g)\\
&  = \frac{1}{|K|} \sum_{h \in G}\left(\begin{cases} \theta(g^{-1}h)f(h) & \mbox{ if } h^{-1}g \in K\\
0 & \mbox{ otherwise}\end{cases}\right)\\
& = \frac{1}{|K|} \sum_{k \in K}\left(\begin{cases} \theta(k)f(gk) = f(g) & \mbox{ if } h^{-1}g = k^{-1} \in K\\
0 & \mbox{ otherwise}\end{cases}\right)\\
\mbox{(by \eqref{H31})} \ & = f(g).
\end{split}
\] 

Let then $T \in \End_G(\Ind_K^G V)$, $f \in \Ind_K^G$ and $g \in G$. 
We have
\[
\begin{split}
\left[\left(\xi \circ \Xi\right(T)f\right](g) & = \sum_{h \in G} [\Xi(T)(h^{-1}g)]f(h)\\
\mbox{(by \eqref{e;indtau2})} \ & = \frac{1}{|K|}  \sum_{h \in G} \left[Tf_{f(h)}\right] (h^{-1}g)\\
& = \frac{1}{|K|}  \sum_{h \in G} \left[\lambda(h)Tf_{f(h)}\right] (g)\\
\mbox{(since $T\in \End_G(\Ind_K^G V)$)} \ & = \frac{1}{|K|}  \sum_{h \in G} \left[T\lambda(h)f_{f(h)}\right] (g)\\
& =  \left[T\left(\frac{1}{|K|} \sum_{h \in G} \lambda(h)f_{f(h)}\right)\right] (g)\\
\mbox{(by \eqref{e:lambda-f})} \ & =  \left[Tf\right](g).
\end{split}
\]
This shows that $\xi(\Xi(T)) = T$, and surjectivity of $\xi$ follows.
\end{proof}

Recalling Proposition \ref{CorbasisL}, we immediately get the following:

\begin{corollary}[Mackey's intertwining number theorem]
\label{c:TUCSD}
Let $(\theta, V)$ be a $K$-representation. Then
\[
\diim \End_G(\Ind_K^G V) = \sum_{s \in {\mathcal S}_0} \diim \Hom_{K_s}(\Res^K_{K_s} \theta, \theta^s).
\]
\end{corollary}

\begin{remark} {\rm
Now we prove that $\Xi(T)$ given by (\ref{e;indtau2}) is the unique $\xi^{-1}$-image of $T \in \End(\Ind_K^GV)$ (note that this yields a second proof of injectivity of $\xi$). 
Since the functions $\lambda(h)f_v$, $h \in G$ and $v \in V$, span $\Ind_K^GV$, for $F_1\in \widetilde{\mathcal{H}}(G,K,\theta)$ we have that $\xi(F_1) = T$ if and only if
\begin{equation}\label{stella16}
\xi(F_1) \lambda(h)f_v = T \lambda(h)f_v \ \mbox{ for all }  h \in G \mbox{ and } v \in V.
\end{equation}
Let $\Xi(T)$ be given by \eqref{e;indtau2}. If (\ref{stella16}) holds, then
\[ 
\begin{split}
|K| \Xi(T)(g)v & = [T f_v](g) \\
& =[\lambda(g^{-1})Tf_v](1_G)\\
\mbox{(since $T\in \End_G(\Ind_K^G V)$)} \  
& =[T\lambda(g^{-1})f_v](1_G)\\
\mbox{(by \eqref{stella16})} \  & = [\xi(F_1)\lambda(g^{-1})f_v](1_G)\\
\mbox{(by \eqref{e;indtau4})} \   & = \sum_{h \in G} F_1(h^{-1}) f_v(gh) \\
\mbox{(by \eqref{definizione stellata1}) with $gh=k$)}  \ & =  \sum_{k \in K} F_1(k^{-1}g) \theta(k^{-1})v \\
\mbox{(by \eqref{e:F-def})} \  & = |K|F_1(g)v.
\end{split}
\]
\noindent
This shows that $F_1 = \Xi(T)$, that is, (\ref{e;indtau2}) defines the unique element in $\xi^{-1}(T)$.}
\end{remark}

We end this section by giving an explicit formula for the $\xi$-image of the map ${\mathcal L}_T \in
\widetilde{\mathcal{H}}(G,K,\theta)$, where $T \in \Hom_{K_s}(\Res^K_{K_s}\theta,\theta^s)$ and $s \in {\mathcal S}$, defined in \eqref{LTg}, and therefore, for the elements in the orthogonal basis of $\widetilde{\mathcal{H}}(G,K,\theta)$ (cf.\
Proposition \ref{CorbasisL}). 

Fix $s\in\mathcal{S}$ and choose a transversal $\mathcal{R}_s$ for the right-cosets of $K_{s^{-1}} \equiv K \cap s^{-1}Ks =
s^{-1}K_s s$ in  $K$ so that $K=\bigsqcup_{r\in\mathcal{R}_s}K_{s^{-1}}r$. 
Then, for all $f\in\Ind_K^GV$ and $g\in G$, we have
\[
\begin{split}
\left[\xi(\mathcal{L}_T)f\right](g) & = \sum_{h \in G}\mathcal{L}_T(h^{-1}g)f(h)\\
(z=h^{-1}g) \ & = \sum_{z \in G}\mathcal{L}_T(z)f(gz^{-1})\\
\mbox{(by \eqref{sumPhi} and \eqref{LTg})} \ & = \frac{1}{\lvert K_s\rvert}\sum_{k,k_1\in K}\theta(k^{-1})T\theta(k_1^{-1})f(gk^{-1}s^{-1}k_1^{-1})\\
\mbox{(by \eqref{H31})} \ & =  \frac{\lvert K\rvert}{\lvert K_s\rvert}\sum_{k\in K}\theta(k^{-1})Tf(gk^{-1}s^{-1})\\
(k=xr) \ & = \frac{\lvert K\rvert}{\lvert K_s\rvert}\sum_{x\in K_{s^{-1}}}\sum_{r\in\mathcal{R}_s}\theta(r^{-1})\theta(x^{-1})Tf(gr^{-1}x^{-1}s^{-1})\\
\mbox{(by \eqref{e:theta-s} and $sxs^{-1} \in K_s$)} \ & = \frac{\lvert K\rvert}{\lvert K_s\rvert}\sum_{x\in K_{s^{-1}}}\sum_{r\in\mathcal{R}_s}\theta(r^{-1})\theta^s(sx^{-1}s^{-1})Tf(gr^{-1}x^{-1}s^{-1})\\
(T\in\Hom_{K_s}(\text{Res}^K_{K_s}\theta,\theta^s))\ & = \frac{\lvert K\rvert}{\lvert K_s\rvert}\sum_{x\in K_{s^{-1}}}\sum_{r\in\mathcal{R}_s}\theta(r^{-1})T\theta(sx^{-1}s^{-1})f(gr^{-1}s^{-1}(sx^{-1}s^{-1}))\\
\mbox{(by \eqref{H31})} & = \lvert K\rvert\sum_{r\in \mathcal{R}_s}\theta(r^{-1})Tf(gr^{-1}s^{-1}).
\end{split}
\]

\subsection{The Hecke algebra revisited}
\label{s:HaR}
Let $(\theta,V)$ be a $K$-representation as in the previous section, but we now assume that 
$\theta$ is {\em irreducible}. 
We choose $v \in V$ with $\|v\| = 1$ and an orthonormal basis $\{v_j:j=1,2,\dotsc,d_\theta\}$ in $V$ with $v_1 = v$. We begin with a technical but quite useful lemma.

\begin{lemma}
For all $w\in V$ and $1 \leq i,j\leq d_\theta$, we have:
\begin{equation}\label{estar:ACT2}
\frac{d_\theta}{|K|}\sum_{k \in K}\langle\theta(k) v_i, v_j\rangle\theta(k^{-1})w = \langle w, v_j\rangle v_i.
\end{equation}
\end{lemma}
\begin{proof}
Let $w = \sum_{r = 1}^{d_\theta}\alpha_rv_r$. Then, for all $1 \leq s \leq d_\theta$, we have 
\[
\begin{split}
\left\langle \frac{d_\theta}{|K|}\sum_{k \in K}\langle\theta(k) v_i, v_j\rangle\theta(k^{-1})w, v_s\right\rangle  & = \sum_{r = 1}^{d_\theta}\alpha_r\frac{d_\theta}{|K|}\sum_{k \in K}\langle\theta(k)v_i, v_j\rangle\langle v_r, \theta(k)v_s\rangle\\
(\mbox{by \eqref{ORT}}) & = \sum_{r = 1}^{d_\theta}\alpha_r\delta_{i,s}\delta_{j,r}\\
& = \langle w , v_j\rangle\delta_{i,s}
\end{split}
\]
so that  the left hand side of \eqref{estar:ACT2} is equal to 
\[
\sum_{s=1}^{d_\theta} \langle w, v_j\rangle \delta_{i,s}v_s = \langle w,v_j\rangle v_i.
\]
\end{proof}

In the sequel, we shall use several times the following particular case (obtained from  \eqref{estar:ACT2} after takig $w = v_j$):
\begin{equation}\label{projvivj}
v_i=\frac{d_\theta}{\lvert K\rvert}\sum_{k\in K}\langle \theta(k)v_i, v_j\rangle_V\theta(k^{-1})v_j.
\end{equation}

Define $\psi \in L(K) = \{f \colon K \to \CC\}$  by setting
\begin{equation}\label{defpsi}
\psi(k) = \frac{d_\theta}{|K|}\langle v , \theta(k) v\rangle_V
\end{equation}
for all $k \in K$. From now on, we identify $L(K)$ as the subalgebra of $L(G)$ consisting of all functions
supported on $K$. Thus, we regard $\psi$ also as an element in $L(G)$. 
We then define the convolution operator $P\colon L(G) \to L(G)$ by setting
\[
Pf = f*\psi
\]
for all $f \in L(G)$ (in fact, $P = T_\psi$, cf.\ \eqref{convoper}). We also define the operator 
\[
T_v \colon \Ind_K^GV \to  L(G)
\]
by setting
\begin{equation}\label{defTv}
[T_vf](g) = \sqrt{d_\theta/|K|}\langle f(g), v\rangle_V
\end{equation}
for all $f \in \Ind_K^GV$ and $g \in G$, and denote its range by
\[
\mathcal{I}(G,K,\psi) = T_v\left(\Ind_K^GV\right) \subseteq L(G).
\]
Finally, we define a map
\[
S_v \colon \widetilde{\mathcal{H}}(G,K, \theta)\longrightarrow L(G)
\]
by setting
\[
[S_vF](g)=d_\theta\langle F(g)v,v\rangle_V
\]
for all $F\in \widetilde{\mathcal{H}}(G,K, \theta)$ and $g\in G$. 
The first and the second statement in the following theorem are taken from \cite{st5}. We reproduce here the proofs for the sake of completeness.

\begin{theorem}\label{pD2}
\begin{enumerate}[{\rm (1)}]
\item
The operator $T_v$ belongs to $\Hom_G(\Ind_K^GV, L(G))$ and it is an isometry; in particular, $\mathcal{I}(G,K,\psi)$
is a $\lambda_G$-invariant subspace of $L(G)$, which is $G$-isomorphic to $\Ind_K^GV$.
\item
The function $\psi$ satisfies the identities
\begin{equation}\label{idenpsi}
\psi * \psi = \psi \ \mbox{ and } \ \psi^*= \psi;
\end{equation}
moreover, $P$ is the orthogonal projection of $L(G)$ onto $\mathcal{I}(G,K,\psi)$. In other words,
\begin{equation}\label{charactTvIndV}
\mathcal{I}(G,K,\psi)=\{f*\psi:f\in L(G)\}\equiv\{f\in L(G):f*\psi=f\}.
\end{equation}
\item\label{Svisom}
Define
\[
\mathcal{H}(G,K, \psi) = \{\psi*f*\psi: f \in L(G)\} \equiv\{f \in L(G): f =
\psi*f*\psi\}.
\]
Then $\mathcal{H}(G,K, \psi)$ is an involutive subalgebra of $\mathcal{I}(G,K,\psi) \subseteq L(G)$ and $S_v$ yields a $*$-antiisomorphism from $\widetilde{\mathcal{H}}(G,K, \theta)$ onto $\mathcal{H}(G,K, \psi)$.
Every $f \in \mathcal{H}(G,K, \psi)$ is supported in $\bigsqcup_{s \in \mathcal{S}_0} KsK$ $($cf.\ \eqref{e:s-0}$)$.
Moreover, $\frac{1}{\sqrt{d_\theta}} S_v$ is an isometry and we have
\begin{equation}\label{TvSvxi}
T_v\left[\xi(F)f\right]=(T_vf)*(S_vF)
\end{equation}
for all $f\in\Ind_K^G V$ and $F\in \widetilde{\mathcal{H}}(G,K,\theta)$. 
The inverse $S_v^{-1}$ of $S_v$ is given by
\begin{equation}
\label{e:formula-per-sv-inv}
\langle [S_v^{-1}f](g)v_i,v_j\rangle=\frac{d_\theta}{\lvert K\rvert^2}\sum_{k_1,k_2\in K}f(k_1^{-1}gk_2)\overline{\langle \theta(k_1)v_1,v_i\rangle}\langle \theta(k_2)v_1,v_j\rangle
\end{equation}
for all $f\in\mathcal{H}(G,K, \psi)$, $g\in G$, and $i,j=1,2,\dotsc,d_\theta$.
\item\label{TfHecke}
The map
\begin{equation}
\label{e:TfHecke}
\begin{array}{ccc}
\mathcal{H}(G,K,\psi)&\longrightarrow  & \End_G\left(\mathcal{I}(G,K,\psi)\right)\\
f&\longmapsto & T_f|_{\mathcal{I}(G,K,\psi)}\\
\end{array}
\end{equation}
is a $\ast$-antiisomorphism of algebras and $\ker T_f$ contains $\left\{\mathcal{I}(G,K,\psi)\right\}^\bot$ for all $f\in\mathcal{H}(G,K,\psi)$.  Moreover,
\begin{equation}
\label{e:abilitazioniTF}
T_f\vert_{\mathcal{I}(G,K,\psi)} = T_v \circ \xi\left(S_v^{-1} f\right) \circ {T_v}^{-1}
\end{equation}
and 
\begin{equation}
\label{e:tutte-isometrie}
\langle T_{f_1},T_{f_2}\rangle_{{\tiny \End}_G\left(\mathcal{I}(G,K,\psi)\right)}=\frac{\lvert G\rvert}{\dim \mathcal{I}(G,K,\psi)}\langle f_1,f_2\rangle_{L(G)}
\end{equation}
for all $f,f_1,f_2\in\mathcal{H}(G,K,\psi)$.
\end{enumerate}
\end{theorem}

Before starting the proof, for the convenience of the reader, in diagrams
\[
\Ind^G_K V \overset{T_v}{\longrightarrow} \mathcal{I}(G,K,\psi)
\]
and

\[
\xymatrix{
&\mathcal{H}(G,K,\psi)\ar[r]^{* \ \ \ \ }\ar[dd]^{\xi \circ {S_v}^{-1}}]&\End_G(\mathcal{I}(G,K,\psi))\\
\widetilde{\mathcal{H}}(G,K,\theta)\ar[ur]_{S_v}\ar[dr]_{\xi} & & \\
&\End_G(\Ind_K^G V) &}
\]

\noindent
where $\overset{*}{\to}$ denotes the map defined in \eqref{e:TfHecke} and \eqref{e:abilitazioniTF},  we summarize the various maps involved in the theorem.

\begin{proof} 
(1) It is immediate to check that 
\[
T_v\lambda(g) f = \lambda_G(g)T_v f
\]
for all $g \in G$ and $f \in \Ind^G_KV$, showing that $T_v \in \Hom_G(\Ind^G_KV,L(G))$.
We now show that  $T_v$ is an isometry by using the basis in \eqref{orthbasisind} (recall that
$\mathcal{T}$ is a transversal for the left-cosets of $K$ in $G$). 
For $t_1,t_2 \in \mathcal{T}$ and $i,j = 1, 2,\ldots, d_\theta$, we have
\[
\begin{split}
 \langle T_v\lambda(t_1)f_{v_i}, T_v\lambda(t_2)f_{v_j}\rangle _{L(G)} & = 
\frac{d_\theta}{|K|}\sum_{g \in G}\langle f_{v_i}(t_1^{-1}g),
v\rangle\overline{\langle f_{v_j}(t_2^{-1}g), v\rangle}\\
\mbox{(by \eqref{definizione stellata1})} \ \ \ \  & =
\frac{d_\theta}{|K|}\delta_{t_1, t_2}\sum_{k  \in K}\langle \theta(k^{-1}) v_i,
v_1\rangle \overline{\langle \theta(k^{-1})v_j, v_1\rangle}\\
\mbox{(by \eqref{ORT})} \ \ \ \  & = \delta_{t_1, t_2}\delta_{i,j}\\
& =  \langle \lambda(t_1)f_{v_i}, \lambda(t_2)f_{v_j}\rangle _{L(G)}.
\end{split}
\]

(2) The first identity in \eqref{idenpsi} follows from the convolution properties of matrix coefficients (\eqref{CON} with $\sigma = \rho = \theta$ and $i=j=h=k=1$) and ensures that $P$ is an idempotent; the second identity in \eqref{idenpsi} follows from the fact that $\psi$ is (the conjugate of) a diagonal matrix coefficient, and \eqref{e:conv-*} ensures that $P$ is self-adjoint. As a consequence, $P=P^2=P^*$ is an orthogonal projection. Moreover, for all $f \in \Ind_K^GV$ and $g \in G$ we have 
\[
\begin{split}
[(T_vf)*\psi](g) & = \left(\frac{d_\theta}{|K|}\right)^{3/2}\sum_{k \in K} \langle
f(gk^{-1}), v\rangle \langle v, \theta(k) v\rangle\\
\mbox{(by \eqref{H31})} \ \ & = \left(\frac{d_\theta}{|K|}\right)^{3/2}\sum_{k \in K} \langle
\theta(k) f(g), v\rangle \langle v, \theta(k) v\rangle\\
& = \left(\frac{d_\theta}{|K|}\right)^{3/2}\left\langle f(g), \sum_{k \in K} \langle 
\theta(k)v,v \rangle \theta(k^{-1})v \right\rangle\\
\mbox{(by \eqref{projvivj})} \ \ & = \sqrt{\frac{d_\theta}{|K|}}\left\langle f(g), v \right\rangle\\
& = (T_vf)(g),
\end{split}
\]
that is, $PT_v f = T_v f$. We deduce that the range $\ran P$ of $P$ contains $T_v(\Ind_K^GV) = \mathcal{I}(G,K,\psi)$. Let now show that, in fact, the range of $P$ is contained in (and therefore equals) $\mathcal{I}(G,K,\psi)$. 
Indeed, for all $\phi \in L(G)$ and $g\in G$ we have
\[
\begin{split}
[P\phi](g) & =  \sum_{k \in K}\phi(gk)\psi(k^{-1}) \\
& = \frac{d_\theta}{|K|}\left\langle\sum_{k\in K}\phi(gk)\theta(k)v, v\right\rangle\\
& = [T_vf](g)
\end{split}
\]
where $f \colon G \to V$ is defined by $f(g) = \sqrt{\frac{d_\theta}{|K|}} \sum_{k \in K}\phi(gk)\theta(k)v$. 
Since $f$ belongs to $\Ind_K^G V$, as one immediately checks, we conclude that $\ran P = T_v(\Ind_K^GV)$.

(3) It is easy to check that $\mathcal{H}(G,K, \psi) \subseteq \mathcal{I}(G,K, \psi)$ and
 that $\mathcal{H}(G,K, \psi)$ is closed under the convolution product and the involution 
$f \mapsto f^*$. As a consequence, $\mathcal{H}(G,K, \psi)$ is a $*$-subalgebra of $L(G)$ contained in $\mathcal{I}(G,K, \psi)$. 

Now we prove the statements for $S_v$. 
Let $F,F_1,F_2\in\widetilde{\mathcal{H}}(G,K, \theta)$, $f\in\Ind_K^G V$ and $g\in G$. 
We then have
\[
\begin{split}
[\psi*(S_vF)*\psi](g)& = \frac{d_\theta^3}{\lvert K\rvert^2}\sum_{k_1,k_2\in K}\langle v_1,\theta(k_1)v_1\rangle\langle F(k_1^{-1}gk_2^{-1})v_1,v_1\rangle\langle v_1,\theta(k_2)v_1\rangle\\
\mbox{(by \eqref{e:F-def})} \ \ & = \frac{d_\theta^3}{\lvert K\rvert^2}\left\langle F(g)\sum_{k_1\in K}\langle \theta(k_1^{-1})v_1,v_1\rangle\theta(k_1)v_1, \sum_{k_2\in K}\langle \theta(k_2)v_1,v_1\rangle \theta(k_2^{-1})v_1\right\rangle\\
\mbox{(by \eqref{projvivj})} \ \ & = (S_vF)(g)
\end{split}
\]
and therefore $S_v\left[\widetilde{\mathcal{H}}(G,K, \theta)\right]\subseteq \mathcal{H}(G,K, \psi)$. 
Moreover
\[
\begin{split}
\langle S_v(F_1),S_v(F_2)\rangle_{L(G)}& = d_\theta^2\sum_{g\in G}\langle F_1(g)v,v\rangle\overline{\langle F_2(g)v,v\rangle}\\
\mbox{(by \eqref{projvivj})} \ \ \ \ \ \ \ & = \frac{d_\theta^2}{\lvert K\rvert^2}\sum_{g\in G}\sum_{i,j=1}^{d_\theta}\sum_{k_1,k_2\in K}\langle F_1(g)\theta(k_1^{-1})v_i,\theta(k_2^{-1})v_j\rangle\langle\theta(k_1)v_1,v_i\rangle\times \\
&\qquad\times\overline{\langle\theta(k_2)v_1,v_j\rangle}\overline{\langle F_2(g)v_1,v_1\rangle}\\
(h=k_1gk_2^{-1}) \ \ \ \ \ \ \ & = \frac{d_\theta^2}{\lvert K\rvert^2}\sum_{h\in G}\sum_{i,j=1}^{d_\theta}\langle F_1(h)v_i,v_j\rangle \times \\
\times&\overline{\left\langle F_2(h)\sum_{k_1\in K}\langle\theta(k_1^{-1})v_i,v_1\rangle\theta(k_1)v_1,\sum_{k_2\in K}\langle\theta(k_2^{-1})v_j,v_1\rangle\theta(k_2)v_1\right\rangle}\\
\mbox{(by \eqref{projvivj})} \ \ \ \ \ \ \ & = \sum_{h\in G}\sum_{i,j=1}^{d_\theta}\langle F_1(h)v_i,v_j\rangle\overline{\langle F_2(h)v_i,v_j\rangle}\\
& =  \sum_{i=1}^{d_\theta}\sum_{h\in G}\langle F_2(h)^*F_1(h)v_i,v_i\rangle\\
& = d_\theta\langle F_1,F_2\rangle_{\widetilde{\mathcal{H}}(G,K, \theta)}
\end{split}
\]
and this proves that $\frac{1}{\sqrt{d_\theta}} S_v$ is an isometry and, in particular, that $S_v$ is injective. 

Given $f\in L(G)$, we define $F \colon G\longrightarrow \End(V)$ by setting
\[
\langle F(g)v_i,v_j\rangle=\frac{d_\theta}{\lvert K\rvert^2}\sum_{k_1,k_2\in K}f(k_1^{-1}gk_2)\overline{\langle \theta(k_1)v_1,v_i\rangle}\langle \theta(k_2)v_1,v_j\rangle
\]
for all $g\in G$ and $i,j=1,2,\dotsc,d_\theta$. Then it is immediate to check that $F\in \widetilde{\mathcal{H}}(G,K, \theta)$. Moreover, if in addition $f\in \mathcal{H}(G,K, \psi)$, we then have, for all $g \in G$,
\[
d_\theta\langle F(g)v,v\rangle=\frac{d_\theta^2}{\lvert K\rvert^2}\sum_{k_1,k_2\in K}\langle v,\theta(k_1)v\rangle f(k_1^{-1}gk_2)\langle v,\theta(k_2^{-1})v\rangle= \left(\psi*f*\psi\right)(g)=f(g),
\]
that is, $S_vF=f$. This shows that $S_v$ is also surjective (in \eqref{TfHecke} another proof of surjectivity is given).
Incidentally, the above also provides the expression \eqref{e:formula-per-sv-inv} for $S_v^{-1}$. 

It is easy to check that $S_v$ preserves the involutions: indeed,
\[
(S_vF^*)(g)=d_\theta\langle F(g^{-1})^*v,v\rangle=d_\theta\overline{\langle F(g^{-1})v,v\rangle}=(S_vF)^*(g),
\]
where the first equality follows from \eqref{e:star-pag8+}.

We need a little more effort to prove that $S_v$ is an antiisomorphism:
\[
\begin{split}
[S_v(F_1*F_2)](g)& = d_\theta\sum_{h\in G}\langle F_1(h^{-1}g)F_2(h)v,v\rangle\\
& = d_\theta\sum_{h\in G}\sum_{i=1}^{d_\theta}\langle F_2(h)v_1,v_i\rangle \langle F_1(h^{-1}g)v_i,v_1\rangle\\
\mbox{(by \eqref{projvivj})} \ \ \ \ \ \ \ & = \frac{d_\theta^2}{\lvert K\rvert }\sum_{h\in G}\sum_{i=1}^{d_\theta}\sum_{k\in K}\langle F_2(h)v_1,\theta(k^{-1})v_1\rangle \langle v_1,\theta(k)v_i\rangle \langle F_1(h^{-1}g)v_i,v_1\rangle\\
(h=rk) \ \ \ \ \ \ \ & = \frac{d_\theta^2}{\lvert K\rvert}\sum_{r\in G}\sum_{i=1}^{d_\theta}\langle F_2(r)v_1,v_1\rangle \left\langle F_1(r^{-1}g)\sum_{k\in K}\langle \theta(k^{-1})v_1,v_i\rangle \theta(k)v_i,v_1\right\rangle\\
\mbox{(by \eqref{projvivj})} \ \ \ \ \ \ \ & = \left[S_v(F_2)*S_v(F_1)\right](g).
\end{split}
\]

This ends the proof that $S_v$ is a $*$-antiisomorphism.

The fact that every $f \in \mathcal{H}(G,K, \psi)$ is supported in $\bigsqcup_{s \in \mathcal{S}_0} KsK$ follows
immediately from the (anti-) isomorphsm between $\widetilde{\mathcal{H}}(G,K, \theta)$ and ${\mathcal{H}}(G,K, \psi)$
and Proposition \ref{CorbasisL}.

We now prove \eqref{TvSvxi}:
\[
\begin{split}
[T_v\xi(F)f](g)& = \sqrt{\frac{d_\theta}{\lvert K\rvert}}\sum_{h\in G}\langle F(h^{-1}g)f(h),v \rangle\\
& = \sqrt{\frac{d_\theta}{\lvert K\rvert}}\sum_{i=1}^{d_\theta}\sum_{h\in G}\langle f(h),v_i \rangle \langle F(h^{-1}g)v_i,v_1\rangle\\
\mbox{(by \eqref{projvivj})} \ \ \ \ \ \ \ & = \left(\frac{d_\theta}{\lvert K\rvert}\right)^{3/2}\sum_{i=1}^{d_\theta}\sum_{h\in G}\sum_{k\in K}\langle f(h),\theta(k^{-1})v_1 \rangle \langle v_1,\theta(k)v_i\rangle \langle F(h^{-1}g)v_i,v_1\rangle\\
(h=rk) \ \ \ \ \ \ \ & = \left(\frac{d_\theta}{\lvert K\rvert}\right)^{3/2}\sum_{i=1}^{d_\theta}\sum_{r\in G}\langle f(r),v \rangle  \left\langle F(r^{-1}g)\sum_{k\in K} \langle \theta(k^{-1})v_1,v_i\rangle\theta(k)v_i,v_1\right\rangle\\
\mbox{(by \eqref{projvivj})} \ \ \ \ \ \ \ & =   (T_vf)*(S_vF)(g).
\end{split}
\]

(4) Let $f\in\mathcal{H}(G,K, \psi)$ and $f_1 \in L(G)$. Then,  if $f_1=Pf_1$, that is $f_1=f_1*\psi$, we have $P(T_{f}f_1) = P(f_1*f)=((f_1*\psi)*f)*\psi=f_1*f = T_{f}f_1$. In other words $P T_{f} P = T_{f} P$, equivalently
$PT_{f}\vert_{\mathcal{I}(G,K,\psi)} = T_{f}P\vert_{\mathcal{I}(G,K,\psi)}$.
On the other hand, if $f_1\in\ker P$ we have $T_{f}f_1=f_1*f=f_1*(\psi*f*\psi)
=(f_1*\psi)*f*\psi = (Pf_1)*f*\psi =0$, that is, $T_{f}\ker P \subseteq \ker P$.

This shows that the convolution operator $T_{f}$ intertwines $\mathcal{I}(G,K,\psi)$ with itself and annihilates the orthogonal complement of $\mathcal{I}(G,K,\psi)$. This can be used to give a second proof of the surjectivity of $S_v$. Indeed, the fundamental relation \eqref{TvSvxi} and the injectivity of $S_v$ ensure that the antiisomorphism $S_v$ identifies the commutant of $\mathcal{I}(G,K,\psi)$ with the convolution algebra $\mathcal{H}(G,K, \psi)$, so that $\dim {\mathcal{H}(G,K,\psi)} = \dim \End_G(\Ind_K^G V) \equiv 
\dim {\widetilde{\mathcal{H}}(G,K,\theta)}$. Moreover, keeping in mind \eqref{antiTf}, if $\bar{f} \in \Ind^G_K V$ we have
\[
\begin{split}
\left[T_v \circ \xi\left(S_v^{-1} f\right)\right](\bar{f}) & =  T_v\left(\xi\left(S_v^{-1} f\right)\bar{f}\right)\\
\mbox{(by \eqref{TvSvxi})} \ \ &  = T_v\bar{f} * f\\ 
& = \left[T_f \circ T_v\right](\bar{f})
\end{split}
\]
which gives \eqref{e:abilitazioniTF}.

Finally, \eqref{e:tutte-isometrie} follows from \eqref{Tf1Tf2}.
\end{proof}

We end this section with a useful computational rule.

\begin{lemma}
Let $f_1\in\mathcal{H}(G,K, \psi)$ and $f_2\in L(G)$. Then
\begin{equation}\label{f1psif2}
[f_1*\psi*f_2*\psi](1_G)=[f_1*f_2](1_G).
\end{equation}
\end{lemma}
\begin{proof}
We have
\[
\begin{split}
[f_1*\psi*f_2*\psi](1_G) & = \sum_{h\in G}[f_1*\psi*f_2](h)\psi(h^{-1})\\
&  = \sum_{h\in G}\psi(h^{-1})[f_1*\psi*f_2](h)\\
& = [\psi*f_1*\psi*f_2](1_G)\\
& = [f_1*f_2](1_G).
\end{split}
\]
\end{proof}

\begin{remark}\label{follremark}
{\rm
In the terminology of \cite{CURFOS, CR2, st4}, the function $\psi$ is a \emph{primitive idempotent} in $L(K)$ and
$\mathcal{I}(G,K, \psi) \cap L(K) = \{f \in L(K): f *\psi = f\}$ is the \emph{minimal left ideal} in $L(K)$ generated
by $\psi$ (for, $\mathcal{I}(G,K,\psi)$ is generated by $\psi$ as a left ideal in $L(G)$). Moreover, $\mathcal{H}(G,K, \psi)$
is the {\it Hecke algebra} associated with $G,K$ and the primitive idempotent $\psi$.
}
\end{remark}

\begin{remark}
For an element $\mathcal{L}_T$ of $\widetilde{\mathcal{H}}(G,K,\theta)$ as in \eqref{LTg} we have:

\[
(S_v\mathcal{L}_T)(g)=\left\{\begin{array}{ll}
d_\theta\langle T\theta(k_1^{-1})v,\theta(k_2)v\rangle&\text{ if }g=k_1sk_2\in KsK\\
0&\text{ otherwise}.
\end{array}\right.
\]
\end{remark}

\section{Multiplicity-free triples}
This section is devoted to the study of multiplicity-free triples and their associated spherical functions.
After the characterization of multiplicity-freens in terms of commutativity of the associated Hecke algebra
(Theorem \ref{GKtwisted}), in Section \ref{ss:gBG} we present a generalization of a criterion due to Bump and 
Ginzburg \cite{BumpGinz}.
In the subsequent section, we develop the intrinsic part of the theory of spherical functions, that is, we determine all their properties (e.g.\ the Functional Equation in Theorem  \ref{t:functional-eq}) that may be deduced without their explicit form as matrix coefficients, as examined in Section \ref{section 7}.
In Section \ref{s:case-dim-1} we consider the case when the $K$-representation $(V,\theta)$ is one-dimensional. This case is treated
in full details in Chapter 13 of our monograph \cite{book4}.
The results presented here will provide some of the necessary tools to present our main new examples of multiplicity-free triples 
(see  Section \ref{s:HAMFT1} and Section \ref{s:IItrippa}). 
Finally, in Section \ref{s:FS-mft} we present a Frobenius-Schur type theorem for multiplicity-free triples: it provides a criterion for determining the type of a given irreducible spherical representation, namely, being real, quaternionic, or complex.

Let $G$ and $K \leq G$ be finite groups. Let $\theta \in \widehat{K}$.

\begin{theorem}
\label{GKtwisted}
The following conditions are equivalent:
\begin{enumerate}[{\rm (a)}]
\item The induced representation $\Ind^G_K \theta$ is multiplicity-free;
\item the Hecke algebra $\mathcal{H}(G,K,\psi)$ is commutative.
\end{enumerate}
Moreover, suppose that one of the above condition is satisfied and that \eqref{decIndmW}
is an explicit multiplicity-free decomposition of $\Ind_K^GV$ into irreducibles.
For every $f\in\mathcal{H}(G,K,\psi)$ consider the restriction $T_f' = T_f\vert_{\mathcal{I}(G,K,\psi)}$
to the invariant subspace $\mathcal{I}(G,K,\psi)$. Then 
\begin{equation}\label{TvTsigma}
\mathcal{I}(G,K,\psi)=\bigoplus_{\sigma\in J}T_v\left[T_\sigma W_\sigma\right]
\end{equation}
is the decomposition of $\mathcal{I}(G,K,\psi)$ into $T_f'$-eigenspaces (cf. \eqref{TfHecke} in Theorem \ref{pD2}). 
Also, if $\lambda_{\sigma,f}$ is the eigenvalue of $T_f'$ associated with the subspace $T_v\left[T_\sigma W_\sigma\right]$, then the map 
\begin{equation}
\label{e:spazzzi}
f\mapsto (\lambda_{\sigma,f})_{\sigma \in J}
\end{equation} 
is an algebra isomorphism from $\mathcal{H}(G,K,\psi)$ onto $\mathbb{C}^J$.

Finally, with the notation of Corollary \ref{c:TUCSD}, we have 
\begin{equation}
\label{e:c}
\dim \mathcal{H}(G,K,\psi)=\lvert J\rvert=\sum_{s\in \mathcal{S}}\dim\Hom_{K_s}(\Res^K_{K_s}\theta,\theta^s).
\end{equation}
\end{theorem}
\begin{proof}
The Hecke algebra $\mathcal{H}(G,K,\psi)$ is antiisomorphic to $\End_G(\Ind^G_K V)$ via the map
$\xi \circ S_v^{-1}$ by virtue of Theorem \ref{isomtildeH} and Theorem \ref{pD2}.(3). Then the equivalence (a) $\Leftrightarrow$ (b) follows from the isomorphism \eqref{isomHomIndMat}.

By Theorem \ref{pD2}.\eqref{TfHecke}, $T_f'$ intertwines each irreducible representation $T_v\left[T_\sigma W_\sigma\right]$ with itself and therefore, by Schur's lemma, $T_f\vert_{T_v\left[T_\sigma W_\sigma\right]} =
T_f'\vert_{T_v\left[T_\sigma W_\sigma\right]}$ is a multiple of the identity on this space.

Suppose that $f_1\in \mathcal{I}(G,K,\psi)$ and $f_1=\sum_{\sigma\in J}f_\sigma$ with 
$f_\sigma\in T_v\left[T_\sigma W_\sigma\right]$. Then $T_f(f_1)=\sum_{\sigma\in J}\lambda_{\sigma,f}f_\sigma$ proving the required properties for the map \eqref{e:spazzzi} (cf.\ \eqref{e:conv-anti}). In particular, $\dim \mathcal{H}(G,K,\psi) = \dim \mathbb{C}^J=\lvert J\rvert$
(see also Proposition \ref{p:equiv-MF-pre}). Since $\dim \mathcal{H}(G,K,\psi) = \dim \widetilde{\mathcal{H}}(G,K,\psi)$, the second equality in \eqref{e:c} follows from Corollary \ref{c:TUCSD}.
\end{proof}

\begin{definition}
{\rm If one of the equivalent conditions (a) and (b) in Theorem \ref{GKtwisted} is satisfied, we say that
$(G,K,\theta)$ is a \emph{multiplicity-free triple}.}
\end{definition}

Observe that if $\theta = \iota_K$ is the trivial $K$-representation, then $(G,K,\iota_K)$ is a multiplicity-free triple
if and only if $(G,K)$ is a Gelfand pair (cf.\ Definition \ref{d:GP} and Theorem \ref{t:GP}).
More generally, for $\dim \theta = 1$ we also refer to \cite{Macdonald}, \cite{Stembridge}, and \cite{Mizukawa1, Mizukawa2}.

The following proposition provides a useful tool for proving multiplicity-freeness for certain triples. 
The map $T \mapsto T^\sharp$ therein is a fixed antiautomorphism of the algebra $\End(V)$ (for instance, adjunction, or transposition).
\begin{proposition}
\label{e;smft} 
Suppose there exists an antiautomorphism $\tau$ of $G$ such that $f(\tau(g))=f(g)$ for all $f \in {\mathcal H}(G,K,\psi)$ and $g \in G$. Then ${\mathcal H}(G,K,\psi)$ is commutative. Similarly, if $F(\tau(g)) = F(g)^\sharp$ for all $F \in \widetilde{{\mathcal H}}(G,K,\theta)$ and $g \in G$, then $\widetilde{{\mathcal H}}(G,K,\theta)$ is commutative.
\end{proposition}
\begin{proof} Let $f_1, f_2 \in {\mathcal H}(G,K,\psi)$ and $g \in G$. We have
\[
\begin{split} [f_1 * f_2] (g) & = \sum_{h \in G} f_1(gh) f_2(h^{-1})\\
(f \circ \tau = f) \ \ \ \ & = \sum_{h \in G} f_1(\tau(gh)) f_2(\tau(h^{-1}))\\
\ & = \sum_{h \in G} f_1(\tau(h)\tau(g)) f_2(\tau(h^{-1}))\\
\ & = \sum_{h \in G} f_2(\tau(h^{-1}))f_1(\tau(h)\tau(g))\\
\mbox{(setting $t = \tau(h)$) \ \ }
\ &  =\sum_{t \in G} f_2(t^{-1}) f_1(t\tau(g))\\
\ & = [f_2 * f_1] (\tau(g)) = [f_2 * f_1] (g),
\end{split}
\] 
showing that ${\mathcal H}(G,K,\psi)$ is commutative. Similarly, for $F_1, F_2 \in \widetilde{{\mathcal H}}(G,K,\theta)$ and $g \in G$, 
we have
\[
\begin{split} [F_1 * F_2] (g) & = \sum_{h \in G} F_1(h^{-1}g) F_2(h)\\
(F \circ \tau = F^\sharp) \ \ \ \ & = \sum_{h \in G} F_1(\tau(g)\tau(h)^{-1})^\sharp F_2(\tau(h))^\sharp\\
\ & = \sum_{h \in G} [F_2(\tau(h))F_1(\tau(g)\tau(h^{-1}))]^\sharp\\
\mbox{(setting $\tau(g)\tau(h)^{-1} = t$) \ \ } 
\ & = \left[\sum_{t \in G} F_2(t^{-1}\tau(g))F_1(t)\right]^\sharp\\
\ & = [F_2 * F_1] (\tau(g))^\sharp = [F_2 * F_1] (g),
\end{split}
\] 
showing that $\widetilde{{\mathcal H}}(G,K,\theta)$ is commutative as well.
\end{proof}

If the first condition in the above proposition is satisfied, we say that the multiplicity-free triple $(G,K,\theta)$ is {\it weakly symmetric}. When $\tau(g) = g^{-1}$ we say that $(G,K,\theta)$ is {\it symmetric}.

\begin{example}[Weakly symmetric Gelfand pairs] 
{\rm  Suppose that $\theta = \iota_K$ is the trivial representation of $K$. 
Then ${\mathcal H}(G,K,\psi) = \ ^K\!L(G)^K$ and condition $f(g^{-1})=f(g)$
(resp. $f(\tau(g)) = f(g)$), for all $f \in{\mathcal H}(G,K,\psi)$ and $g \in G$, is equivalent to
$g^{-1} \in KgK$ (resp. $\tau(g) \in KgK$), for all $g \in G$. If this holds, one says that
$(G,K)$ is a \emph{symmetric}  (resp. \emph{weakly symmetric}) Gelfand pair 
(cf.\ \cite[Example 4.3.2 and Exercise 4.3.3]{book}).}
\end{example}

\subsection{A generalized Bump-Ginzburg criterion}
\label{ss:gBG}
We now present a generalization of a criterion due to Bump and Ginzburg \cite{BumpGinz} (see also \cite[Corollary 13.3.6.(a)]{book4}). 
We use the notation in \eqref{LTg} and we assume that $T \mapsto T^\sharp$ is a fixed antiautomorphism of the algebra $\End(V)$.

\begin{theorem}[A generalized Bump-Ginzburg criterion]
\label{t:Bump-Ginz2}
Let $G$ be a finite group, $K \leq G$ a subgroup, and $(\theta,V) \in \widehat{K}$. 
Suppose that there exists an antiautomorphism $\tau \colon G \to G$ such that: 
\begin{itemize}
\item $K$ is $\tau$-invariant: $\tau(K) = K$,
\item $\theta(\tau(k)) = \theta(k)^\sharp$ for all $k \in K$,
\item for every $s \in {\mathcal S}_0$ (cf.\ \eqref{e:s-0}) there exist $k_1, k_2 \in K$ such that 
\begin{equation}
\label{estar:PTPT2}
\tau(s) = k_1 s k_2
\end{equation}
and 
\begin{equation}
\label{estar2:PTPT2}
\theta(k_2)^*T \theta(k_1)^* = T^\sharp
\end{equation}
for all $T \in \Hom_{K_s}(\Res^K_{K_s}\theta, \theta^s)$,
\item $(T^\sharp)^* = (T^*)^\sharp$ for all $T \in \End(V)$ \ (i.e., $\sharp$ and $*$ commute).
\end{itemize}
Then the triple $(G,K,\theta)$ is multiplicity-free.
\end{theorem}

\begin{proof} Let $s \in  {\mathcal S}_0$.
We show that ${\mathcal L}_T(\tau(g)) = {\mathcal L}_T(g)^\sharp$ for all $g \in G$ and $T \in \Hom_{K_s}(\Res^K_{K_s}\theta, \theta^s)$
(cf.\ \eqref{LTg}), so that we may apply the second condition in Proposition \ref{e;smft}. 
First of all, we have
\begin{equation}
\label{esquare:agg3}
{\mathcal L}_T(\tau(s)) = {\mathcal L}_T(k_1 s k_2) = \theta(k_2)^* T \theta(k_1)^* = T^\sharp = {\mathcal L}_T(s)^\sharp
\end{equation}
Let $g\in G$ and suppose that $g = h_1sh_2$ with $s \in \mathcal{S}_0$ and $h_1, h_2 \in K$. Then 
\[
\begin{split}
{\mathcal L}_T(\tau(g)) & = {\mathcal L}_T(\tau(h_2)\tau(s)\tau(h_1))\\
& = \theta(\tau(h_1))^*{\mathcal L}_T(\tau(s))\theta(\tau(h_2))^*\\
\mbox{(by \eqref{esquare:agg3})} \ & = (\theta(h_1)^\sharp)^*T^\sharp(\theta(h_2)^\sharp)^*\\
& = (\theta(h_1)^*)^\sharp T^\sharp(\theta(h_2)^*)^\sharp\\
& = [\theta(h_2)^*{\mathcal L}_T(s)\theta(h_1)^*]^\sharp\\
& = {\mathcal L}_T(g)^\sharp.
\end{split}
\]
\end{proof}

\begin{remark}{\rm
The Bump-Ginzburg criterion, that concerns the case $\dim \theta = 1$ so that $\End(V_\theta) \cong \CC$ is commutative, may be obtained by taking the operator $\sharp$ as the identity. As in the following we shall only make use of this version, rather than its generalization
(Theorem \ref{t:Bump-Ginz2}), for the convenience of the reader we state it as in its original form.

\begin{theorem}[Bump-Ginzburg criterion]
\label{t:Bump-Ginz}
Let $G$ be a finite group, $K \leq G$ a subgroup, and $\chi$ a one-dimensional $K$-representation. 
Suppose that there exists an antiautomorphism $\tau \colon G \to G$ such that: 
\begin{itemize}
\item $K$ is $\tau$-invariant: $\tau(K) = K$,
\item $\chi(\tau(k)) = \chi(k)$ for all $k \in K$,
\item for every $s \in {\mathcal S}_0$ (cf.\ \eqref{e:s-0}) there exist $k_1, k_2 \in K$ such that 
\begin{equation}
\label{e:bump1}
\tau(s) = k_1 s k_2
\end{equation}
and 
\begin{equation}
\label{e:bump2}
\chi(k_1)\chi(k_2) = 1.
\end{equation}
\end{itemize}
Then the triple $(G,K,\chi)$ is multiplicity-free.
\end{theorem}
}
\end{remark} 

\noindent
{\bf Open problem.}  In \cite[Section 4]{AM}, there is a quite deep analysis of the Mackey-Gelfand criterion for finite Gelfand pairs.  
It should be interesting to extend that analysis in the present setting. 
Also, the twisted Frobenius-Schur theorem (cf.\ \cite[Section 9]{AM}) deserves to be analyzed in the present setting 
(cf.\ Section \ref{s:FS-mft}).

\subsection{Spherical functions: intrinsic theory}
\label{s:sf:it}
In this section we develop the intrinsic part of the theory of spherical functions, that is, we determine all their properties that may be deduced without their explicit form as matrix coefficients. 

Suppose now that $(G,K,\theta)$ is a multiplicity-free triple.

\begin{definition}
\label{def:spher-f}
A function $\phi\in\mathcal{H}(G,K,\psi)$ is \emph{spherical} provided it satisfies the following conditions:
\begin{equation}
\label{phi1G}
\phi(1_G)=1
\end{equation}
and, for all $f\in\mathcal{H}(G,K,\psi)$, there exists $\lambda_{\phi,f}\in\mathbb{C}$ such that
\begin{equation}
\label{lambdaf}
\phi*f=\lambda_{\phi,f}\phi.
\end{equation}
We denote by ${\mathcal S}(G,K,\psi) \subseteq \mathcal{H}(G,K,\psi)$ the set of all spherical functions
associated with the multiplicity-free triple $(G,K,\theta)$.
\end{definition}

Condition \eqref{lambdaf} may be reformulated by saying that $\phi$ is an eigenvector for the convolution operator
$T_f$ for every $f\in\mathcal{H}(G,K,\psi)$. Moreover, combining \eqref{lambdaf} and \eqref{phi1G}, we get $\lambda_{\phi,f}=[\phi*f](1_G)$, and, taking into account the commutativity of $\mathcal{H}(G,K,\psi)$, we deduce the following equivalent reformulation of \eqref{lambdaf}:
\begin{equation}\label{lambdaf2}
\phi*f= f*\phi = \left([\phi*f](1_G)\right)\phi.
\end{equation}

We now give a basic characterization of the spherical functions that  makes use of the function $\psi \in L(G)$ defined in \eqref{defpsi}.

\begin{theorem}[Functional Equation]
\label{t:functional-eq}
A non-trivial function $\phi\in L(G)$ is spherical if and only if
\begin{equation}
\label{funcidensph}
\sum_{k\in K}\phi(gkh)\overline{\psi(k)}=\phi(g)\phi(h),\qquad\text{ for all }g,h\in G.
\end{equation}
\end{theorem}
\begin{proof}
Suppose that $0\neq \phi\in L(G)$ satisfies \eqref{funcidensph}. 
Choose $h\in G$ such that $\phi(h)\neq 0$. Then writing \eqref{funcidensph} in the form $\phi(g)=\frac{1}{\phi(h)}\sum_{k\in K}\phi(gkh)\overline{\psi(k)}$ we get
\[
\begin{split}
[\phi*\psi](g) & =\frac{1}{\phi(h)}\sum_{k,k_1\in K}\phi(gk_1kh)\overline{\psi(k)}\psi(k_1^{-1})\\
(k_1k=k_2)\  & = \frac{1}{\phi(h)}\sum_{k_2\in K}\phi(gk_2h)\overline{[\psi*\psi](k_2)}\\
(\text{by }\eqref{idenpsi})\ & = \frac{1}{\phi(h)}\sum_{k_2\in K}\phi(gk_2h)\overline{\psi(k_2)}\\
& = \phi(g)
\end{split}
\]
for all $g\in G$, that is, $\phi*\psi = \phi$.  Similarly one proves that $\psi*\phi=\phi$. 
Combining these two facts together we get $\psi*\phi*\psi=\phi$ yielding $\phi\in\mathcal{H}(G,K,\psi)$. 
Then setting $h=1_G$ in \eqref{funcidensph} we get
\[
\phi(g)\phi(1_G)=\sum_{k\in K}\phi(gk)\overline{\psi(k)}=[\phi*\psi](g)=\phi(g)
\]
for all $g\in G$, and therefore \eqref{phi1G} is satisfied. 
Finally, for $f\in\mathcal{H}(G,K,\psi)$ and $g\in G$ we have
\[
\begin{split}
[\phi*f](g)& = [\phi*f*\psi](g)\\
& = \sum_{h\in G}\sum_{k\in K}\phi(gkh)f(h^{-1})\overline{\psi(k)}\\
(\text{by }\eqref{funcidensph})\ & = \phi(g)\sum_{h\in G}\phi(h)f(h^{-1})\\
& = [\phi*f](1_G)\phi(g)
\end{split}
\]
and also \eqref{lambdaf2} is satisfied. This shows that $\phi$ is spherical.

Conversely, suppose that $\phi$ is spherical and set
\[
F_g(h)=\sum_{k\in K}\phi(gkh)\overline{\psi(k)},
\]
for all $h,g\in G$. If $f\in\mathcal{H}(G,K,\psi)$ and $g,g_1\in G$, then we have
\begin{equation}\label{Fgfg1}
\begin{split}
[F_g*f](g_1)& = \sum_{k\in K}\sum_{h\in G}\phi(gkg_1h)f(h^{-1})\overline{\psi(k)}\\
(\text{by }\eqref{lambdaf2})\ & = [\phi*f](1_G)\sum_{k\in K}\phi(gkg_1)\overline{\psi(k)}\\
& = [\phi*f](1_G)F_g(g_1).
\end{split}
\end{equation}
Let now show that the function $J_g \in L(G)$, defined by
\[
J_g(h)=\sum_{k\in K}f(hkg)\overline{\psi(k)}
\]
for all $h\in G$, is in $\mathcal{H}(G,K,\psi)$ for every $g \in G$. Indeed, for all $h \in G$,
\[
\begin{split}
[\psi*J_g*\psi](h)& = \sum_{k,k_1,k_2\in K}\psi(k_1)f(k_1^{-1}hk_2^{-1}kg)\psi(k_2)\overline{\psi(k)}\\
(k_3=k_2^{-1}k)\ & = \sum_{k,k_3\in K}[\psi*f](hk_3g)\psi(kk_3^{-1})\psi(k^{-1})\\
& = \sum_{k_3\in K}f(hk_3g)[\psi*\psi](k_3^{-1})\\
\mbox{(by \eqref{idenpsi})} \ \ & = \sum_{k_3\in K}f(hk_3g)\overline{\psi(k_3)}\\
& = J_g(h).
\end{split}
\]
Moreover,
\begin{equation}
\label{phiJg}
\begin{split}
[\phi*J_g](1_G) & = \sum_{h\in G}\phi(h^{-1})\sum_{k\in K}f(hkg)\overline{\psi(k)}\\
(hk=t)\ & = \sum_{t\in G}\left(\sum_{k\in K}\psi(k^{-1})\phi(kt^{-1})\right)f(tg)\\
& = \sum_{t\in G}\phi(t^{-1})f(tg)\\
& = [\phi*f](g).
\end{split}
\end{equation}
It follows that, for $g,g_1\in G$, 
\begin{equation}
\label{Fgfg12}
\begin{split}
[F_g*f](g_1)& = \sum_{h\in G}\sum_{k\in K}\phi(gkg_1h)\overline{\psi(k)}f(h^{-1})\\
(kg_1h=t) \ & = \sum_{t\in G}\phi(gt)\sum_{k\in K}\overline{\psi(k)}f(t^{-1}kg_1)\\
& = \phi*J_{g_1}(g)\\
(\text{by }\eqref{lambdaf2}) \ & = [\phi*J_{g_1}](1_G)\phi(g)\\
(\text{by }\eqref{phiJg})\ & = [\phi*f](g_1)\phi(g)\\
(\text{again by }\eqref{lambdaf2})\ & = [\phi*f](1_G)\phi(g_1)\phi(g).
\end{split}
\end{equation} 
From \eqref{Fgfg1} and \eqref{Fgfg12} we get
$[\phi*f](1_G)F_g(g_1)=[\phi*f](1_G)\phi(g_1)\phi(g)$ and taking $f\in\mathcal{H}(G,K,\psi)$ such that $[\phi*f](1_G)\neq 0$ (take, for instance, $f=\psi*\delta_{1_G}*\psi$) this yields $F_g(g_1)=\phi(g_1)\phi(g)$
which is nothing but \eqref{funcidensph}. 
\end{proof}

A linear functional $\Phi \colon {\mathcal{H}}(G,K,\psi)\to\mathbb{C}$ is said to be {\em multiplicative} provided that
\[
\Phi(f_1*f_2)=\Phi(f_1)\Phi(f_2)
\]
for all $f_1,f_2\in\mathcal{H}(G,K,\psi)$.

\begin{theorem}
Let $\phi$ be a spherical function and set
\begin{equation}\label{linfunct}
\Phi(f)= [f*\phi](1_G)
\end{equation}
for all $f\in\mathcal{H}(G,K,\psi)$. 
Then $\Phi$ is a linear multiplicative functional on $\mathcal{H}(G,K,\psi)$. 
Conversely, any non-trivial multiplicative linear functional on $\mathcal{H}(G,K,\psi)$ is of this form.
\end{theorem}
\begin{proof}
Let $f_1,f_2\in\mathcal{H}(G,K,\psi)$. Then, 
 we get:
\[
\begin{split}
\Phi(f_1 * f_2) & = [(f_1*f_2)*\phi](1_G)\\
& = [f_1*(f_2*\phi)](1_G) \\
\mbox{(by \eqref{lambdaf2})} \ & = \left[f_1*\left([f_2*\phi](1_G)\phi\right)\right](1_G)\\
& = [f_1*\phi](1_G)[f_2*\phi](1_G)\\
&  = \Phi(f_1) \Phi(f_2).
\end{split}
\]
Conversely, suppose that $\Phi$ is a non-trivial multiplicative linear functional on $\mathcal{H}(G,K,\psi)$. 
We can extend $\Phi$ to a linear functional $\overline{\Phi}$ on $L(G)$ by setting
$\overline{\Phi}(f) =  \Phi(\psi*f*\psi)$ for all $f \in L(G)$. By Riesz' theorem, we can find 
$\varphi \in L(G)$ such that
\[
\Phi(\psi*f*\psi) = \sum_{g \in G} f(g) \varphi(g^{-1}) = [f*\varphi](1_G),
\]
for all $f\in L(G)$. From \eqref{f1psif2} we deduce that if $f_1\in \mathcal{H}(G,K,\psi)$ then
\[
\Phi(f_1)=[f_1*\varphi](1_G)= [f_1*\psi*\varphi*\psi](1_G)
\]
and therefore the function $\varphi \in L(G)$ may be replaced by the function $\phi=\psi*\varphi*\psi\in\mathcal{H}(G,K,\psi)$. 
With this position, \eqref{f1psif2} also yields
\begin{equation}
\label{e:gaeta}
\Phi(\psi*f*\psi)=[\phi*\psi*f*\psi](1_G)=[\phi*f](1_G)
\end{equation}
for all $f\in L(G)$. 

We are only left to show that $\phi \in {\mathcal S}(G,K,\psi)$.
Let then $f_1\in\mathcal{H}(G,K,\psi)$ and $f_2\in L(G)$.
Since $\Phi$ is multiplicative, the quantity
\[
\begin{split}
\Phi(f_1 * \psi*f_2*\psi) & = [\phi*f_1*\psi*f_2*\psi](1_G)\\
(\text{by }\eqref{f1psif2})\ & = [\phi*f_1*f_2](1_G)\\
& = \sum_{h \in G} [\phi * f_1](h) f_2(h^{-1})
\end{split}
\]
must be equal to
\[
\Phi(f_1) \Phi(\psi*f_2*\psi) = \Phi(f_1) [\phi * f_2](1_G) = \sum_{h \in G} \Phi(f_1) \phi(h)f_2(h^{-1}),
\]
where the first equality follows from \eqref{e:gaeta}.
Since $f_2\in L(G)$ was arbitrary, we get the equality $[\phi*f_1](h) = \Phi(f_1)\phi(h) = [f_1*\phi](1_G) \phi(h)
= [\phi*f_1](1_G) \phi(h)$. Thus $\phi$ satisfies condition \eqref{lambdaf2}.
Moreover, taking $h = 1_G$, one also obtains $\phi(1_G) = 1$, and \eqref{phi1G} is also satisfied. 
In conclusion, $\phi$ is a spherical function. 
\end{proof}

\begin{corollary}
\label{CorJsph}
The number $\vert {\mathcal S}(G,K,\psi)\vert$ of spherical functions in $\mathcal{H}(G,K,\psi)$ equals the number $|J|$ of pairwise inequivalent irreducible $G$-representations contained in $\Ind_K^G\theta$.
\end{corollary}
\begin{proof}
We have $\mathcal{H}(G,K,\psi)\cong\mathbb{C}^J$ (see Theorem \ref{GKtwisted}) and every linear multiplicative functional on $\mathbb{C}^J$ is of the form $\lambda = (\lambda_\sigma)_{\sigma \in J} \mapsto\lambda_{\overline{\sigma}}$, for a fixed $\overline{\sigma}\in J$.
\end{proof}

\begin{proposition}
\label{propsphfunct}
Let $\phi$ and $\mu$ be two distinct spherical functions. Then
\begin{enumerate}[{\rm (i)}]
\item\label{propsphfunct1}
$\phi^* = \phi$;
\item
$\phi*\mu = 0$;
\item\label{propsphfunct3}
$\langle\lambda_G(g_1)\phi, \lambda_G(g_2)\mu\rangle_{L(G)} = 0$ for all $g_1, g_2 \in G$; in particular $\langle\phi, \mu\rangle_{L(G)} = 0$.
\end{enumerate}
\end{proposition}

\begin{proof} (i) By definition of a spherical function, we have
\[
\phi^* * \phi = [\phi^* * \phi](1_G) \phi = \left(\sum_{g \in G} \overline{\phi(g^{-1})}\phi(g^{-1})\right) \phi =
\|\phi\|^2 \phi.
\]
On the other hand, since $(\phi^* * \phi)^*=\phi^* * \phi$, we have
\[
[\phi^* * \phi](g) = \overline{[\phi^* * \phi](g^{-1})} = \overline{[\phi^* * \phi](1_G)} \cdot \overline{\phi(g^{-1})} = \|\phi\|^2\overline{\phi(g^{-1})}
\] 
and therefore we must have $\phi=\phi^*$.

(ii) By commutativity of $\mathcal{H}(G,K,\psi)$, we have that
\begin{equation}
\label{e:detto}
\begin{split}
[\phi*\mu](g) & = \left([\phi*\mu](1_G)\right)\phi(g)\\ & \mbox{ must be equal to }\\ [\mu*\phi](g) & = \left([\mu*\phi](1_G)\right) \mu(g) = 
\left([\phi*\mu](1_G)\right) \mu(g)
\end{split}
\end{equation}
for all $g \in G$. Therefore, if $\phi \neq \mu$, \eqref{e:detto} implies $[\phi*\mu](1_G)=  [\mu*\phi](1_G)=0$. 
But then, again, \eqref{e:detto} yields $\phi*\mu = 0$. 

(iii) Let $g_1, g_2 \in G$. Applying (i) and (ii) we have 
\[
\langle\lambda_G(g_1)\phi, \lambda_G(g_2)\mu\rangle_{L(G)} = \langle\phi, \lambda_G(g_1^{-1}g_2) \mu\rangle_{L(G)} = [\phi * \mu^*](g_1^{-1}g_2) =  [\phi * \mu](g_1^{-1}g_2)=0.
\]
\end{proof}

\begin{corollary} 
\label{c:gaeta11}
The spherical functions form an orthogonal basis for ${\mathcal H}(G,K,\psi)$.
\end{corollary}
\begin{proof}
This is an immediate consequence of the first equality in \eqref{e:c}.
\end{proof}
In Theorem \ref{t:supino} we shall copute the orthogonality relations for the spherical functions.
\par
Recall (cf.\ Theorem \ref{pD2}.(a)) that $\mathcal{I}(G,K,\psi)$ is a subrepresentation of the left-regular
representation of $G$. For $\phi \in {\mathcal S}(G,K,\psi)$ we denote by $U_\phi = \spann\{\lambda_G(g)\phi: g\in G\}$ 
the subspace of $L(G)$ generated by all $G$-translates of $\phi$. Then the following holds.

\begin{theorem}\label{Thmspherapp}
For every $\phi \in {\mathcal S}(G,K,\psi)$ the space $U_\phi$ is an irreducible $G$-representation and
\begin{equation}
\label{direct-sum}
\mathcal{I}(G,K,\psi)=\bigoplus_{\phi \in {\mathcal S}(G,K,\psi)}U_\phi
\end{equation}
is the decomposition of $\mathcal{I}(G,K,\psi)$ into irreducibles.
\end{theorem}
\begin{proof}
Let $\phi, \mu \in {\mathcal S}(G,K,\psi)$.
Then $U_\phi$ is $G$-invariant and contained in $\mathcal{I}(G,K,\psi)$. 
Moreover, by Proposition \ref{propsphfunct}.(iii), if $\phi \neq \mu$, then the spaces $U_\phi$ and $U_\mu$ are orthogonal. Finally, we can invoke Corollary \ref{CorJsph} and the fact that $\mathcal{I}(G,K,\psi)$ is multiplicity-free (cf.\ Theorem \ref{pD2}.(1) and Theorem \ref{GKtwisted}), to conclude that  the direct sum in the Right Hand Side of \eqref{direct-sum} exhausts the whole of $\mathcal{I}(G,K,\psi)$ and that the $U_\phi$s are irreducible.
\end{proof}

The space $U_\phi$ is called the {\em spherical representation} associated with $\phi \in {\mathcal S}(G,K,\psi)$. 
In the next section different realizations of the spherical functions  and of the  representations are discussed.

\subsection{Spherical functions as matrix coefficients}\label{section 7}
Let $(G,K,\theta)$ be a multiplicity-free triple. Let also $(\sigma,W) \in \widehat{G}$ be a spherical representation (i.e.\ 
$\sigma\in J$, where $J$ is as in Theorem \ref{GKtwisted}).

By Frobenius reciprocity (cf.\ \eqref{e:FR}), $\sigma$ is contained in $\Ind_K^G \theta$ if and only if $\Res^G_K \sigma$ 
contains $\theta$.  Moreover, if this is the case, since $(G,K,\theta)$ is multiplicity-free, then the multiplicity of $\theta$ in $\Res^G_K \sigma$ must be exactly one. 

This implies that there exists an {\em isometric} map $L_\sigma \colon V \rightarrow W$ such that $\Hom_K(V,\Res^G_KW) = \CC L_\sigma$. Now we show that, by means of Frobenius reciprocity, the operator $T_\sigma$ in \eqref{decIndmW} may be expressed in terms of $L_\sigma$; see \cite{st5} for more general results.

\begin{proposition}
\label{p:7.1}
The map $T_\sigma \colon W\rightarrow \{f\colon G \to V\}$ defined by setting
\begin{equation}
\label{defTsigma}
[T_\sigma w](g)=\sqrt{\frac{d_\sigma \vert K\vert}{d_\theta\vert G\vert}}L_\sigma^*\sigma(g^{-1})w,
\end{equation}
for all $g\in G$ and $w\in W$, is an isometric immersion of $W$ into $\Ind_K^GV$.
\end{proposition}
\begin{proof}
We first note that, by \eqref{adjoint}, $L_\sigma^*\in \Hom_K(\Res^G_KW,V)$, that is $L_\sigma^*\sigma(k)=\theta(k)L_\sigma^*$ for all $k\in K$. This implies that $T_\sigma w\in\Ind_K^GV$ since 
\[
[T_\sigma w](gk)= \sqrt{\frac{d_\sigma \vert K\vert}{d_\theta\vert G\vert}}L_\sigma^*\sigma(k^{-1}g^{-1})w =
\sqrt{\frac{d_\sigma \vert K\vert}{d_\theta\vert G\vert}}\theta(k^{-1})L_\sigma^*\sigma(g^{-1})w =
\theta(k^{-1})[T_\sigma w](g)
\] 
for all $g \in G$ and $k \in K$. 
It is also easy to check that $T_\sigma\sigma(g)=\lambda(g)T_\sigma$ for all $g \in G$, in other words $T_\sigma\in\Hom_G(W,\Ind_K^GV)$.

It remains to show that $T_\sigma$ is an isometric embedding. First of all, it is straightforward to check that the operator 
\[
\sum_{g\in G}\sigma(g)L_\sigma L_\sigma^*\sigma(g^{-1})
\]
belongs to $\End_G(W)$ and therefore, by Schur's lemma, it is a multiple of the identity: there exists $\alpha\in\mathbb{C}$ such that $\sum_{g\in G}\sigma(g)L_\sigma L_\sigma^*\sigma(g^{-1})=\alpha I_W$. Since $L_\sigma \colon V \to W$ is an isometric embedding, we have that
$L_\sigma^*L_\sigma  = I_V$ yielding $\tr(L_\sigma^* L_\sigma) = d_\theta$, so that
\[
\tr\left(\sum_{g\in G}\sigma(g) L_\sigma L_\sigma^*\sigma(g^{-1})\right)=\lvert G\rvert \tr\left(L_\sigma L_\sigma^*\right)
=\lvert G\rvert \tr\left(L_\sigma^* L_\sigma \right) =\lvert G\rvert d_\theta.
\]
Since $\tr(\alpha I_W)=\alpha d_\sigma$, we get $\alpha=\frac{\lvert G\rvert d_\theta}{d_\sigma}$, that is,
\[
\sum_{g\in G}\sigma(g)L_\sigma L_\sigma^*\sigma(g^{-1})=\frac{\lvert G\rvert d_\theta}{d_\sigma} I_W.
\]
We deduce that
\[
\begin{split}
\langle T_\sigma w_1,T_\sigma w_2\rangle_{{\tiny \Ind}_K^GV}&= \frac{d_\sigma}{d_\theta\lvert G\rvert}\sum_{g\in G}\left\langle L_\sigma^*\sigma(g^{-1}) w_1,L_\sigma^*\sigma(g^{-1})w_2\right\rangle_V\\
& = \frac{d_\sigma}{d_\theta\lvert G\rvert}\left\langle  w_1,\sum_{g\in G}\sigma(g)L_\sigma L_\sigma^*\sigma(g^{-1})w_2\right\rangle_W\\
& = \langle w_1,w_2\rangle_{W}
\end{split}
\]
for all $w_1, w_2 \in W$, thus showing that $T_\sigma$ is isometric.
\end{proof}

From now on, in order to emphasize the dependence of the representation space $W$ on $\sigma \in J$, we denote the former by $W_\sigma$.
Moreover, with the notation in Section \ref{s:HaR} (cf.\ \eqref{defpsi}, so that $v = v_1$), for each $\sigma\in J$ we set
\begin{equation}
\label{defwsigma}
w^\sigma=L_\sigma v \in W_\sigma.
\end{equation}

\begin{lemma}
\label{lemSsigma}
Let $\sigma\in J$. Consider the map $S_\sigma\colon W_\sigma\rightarrow L(G)$ defined by setting
\begin{equation}\label{defSsigma}
[S_\sigma w](g)=\sqrt{\frac{d_\sigma}{\lvert G\rvert}}\langle w,\sigma(g)w^\sigma\rangle_{W_\sigma},
\end{equation}
for all $w\in W_\sigma$ and $g\in G$. Then $S_\sigma = T_vT_\sigma$, that is, the following diagram
\[
\xymatrix{
W_\sigma\ar[r]^{T_\sigma}\ar[dr]_{S_\sigma}&\Ind_K^GV\ar[d]^{T_v}\\
&\mathcal{I}(G,K,\psi)}
\]
is commutative. As a consequence, $S_\sigma$ is an isometric immersion of $W_\sigma$ into $\mathcal{I}(G,K,\psi)$.  
Moreover, \eqref{TvTsigma} may be written in the form
\begin{equation}\label{IGKSsigma}
\mathcal{I}(G,K,\psi)=\bigoplus_{\sigma\in J}S_\sigma W_\sigma.
\end{equation}
\end{lemma}
\begin{proof}
For all $w\in W_\sigma$ and $g\in G$, we have
\[
\begin{split}
(T_vT_\sigma w)(g)& = \sqrt{d_\theta/\lvert K\rvert}\langle [T_\sigma w](g),v\rangle_V \ \ \ \ \qquad(\text{by }\eqref{defTv})\\
& = \sqrt{d_\sigma/\lvert G\rvert}\langle L_\sigma^*\sigma(g^{-1})w,v\rangle_V \ \qquad(\text{by }\eqref{defTsigma})\\
& = \sqrt{d_\sigma/\lvert G\rvert}\langle w,\sigma(g)L_\sigma v\rangle_{W_\sigma}  \\
& = [S_\sigma w](g) \ \   \ \ \ \qquad\qquad\qquad\qquad(\text{by }\eqref{defSsigma}).
\end{split}
\]
\end{proof}
\begin{remark}{\rm 
Clearly, $\psi$ and $\mathcal{I}(G,K,\psi)$ depend on the choice of $v \in V$. From  \eqref{defSsigma} and the orthogonality relations 
\eqref{ORT} it follows that the  replacement of $v$ with an orthogonal vector $u$ leads to an isomorphic orthogonal realization of 
 $\mathcal{I}(G,K,\psi)$.}
\end{remark}

Equations \eqref{direct-sum} and \eqref{IGKSsigma} provide two different decompositions of $\mathcal{I}(G,K,\psi)$ inirreducible representations: in the first one these are constructed in terms of spherical functions, in the second one they come from a decomposition of
$\Ind^G_K V$.
 The connection between these two decompositions requires an explicit expression for the spherical functions. To this end, given $\sigma \in J$, we define $\phi^\sigma\in L(G)$ by setting
\begin{equation}
\label{defphisigma}
\phi^\sigma(g) = \langle w^\sigma, \sigma(g)w^\sigma\rangle_{W_\sigma},
\end{equation}
for all $g\in G$, where $w^\sigma$ is as in \eqref{defwsigma}.

Our next task is to show that the above defined functions $\phi^\sigma$, $\sigma \in J$, are spherical and that, in fact, any spherical function is one of these. In other words, ${\mathcal S}(G,K,\psi)  = \{\phi^\sigma: \sigma \in J\}$.

We need to prove a preliminary identity. 
We choose an orthonormal basis $\{u_i:i=1,2,\dotsc,d_\sigma\}$ for $W_\sigma$ in the following way. Let $\Res_K^G W_\sigma = L_\sigma V\oplus\left(\oplus_{\eta}m_\eta U_\eta\right)$ an explicit (cf.\ Remark \ref{r:1111}) decomposition of $\Res_K^GW_\sigma$ into irreducible $K$-representations
(the $U_\eta$'s are pairwise distinct and each of them is non-equivalent to $V$; $m_\eta$ is the multiplicity of $U_\eta$). We set $u_1=w^\sigma = L_\sigma v$, and let $\{u_i:1\leq i\leq d_\theta\}$ form an orthonormal basis for $L_\sigma V$. Finally, the remaining $u_i$'s  are only subject to the condition of belonging to some irreducible $U_\eta$ of the explicit decomposition.
Then, by \eqref{ORT}, we have
\[
\sum_{k\in K}\langle u_1,\sigma(k)u_1\rangle \langle \sigma(k)u_i,u_j\rangle=\frac{\lvert K\rvert}{d_\theta}\delta_{1i}\delta_{1j}.
\]
Since $\psi(k)=\frac{d_\theta}{\lvert K\rvert}\langle u_1,\sigma(k)u_1\rangle_{W_\sigma}$, the above may be written in the form
\[
\left\langle\sum_{k\in K}\psi(k) \sigma(k)u_i,u_j\right\rangle=\delta_{1i}\delta_{1j}
\]
yielding
\begin{equation}\label{stellap18}
\sum_{k\in K}\psi(k) \sigma(k)u_i = \delta_{i1}w^\sigma.
\end{equation}

\begin{theorem}
\label{t:supino}
$\phi^\sigma \in L(G)$ is the spherical function associated with $W_\sigma$, that is, in the notation of Theorem \ref{Thmspherapp} and Lemma \ref{lemSsigma}, $U_{\phi^\sigma}=S_\sigma W_\sigma$. 
Moreover, the following orthgonality relations hold:
\begin{equation}\label{orthrelsph}
\langle \phi^\sigma, \phi^\rho\rangle_{L(G)}=\frac{\lvert G\rvert}{d_\sigma}\delta_{\sigma,\rho},
\end{equation}
for all $\sigma,\rho\in J$.
\end{theorem}
\begin{proof}
By \eqref{defSsigma} we have $\phi^\sigma=\sqrt{\frac{\lvert G\rvert}{d_\sigma}}S_\sigma w^\sigma$ and therefore, by Lemma \ref{lemSsigma}, $\phi^\sigma$ belongs to the subspace of $\mathcal{I}(G,K,\psi)$ isomorphic to $W_\sigma$, namely to $S_\sigma W_\sigma$. Now we check the functional identity \eqref{funcidensph} to show that $\phi^\sigma$ is a spherical function: for all $g,h \in G$,
\[
\begin{split}
\sum_{k \in K}\phi^\sigma(gkh)\overline{\psi(k)}   & = \sum_{k \in K}\langle w^\sigma, \sigma(gkh)w^\sigma\rangle\overline{\psi(k)} \\
& = \sum_{i = 1}^{d_\sigma}\langle \sigma(g^{-1  })w^\sigma, u_i\rangle
\sum_{k\in K}\overline{\langle\sigma(kh)w^\sigma, u_i\rangle}\overline{\psi(k)} \\
& = \sum_{i = 1}^{d_\sigma}\langle \sigma(g^{-1  })w^\sigma, u_i\rangle
\overline{\left\langle\sigma(h)w^\sigma,\sum_{k\in K}\psi(k^{-1})\sigma(k^{-1}) u_i\right\rangle}\\
\mbox{(by \eqref{stellap18})}\ \ \ \ \  &= \phi^\sigma(g)\phi^\sigma(h).
\end{split}
\]
Finally, \eqref{orthrelsph} is a particular case of \eqref{ORT}.
\end{proof}

The {\em spherical Fourier transform} is the map
\[
\mathcal{F}:\mathcal{H}(G,K,\psi)\longrightarrow L(J) \equiv \CC^J
\]
defined by setting
\[
[\mathcal{F}f](\sigma) = \langle f, \phi^\sigma \rangle  = \sum_{g\in G}f(g)\overline{\phi^\sigma(g)}
\]
for all $f\in \mathcal{H}(G,K,\psi)$ and $\sigma\in J$.
Note that 
\begin{equation}
\label{SFT-tullio}
[\mathcal{F}f](\sigma) = \sum_{g\in G}f(g)\overline{\phi^\sigma(g)} = \sum_{g\in G}f(g)
\phi^\sigma(g^{-1}) = [f * \phi^\sigma](1_G).
\end{equation}

From Corollary \ref{c:gaeta11} and the orthogonality relations \eqref{orthrelsph} we immediately deduce the {\em inversion formula}:

\begin{equation}
\label{e:inverion-Fourier}
f=\frac{1}{\lvert G\rvert}\sum_{\sigma\in J}d_\sigma[\mathcal{F}f](\sigma)\phi^\sigma
\end{equation}
and the {\em Plancherel formula}:
\[
\langle f_1,f_2\rangle_{L(G)}=\frac{1}{\lvert G\rvert}\sum_{\sigma\in J}d_\sigma[\mathcal{F}f_1](\sigma)\overline{[\mathcal{F}f_2](\sigma)},
\]
for all $f,f_1,f_2\in \mathcal{H}(G,K,\psi)$. Moreover, the {\em convolution formula}

\[
\mathcal{F}(f_1*f_2)=(\mathcal{F}f_1)(\mathcal{F}f_2)
\]
follows from the inversion formula and \eqref{CON}.

\begin{proposition}
Let $f\in\mathcal{H}(G,K,\psi)$. Then the eigenvalue of $T_f\in\End_G(\mathcal{I}(G,K,\psi))$ $($see Theorem \ref{pD2}.\eqref{TfHecke}$)$ associated with the eigenspace 
$S_\sigma W_\sigma$ $($see \eqref{IGKSsigma}$)$ is equal to $[\mathcal{F}f](\sigma)$, for all  $\sigma\in J$.
\end{proposition}
\begin{proof}
This is a simple calculation:
\[
\begin{split}
[T_f\phi^\sigma](g)& = [f*\phi^\sigma](g)\\
(\text{by }\eqref{lambdaf2})\ & = [f*\phi^\sigma](1_G)\phi^\sigma(g)\\
\mbox{(by \eqref{SFT-tullio})} \ & = [\mathcal{F}f](\sigma)\phi^\sigma(g).
\end{split}
\]
\end{proof}

\begin{proposition}
The operator $E_\sigma:\mathcal{I}(G,K,\psi)\longrightarrow L(G)$, defined by setting
\[
E_\sigma f=\frac{d_\sigma}{\lvert G\rvert}[f*\phi^\sigma]
\]
for  all $f \in  \mathcal{I}(G,K,\psi)$, is the orthogonal projection from $\mathcal{I}(G,K,\psi)$ onto $S_\sigma W_\sigma$.
\end{proposition}
\begin{proof}
First of all note that, for $g\in G$ and $f\in \mathcal{I}(G,K,\psi)$, we have:
\[
[E_\sigma f](g)=\frac{d_\sigma}{\lvert G\rvert}\sum_{h\in G}f(h)\phi^\sigma(h^{-1}g)=\frac{d_\sigma}{\lvert G\rvert}\sum_{h\in G}f(h)\overline{\phi^\sigma(g^{-1}h)}=\frac{d_\sigma}{\lvert G\rvert}\langle f,\lambda_G(g)\phi^\sigma\rangle.
\]
Therefore, for $\eta\neq \sigma$ and $h\in G$,
\[
\left[E_\sigma \lambda_G(h)\phi^\eta\right](g)=\frac{d_\sigma}{\lvert G\rvert}\langle \lambda_G(h)\phi^\eta,\lambda_G(g)\phi^\sigma\rangle=0
\]
by Proposition \ref{propsphfunct}.\eqref{propsphfunct3}, that is, $\bigoplus\limits_{\eta\in J, \eta\neq \sigma} S_\eta W_\eta\subseteq \ker E_\sigma$. Similarly, 
\[
\begin{split}
\left[E_\sigma \lambda_G(h)\phi^\sigma\right](g)& = \frac{d_\sigma}{\lvert G\rvert}\langle \lambda_G(h)\phi^\sigma,\lambda_G(g)\phi^\sigma\rangle\\
& = \frac{d_\sigma}{\lvert G\rvert}\langle \phi^\sigma,\lambda_G(h^{-1}g)\phi^\sigma\rangle\\
& = \frac{d_\sigma}{\lvert G\rvert}[\phi^\sigma*\phi^\sigma](h^{-1}g)\\
(\text{by }\eqref{lambdaf2})\qquad& = \frac{d_\sigma}{\lvert G\rvert}[\phi^\sigma*\phi^\sigma](1_G)\phi^\sigma(h^{-1}g)\\
(\mbox{since } [\phi^\sigma*\phi^\sigma](1_G)=\lVert\phi^\sigma\rVert^2=\lvert G\rvert/d_\sigma)	\qquad& = \lambda_G(h)\phi^\sigma(g).
\end{split}
\]
\end{proof}

We now prove the relations between the spherical function $\phi^\sigma$ and the character $\chi^\sigma$ of $\sigma$.

\begin{proposition}\label{propcharsphfun}
For all $g\in G$ we have:
\begin{equation}\label{eD21sigma}
\chi^\sigma(g) = \frac{d_\sigma}{|G|}\sum_{h \in G}\overline{\phi^\sigma(h^{-1}gh)}
\end{equation}
and
\begin{equation}\label{eD21chi}
\phi^\sigma(g) = [\overline{\chi^\sigma}*\psi](g).
\end{equation}
\end{proposition}
\begin{proof}
Clearly, \eqref{eD21sigma} is just a particular case of \eqref{eD21}, keeping into account the definition of $\phi^\sigma$ (see \eqref{defphisigma}). Using the bases in \eqref{stellap18} we have
\[
\begin{split}
[\overline{\chi^\sigma}*\psi](g)& = \sum_{k\in K}\sum_{i=1}^{d_\sigma}\overline{\langle\sigma(gk^{-1})u_i,u_i \rangle}\psi(k)\\
& = \sum_{k\in K}\sum_{i=1}^{d_\sigma}\overline{\langle\psi(k^{-1})\sigma(k^{-1})u_i,\sigma(g^{-1})u_i \rangle}\\
(\text{by }\eqref{stellap18})\qquad& = \phi^\sigma(g). 
\end{split}
\]
\end{proof}

\subsection{The case $\dim \theta=1$}\label{s:case-dim-1}
In this section we consider the case when the $K$-representation $(V,\theta)$ is one-dimensional (this case is also treated, in a more detailed way, in \cite[Chapter 13]{book4}). We then denote by $\chi = \chi^\theta$ its character. 
Let $\mathcal{S}_0 \subseteq \mathcal{S} \subseteq G$ and $K_s$, $s \in \mathcal{S}$, 
be as in Section \ref{s:MFIR}, and let $\psi \in L(G)$ and $\mathcal{H}(G,K,\psi)$ be as in  Section \ref{s:HaR}.

Note that, in our setting, we have (cf.\ \eqref{defpsi}) $\psi(k) = \frac{1}{|K|} \overline{\chi(k)}$  for all $k \in K$.

\begin{theorem}
\label{t:caso-dim-1}
\begin{enumerate}[{\rm (1)}]
\item $\mathcal{H}(G,K,\psi) = \{f \in L(G): f(k_1gk_2) = \overline{\chi(k_1)}\overline{\chi(k_2)} f(g)$ for all  $k_1, k_2 \in K$ and  $g \in G\}$;
\item $\mathcal{S}_0 = \{s \in \mathcal{S}: \chi(s^{-1}xs) = \chi(x), \mbox{ for all } x \in K_s\}$; 
\item every function $f \in \mathcal{H}(G,K,\psi)$ only depends on its values on $\mathcal{S}_0$, namely
\[
f(g) = \begin{cases}
\overline{\chi(k_1)}f(s)\overline{\chi(k_2)}& \mbox{if } g = k_1sk_2 \mbox{ with } s \in \mathcal{S}_0\\
0 & \mbox{otherwise.}
\end{cases}
\]
\end{enumerate}
\end{theorem}
\begin{proof}
(1) Let $f \in \mathcal{H}(G,K,\psi)$. Then by Theorem \ref{pD2}.(3) we can find a unique function $F \in
\widetilde{\mathcal{H}}(G,K,\theta)$ such that $f = S_v(F)$. Since $V$ is one-dimensional, we have
$F \in L(G)$ and $\theta = \chi$ so that \eqref{e:F-def} yields
\[
F(k_1gk_2) = \overline{\chi(k_2)}F(g) \overline{\chi}(k_1) = \overline{\chi(k_1)} \overline{\chi(k_2)}F(g),
\]
which, by linearity of $S_v$, yields
\[
f(k_1gk_2) =  \overline{\chi(k_1)} \overline{\chi(k_2)}f(g)
\]
for all $k_1, k_2 \in K$ and $g \in G$.
This shows the inclusion $\subseteq$. Since $S_v$ is bijective, (1) follows.

(2) We first observe that since $\theta$ is one-dimensional, so are $\Res^K_{K_s}\theta$ and $\theta^s$
for all $s \in \mathcal{S}$. As a consequence, for $s \in \mathcal{S}$ we have that $\Hom_{K_s}(\Res^K_{K_s}\theta,\theta^s)$ is non-trivial if and only if $\Res^K_{K_s}\theta$ equals $\theta^s$ 
and this is in turn equivalent to $\chi(x) = \chi(s^{-1}xs)$ for all $x \in K_s$.  

(3) This follows immediately from (1), (2), and the fact that any $f \in \mathcal{H}(G,K,\psi)$ is supported in $\sqcup_{s \in \mathcal{S}_0} KsK$ (cf.\ Theorem \ref{pD2}.(3)).
\end{proof}                                     
 
\subsection{An example: the Gelfand-Graev representation of $\GL(2,\FF_q)$}
\label{ss:GGrep}
We now illustrate a fundamental example, which is completely examined in \cite[Chapter 14]{book4}. 
This gives us the opportunity to introduce some notation and basic notions that shall be widely used in Sections \ref{s:HAMFT1} 
and \ref{s:IItrippa}.
\par
Let $p$ be a prime number, $n$ a positive integer, and denote by $\FF_q$ the field with $q:=p^n$ elements.
Let $G = \GL(2,\FF_q)$ denote the group of invertible $2 \times 2$ matrices with coefficients in $\FF_q$
and consider the following subgroups:
\[
\begin{split}
B  &= \left\{\begin{pmatrix}  a & b\\ 0& d \end{pmatrix}: a, d \in \FF_q^*, b \in \FF_q\right\} \quad  \mbox{(the {\it Borel} subgroup)}\\
C  &= \left\{\begin{pmatrix}  a & \eta b\\ b& a \end{pmatrix}: a,b \in \FF_q, (a, b)  \neq (0,0)\right\} \quad \mbox{(the {\it Cartan} subgroup)}\\
D & = \left\{\begin{pmatrix}  a & 0\\ 0& d \end{pmatrix}: a, d \in \FF_q^*\right\} \quad \mbox{(the subgroup of {\it diagonal} matrices)}\\
U &= \left\{\begin{pmatrix}  1 & b\\ 0& 1 \end{pmatrix}: b \in \FF_q\right\} \quad  \mbox{(the subgroup of {\it unipotent} matrices)}\\
Z &= \left\{\begin{pmatrix}  a & 0\\ 0& a \end{pmatrix}: a \in \FF_q^*\right\} \quad  \mbox{(the {\it center})}
\end{split}
\]
where $\FF_q^*$ denotes the multiplicative subgroup of $\FF_q$ consisting of all non-zero elements, and $\eta$ is a generator of the
multiplicative group $\FF_q^*$. As for subgroup $C$, we suppose that $q$ is odd (for $q$ even we refer to \cite[Section 6.8]{book4}):
then we have the isomorphism
\begin{equation}
\label{estar:PTPT3}
\begin{array}{ccc}
C & \longrightarrow & \FF_{q^2}^*\\
\begin{pmatrix}  a & \eta b\\ b& a \end{pmatrix} & \mapsto & a + ib,
\end{array}
\end{equation}
where $\FF_{q^2}$ is the quadratic extension of $\FF_q$ and  $\pm i$ are the square roots (in $\FF_{q^2}$) of $\eta$.

An irreducible $\GL(2,\FF_q)$-representation $(\rho,V)$ such that the subspace $V^U$ of $U$-invariant vectors is trivial is called a \emph{cuspidal representation}.

\begin{theorem}
\label{t:cuspidal}
Let $\chi$ be a non-trivial character of the (Abelian) group $K=U$. Then $\Ind_K^G \chi$ is multiplicity-free.
\end{theorem}
In order to prove the above theorem we shall make use of the so-called \emph{Bruhat decomposition} of $G$:
\begin{equation}
\label{e:bruhat}
G = B \bigsqcup UwB
\end{equation}
where $w = \begin{pmatrix}  0 & 1\\ 1& 0 \end{pmatrix}$ (this follows from elementary calculations, cf.\ \cite[Lemma 14.2.4.(iv)]{book4}).
\begin{proof}[Proof of Theorem \ref{t:cuspidal}]
We first observe that $U$ is a normal subgroup of $B$ and that one has $B = \bigsqcup_{d \in D}dU = \bigsqcup_{d \in D}UdU$. 
From \eqref{e:bruhat} we then get
\[
G =  \left(\bigsqcup_{d \in D}dU\right) \bigsqcup \left(\bigsqcup_{d \in D}UwdU\right) = \left(\bigsqcup_{d \in D}UdU\right) \bigsqcup \left(\bigsqcup_{d \in D}UwdU\right).
\]
As a consequence, we can take $\mathcal{S}:= D \bigsqcup wD$ as a complete set of representatives for the double $K$-cosets in $G$. Moreover, it is easy to check that $dUd^{-1} \cap U = U$ and that $wdUd^{-1}w \cap U = \{1_G\}$
for all $d \in D$. As a consequence, cf.\ Theorem \ref{t:caso-dim-1}.(2), we have that $\mathcal{S}_0 = Z \bigsqcup wD = \mathcal{S} \setminus (D \setminus Z)$. From Theorem \ref{t:caso-dim-1}.(3) we deduce that every function $f \in \mathcal{H}(G,K,\psi)$ vanishes on $\bigsqcup_{d \in D \setminus Z}dU$.

Consider the map $\tau \colon G \to G$ defined by setting
\[
\tau\left(\begin{pmatrix}  a & b\\ c& d \end{pmatrix}\right) = \begin{pmatrix}  d & b\\ c& a \end{pmatrix}
\]
for all $\begin{pmatrix}  a & b\\ c& d \end{pmatrix} \in G$. It is easy to check that $\tau(g_1g_2) = \tau(g_2)\tau(g_1)$ and $\tau^2(g) = g$ for all $g_1,g_2, g \in G$. Thus, $\tau$ is an involutive antiautomorphism of $G$.

Let $f \in \mathcal{H}(G,K,\psi)$. We claim that
\begin{equation}
\label{e:f-tau-2}
f(\tau(g)) = f(g) \ \ \mbox{ for all } g \in G.
\end{equation}
In order to show \eqref{e:f-tau-2}, we recall that $f$ is supported in $\bigsqcup_{s \in Z \bigsqcup wD} UsU$ and observe that $\tau$ fixes all the elements of the subgroup $U$. As a consequence, it suffices to show that $\tau$ also fixes all elements in $Z \bigsqcup wD$. This is a simple calculation:
\[
\tau(wd) = \tau\left(\begin{pmatrix}  0 & 1\\ 1& 0 \end{pmatrix} \begin{pmatrix}  a & 0\\ 0& b \end{pmatrix}\right) =
\tau\left(\begin{pmatrix}  0 & b\\ a& 0 \end{pmatrix}\right) = \begin{pmatrix}  0 & b\\ a& 0 \end{pmatrix} = wd
\]
for all $d = \begin{pmatrix}  a & 0\\ 0& b \end{pmatrix} \in D$. On the other hand, it is obvious that $\tau(z) = z$
for all $z \in Z$. This proves the claim.
As a consequence, by virtue of Proposition \ref{e;smft}, we have that the Hecke algebra $\mathcal{H}(G,K,\psi)$ is commutative. From
Theorem \ref{GKtwisted} we deduce that the induced representation $\Ind^G_K \chi$ is multiplicity-free.
\end{proof}

\begin{remark}{\rm 
The conditions in the Bump-Ginzburg criterion are trivially satisfied, because  $\tau(k) = k$ for all $k \in K$ and $\tau(s) = s$ for all $s \in \mathcal{S}_0$.}
\end{remark}

\subsection{A Frobenius-Schur theorem for multiplicity-free triples}
\label{s:FS-mft}
Let $G$ be  a finite group. 
Recall that the \emph{conjugate} of a $G$-representation $(\sigma, W)$
is the $G$-representation $(\sigma',W')$ where  $W'$ is the dual of $W$ and $[\sigma'(g)w'](w) = w'[\sigma(g^{-1})w]$
for all $g \in G$, $w \in W$, and $w'\in W'$. The matrix ceofficients of the conjugate representation are the conjugate
of the matrix coefficients, in formulae:
\begin{equation}
\label{e:conj-matr-coeff}
u_{i,j}^{\sigma'} = \overline{u_{i,j}^\sigma}
\end{equation} 
for all $i,j=1,2,\ldots,d_\sigma$ (cf.\ \cite[Equation (9.14)]{book}).

One then says that $\sigma$ is \emph{self-conjugate} provided
$\sigma \sim \sigma'$; this is in turn equivalent to the associated character $\chi^\sigma$ being real-valued. 
When $\sigma$ is not self-conjugate, one says that it is {\it complex}.
The class of self-conjugate $G$-representations splits into two subclasses according to the associated matrix
coefficients of the representation $\sigma$ being real-valued or not: in the first case, one says that $\sigma$ is
\emph{real}, in the second case $\sigma$ is termed \emph{quaternionic}.

A fundamental theorem of Frobenius and Schur provides a criterion for determining the type of a given 
irreducible $G$-representation $\sigma$, namely
\begin{equation}
\label{eq:Frobenius-introduction}
\frac{1}{|G|}\sum_{g \in G} \chi^\sigma(g^2) = 
\begin{cases} 
1 & \mbox{ if $\sigma$ is real}\\
-1 & \mbox{ if $\sigma$ is quaternionic}\\
0 & \mbox{ if $\sigma$ is complex}
\end{cases}
\end{equation}
see, for instance, \cite[Theorem 9.7.7]{book}.

Let now $K \leq G$ be a subgroup, $(\theta,V)$ be an irreducible $K$-representation, and suppose that
$(G,K,\theta)$ is a multiplicity-free triple. In the following we prove a generalization of the
Frobenius-Schur theorem for spherical representations.

\begin{theorem} 
Let $(G,K,\theta)$ be a multiplicity-free triple and suppose that $(\sigma,W)$ is a spherical representation
(i.e., $\sigma \in \widehat{G}$ and $\sigma \preceq \Ind^G_K \theta$). Then we have
\begin{equation}
\label{eq:Frobenius-schur}
\frac{d_\sigma}{|G|}\sum_{g \in G} \phi^\sigma(g^2) = 
\begin{cases} 
1 & \mbox{ if $\sigma$ is real}\\
-1 & \mbox{ if $\sigma$ is quaternionic}\\
0 & \mbox{ if $\sigma$ is complex.}
\end{cases}
\end{equation}
\end{theorem}
\begin{proof}
We first fix an orthonormal basis $\{u_i:i=1,2,\dotsc,d_\sigma\}$ for $W$ as in the paragraph preceding
Theorem \ref{t:supino} so that $\phi^\sigma = \overline{u_{1,1}^\sigma}$.
We then have
\begin{equation}
\label{e:cl-ol}
\begin{split}
\frac{1}{|G|} \sum_{g \in G} \phi^\sigma(g^2) & = \frac{1}{|G|} \sum_{g \in G} \overline{u_{1,1}^\sigma(g^2)}\\
\mbox{(by \eqref{CONproduct})} \ \ \ \ & =  \frac{1}{|G|} \sum_{g \in G} \sum_{h=1}^{d_\sigma}  \overline{u_{1,h}^\sigma(g)} \overline{u_{h,1}^\sigma(g)}.
\end{split}
\end{equation}

If $\sigma$ is real, then $\overline{u_{1,h}^\sigma(g)} = u_{1,h}^\sigma(g)$ for all $h=1,2,\ldots,d_\sigma$ and
$g \in G$, so that, by \eqref{ORT},
\[
\frac{1}{|G|} \sum_{g \in G}\overline{u_{1,h}^\sigma(g)} \overline{u_{h,1}^\sigma(g)} = \frac{1}{|G|}
\langle u_{1,h}^\sigma, u_{h,1}^\sigma \rangle_{L(G)} = \frac{1}{d_\sigma} \delta_{1,h}
\]
and \eqref{e:cl-ol} yields $\frac{d_\sigma}{|G|} \sum_{g \in G} \phi^\sigma(g^2) = 1$.

If $\sigma$ is complex then, by \eqref{e:conj-matr-coeff}, 
\[
\frac{1}{|G|} \sum_{g \in G}\overline{u_{1,h}^\sigma(g)} \overline{u_{h,1}^\sigma(g)} = \frac{1}{|G|}
\langle u_{1,h}^{\sigma'}, u_{h,1}^\sigma \rangle_{L(G)} = 0,
\]
since $\sigma \not\sim \sigma'$ and \eqref{ORT} applies. Thus in this case \eqref{e:cl-ol} yields 
$\frac{d_\sigma}{|G|} \sum_{g \in G} \phi^\sigma(g^2) = 0$.

Suppose, finally, that $\sigma$ is quaternionic. Then, see \cite[Lemma 9.7.6]{book}, we can find a $d_\sigma \times
d_\sigma$ complex matrix $W$ such that $W \overline{W} = \overline{W} W = -I$ and $\overline{U(g)} = WU(g)W^*$,
where $U(g) = \left(u_{i,j}^\sigma(g)\right)_{i,j=1}^{d_\sigma}$, for all $g \in G$. Then, for every $h = 1,2,\ldots, d_\sigma$ we have 
\[
\begin{split}
\frac{d_\sigma}{|G|} \sum_{g \in G}\overline{u_{1,h}^\sigma(g)} \overline{u_{h,1}^\sigma(g)} & = 
\frac{d_\sigma}{|G|} \sum_{g \in G}\sum_{j,\ell=1}^{d_\sigma} w_{1,\ell}u_{\ell,j}^\sigma(g)\overline{w_{h,j}} \overline{u_{h,1}^\sigma(g)}\\& = \sum_{j,\ell=1}^{d_\sigma} w_{1,\ell}\overline{w_{h,j}} \frac{d_\sigma}{|G|} \langle u_{\ell,j}^\sigma, u_{h,1}^\sigma \rangle_{L(G)}\\
\mbox{(by \eqref{ORT})} \ \ \ & = w_{1,h}\overline{w_{h,1}}
\end{split}     
\]       
and \eqref{e:cl-ol} yields 
\[
\frac{d_\sigma}{|G|} \sum_{g \in G} \phi^\sigma(g^2) =  \sum_{h=1}^{d_\sigma} w_{1,h}\overline{w_{h,1}} = -1,
\]
since $W \overline{W}= -I$.                            
\end{proof}

\begin{remark}
{\rm As mentioned at the end of Section \ref{ss:gBG}, 
the twisted Frobenius-Schur theorem (cf.\ \cite[Section 9]{AM}) deserves to be analyzed in the present setting.}
\end{remark}

\section{The case of a normal subgroup}\label{s:normal}
In this section we consider triples of the form $(G,N,\theta)$ in the particular case when the subgroup $N \leq G$ is normal.

It is straightforward to check that $\Ind_N^G \iota_N$, the induced representation of the trivial representation of $N$, is equivalent to the regular representation of the quotient group $G/N$.
Therefore, if $(G,N)$ is a Gelfand pair, its analysis is equivalent to the study of the representation theory of the quotient group $G/N$ (which is necessarily Abelian). 

In this section we treat the general case, namely, $\Ind_N^G \theta$, where $\theta\in\widehat{N}$ is arbitrary. 
Now, $G$ acts by conjugation on $\widehat{N}$, and we denote by $I_G(\theta)$ the stabilizer of
$\theta$, called the {\em inertia group} (cf.\ Section \ref{Subsecinercocy}).
We then study the commutant $\End_G(\Ind_N^G\theta)$ of $\Ind_N^G\theta$.

From Clifford theory (cf.\ \cite[Theorem 2.1(2)]{CSTCli} and \cite[Theorem 1.3.2(ii)]{book3}) it is known that 
$\dim\End_G(\Ind_N^G\theta)=\lvert I_G(\theta)/N\rvert$.
In Theorem \ref{t:teorema1.9} we will show that, indeed, $\End_G(\Ind_N^G\theta)$ is isomorphic to the algebra $L(I_G(\theta)/N)$ equipped with a modified convolution product. Moreover, in Section \ref{s:H-and-sf} we shall study in detail this Hecke algebra and its associated spherical functions. Most of this section is indeed devoted to the general case, where we do not assume that $\Ind_N^G \theta$ is multiplicity-free. This is quite natural and has some interest on its own.
In the last section we finally examine the multiplicity-free case and we prove that if $\End_G(\Ind_N^G \theta)$ is commutative, then $I_G(\theta)/N$ is Abelian and $\End_G(\Ind_N^G \theta) \cong L(I_G(\theta)/N)$ (where the latter is the usual group algebra); therefore, also for general representations induced from normal subgroups, the multiplicity-free case reduces to the  analysis on an Abelian group. For more precise statements, see Theorem \ref{theoremA:ACCT} and Theorem \ref{theoremC:ACCT}. 

\subsection{Unitary cocycles}\label{Subsectionuncocy}

Let $H$ be a finite group.

\begin{definition}
{\rm A {\em unitary cocycle} on $H$ is a map $\tau \colon H\times H\rightarrow \mathbb{T}$ such that
\begin{equation}\label{cocyc1}
\tau(1_H,h)=\tau(h,1_H)=1_\mathcal{\TT}\qquad\qquad \text{\rm (normalization)}
\end{equation}
and
\begin{equation}\label{cocyc2}
\tau(h_1h_2,h_3)\tau(h_1,h_2)=\tau(h_1,h_2h_3)\tau(h_2,h_3) \qquad\qquad \text{\rm (cocycle identity) }
\end{equation}
for all $h,h_1,h_2,h_3\in H$. 

We denote by $\mathcal{C}(\mathbb{T},H)$ the set of all unitary cocycles.}
\end{definition}

It is easy to prove that $\mathcal{C}(\mathbb{T},H)$ is an Abelian group under pointwise multiplication. Moreover, if a function $\rho \colon H\to\mathbb{T}$ satisfies the condition $\rho(1_H)=1$, then $\tau_\rho \colon H\times H\to \mathbb{T}$, defined by setting
\begin{equation}\label{cobound}
\tau_\rho(h_1,h_2)=\rho(h_1h_2)\left[\rho(h_1)\rho(h_2)\right]^{-1},
\end{equation}
for all $h_1,h_2\in H$, is a unitary cocycle, called a {\em coboundary}. 
We denote by $\mathcal{B}(\mathbb{T},H)$ the set of all coboundaries: it is a subgroup of $\mathcal{C}(\mathbb{T},H)$.
 The corresponding quotient group $\mathcal{H}^2(\mathbb{T},H)=\mathcal{C}(\mathbb{T},H)/\mathcal{B}(\mathbb{T},H)$ is called the {\em second cohomology group of} $H$ {\em with values in} $\mathbb{T}$. The elements of $\mathcal{H}^2(\mathbb{T},H)$ are called {\em cocycle classes} and two cocycles belonging to the same class are said to be {\em cohomologous}.
 
Note that if $\tau \in \mathcal{C}(\mathbb{T},H)$, then, setting $h_1=h_3=h$ and $h_2=h^{-1}$ in \eqref{cocyc2} and using \eqref{cocyc1}, we get that $\tau(h,h^{-1})=\tau(h^{-1},h)$, for all $h\in H$. In particular, a unitary cocycle is said to be {\em equalized} 
(cf.\ \cite{Mihailovs}) if $\tau(h,h^{-1})=1$ for all $h\in H$. Every unitary cocycle $\tau$ is 
cohomologous to an equalized cocycle $\tau'$. Indeed,
if $\rho \colon H \to \TT$ is defined by setting $\rho(1_H) = 1$ and $\rho(h)=\rho(h^{-1})=\alpha$, where $\alpha\in\mathbb{T}$ satisfies $\alpha^2=\tau(h,h^{-1})^{-1}$, for all $h \in H$, then one immediately checks that the function $\tau' \colon H\times H\to \mathbb{T}$, defined by setting
\begin{equation}\label{NS:asterisco39}
\tau'(h_1,h_2)=\rho(h_1)\rho(h_2)\rho(h_1h_2)^{-1}\tau(h_1,h_2),
\end{equation}
 for all $h_1, h_2 \in H$, is an equalized unitary cocycle cohomologous to $\tau$.

\begin{lemma}
If $\tau\in \mathcal{C}(\mathbb{T},H)$ is equalized, then for all $g,h\in H$ we have:
\begin{equation}\label{taumin11}
\tau(g,h)^{-1}=\tau(g^{-1},gh)=\tau(h^{-1},g^{-1}).
\end{equation}
Moreover, for all $k,r,s \in H$ we have
\begin{equation}\label{cocycrsk}
\tau(k^{-1}s,s^{-1}r)\tau(k^{-1},s) =\tau(k^{-1},r)\tau(r^{-1},s),
\end{equation}
\end{lemma}
\begin{proof} Let $g,h,k,r,s\in H$.
Setting $h_1=g^{-1},h_2=g$ and $h_3=h$ in \eqref{cocyc2}, we get 
\[
\tau(g^{-1},gh)\tau(g,h) = \tau(1_H,h)\tau(g^{-1},g) = 1,
\] 
and this proves the first equality in \eqref{taumin11}. Similarly, 
\[
\begin{split}
\tau(g,h)\tau(h^{-1},g^{-1}) & = \tau(g,h)\tau(gh,h^{-1}g^{-1})\tau(h^{-1},g^{-1})\\
(h_1=g,h_2=h, \text{ and } h_3=h^{-1}g^{-1}\text{ in }\eqref{cocyc2})\qquad & = \tau(g,g^{-1})\tau(h,h^{-1}g^{-1})\tau(h^{-1},g^{-1})\\
(h_1=h,h_2=h^{-1}, \text{ and } h_3=g^{-1}\text{ in }\eqref{cocyc2})\qquad & = \tau(1,g^{-1})\tau(h,h^{-1})\\
& = 1,
\end{split}
\]
yields the second equality in \eqref{taumin11}.  

Finally, we have
\begin{equation*}
\tau(k^{-1}s,s^{-1}r)\tau(k^{-1},s) = \tau(k^{-1},r)\tau(s,s^{-1}r) =\tau(k^{-1},r)\tau(r^{-1},s),
\end{equation*}
where the first equality may be obtained by setting $h_1=k^{-1}, h_2=s$, and $h_3=s^{-1}r$ in \eqref{cocyc2}, and the second equality follows from the second equality in \eqref{taumin11}, by setting $g=s^{-1}$ and $h=r$. 
\end{proof}

\subsection{Cocycle convolution}\label{ss:cocycle}
In this section we introduce a convolution product on $L(H)$ modified by means of an equalized unitary cocycle in 
$\mathcal{C}(\mathbb{T},H)$ (see also \cite{BZ,Isaacs} where, for a similar algebra, the authors use the term {\it twisted convolution}). 

Let $\eta\in\mathcal{C}(\mathbb{T},H)$ be an equalized unitary cocycle. Given $f_1,f_2\in L(H)$ we define their $\eta$-{\em cocycle convolution} by setting
\[
[f_1*_\eta f_2] (k)=\sum_{h\in H}f_1(h^{-1}k)f_2(h)\eta(k^{-1},h),
\]
for all $k\in H$.
Also, as usual, we let $f\mapsto f^*$ denote the involution on $L(H)$ defined by setting $f^*(h)=\overline{f(h^{-1})}$ for all $h \in H$.

\begin{proposition} 
The space $L(H)$ with the $\eta$-{\em cocycle convolution} and the involution defined above is an involutive, unital, associative algebra, which we denote by $L(H)_\eta$.
\end{proposition}
\begin{proof}
We first prove that the $\eta$-cocycle convolution is associative.

Let $f_1,f_2,f_3\in L(H)$ and $k\in H$. Then we have
\[
\begin{split}
\left[\left(f_1*_\eta f_2\right)*_\eta f_3\right] (k)&=\sum_{s\in H}\sum_{h\in H}f_1(h^{-1}s^{-1}k)f_2(h)f_3(s)\eta(k^{-1}s,h)\eta(k^{-1},s)\\
(h=s^{-1}r)\qquad&=\sum_{s\in H}\sum_{r\in H}f_1(r^{-1}k)f_2(s^{-1}r)f_3(s)\eta(k^{-1}s,s^{-1}r)\eta(k^{-1},s)\\
(\text{by }\eqref{cocycrsk})\qquad&=\sum_{s\in H}\sum_{r\in H}f_1(r^{-1}k)f_2(s^{-1}r)f_3(s)\eta(k^{-1},r)\eta(r^{-1},s)\\
& = \left[f_1*_\eta \left(f_2*_\eta f_3\right)\right] (k).
\end{split}
\]
This proves associativity of the convolution product $*_\eta$. Moreover,
\[
\begin{split}
[f_1^**_\eta f_2^*](k)& =\sum_{h\in H}\overline{f_1(k^{-1}h)}\overline{f_2(h^{-1})}\eta(k^{-1},h)\\
(\text{setting } s=k^{-1}h\text{ and by }\eqref{taumin11})\ \ \ & = \sum_{s\in H}\overline{f_1(s)}\overline{f_2(s^{-1}k^{-1})}\overline{\eta(k,s)}\\
& =  [f_2*_\eta f_1]^*(k).
\end{split}
\]
This shows that $L(H)_\eta$ is involutive.

We leave it to the reader to check that the identity element is $\delta_{{1_H}}$.
\end{proof}

\subsection{The inertia group and unitary cocycle representations}\label{Subsecinercocy}

We recall some basic facts on Clifford theory; we refer to \cite{CSTCli, book3} for more details and further results. 

Let $G$ be a finite group and suppose that $N \trianglelefteq G$ is a normal subgroup.
Also let $(\theta, V)$ be an irreducible $N$-representation. Given $g \in G$, we denote by $( ^g\!\theta, V)$ the $N$-representation
defined by setting
\begin{equation}\label{e;clifgconj}
 ^g\!\theta(n) = \theta(g^{-1}ng)
\end{equation}
for all $n \in N$ (cf.\ \eqref{e:theta-s}). This is called the {\it $g$-conjugate} representation of $\theta$.

Observe that \eqref{e;clifgconj} defines a left action of $G$ on $\widehat{N}$, i.e., $ ^{{1_G} \,}\!\theta = \theta$ and 
$^{g_1g_2 \,}\!\theta = \ ^{g_1}\!( ^{g_2 \, }\!\theta)$ for all $g_1, g_2 \in G$. The stabilizer of this action is the subgroup
\begin{equation}
\label{e:inerzia-sgr}
I_G(\theta) = \{h \in G:\  ^h\!\theta \sim \theta\},
\end{equation}
which is called the {\it inertia group} of $\theta$. Note that $N\trianglelefteq I_G(\theta)$.
Finally, we fix $Q \subseteq N$ a complete set of representatives for the cosets of $N$ in $I_G(\theta)$ such that $1_G \in Q$,
so that we have
\begin{equation}\label{cosetQ}
I_G(\theta) = \bigsqcup_{q \in Q}qN = \bigsqcup_{q \in Q}Nq.
\end{equation} 

From now on, our exposition is based on \cite[Section XII.1]{FellDoran}; see also \cite{CST4}. The same material is treated in 
\cite{CR1, Hu, Isaacs} (where unitarity is not assumed) under the name of projective representations. 
Actually, we only need some specific portions of the theory but expressed in our language, so that we include complete proofs.

Since $^{q \,}\!\theta\sim\theta$ for all $q\in Q$, there exists a unitary operator $\Theta(q) \in \End(V)$ such that
\begin{equation}\label{form16}
^{q \,}\!\theta(n)=\theta(q^{-1}nq)=\Theta(q)^{-1}\theta(n)\Theta(q),
\end{equation}
for all $n\in N$. By Schur's lemma $\Theta(q)$ is determined up to a unitary multiplicative constant; we fix an arbitrary such a choice but we assume that $\Theta(1_G)=I_V$ (see also Remark \ref{r:CoreqcocPsi}).  Then, on the whole of $I_G(\theta)$ (cf.\ \eqref{cosetQ}) we set 
\begin{equation}\label{form18}
\Theta(nq)=\theta(n)\Theta(q)
\end{equation}
$n \in N$ and $q \in Q$.
Therefore, for all $h=nq\in I_G(\theta)$ and $m\in N$, we have:
\begin{equation}\label{form16bis}
\theta(h^{-1}mh)=\theta(q^{-1}n^{-1}mnq)=\Theta(q)^{-1}\theta(n^{-1})\theta(m)\theta(n)\Theta(q)=\Theta(h)^{-1}\theta(m)\Theta(h),
\end{equation}
and
\begin{equation}\label{form18bis}
\Theta(mh)=\Theta(mnq)=\theta(mn)\Theta(q)=\theta(m)\theta(n)\Theta(q)=\theta(m)\Theta(h).
\end{equation}
It follows that, for all $h,k\in I_G(\theta)$ and $n\in N$,
\[
\Theta(h)^{-1}\Theta(k)^{-1}\theta(n)\Theta(k)\Theta(h)
=\theta(h^{-1}k^{-1}nkh)=\Theta(kh)^{-1}\theta(n)\Theta(kh),
\] 
that is, $\Theta(kh)\Theta(h)^{-1}\Theta(k)^{-1}\theta(n)=\theta(n)\Theta(kh)\Theta(h)^{-1}\Theta(k)^{-1}$,
so that, by Schur's lemma, there exists a constant $\tau(k,h)\in\mathbb{T}$ such that:
\begin{equation}\label{form20}
\Theta(kh)=\tau(k,h)\Theta(k)\Theta(h).
\end{equation}
A map $I_G(\theta) \ni h  \mapsto \Theta(h)$ (where each $\Theta(h)$ is unitary and $\tau$ is a unitary cocycle) satisfying \eqref{form20} is called a {\it unitary $\tau$-representation} (or a {\it cocycle} representation). 
We shall prove that $\tau$ is a unitary cocycle in Proposition \ref{prop2cocinv}.
The notions of invariant subspaces, (unitary) equivalence, and irreducibility may be easily introduced for cocycle representations and we leave the corresponding details to the reader.
\begin{lemma}
The maps $\tau \colon I_G(\theta)\times I_G(\theta)\to\mathbb{T}$ and $\Theta \colon I_G(\theta) \to \End(V)$  satisfy the following identities:
\begin{equation}\label{starr}
\Theta(k)^*\equiv\Theta(k)^{-1}=\tau(k,k^{-1})\Theta(k^{-1})
\end{equation}
and
\begin{equation}\label{form22}
\Theta(h)\theta(m)\Theta(h)^{-1}=\theta(hmh^{-1}),
\end{equation}
for all $k,h\in I_G(\theta), m\in N$. Moreover, for $k=nq_1,h=mq_2\in I_G(\theta), q_1,q_2\in Q, n,m\in N$,
\begin{equation}\label{form20bis}
\tau(k,h) I_V =\Theta(q_2)^{-1}\Theta(q_1)^{-1}\Theta(q_1q_2).
\end{equation}
\end{lemma}
\begin{proof}
Since $\Theta(1_G)=I_V$, from \eqref{form20} with $h=k^{-1}$ we deduce that $I_V=\tau(k,k^{-1})\Theta(k)\Theta(k^{-1})$, and the first identity follows. We now prove the second identity:
\begin{align*}
\Theta(h)\theta(m)\Theta(h)^{-1}&=\Theta(h^{-1})^{-1}\theta(m)\Theta(h^{-1})&(\text{by }\eqref{starr})\\
&=\theta(hmh^{-1})&(\text{by }\eqref{form16bis}).
\end{align*}
Finally,
\begin{equation}\label{Psikh}
\begin{split}
\Theta(kh)=\Theta(nq_1mq_2)&=\Theta(nq_1mq_1^{-1}\cdot q_1q_2)\\
(\text{by }\eqref{form18}\text{ and }\eqref{form22})\quad&=\theta(n)\Theta(q_1)\theta(m)\Theta(q_1)^{-1}\Theta(q_1q_2)
\end{split}
\end{equation}
so that
\begin{align*}
\tau(k,h) I_V & = \Theta(h)^{-1}\Theta(k)^{-1}\cdot\Theta(kh)&(\text{by }\eqref{form20})\\
&=\Theta(q_2)^{-1}\theta(m)^{-1}\Theta(q_1)^{-1}\theta(n)^{-1}\cdot&(\text{by }\eqref{form18})\\
&\qquad\cdot\theta(n)\Theta(q_1)\theta(m)\Theta(q_1)^{-1}\Theta(q_1q_2)&(\text{by }\eqref{Psikh})\\
&=\Theta(q_2)^{-1}\Theta(q_1)^{-1}\Theta(q_1q_2).
\end{align*}

\end{proof}

In Remark \ref{r:CoreqcocPsi} we will give a simplified version of \eqref{starr}.

\begin{proposition}\label{prop2cocinv}
\begin{enumerate}[{\rm (1)}]
\item
The function $\tau$ defined by \eqref{form20} is a 2-cocycle on $I_G(\theta)$.
\item\label{prop2cocinvb}
The cocycle $\tau$ is bi-$N$-invariant: 
$\tau(nhn',mkm')=\tau(h,k)$, for all $h,k\in I_G(\theta)$ and $n,n',m,m'\in N$.
\end{enumerate}
\end{proposition}
\begin{proof}
(1) From the condition $\Theta(1_G)=I_V$ and \eqref{form20} it follows that $\tau(k,1_G)=\tau(1_G,h)=1$, for $k,h\in I_G(\theta)$. Moreover, from \eqref{form20} it also follows that, for all $h_1,h_2,h_3\in I_G(\theta)$,
\[
\begin{split}
\tau(h_1h_2,h_3)\tau(h_1,h_2) I_V &=\Theta(h_1h_2h_3)\Theta(h_3)^{-1}\Theta(h_1h_2)^{-1}\cdot \Theta(h_1h_2)\Theta(h_2)^{-1}\Theta(h_1)^{-1}\\
&=\Theta(h_1h_2h_3)\Theta(h_2h_3)^{-1}\tau(h_2,h_3)\Theta(h_1)^{-1}\\
&=\tau(h_1,h_2h_3)\tau(h_2,h_3) I_V.
\end{split}
\]

\noindent
(2) Let $h,k\in I_G(\theta)$. Given $n,m\in N$, from
\begin{align*}
\tau(nh,k)\Theta(nh)\Theta(k)&=\Theta(nhk)&(\text{by }\eqref{form20})\\
&=\tau(h,k)\theta(n)\Theta(h)\Theta(k)&(\text{by }\eqref{form18bis}\text{ and }\eqref{form20})\\
&=\tau(h,k)\Theta(nh)\Theta(k)&(\text{by }\eqref{form18bis})
\end{align*}
it follows that $\tau(nh,k)=\tau(h,k)$, while from
\begin{align*}
\tau(h,k)\Theta(hmk)&=\tau(h,k)\tau(h,mk)\Theta(h)\Theta(mk)&(\text{by }\eqref{form20})\\
&=\tau(h,k)\tau(h,mk)\Theta(h)\theta(m)\Theta(k)&(\text{by }\eqref{form18bis})\\
&=\tau(h,k)\tau(h,mk)\theta(hmh^{-1})\Theta(h)\Theta(k)&(\text{by }\eqref{form22})\\
&=\tau(h,mk)\theta(hmh^{-1})\Theta(hk)&(\text{by }\eqref{form20})\\
&=\tau(h,mk)\Theta(hmk),&(\text{by }\eqref{form18bis})
\end{align*}
we deduce that $\tau(h,mk) = \tau(h,k)$. This shows left-$N$-invariance of $\tau$.
Given $n',m' \in N$ we can find $n,m \in N$ such that $hn' = nh$ and $km' = mk$. Using left-$N$-invariance we have
\[
\tau(hn',km') = \tau(nh,mk) = \tau(h,k)
\]
and right-$N$-invariance follows as well.
\end{proof}

\begin{corollary}
\label{c:form18tris}
For all $h\in I_G(\theta)$ and $n\in N$
\begin{equation}\label{form18tris}
\Theta(hn)=\Theta(h)\theta(n).
\end{equation}
\end{corollary}
\begin{proof}
We have $\Theta(hn)=\tau(h,n)\Theta(h)\theta(n)$ and $\tau(h,n)=\tau(h,1_G)=1$.
\end{proof}

As a consequence of Proposition \ref{prop2cocinv}.(2), we define a unitary cocycle 
$\eta\in\mathcal{C}(\mathbb{T},I_G(\theta)/N)$ by setting
\begin{equation}\label{defeta}
\eta(hN,kN)=\tau(h,k),
\end{equation}
for all $h,k\in I_G(\theta)$. 

For the following proposition we keep the notation above, referring to a fixed $Q\subset I_G(\theta)$ complete set of representatives of the $N$-cosets in $I_G(\theta)$. We want to show that $\eta$ is indipendent  of the choices of $Q$ and $\Theta$.

\begin{proposition}
\label{p:q-q'}
Let $Q' \subset I_G(\theta)$ be another complete set of representatives of the $N$-cosets in $I_G(\theta)$ (possibly, $Q' = Q$), 
and denote by $\Theta'(q'), q'\in Q'$, a family of unitary operators on $V$ satisfying
\begin{equation}
\label{e:equazione'}
^{q'}\!\theta(n)=\Theta'(q')^{-1}\theta(n)\Theta'(q')
\end{equation}
for all $n \in N$ $($cf.\ \eqref{form16}$)$. Also set 
\begin{equation}\label{form18quater}
\Theta'(nq')=\theta(n)\Theta'(q')
\end{equation}
for all $n \in N$ and $q' \in Q'$ $($cf.\ \eqref{form18}$)$.
Then the corresponding unitary cocycle $\eta'$ is cohomologous to $\eta$.
\end{proposition}
\begin{proof}
For each $q' \in Q'$ there exist unique $n_{q'} \in N$ and $q \in Q$ such that
\begin{equation}
\label{e:equazione-q'}
q' = n_{q'}q.
\end{equation}
By \eqref{form18} we then have
\begin{equation}
\label{e:equazione2-q'}
\Theta(q') = \Theta(n_{q'}q) =  \theta(n_{q'})\Theta(q).
\end{equation}
Moreover,
\begin{equation}
\label{e:resta}
\begin{split}
^{q'}\!\theta(n) & = \theta(q^{-1} n_{q'}^{-1} n n_{q'}q) \\
& = \Theta(q)^{-1} \theta(n_{q'}^{-1} n n_{q'}) \Theta(q) \\
& = \Theta(q)^{-1} \theta(n_{q'})^{-1} \theta(n) \theta(n_{q'}) \Theta(q)\\
\mbox{(by \eqref{e:equazione2-q'})} \ \ & = \Theta(q')^{-1} \theta(n) \Theta(q').
\end{split}
\end{equation}
Comparing \eqref{e:equazione'} and \eqref{e:resta}, from Schur's lemma we deduce that there exists
$\psi \colon I_G(\theta) \to \TT$ such that
\[
\Theta'(h) = \psi(h) \Theta(h)
\]
for all $h \in I_G(\theta)$. Note that $\psi$ is bi-$N$-invariant: $\psi(nq') = \psi(q') = \psi(q'n) $ for all $n \in N$ and $q' \in Q'$,
where the second equality follows from Corollary \ref{c:form18tris}. In particular, $\psi(1_G)=1$. 
Then for $k,h\in I_G(\theta)$
\begin{equation}\label{taucohomol}
\begin{split}
\tau'(k,h) I_V &=\Theta'(h)^{-1}\Theta'(k)^{-1}\Theta'(kh)=\Theta(h)^{-1}\Theta(k)^{-1}\Theta(kh)\psi(h)^{-1}\psi(k)^{-1}\psi(kh)\\
&=\tau(k,h)\psi(h)^{-1}\psi(k)^{-1}\psi(kh) I_V,
\end{split}
\end{equation}
that is, $\tau$ and $\tau'$ are cohomologous. 
Finally, setting $\rho(hN)=\psi(h)$, for all $h\in N$ (cf.\ Equation \eqref{defeta}) we deduce that $\eta$ and $\eta'$ are cohomologous: \eqref{taucohomol} becomes $\eta'(kN,hN)=\eta(kN,hN)\rho(hN)^{-1}\rho(kN)^{-1}\rho(khN)$. 
\end{proof}

\begin{remark}
\label{r:CoreqcocPsi}
Let $\Theta \colon I_G(\theta) \to \End(V)$ be as in \eqref{form16} and denote by $\tau \colon I_G(\theta) \times I_G(\theta) \to \TT$ the corresponding unitary cocycle as in \eqref{form20}. For each $k \in I_G(\theta)$ let $\tau'(k) = \tau'(k^{-1})$ be a square root of
$\tau(k,k^{-1}) = \tau(k^{-1},k)$ and set
\begin{equation*}
\label{e:theta-primo}
\Theta'(k) = \tau'(k) \Theta(k).
\end{equation*}
Then, from \eqref{starr} we deduce
\[
\Theta'(k)^{-1} = \left(\tau'(k) \Theta(k)\right)^{-1} = \tau'(k)^{-1} \tau(k,k^{-1}) \Theta(k^{-1}) =
\tau'(k^{-1}) \Theta(k^{-1}) = \Theta'(k^{-1}).
\]
Note that $\Theta$ and $\Theta'$ give rise to cohomologous cocycles $\eta$ and $\eta'$, by Proposition \ref{p:q-q'}, but  $\eta'$ is equalized; see also \eqref{NS:asterisco39}.
\end{remark}

\subsection{A description of the Hecke algebra $\tilde{{\mathcal H}}(G,N,\theta)$}
\label{s:description}

In the present setting, the Hecke algebra in Section \ref{s:MFIR} is made up of all functions 
$F \colon G \to \End(V)$ such that
\begin{equation}\label{Fngm}
F(ngm) = \theta(m^{-1})F(g)\theta(n^{-1}),\qquad \text{for all }\quad g \in G,  n,m \in N.
\end{equation}
We also suppose that the unitary cocycle $\eta$ is equalized; see Remark \ref{r:CoreqcocPsi}.

In the following theorem we prove the normal subgroup version of Mackey's formula for invariants.

\begin{theorem}\label{t:teorema1.9}
For each $F\in\widetilde{\mathcal{H}}(G,N,\theta)$ there exists $f\in L(I_G(\theta)/N)$ such that:
\begin{equation}\label{FhPsihfhN}
F(h)=\frac{1}{\lvert N\rvert}\Theta(h)^*f(hN),
\end{equation}
for all $h\in I_G(\theta)$, while $F(g)=0$ for $g\notin I_G(\theta)$. Moreover, the map
\begin{equation}\label{mapFhPsihfhN}
\begin{array}{cccc}
\Phi \colon & L(I_G(\theta)/N)_\eta&\longrightarrow&\widetilde{\mathcal{H}}(G,N,\theta)\\
&f&\longmapsto&F,
\end{array}
\end{equation}
where $F$ is as in \eqref{FhPsihfhN}, is a *-isomorphism of algebras such that $\sqrt{|N|}\Phi$ is isometric.
\end{theorem}
\begin{proof}
From \eqref{form18bis} and \eqref{form18tris} it follows that the function $I_G(\theta)\ni h\mapsto\Theta(h)^*\in\End(V)$ belongs to $\widetilde{\mathcal{H}}(G,N,\theta)$. Moreover, from \eqref{form16bis} it follows that for each $h\in I_G(\theta)$ the operator $\Theta(h)^*$ belongs to $\Hom_N( ^{h^{-1}}\!\theta,\theta)$ and in fact spans it: this follows from the fact that this space is one-dimensional (by Schur's lemma). Recall also that by \eqref{starr} and by Remark \ref{r:CoreqcocPsi} we may suppose that 
\begin{equation}
\label{e:tieni-in-avount}
\Theta(h)^* \equiv \Theta(h)^{-1} = \Theta(h^{-1})
\end{equation} 
for all $h \in I_G(\theta)$.
 Similarly, from \eqref{Fngm} it follows that if $F\in\widetilde{\mathcal{H}}(G,N,\theta)$ then
\[
\theta(n)F(g)=F(gn^{-1})=F(gn^{-1}g^{-1}\cdot g)=F(g)\left[ ^{g^{-1}}\!\theta\right](n),
\]
that is, $F(g)\in\Hom_N( ^{g^{-1}}\!\theta,\theta)$. In particular, $F(g)=0$ if $g\notin I_G(\theta)$ and there exists $f\in L(I_G(\theta)/N)$ such that $F$ is of the form \eqref{FhPsihfhN}. Indeed, $F(h)=\widetilde{f}(h)\Theta(h)^*$ for some constant $\widetilde{f}(h)\in\mathbb{C}$ and from
\[
\theta(n^{-1})\widetilde{f}(h)\Theta(h)^*=\theta(n^{-1})F(h)=F(hn)=\widetilde{f}(hn)\Theta(hn)^*=\theta(n^{-1})\widetilde{f}(hn)\Theta(h)^*
\]
we get that $\widetilde{f}$ in $N$-invariant; we then set $f(hN)=\lvert N\rvert \widetilde{f}(h)$ 
(where $\lvert N\rvert$ is a normalization constant). Conversely, any function of the form \eqref{FhPsihfhN} clearly belongs to $\widetilde{\mathcal{H}}(G,N,\theta)$.

Let $f_1, f_2 \in L(I_G(\theta)/N)$ and $k\in I_G(\theta)$.
Let us show that the map $\Phi$  preserves the convolution. 
Taking into account \eqref{e:star-pag8-} and \eqref{form20}), we have: 
\begin{align*}
[\Phi(f_1)*\Phi(f_2)](k)&=\frac{1}{\lvert N\rvert^2}\sum_{h\in I_G(\theta)}f_1(h^{-1}kN)f_2(hN)\Theta(k^{-1}h)\Theta(h^{-1})\\
&=\frac{1}{\lvert N\rvert^2}\sum_{h\in I_G(\theta)}f_1(h^{-1}kN)f_2(hN)\Theta(k^{-1})\tau(k^{-1},h)\Theta(h)\Theta(h^{-1})\\
&=_*\frac{1}{\lvert N\rvert}\Theta(k)^*\sum_{hN\in I_G(\theta)/N}f_1(h^{-1}N\cdot kN)f_2(hN)\eta(k^{-1}N,hN) \\
&= \frac{1}{\lvert N\rvert}\Theta(k)^*\left[f_1*_\eta f_2\right](kN)\\
& = \Phi(f_1*_\eta f_2)(k),
\end{align*}
where $=_*$ follows from \eqref{e:tieni-in-avount} and \eqref{defeta}).
Similarly, 
\begin{align*}
\langle \Phi(f_1),\Phi(f_2) \rangle_{\widetilde{\mathcal{H}}(G,N,\theta)}&=\sum_{h\in I_G(\theta)}\frac{1}{\dim V}\tr\left[\Phi(f_2)(h)^*\Phi(f_1)(h)\right]&(\text{by }\eqref{e:star-pag8})\\
&=\sum_{h\in I_G(\theta)}\frac{1}{\lvert N\rvert^2\dim V}f_1(hN)\overline{f_2(hN)}\tr\left[\Theta(h)\Theta(h)^*\right]&\\
&=\sum_{hN\in I_G(\theta)/N}\frac{1}{\lvert N\rvert}f_1(hN)\overline{f_2(hN)}\\
&=\frac{1}{\lvert N\rvert} \langle f_1,f_2\rangle_{L(I_G(\theta)/N)},
\end{align*}
so that the map $\sqrt{|N|}\Phi$ is an isometry.

We are only left to show that $\Phi$ preserves the involutions. 
Let $f\in L(I_G(\theta)/N)$ and $k\in I_G(\theta)$. Keeping in mind \eqref{e:star-pag8+} and \eqref{e:tieni-in-avount} we have
\[
\Phi(f)^*(k)=[\Phi(f)(k^{-1})]^*=\frac{1}{{\lvert N\rvert}}\Theta(k^{-1})\overline{f(k^{-1}N)}=\frac{1}{{\lvert N\rvert}}\Theta(k)^*f^*(kN) = \Phi(f^*)(k).
\]
\end{proof}
\begin{corollary}\label{corollarypag46} $\dim\End_G(\Ind_N^G \theta) = |I_G(\theta)/N|$.
\end{corollary}
Moreover, from Theorem  \ref{isomtildeH}, we can also deduce the following 
\begin{corollary}\label{corollary2pag46}
 The composition $\xi\circ \Phi \colon L(I_G(\theta)/N)_\eta  \to \End_G(\Ind_N^GV)$ is indeed an  isometric *-isomorphism of algebras.
\end{corollary}

\subsection{The Hecke algebra $\widetilde{{\mathcal H}}(G,N,\psi)$}
\label{s:H-and-sf}

In order to describe, in the present framework, the Hecke algebra $\mathcal{H}(G,N,\psi)$ (cf.\ Section \ref{s:HaR}), we introduce the function $\Psi \colon I_G(\theta) \to \CC$ by setting
\begin{equation}\label{estar:ACT3}
\Psi(h) = \langle v, \Theta(h)v\rangle 
\end{equation}
for all $h\in I_G(\theta)$, where $v \in V$ is a fixed vector with $\|v\|= 1$, and $\Theta \colon I_G(\theta) \to \End(V)$ is as in Section \ref{Subsecinercocy}.
Clearly, in the notation of \eqref{defpsi}, we have, keeping in mind \eqref{form18}, 
\begin{equation}\label{estar2:ACT3}
\psi(n) = \frac{d_\theta}{|N|}\Psi(n)
\end{equation}
for all $n \in N$.

\begin{theorem}
\label{TA:ACT3}
We have 
\begin{equation}
\label{e:perche}
\Psi = \Psi*\psi = \psi*\Psi = \psi*\Psi*\psi
\end{equation}
and, for all $k,h \in I_G(\theta)$,
\begin{equation}\label{estar:ACT4}
\sum_{n \in N}\Psi(knh)\overline{\psi(n)} = \overline{\tau(k,h)}\Psi(k)\Psi(h).
\end{equation}
\end{theorem}
\begin{proof}
Let $k \in I_G(\theta)$. Then
\[
\begin{split}
[\Psi* \psi](k) & = \sum_{n \in N}\Psi(kn^{-1})\psi(n)\\
& =\sum_{n \in N}\langle v, \Theta(kn^{-1}) v\rangle \frac{d_\theta}{|N|}\cdot \langle v, \theta(n) v\rangle\\
(\mbox{by \eqref{form18tris} and \eqref{e:tieni-in-avount}})\  & = \left\langle \Theta(k^{-1})v, \frac{d_\theta}{|N|}\sum_{n \in N}\langle\theta(n)v, v\rangle\theta(n^{-1})v\right\rangle \\
(\mbox{by \eqref{projvivj}})\   & = \langle \Theta(k^{-1})v, v\rangle\\
& = \Psi(k).
\end{split}
\]
This shows the first equality in \eqref{e:perche}, the other ones follow after similar computations.
 
In order to show \eqref{estar:ACT4}, let $k,h \in I_G(\theta)$ and $n \in N$.
We first note that
\begin{equation}\label{estar:ACT5}
\begin{split}
\Theta(knh) & = \Theta(knk^{-1}\cdot kh)\\
(\mbox{by  \eqref{form18bis}}) \ & = \theta(knk^{-1})\Theta(kh)\\
(\mbox{by \eqref{form22}}) \ &= \Theta(k)\theta(n)\Theta(k)^{-1}\Theta(kh)\\
(\mbox{by \eqref{form20}}) \ & = \tau(k,h)\Theta(k)\theta(n)\Theta(h).
\end{split}
\end{equation}
We deduce that
\[
\begin{split}
\sum_{n \in N}\Psi(knh)\overline{\psi(n)} & = \sum_{n \in N}\langle v, \Theta(knh)v\rangle \frac{d_\theta}{|N|}\langle \theta(n)v,v\rangle \\
(\mbox{by \eqref{estar:ACT5}}) \ & = \overline{\tau(k,h)}\left\langle \Theta(k^{-1})v, \frac{d_\theta}{|N|}\sum_{n\in N}\langle \theta(n^{-1})v,v\rangle \theta(n)\Theta(h)v\right\rangle\\
(\mbox{by \eqref{estar:ACT2}}) \ &=  \overline{\tau(k,h)}\left\langle \Theta(k^{-1})v, \langle \Theta(h)v,v\rangle v \right\rangle\\
& = \overline{\tau(k,h)}\Psi(k)\Psi(h).
\end{split}
\]
\end{proof}

For every $\phi \in \widehat{I_G(\theta)/N}$ denote by $\widetilde{\phi}$ its {\it inflation} to  $I_G(\theta)$. This is defined by setting
$\widetilde{\phi}(h) = \phi(hN)$ for all $h \in I_G(\theta)$ (i.e.\ we compose the projection map $I_G(\theta) \to I_G(\theta)/N$ with 
$\phi \colon I_G(\theta)/N \to  \End(V_\phi)$).

\begin{proposition}\label{pB:ACT7}
We have 
\[
\mathcal{H}(G,N,\psi)= \{\widetilde{\phi}\Psi: \phi \in L(I_G(\theta)/N)\}
\]
where
\[
[\widetilde{\phi}\Psi](h) = \phi(hN)\Psi(h)
\] 
for all $h \in H$. In particular, every $f \in \mathcal{H}(G,N,\psi)$ vanishes outside $I_G(\theta)$.
\end{proposition}
\begin{proof}
From Theorem  \ref{pD2}.(3)  we deduce that $\mathcal{H}(G,N,\psi)$ is made up of all $S_vF$, where $F$ is as in  \eqref{FhPsihfhN}
with $f$ replaced by $\phi$, that is, for all $h \in I_G(\theta)$, 
\[
\begin{split}
[S_v F](h) & = d_\theta\langle F(h)v,v\rangle\\
& = \frac{d_\theta}{|N|}\phi(hN)\langle v, \Theta(h)v\rangle\\
& = \frac{d_\theta}{|N|}\widetilde{\phi}(h)\Psi(h).
\end{split}
\]
For the last statement, recall that a function  $F \in\widetilde{\mathcal{H}}(G,N,\psi)$ vanishes outside $I_G(\theta)$, 
by Theorem \ref{t:teorema1.9}.
\end{proof}

\subsection{The multiplicity-free case and the spherical functions}
\label{s:MFcase-sf}
In this section we study the multiplicity-free case of a triple $(G,N,\theta)$ and determine the associated spherical functions.

We first introduce another notation from Clifford theory (see \cite{CSTCli}  and \cite{book3}):
\[
\widehat{I}(\theta) = \{\xi \in \widehat{I_G(\theta)}: \xi \preceq \Ind_N^{I_G(\theta)} \theta\}.
\]

\begin{remark}
\label{r:ACCT2}
{\rm Note that if  $(G,N,\theta)$ is a multiplicity-free triple,  then also $(I_G(\theta), N, \theta)$ is multiplicity-free. Indeed the commutant  of $\Ind_N^G\theta$ coincides, by virtue of Theorem \ref{isomtildeH} and Theorem \ref{t:teorema1.9}, with $L(I_G(\theta)/N)_\eta$.
Since $I_{I_G(\theta)}(\theta) = I_G(\theta)$, the previous argument shows that the latter is also the commutant of $\Ind_N^{I_G(\theta)}\theta$. In other words, the commutant depends only on $I_G(\theta)$. 
Alternatively, this also follows from transitivity of induction.}
\end{remark}

\begin{theorem}\label{theoremA:ACCT}
$\Ind_N^G \theta$  decomposes without multiplicities if and only if the following  conditions are satisfied:
\begin{enumerate}[{\rm(i)}]
\item $I_G(\theta)/N$ is Abelian
\item $\Res ^{I_G(\theta)}_N \xi = \theta \mbox{ for all } \xi \in \widehat{I}(\theta).$
\end{enumerate}
\end{theorem}
\begin{proof}
If $\Ind_N^G \theta$ is multiplicity-free then also  $\Ind_N^{I_G(\theta)}\theta$ is multiplicity-free (see Remark \ref{r:ACCT2}). Then 
\begin{equation}\label{estar:ACCT3}
\Ind_N^{I_G(\theta)} \theta  \sim \bigoplus_{\xi \in \widehat{I}(\theta)}\xi.
\end{equation}
By Frobenius reciprocity
\[
\Hom_{I_G(\theta)}(\theta, \Res_N^{I_G(\theta)}\Ind_N^{I_G(\theta)} \theta)   \cong \End_{I_G(\theta)}(\Ind_N^{I_G(\theta)} \theta).
\]
By taking dimensions, the multiplicity of $\theta$ in $\Res^{I_G(\theta)}\Ind_N^{I_G(\theta)} \theta$ equals
\begin{equation}\label{estar:ACCT4}
\dim \End_{I_G(\theta)}(\Ind_N^{I_G(\theta)}\theta) = |I_G(\theta)/N|
\end{equation}
by Corollary \ref{corollarypag46}.
Again by computing dimension, since $\dim \Ind_N^{I_G(\theta)}\theta = |I_G(\theta)/N| \cdot \dim \theta$, we deduce that indeed 
\begin{equation}\label{equad:ACCT5}
\Res^{I_G(\theta)}_N \Ind^{I_G(\theta)}_N \theta \sim |I_G(\theta)/N| \theta.
\end{equation}
But \eqref{estar:ACCT3} and \eqref{estar:ACCT4} force $|\widehat{I}(\theta)| = |I_G(\theta)/N|$ (cf.\ \eqref{isomIndmW} and 
\eqref{isomHomIndMat} for $m_\sigma = 1$) and therefore,  by \eqref{equad:ACCT5}, necessarily
$\Res^{I_G(\theta)}_N \xi = \theta$ for all $\xi \in \widehat{I}(\theta)$, and (ii) is proved.

Then, for all $\xi \in \widehat{I}(\theta)$, we have 
\[
\begin{split}
\Ind_N^{I_G(\theta)}\theta &  \sim  \Ind_N^{I_G(\theta)} (\theta \otimes \iota_N)\\
(\mbox{by (ii)}) \ & \sim \Ind_N^{I_G(\theta)}[(\Res^{I_G(\theta)}_N \xi)\otimes  \iota_N]\\
(\mbox{by \cite[Theorem 11.1.16]{book4}}) \ & \sim \xi \otimes \Ind_N^G\iota_N\\
& \sim \xi \otimes \widetilde{\lambda}.
\end{split}
\]
From the decomposition of the regular representation of $I_G(\theta)/N$
\[
\lambda \sim \bigoplus_{\phi \in \widehat{I_G(\theta)/N}} d_\phi\phi
\]
we deduce that
\begin{equation}\label{estar:ACCT7}
\Ind_N^{I_G(\theta)}\theta \sim \bigoplus_{\phi\in  \widehat{I_G(\theta)/N}} d_\phi \xi \otimes \widetilde{\phi}
\end{equation}
 so that  \eqref{estar:ACCT3} forces $d_\phi \equiv 1$ for all $\phi \in \widehat{I_G(\theta)/N}$ and therefore  $I_G(\theta)/N$ is Abelian (see \cite[Exercise 10.3.16]{book4}). So that also (i) is proved.

We now assume (i) and (ii). Then in \eqref{form16} we can take 
\begin{equation}\label{estar2:ACCT7}
\Theta = \xi
\end{equation}
choosing any $\xi \in \widehat{I}(\theta)$. Therefore in \eqref{form20} we have $\tau \equiv 1$ so that also $\eta \equiv 1$ in Theorem \ref{t:teorema1.9}. Then (cf.\ Corollary \ref{corollary2pag46}) 
\[
\End_{I_G(\theta)}(\Ind^{I_G(\theta)}_N \theta) \cong L(I_G(\theta)/N)
\]
with the usual convolution structure, so that it is commutative and $\Ind_N^{I_G(\theta)}\theta$ is multiplicity-free (cf.\  Proposition \ref{p:equiv-MF-pre}) and therefore  (see Remark \ref{r:ACCT2}) also $\Ind_N^G\theta$ is multiplicity-free.
\end{proof}
\begin{remark}\label{r:ACCT8}
{\rm The proof of Theorem  \ref{theoremA:ACCT} is based on Theorem \ref{t:teorema1.9} and its corollaries, that is, on the normal subgroup version of Mackey's formula for invariants. Another proof, based on Clifford Theory, which is a normal subgroup version of Mackey's lemma, is in \cite{CST4}. In that paper we plan to develop a Fourier analysis of the algebra $L(H)_\eta$ (see Section \ref{ss:cocycle}) in the noncommutative  case and therefore  also of $\End_G(\Ind_N^G \theta)$ when $\Ind_N^G\theta$ is not multiplicity-free (cf.\ \cite{st5} for a different approach to the analysis of $\Ind_K^G\theta$, where $K \leq G$ is not necessarily a normal subgroup and the representation decomposes with multiplicities).}
\end{remark}

In what follows, $\chi \in \widehat{I_G(\theta)/N}$ denotes a character of the Abelian group $I_G(\theta)/N$ and $\widetilde{\chi}$ its inflation to  $I_G(\theta)$, that is, $\widetilde{\chi}(h) = \chi(hN)$ for all $h \in I_G(\theta)$, while 
$\overline{\chi}$ is the {\it complex conjugate}, that is, $\overline{\chi}(hN) = \overline{\chi(hN)}$ (as a complex number).

Recall (cf.\ Definition \ref{def:spher-f}) that ${\mathcal S}(G,N,\psi) \subseteq {\mathcal H}(G,N,\psi)$ denotes the set of all
spherical functions associated with the multiplicity-free triple $(G,N,\theta)$.

\begin{proposition}\label{pD:ACT9}
Let $(G,N,\theta)$ be a multiplicity-free triple. Then 
\[
{\mathcal S}(G,N,\psi) = \{\widetilde{\chi}\Psi: \chi \in \widehat{I_G(\theta)/N}\}.
\] 
\end{proposition}
\begin{proof}
Recall that a function $\phi \in L(G)$ is spherical if and only if it satisfies \eqref{funcidensph} and, moreover, if this
is the case, $\phi$ must be of the form $\widetilde{\varphi} \Psi$ for some $\varphi\in L(I_G(\theta)/N)$, by Proposition \ref{pB:ACT7}. 

We then show that, given $\varphi \colon I_G(\theta)/N \to \TT$, then $\phi = \widetilde{\varphi} \Psi$ satisfies \eqref{funcidensph} if and only if 
$\varphi$ is a character. This follows immediately after comparing 
\[
\sum_{n \in N}\phi(knh)\overline{\psi(n)} = \sum_{n \in N}\varphi(knhN)\Psi(knh)\overline{\psi(n)}
= \varphi(khN)\Psi(k)\Psi(h)
\]
(where the last equality follows from \eqref{estar:ACT4} with $\tau \equiv 1$) and the last expression is equal to $\varphi(kN)\Psi(k)\varphi(hN)\Psi(h)$ if and only if $\varphi$ is a character belonging to $\widehat{I_G(\theta)/N}$.
\end{proof}

\begin{theorem}\label{theoremC:ACCT}
Suppose that $(G,N,\theta)$ is a multiplicity-free triple. Fix $\xi\in \widehat{I}(\theta)$ and $v\in V$ with $\|v\| =1$, and set $\Psi(h) = \langle v, \xi(h) v\rangle$ for all $h \in I_G(\theta)$ (cf.\  \eqref{estar:ACT3} and \eqref{estar2:ACCT7}).
Then  the following hold.
\begin{enumerate}[{\rm (1)}]
\item The map
\[
\begin{array}{ccc}
L(I_G(\theta)/N) &  \to &  {\mathcal H}(G,N,\psi)\\
f & \mapsto & \frac{d_\theta}{|I_G(\theta)|}\widetilde{f} \Psi
\end{array}
\]
is an isomorphism of commutative algebras.
\item 
\begin{equation}\label{estar:ACCT10}
\Ind_N^G \theta \sim \bigoplus_{\chi \in \widehat{I_G(\theta)/N}}\Ind_{I_G(\theta)}^G(\xi \otimes \widetilde{\chi})
\end{equation}
is the decomposition of $ \Ind_N^G \theta$ into irreducibles (that is, into the spherical representations of the multiplicity-free triple $(G,N,\theta)$).
\item The spherical function associated with  $\Ind_{I_G(\theta)}^G(\xi \otimes \widetilde{\chi})$ is given by 
\[
\phi^\chi =\widetilde{\overline{\chi}}\Psi.
\]
\end{enumerate}
\end{theorem}

\begin{proof}
(1) This follows from Proposition  \ref{pB:ACT7}, by taking into account that now $\eta \equiv 1$ (see the proof of Theorem \ref{theoremA:ACCT}), that $I_G(\theta)/N$ is commutative, and that $\Psi * \Psi = \frac{|I_G(\theta)|}{d_\theta} \Psi$ (cf.\ \eqref{CON}).
Note also that, in order to compute the convolution in ${\mathcal H}(G,N,\psi)$, we may use the expression for the spherical functions in Proposition \ref{pD:ACT9}, \eqref{CON}, and \eqref{estar2:ACCT7} applied to the $I_G(\theta)$-representations $\xi \otimes \chi$, with
$\chi \in \widehat{I_G(\theta)/N}$. Alternatively, this isomorphism is the composition of the map \eqref{mapFhPsihfhN} with the map $S_v \colon \widetilde{\mathcal{H}}(G,K, \theta)\longrightarrow L(G)$ in
Theorem \ref{pD2}.(3) (cf.\ the proof of Proposition \ref{pD:ACT9}).

(2) The decomposition  \eqref{estar:ACCT10} follows from transitivity of induction and \eqref{estar:ACCT7}, taking into account that $d_\phi = 1$ and using $\chi$ to denote a generic character of $I_G(\theta)/N$. Moreover, 
$$\dim \End_G(\Ind_N^G \theta) = |I_G(\theta)/N| = \dim \End_{I_G(\theta)}(\Ind_N^{I_G(\theta)}\theta)$$
(see Corollary  \ref{corollarypag46},  Corollary \ref{corollary2pag46}, and Remark \ref{r:ACCT2}) forces each representation in the right hand side of \eqref{estar:ACCT10} be irreducible.

(3) From Proposition \ref{pD:ACT9} applied  to  $(I_G(\theta), N, \theta)$ and  \eqref{estar:ACT3} (with $\Theta$ replaced  by $\xi$; see \eqref{estar2:ACCT7}) we deduce that  the function
\[
H \ni h   \mapsto \widetilde{\overline{\chi}}(h) \Psi(h) =  \widetilde{\overline{\chi}}(h) \langle v, \xi(h)v\rangle = \langle v,    \widetilde{\chi}(h)\xi(h)v\rangle 
\]
is the spherical functions of $(I_G(\theta), N, \theta)$ associated with $\xi \otimes \widetilde{\chi}$ (cf.\  \eqref{defphisigma}; now $L_\sigma$ is trivial because $\Res^{I_G(\theta)}_N \xi \sim \theta$ act on the same space $V_\theta$).
From Proposition  \ref{pD:ACT9} it follows that  $\mathcal{S}(G,N,\psi) \equiv \mathcal{S}(I_G(\theta), N,\psi)$ (recall also that each element of ${\mathcal{H}}(G,N, \psi)$ vanishes  outside $I_G(\theta)$; cf.\ Proposition \ref{pB:ACT7}).
Moreover, from the characterization of the spherical representations in Theorem \ref{Thmspherapp} and the characterization (or definition) of induced representations in \cite[Proposition 11.1.2]{book4} it follows that the spherical function associated with 
$\Ind_{I_G(\theta)}^G(\xi \otimes \widetilde{\chi})$ coincides with that associated with $\xi \otimes \widetilde{\chi}$.
\end{proof}

\section{Harmonic Analysis of the multiplicity-free triple $(\GL(2,\FF_q), C, \nu)$} 
\label{s:HAMFT1}
In this section we study a family  of multiplicity-free triples on $\GL(2, \FF_q)$ that generalize the well known Gelfand pair associated  with the finite hyperbolic plane (see \cite[Chapters 19, 20, 21, and 23]{Terras}). We suppose that $q$ is an odd prime power (cf.\ Section \ref{ss:GGrep}) and we denote by $\widehat{\FF_q^*}$ (respectively $\widehat{\FF_{q^2}^*}$) the multiplicative characters of $\FF_q$ 
(respectively $\FF_{q^2}$). If $\nu \in \widehat{\FF_{q^2}^*}$, we may think of $\nu$ as a one-dimensional representation of the Abelian group $C$ by setting (cf.\ \eqref{estar:PTPT3}) 
\[
\nu\begin{pmatrix} \alpha & \eta \beta\\
\beta & \alpha 
\end{pmatrix}=  \nu(\alpha+i\beta),
\]
where $\alpha, \beta \in \FF_q$, $(\alpha,\beta) \neq (0,0)$, $\eta$ is a generator of the
multiplicative group $\FF_q^*$, and  $\pm i$ are the square roots (in $\FF_{q^2}$) of $\eta$.

\subsection{The multiplicity-free triple $(\GL(2,\FF_q), C, \nu)$}

We begin with an elementary, but quite useful lemma.

\begin{lemma}\label{l1:PTPT5}
\begin{enumerate}[{\rm (1)}]
\item Every $\begin{pmatrix} \alpha &  \beta\\
\gamma & \delta 
\end{pmatrix} \in \GL(2, \FF_q)$ may be written uniquely as the product of a matrix in $\Aff(\FF_q)$  and a matrix in $C$, namely
\[
\begin{pmatrix} \alpha &  \beta\\
\gamma & \delta 
\end{pmatrix} = \begin{pmatrix} x & y \\0 & 1 \end{pmatrix} \begin{pmatrix} a & \eta b \\ b & a \end{pmatrix}, 
\]
where $a= \delta$, $b = \gamma$, $x = \frac{\alpha\delta- \beta \gamma}{\delta^2-\eta\gamma^2}$,
$y = \frac{\beta\delta- \alpha\gamma\eta}{\delta^2-\eta\gamma^2}$.
\item
Every $\begin{pmatrix} a & b \\c & d \end{pmatrix}\in \GL(2, \FF_q)$  may be written uniquely as the product of a matrix in 
$C$ by a matrix in  $\Aff(\FF_q)$, namely
\[
\begin{pmatrix} a &  b\\
c &  d 
\end{pmatrix} =\begin{pmatrix} \alpha & \eta \beta \\ \beta & \alpha \end{pmatrix}  \begin{pmatrix} x & y \\0 & 1 \end{pmatrix}.
\]
\item
For all $x, y,\alpha, \beta\in \FF_q$ with $x \neq 0$ and $(\alpha, \beta) \neq (0,0)$ we have
\begin{equation}\label{estar:PTPT6}
 \begin{pmatrix} x & y \\0 & 1 \end{pmatrix} \begin{pmatrix} \alpha & \eta \beta \\ \beta & \alpha \end{pmatrix}
= \begin{pmatrix} u & \eta v  \\ v & u \end{pmatrix} \begin{pmatrix} a & b \\0 & 1 \end{pmatrix},
\end{equation}
 where 
\begin{equation}\label{estar2:PTPT6}
\begin{split}
u & = x (x\alpha+ y\beta)\frac{\alpha^2-\beta^2\eta}{(x\alpha+ y\beta)^2-\eta\beta^2}\\
v & = x\beta \frac{\alpha^2- \eta\beta^2}{(x\alpha+ y\beta)^2-\eta\beta^2}\\
a & = \frac{(x\alpha+y\beta)^2-\eta\beta^2}{x(\alpha^2-\eta\beta^2)}\\
b & = \frac{(x\alpha+y\beta)(y\alpha+\eta\beta x)-\eta\alpha\beta}{x(\alpha^2-\eta\beta^2)}.
\end{split}
\end{equation}
\end{enumerate}
\end{lemma}

\begin{proof} The proof consists just of elementary but tedious computations.

(1) Since  
\[
 \begin{pmatrix} x & y \\0 & 1 \end{pmatrix} \begin{pmatrix} a & \eta b \\ b & a \end{pmatrix}  = \begin{pmatrix}
xa+yb & \eta bx+ya\\ b & a\end{pmatrix}, 
\] 
 we immediately get $b = \gamma$ and $a = \delta$. Moreover, the linear system 
\[
\begin{cases}
x\delta+ y \gamma &  = \alpha\\
x\eta\gamma + y \delta & = \beta
\end{cases}
\]
may be solved by Cramer's rule, yielding the expression for $x$ and $y$ in the statement 
(just note that $\eta$ is not a square, because  $q$ is odd, and therefore  $\delta^2-\eta \gamma^2 \neq 0$).

(2) This follows from (1), simply by taking inverses.

(3) We just sketch the calculations. \eqref{estar:PTPT6} is equivalent to the system
\[
\begin{cases}
ua & = x\alpha+ y \beta\\
ub+ v \eta & = x\beta\eta+ y \alpha\\
va & = \beta \\
vb+ u & = \alpha.
\end{cases}
\]
For the moment, assume $\beta \neq 0$. Then from the third and the fourth equation, we find  
\[
a = v^{-1}\beta \quad b  = v^{-1}\alpha-v^{-1}u
\]
and, from the first equation, we deduce  that 
\[
u = (x\alpha\beta^{-1}+ y)v.
\]
A substitution of these expressions in the second equation leads to the explicit formula  for $v$. Then one can derive the expressions for $u$ and $a$, and finally the expression for $b$. The final formulas also include the case $\beta = 0$. 
\end{proof}

\begin{proposition}\label{p2:PTPT9}
$(\GL(2,\FF_q),C, \nu)$ is a multiplicity-free triple for all $\nu \in  \widehat{\FF_{q^2}^*}$.
\end{proposition}
\begin{proof}
We use the Bump-Ginzburg criterion (Theorem \ref{t:Bump-Ginz}) with $\tau \colon \GL(2,\FF_q) \to \GL(2,\FF_q)$ given by
\[
\tau \begin{pmatrix} a & b\\ c & d \end{pmatrix} =  \begin{pmatrix} d & b\\ c & a\end{pmatrix}
\]
for all $\begin{pmatrix} a & b\\ c & d \end{pmatrix}\in \GL(2,\FF_q)$.
It is easy to check that $\tau$ is an involutive antiautomorphism (cf.\ Theorem \ref{t:cuspidal} and \cite[Theorem 14.6.3]{book4}). Moreover, $\tau(k) = k$ for all $k \in C$ and, consequently, $\nu(\tau(k)) = \nu(k)$ for all $k \in C$ and $\nu \in \widehat{\FF_{q^2}^*}$ (viewed, as remarked above, as a one-dimensional representation of $C$). It only remains to prove conditions \eqref{e:bump1} and \eqref{e:bump2}. Actually, $\mathcal{S}$ is quite difficult to describe and we refer to \cite[pp.\ 311-322]{Terras}, where  a complete geometric description is developed. From our point of view, however, it is enough to know that $\mathcal{S}$ is a subset of $\Aff(\FF_q)$: this follows from  Lemma \ref{l1:PTPT5}.(1) and (2). Therefore we prove conditions \eqref{e:bump1} and \eqref{e:bump2} for all $s \in \Aff(\FF_q)$. 
We shall see that, given $s = \begin{pmatrix} x & y\\ 0 & 1 \end{pmatrix}$, there exists $k = k(x,y) = \begin{pmatrix} \alpha & \eta \beta \\ \beta & \alpha \end{pmatrix}$ such that $\tau(s) = ksk^{-1}$, that is, \eqref{e:bump1} with $k_1 = k$ and $k_2 = k^{-1}$ (then \eqref{e:bump2} is trivially satisfied: in the present setting it becomes  $\nu(k)\nu(k^{-1}) = \nu(1) = 1$).
First of all, we note that 
\[
\tau(s) = \tau \begin{pmatrix} x&y \\ 0 & 1\end{pmatrix}  =  \begin{pmatrix} 1&y \\ 0 & x\end{pmatrix} 
\]
so that $ks = \tau(s)k$ becomes
\[
 \begin{pmatrix} \alpha & \eta \beta \\ \beta & \alpha \end{pmatrix}  \begin{pmatrix} x & y \\0 & 1 \end{pmatrix} = 
 \begin{pmatrix} 1&y \\ 0 & x\end{pmatrix} \begin{pmatrix} \alpha & \eta \beta \\ \beta & \alpha \end{pmatrix}  
\]
which is equivalent to 
\[
\alpha x  = \alpha + \beta y.
\]
Thus, if $x \neq 1$ (respectively $x=1$), $k(x,y)$ is obtained by choosing $\beta$ arbitrarily and then setting $\alpha = (x-1)^{-1}\beta y$
(respectively by choosing $\alpha $ arbitrarily and setting $\beta = 0$).
\end{proof}

Another proof of Proposition \ref{p2:PTPT9} is given in Remark \ref{r5:PTPT24}.

\begin{remark}\label{r3:PTPT11}{\rm The above result may be expressed by saying,
in the terminology of \cite[Section 2.1.2]{book2}, that $C$ is a {\it multiplicity-free} subgroup of $\GL(2,\FF_q)$. 
This is equivalent (cf.\ \cite[Theorem 2.1.10]{book2}) to $(G \times C, \widetilde{C})$ being a  Gelfand pair, 
where $\widetilde{C} = \{(h,h): H \in C\}$. 
We plan to study the Gelfand pair $(G \times C, \widetilde{C})$ in full details in a future paper. 
Other references on this and similar related constructions include \cite[Section 9.8]{book}, \cite{st3}, and \cite{CSTTohoku}.}
\end{remark}

\subsection[Parabolic representations]{Representation theory of $\GL(2,\FF_q)$: parabolic representations}
We now recall some basic facts on the representation theory of the group $\GL(2,\FF_q)$. We refer to \cite[Chapter 14]{book4} for complete proofs. For simplicity, we set $G = \GL(2,\FF_q)$. 

The $(q-1)$ one-dimensional representations are of the form 
\[
\widehat{\chi}_\psi^0(g) = \psi(\det g)
\]
for all  $g \in G$, where $\psi \in \widehat{\FF_q^*}$.

For $\psi_1, \psi_2 \in \widehat{\FF_q^*}$ consider the one-dimensional representation of the Borel subgroup $B$ given by 
\begin{equation}
\label{e:pag53}
\chi_{\psi_1,\psi_2}\begin{pmatrix} \alpha &\beta \\ 0& \delta \end{pmatrix} = \psi_1(\alpha) \psi_2(\delta)
\end{equation}
for all $\begin{pmatrix} \alpha & \beta \\ 0& \delta \end{pmatrix}  \in B$. Then we set 
\[
\widehat{\chi}_{\psi_1, \psi_2} = \Ind_B^G{\chi}_{\psi_1, \psi_2}.
\]
If $\psi_1 \neq \psi_2$ then $\widehat{\chi}_{\psi_1, \psi_2}$ is irreducible  and  $\widehat{\chi}_{\psi_1, \psi_2} \sim  \widehat{\chi}_{\psi_3, \psi_4}$ if and only if 
$\{\psi_1,\psi_2\} = \{\psi_3,\psi_4\}$. If $\psi_1= \psi_2 = \psi$ then 
\[
 \widehat{\chi}_{\psi, \psi} = \widehat{\chi}_\psi^0\oplus \widehat{\chi}_\psi^1
\]
where $ \widehat{\chi}_\psi^0$ is the one-dimensional representation defined above and $\widehat{\chi}_\psi^1$ is a $q$-dimensional irreducible representation. The representations $\widehat{\chi}_{\psi_1, \psi_2}$  and  $\widehat{\chi}_\psi^1$ are called {\it  parabolic representations}: these are $(q-1)(q-2)/2$ representations of dimension $q+1$ and $q-1$ representations of dimension $q$, respectively.

\subsection[Cuspidal representations]{Representation theory of $\GL(2,\FF_q)$: cuspidal representations}
\label{ss:cuspidal}
We now introduce the last class of irreducible representations of $G$.  We need some more preliminary results from \cite[Chapter 7]{book4}.  Let $\nu\in \widehat{\FF_{q^2}^*}$ and set $\nu^\sharp = \Res^{\FF_{q^2}^*}_{\FF_{q}^*} \nu$ and, for 
$\alpha + i \beta \in \FF_{q^2}^* $ (cf.\ \eqref{estar:PTPT3}), set $\overline{\alpha+ i \beta}= \alpha-i\beta$ (conjugation). 
Then, the map $\alpha+ i \beta \mapsto  \alpha-i\beta$ is precisely the unique non-trivial authomorphism of $\FF_{q^2}$ that fixes each element of $\FF_q$. Clearly, $\nu \mapsto \nu^\sharp$ is a group homomorphism and each $\psi \in \widehat{\FF_q^*}$ is the image of $ \frac{ \widehat{\vert\FF_{q^2}^*}\vert}{| \widehat{\FF_q^*}|} = q+1$ characters of $\FF_{q^2}^*$.
For $\psi \in  \widehat{\FF_q^*}$ we set
\begin{equation}\label{estar:PTPT14}
\Psi(w) = \psi(w\overline{w}) \mbox{ for all } w \in \FF_{q^2}^*.
\end{equation}
Then $\Psi \in  \widehat{\FF_{q^2}^*}$ and we say that $\Psi$ is {\it decomposable}.  If $\nu \in  \widehat{\FF_{q^2}^*}$ but cannot be written in this form, it is called {\it indecomposable}. For  $\nu \in  \widehat{\FF_{q^2}^*}$ we set $\overline{\nu}(w) = \nu(\overline{w})$ and we have that $\overline{\nu} \in \widehat{\FF_{q^2}^*}$. Then, $\nu$ is indecomposable if and only if $\nu \neq \overline{\nu}$ 
(warning: as usual, $\overline{\nu(z)}$ will indicate the conjugate of the complex number $\nu(z)$).

Suppose now that ${\nu} \in \widehat{\FF_{q^2}^*}$ is indecomposable and $\chi \in \widehat{\FF_q}$ is a nontrivial {\it additive character} of $\FF_q$. Following the monograph  by Piatetski-Shapiro \cite{PS} (but we refer to \cite[Chapter 7]{book4}) we introduce  the {\it generalized Kloosterman sum} $j = j_{\chi,\nu} \colon \FF_q^* \to \CC$  defined by setting
\begin{equation*}
j(x) = \frac{1}{q} \sum_{\substack{w \in \FF_{\! \! q^{\! 2}}^*:\\w \bar{w} = x}} \chi(w + \bar{w})\nu(w)
\end{equation*}
for all $x \in \FF_q^*$.

We will use repeatedly the following identities (cf.\ \cite[Proposition 7.3.4 and Corollary 7.36]{book4} ):
\begin{equation}\label{estar:PTP16}
\sum_{z \in \FF_q^*} j(xz)j(yz)\nu(z^{-1})=  \delta_{x,y}\nu(-x)
\end{equation}
\begin{equation}\label{estar2:PTP16}
\sum_{z \in \FF_q^*} j(xz)j(yz)\nu(z^{-1})\chi(z) =  -\chi(-x - y)\nu(-1)j(xy)
\end{equation}
for all $x, y\in \FF_q^*$, and 
\begin{equation}\label{estar3:PTP16}
\sum_{\alpha \in \FF_q^*}\nu(-\alpha)\chi(\alpha^{-1}(z + \overline{z}))j(\alpha^{-2}z \overline{z}) = \nu(z) + \nu(\overline{z})
\end{equation}
for all $z \in \FF_{q^2}^*$. Actually, the last identity was obtained during the computation of the characters of the cuspidal representations; cf.\ the end of the proof of the character table in \cite[Section 14.9]{book4} 
(where $\delta = z+\overline{z}$ and $\beta = -z\overline{z}$).
Finally (cf.\ \cite[Proposition 7.3.3]{book4}),
\begin{equation}\label{estar4:PTP16}
\overline{j(x)}= j(x)\overline{\nu(-x)}.
\end{equation}

We now describe the cuspidal representation associated with an indecomposable character $\nu \in \FF_{q^2}^*$: the representation 
space is $L(\FF_q^*)$ and the representation $\rho_\nu$ is defined by setting, for all $g \in G$, $f \in L(\FF_q^*)$, and $y\in \FF_q^*$, 
\begin{equation}\label{cusp1}
[\rho_\nu(g)f](y) = \nu(\delta)\chi(\delta ^{-1}\beta y^{-1})f(\delta \alpha^{-1}y)
\end{equation} if  $g = \begin{pmatrix} \alpha & \beta \\ 0 & \delta \end{pmatrix}\in B$, and 
\begin{equation}\label{cusp2}
[\rho_\nu(g)f](y) = -\sum_{x \in \FF_q^*}\nu(-\gamma x)\chi(\alpha\gamma^{-1}y^{-1}+ \gamma^{-1}\delta x^{-1})
j(\gamma^{-2}y^{-1}x^{-1}\det(g))f(x)
\end{equation}
if $g = \begin{pmatrix} \alpha & \beta\\ \gamma & \delta \end{pmatrix}\in G\setminus B\equiv BwB$
(that is, if $\gamma \neq 0$). 

Actually, $\rho_\nu \sim \rho_\mu$ if and only if  $\nu = \mu$ or $\nu = \overline{\mu}$, so that we have $q(q-1)/2$ cuspidal representations.

\subsection{The decomposition of $\Ind_C^G\nu_0$}
\label{s:deco-prima-trippa}

\begin{theorem}\label{t4:PTPT19}
For all indecomposable characters $\nu_0\in\widehat{\FF_{q^2}^*}$ we have the decomposition
\begin{equation}
\label{e:decompo}
\Ind_C^G\nu_0 \sim \left(\bigoplus_{\substack{ \psi\in \widehat{\FF_q^*}:\\ \psi^2 = \nu_0^\sharp}}\widehat{\chi}^1_\psi\right)\bigoplus \left(\bigoplus_{\substack{\{\psi_1,\psi_2\}\subseteq  \widehat{\FF_q^*}:\\ \psi_1\neq \psi_2, \\ \psi_1\psi_2 = \nu_0^\sharp}}\widehat{\chi}_{\psi_1,\psi_2}\right)\bigoplus \left(\bigoplus_{\substack{\nu\in \widehat{\FF_{q^2}^*}  \mbox{ \tiny indecomposable}:\\ \nu^\sharp = \nu_0^\sharp\\ \nu \neq \nu_0, \overline{\nu_0}}}\rho_\nu\right).
\end{equation}
\end{theorem}
\begin{proof}
The fact that the above decomposition is multiplicity-free follows from Proposition \ref{p2:PTPT9}.  Actually,  we use Frobenius reciprocity, so that we study the restrictions to $C$ of all irreducible $G$-representations. 
From the proof of \cite[Theorem 14.3.2]{book4}, it follows that, for $\beta \neq 0$, the matrices 
\[
\begin{pmatrix}
\alpha &\eta  \beta\\
\beta & \alpha
\end{pmatrix}  \quad \mbox{ and } \quad
\begin{pmatrix}
0 & \eta\beta^2-\alpha^2\\
1 & 2\alpha
\end{pmatrix}
\]
belong to the same conjugacy class of $G$. Moreover, if $z = \alpha+ i \beta$ then  $z+ \overline{z} = 2\alpha$ and $-z\overline{z} = -\alpha^2+ \eta\beta^2$. From the character table in \cite[Table 14.2, Section 14.9]{book4} we have, for $\psi, \psi_1, \psi_2 \in 
\widehat{\FF_q^*}$ and $\nu \in \widehat{\FF_{q^2}^*}$ indecomposable,

\begin{itemize}
\item  $\widehat{\chi}^0_\psi
\begin{pmatrix}
\alpha &\eta  \beta\\
\beta & \alpha
\end{pmatrix} =
\begin{cases}
\psi(z\overline{z}) &  \mbox{ if } \beta \neq 0\\
\psi(\alpha^2) &   \mbox{ if } \beta  = 0 
\end{cases}
\equiv \psi(z\overline{z});$\\

\item $\chi^{\widehat{\chi}^1_\psi}
\begin{pmatrix}
\alpha &\eta  \beta\\
\beta & \alpha
\end{pmatrix} = 
\begin{cases}
-\psi(z\overline{z}) &  \mbox{ if } \beta \neq 0\\
q\psi(\alpha^2) &   \mbox{ if } \beta  = 0; 
\end{cases}$\\

\item $\chi^{\widehat{\chi}_{\psi_1, \psi_2}}
\begin{pmatrix} 
\alpha &\eta  \beta\\
\beta & \alpha
\end{pmatrix} =
\begin{cases}
0  &\mbox{ if } \beta \neq 0\\
(q+1)\psi_1(\alpha)\psi_2(\alpha) &   \mbox{ if } \beta  = 0; 
\end{cases}$\\

\item $\chi^{\rho_\nu}
\begin{pmatrix} 
\alpha &\eta  \beta\\
\beta & \alpha
\end{pmatrix} =
\begin{cases}
- \nu(z)-\nu(\overline{z}) & \mbox{ if } \beta \neq 0\\
(q-1)\nu(\alpha) &   \mbox{ if } \beta  = 0.
\end{cases}$
\end{itemize}

Let now $\nu_0$ be an indecomposable character of $\FF_{\!q^2}^*$. We compute the multiplicity of $\nu_0$ in the restriction $\Res^G_C \theta$, with $\theta \in \widehat{G}$, by means of the scalar product of characters of $C$ (see \cite[Proposition 10.2.18]{book4}).  Note that $|C| = q^2-1$ and  that the 
condition $\beta \neq 0$  is equivalent to  $z \in \FF_{\!q^2}^*\setminus \FF_q^*$. The multiplicity of $\nu_0$ in $\Res^G_C\widehat{\chi}^0_\psi$ is equal to 
\begin{equation}
\label{estar0:PTPT22}
\frac{1}{q^2-1}\sum_{z \in \FF_{\!q^2}^*}\psi(z\overline{z})\overline{\nu_0(z)} = \delta_{\Psi,\nu_0} = 0,
\end{equation}
because $\nu_0$ is indecomposable ($\Psi$ is as in \eqref{estar:PTPT14}).

The multiplicity of $\nu_0$ in $\Res^G_C\widehat{\chi}^1_\psi$ is equal to  
\begin{equation}
\label{estar:PTPT22}
\begin{split}
 \frac{1}{q^2-1}\sum_{z \in \FF_{q^2}^* \setminus \FF_q^*} & \!\!\!\!\!-\psi(z\overline{z})\overline{\nu_0({z})} 
+\frac{1}{q^2-1}\sum_{\alpha\in \FF_q^*}q\psi(\alpha^2)\overline{\nu_0(\alpha)}\\
&=  \frac{1}{q^2-1}\sum_{z \in \FF_{q^2}^*}-\psi(z\overline{z})\overline{\nu_0({z})} + \frac{1}{q-1}\sum_{\alpha \in \FF_q^*}\psi(\alpha^2)\overline{\nu_0(\alpha)}\\
 & = -\delta_{\Psi, \nu_0}+ \delta_{\psi^2, \nu_0^\sharp}\\
 & = \delta_{\psi^2, \nu_0^\sharp},
\end{split}
\end{equation} 
where the last equality follows, once more, from the fact that $\nu_0$ is indecomposable.

The multiplicity of $\nu_0$ in $\Res^G_C\widehat{\chi}_{\psi_1, \psi_2}$ is equal to 
\begin{equation}
\label{estar2:PTPT22}
\frac{1}{q^2-1}\sum_{\alpha\in \FF_q^*}(q+1)\psi_1(\alpha)\psi_2(\alpha)\overline{\nu_0(\alpha)}
=\delta_{\psi_1\psi_2, \nu_0^\sharp}.
\end{equation}

Finally, the multiplicity of $\nu_0$ in $\Res^G_C \rho_\nu$ is equal to 
\begin{equation}
\label{estar:PTPT23}
\begin{split}
\frac{1}{q^2-1}\sum_{z \in \FF_{q^2}^* \setminus \FF_q^*} & \!\!\!\![-\nu(z)-\nu(\overline{z})]\overline{\nu_0(z)}
 + \frac{1}{q^2-1}\sum_{\alpha\in \FF_q^*}(q-1)\nu(\alpha)\overline{\nu_0(\alpha)}\\
&  =  \frac{1}{q^2-1}\sum_{z \in \FF_{q^2}^*}[-\nu(z)\overline{\nu_0(z)} - \overline{\nu}(z)\overline{\nu_0(z)} ]+ \frac{1}{q-1}\sum_{\alpha\in \FF_q^*}\nu(\alpha)\overline{\nu_0(\alpha)}\\
& = -\delta_{\nu, \nu_0}- \delta_{\overline{\nu},\nu_0} + \delta_{\nu^\sharp, \nu_0^\sharp}\\
& = \begin{cases}1  & \mbox{ if } \nu^\sharp= \nu_0^\sharp \mbox{ and } \nu_0 \neq \nu,\overline{\nu}\\
0 & \mbox{ otherwise}.
\end{cases}
\end{split}
\end{equation}
As we alluded to at the beginning of the proof, the decomposition in \eqref{e:decompo} follows from
\eqref{estar0:PTPT22}, \eqref{estar:PTPT22}, \eqref{estar2:PTPT22}, and \eqref{estar:PTPT23}
by invoking Frobenius reciprocity.
\end{proof}
\begin{remark}\label{r5:PTPT24}{\rm 
Note that in the above proof we have, incidentally, recovered the fact that $C$ is a multiplicity-free subgroup of $G$ (the case of a decomposable $\Psi\in \widehat{\FF_{q^2}^*}$ may be handled similarly). This second approach has the advantage of yielding also the decompositions into irreducible representations both of restriction and induced representations; however, it requires the knowledge of the representation theory of $\GL(2, \FF_q)$.}
\end{remark}

\subsection{Spherical functions for $(\GL(2,\FF_q), C, \nu_0)$: the parabolic case}
\label{s:spher-prima-trippa-par}
In this section we derive an explicit formula for the spherical functions associated with the parabolic representations.
We first find an explicit expression for the decomposition of the restriction to $C$ of the parabolic representations.
Note that the involved representations have already been determined in the proof of Theorem \ref{t4:PTPT19}.

\begin{theorem}\label{t6:PTPT25}
\begin{enumerate}[{\rm(1)}]
\item
Let $\psi_1, \psi_2 \in \widehat{\FF_q^*}$, $\psi_1 \neq \psi_2$. For every  $\nu \in \widehat{\FF_{q^2}^*}$  such that  $\nu^\sharp = \psi_1\psi_2$ define  $F_\nu \colon G \to \CC$ by setting 
\begin{equation}\label{estar:PTPT25}
F_\nu\left[\begin{pmatrix} \alpha & \eta \beta\\ \beta & \alpha \end{pmatrix} \begin{pmatrix} x & y \\  0 & 1 \end{pmatrix}\right] = \overline{\nu(\alpha+ i \beta)}\overline{\psi_1(x)}
\end{equation}
(cf.\  Lemma \ref{l1:PTPT5}.(2)).  Then $F_\nu$ belongs to $V_{\widehat{\chi}_{\psi_1, \psi_2}}$, the representation space of $\widehat{\chi}_{\psi_1, \psi_2}$. Moreover,
\begin{equation}\label{estar2:PTPT25}
\lambda\begin{pmatrix}  \alpha & \eta \beta\\ \beta & \alpha \end{pmatrix} F_\nu = \nu(\alpha+i \beta)F_\nu 
\end{equation}
for all $\alpha+ i \beta \in \FF_{q^2}^*$,
and 
\begin{equation}\label {estar3:PTPT25} V_{\widehat{\chi}_{\psi_1, \psi_2}} \cong
\bigoplus_{\substack{\nu \in \FF_{q^2}^*: \\ \nu^\sharp = \psi_1\psi_2}} \!\! \CC F_\nu
\end{equation}
is an explicit  form of  the decomposition 
\begin{equation}
\label{e:compare-mackey}
\Res^G_C \widehat{\chi}_{\psi_1, \psi_2} \sim \bigoplus_{\substack{\nu \in \FF_{q^2}^*: \\ \nu^\sharp = \psi_1\psi_2}}\nu.
\end{equation}
Finally, $\|F_\nu\| _{V_{\widehat{\chi}_{\psi_1, \psi_2}}}= \sqrt{q-1}$.
\item
Let $\psi \in \widehat{\FF_q^*}$ and, for every {\em indecomposable} $\nu \in \widehat{\FF_{q^2}^*}$ such that $\nu^\sharp = \psi^2$, define $F_\nu \colon G \to \CC$ by setting 
\[
F_\nu\left[\begin{pmatrix} \alpha & \eta \beta\\ \beta & \alpha \end{pmatrix} \begin{pmatrix} x & y \\  0 & 1 \end{pmatrix}\right] = \overline{\nu(\alpha+ i \beta)}\overline{\psi(x)}.
\]
Then $F_\nu$ belongs to $V_{\widehat{\chi}_{\psi}^1}$, the representation space of $\widehat{\chi}_{\psi}^1$, and 
\[
\lambda\begin{pmatrix} \alpha & \eta \beta\\ \beta & \alpha \end{pmatrix} F_\nu = \nu(\alpha+i \beta)F_\nu
\]
for all $\alpha+ i \beta \in \FF_{q^2}^*$.
Moreover 
\[
V_{\widehat{\chi}_{\psi}^1} \cong \bigoplus_{\substack{\nu \in \FF_{q^2}^*: \\ \nu {\tiny \mbox{  indecomposable}}\\ \nu^\sharp = \psi^2}} \!\CC F_\nu
\]
is an explicit  form of  the decomposition 
\[
\Res^G_C \widehat{\chi}_{\psi} ^1 \sim \bigoplus_{\substack{\nu \in \FF_{q^2}^*: \\  \nu {\tiny \mbox{  indecomposable}}\\\nu^\sharp = \psi^2}}\nu.
\]
Finally, $\|F_\nu\| _{V_{\widehat{\chi}_{\psi}^1}}= \sqrt{q-1}$.
\end{enumerate}
\end{theorem}
\begin{proof}
We just prove (1), because the proof of (2) is essentially the same.

We prove that $F_\nu$ belongs to the representation space of $\widehat{\chi}_{\psi_1,\psi_2} = \Ind_B^G\chi_{\psi_1, \psi_2}$ by verifying directly definition \eqref{H31}. For all  $\begin{pmatrix} a & b \\ 0 & d \end{pmatrix} \in B$, we have 
\begin{equation}\label{estar:PTPT27}
\begin{split}
F_\nu\left[\begin{pmatrix} \alpha & \eta \beta\\ \beta & \alpha \end{pmatrix} \begin{pmatrix} x & y \\  0 & 1 \end{pmatrix}\begin{pmatrix} a & b \\ 0 & d \end{pmatrix}\right]  & = 
F_\nu\left[\begin{pmatrix} \alpha & \eta \beta\\ \beta & \alpha \end{pmatrix}\begin{pmatrix} xa & xb+ y d \\ 0 & d \end{pmatrix}\right]  \\
& = F_\nu\left[\begin{pmatrix} \alpha d & \eta \beta d \\ \beta  d & \alpha d  \end{pmatrix}\begin{pmatrix} xad^{-1}& xbd^{-1}+ y  \\ 0 & 1 \end{pmatrix}\right]  \\
(\mbox{by \eqref{estar:PTPT25}}) \ & =  \overline{\nu(\alpha d + i \beta d)} \ \overline{\psi_1(xad^{-1})}\\
& =  \overline{\nu(\alpha  + i \beta)} \ \overline{\psi_1(x)}\cdot \overline{\nu(d)} \ \overline{\psi_1(ad^{-1})}\\
(\nu^\sharp = \psi_1\psi_2) \  & = F_\nu\left[\begin{pmatrix} \alpha & \eta \beta\\ \beta & \alpha \end{pmatrix} \begin{pmatrix} x & y \\  0 & 1 \end{pmatrix}\right] \ \overline{\psi_1(a)}\overline{\psi_2(d)}.
\end{split}
\end{equation}
Then \eqref{estar2:PTPT25} is obvious and, from it, \eqref{estar3:PTPT25} and 
\eqref{e:compare-mackey} follow immediately.
Finally, by \eqref{H33},
\[
\begin{split}
\|F_\nu\|^2_{V_{\widehat{\chi}_{\psi_1, \psi_2}}}  & = \frac{1}{q(q-1)^2}\sum_{\alpha+ i \beta \in \FF_{q^2}^*}\sum_{x \in \FF_q^*}\sum_{y \in \FF_q}\overline{\nu(\alpha+ i \beta)}\overline{ \psi_1(x)}\nu(\alpha+ i \beta)\psi_1(x)\\
& = \frac{(q^2-1)q(q-1)}{q(q-1)^2} = q+1.
\end{split}
\]
\end{proof}
\begin{remark}
\label{remark:12-7}
{\rm Recall that Mackey's lemma (cf.\ \cite[Theorem 1.6.14]{book2}, \cite[Theorem 11.5.1]{book4}, and \cite{CSTind})
states that if $H,K \leq G$ and $(\sigma, V)$ is an irreducible $K$-representation, then 
$\Res^G_H \Ind^G_K V \cong \bigoplus_{s \in {\mathcal S}} \Ind^H_{G_s} V_s$, where ${\mathcal S}$ is a complete set
of representatives of the set $H \backslash G \slash K$ of all $H$-$K$ double cosets in $G$, $G_s = H \cap sKs^{-1}$, and
$(\sigma_s,V_s)$ is defined by setting $V_s = V$ and $\sigma_s(x) = \sigma(s^{-1}xs)$ for all $s \in {\mathcal S}$ and
$x \in G_s$.
\par
Now, Theorem \ref{t6:PTPT25} may be seen as a concrete realization of Mackey's Lemma. Indeed, by Lemma \ref{l1:PTPT5}, we have the decomposition $G = CB\equiv C 1_G B$ and $C\cap B = Z$ (i.e., in our previous notation: $H = C$, $K = B$, ${\mathcal S} = \{1_G\}$, and
$G_{1_G} = H \cap K = Z$, so that $(V_{1_G},\sigma_{1_G}) = (V,\Res^B_Z\sigma)$ for all $\sigma \in \widehat{B}$). 
Therefore we have 
\begin{equation}\label{estar:PTPT32}
\begin{split}
\Res^G_C\widehat{\chi}_{\psi_1,\psi_2} & = \Res^G_C\Ind_B^G\chi_{\psi_1, \psi_2}\\
(\mbox{by Mackey's lemma}) & = \Ind_Z^C\Res_Z^B\chi_{\psi_1, \psi_2}\\
& \sim \Ind_{\FF_q^*}^{\FF_{q^2}^*}\psi_1\psi_2\\
& = \bigoplus_{\substack{\nu \in \widehat{\FF_{q^2}^*}: \\ \nu^\sharp = \psi_1\psi_2}} \nu,
\end{split}
\end{equation}
compare with \eqref{e:compare-mackey}.
Note also that \eqref{estar:PTPT32} or, equivalently, its concrete realization \eqref{estar3:PTPT25} in Theorem \ref{t6:PTPT25}, yields another proof of \eqref{estar2:PTPT22} (similarly, for \eqref{estar:PTPT22}). This may be also interpreted as the first part of an alternative proof that $C$ is a multiplicity-free subgroup of $G$; for the remaining part, involving cuspidal representations, see Theorem \ref{t9:PTPT34}.}
\end{remark}

We now compute the parabolic spherical functions.

\begin{theorem}\label{t7:PTPT28}
\begin{enumerate}[{\rm(1)}]
\item The spherical function associated with the parabolic representation $\widehat{\chi}_{\psi_1,\psi_2}$ $($where $\psi_1,\psi_2 \in \widehat{\FF_q^*}$  satisfy that $\nu_0^\sharp = \psi_1\psi_2)$ is given by $($cf.\ Lemma \ref{l1:PTPT5}.$(1))$:
\begin{multline*}
\phi^{\psi_1, \psi_2}\left[ \begin{pmatrix} x & y \\  0 & 1 \end{pmatrix}\begin{pmatrix} a & \eta b\\ b & a \end{pmatrix}\right] \\
=  \frac{\overline{\nu_0(a+ i b)}}{q+1}\left[\sum_{\gamma \in \FF_q} \nu_0(\gamma+i) \overline{\nu_0(x\gamma+ y +i)}\psi_2\left(\frac{(x\gamma+ y)^2-\eta}{x(\gamma^2-\eta)}\right)+ \overline{\psi_1(x)}\right].
\end{multline*}
\item The spherical function associated with $\widehat{\chi}^1_\psi$ 
(where $\psi \in \widehat{\FF_q^*}$ satisfies that $\nu_0^\sharp = \psi^2$) is given by (cf.\ Lemma \ref{l1:PTPT5}.(1)):
\begin{multline*}
\phi^\psi\left[\begin{pmatrix} x & y \\  0 & 1 \end{pmatrix}\begin{pmatrix} a & \eta \beta\\ b & a \end{pmatrix}\right] \\
= 
\frac{\overline{\nu_0(a+ i b)}}{q+1}\left[\sum_{\gamma \in \FF_q} \nu_0(\gamma+i) \overline{\nu_0(x\gamma+ y +i)}
\psi\left(\frac{(x\gamma+ y)^2-\eta}{x(\gamma^2-\eta)}\right)+ \overline{\psi(x)}\right].
\end{multline*}
\end{enumerate}
\end{theorem}
\begin{proof}
Again, we just prove (1).  By virtue of Theorem \ref{t6:PTPT25} and \eqref{defphisigma} the spherical function $\phi^{\psi_1, \psi_2}$
is given by
\[
\phi^{\psi_1, \psi_2}(g)  = \frac{1}{q+1}\langle F_{\nu_0}, \lambda(g) F_{\nu_0}\rangle_{V_{\widehat{\chi}_{\psi_1, \psi_2}}}
\]
for all $g \in G$. Thus, taking $g = \begin{pmatrix} x & y \\  0 & 1 \end{pmatrix}\begin{pmatrix} a & \eta b\\ b & a \end{pmatrix}$, by \eqref{estar:PTPT25} we have:
\[
\begin{split}
\phi^{\psi_1, \psi_2}\left[\begin{pmatrix} x & y \\  0 & 1 \end{pmatrix}\begin{pmatrix} a & \eta b\\ b & a \end{pmatrix}\right]  & = 
 \frac{\overline{\nu_0(a+ i b)}}{q+1}\left\langle \lambda \begin{pmatrix}  x & y \\  0 & 1 \end{pmatrix}^{\!\!-1}\!\!\!\! F_{\nu_0},  F_{\nu_0}\right\rangle\\
(\mbox{by \eqref{H33} and \eqref{estar:PTPT27}}) \ & =   \frac{\overline{\nu_0(a+ i b)}}{q^2-1} \sum_{\alpha+ i \beta \in \FF_{q^2}^*} F_{\nu_0}\left[\begin{pmatrix} x & y \\  0 & 1 \end{pmatrix}\begin{pmatrix} \alpha & \eta \beta\\ \beta & \alpha \end{pmatrix}\right]
\overline{ F_{\nu_0}\begin{pmatrix} \alpha & \eta \beta\\ \beta & \alpha \end{pmatrix}}\\
(\mbox{by \eqref{estar:PTPT6}})\  &  =   \frac{\overline{\nu_0(a+ i b)}}{q^2-1}
\left(
\sum_{\substack{\alpha+ i \beta \in \FF_{q^2}^*: \\ \beta \neq 0}} F_{\nu_0}\left[\begin{pmatrix} u & \eta v \\  v  & u \end{pmatrix}\begin{pmatrix} a & b\\ 0 & 1 \end{pmatrix}\right]\nu_0(\alpha+ i \beta) \right.\\
& \left. \ \ \ \  \ \ \ \  \ \ \ \  \ \ \ \  \ \ \ \ \ \  \ \ \  + 
\sum_{\alpha  \in \FF_{q}^*} F_{\nu_0}\left[\begin{pmatrix} x & y \\  0 & 1 \end{pmatrix}\begin{pmatrix} \alpha & 0\\ 0 & \alpha \end{pmatrix}\right]\nu_0(\alpha) \right)\\
(\mbox{by \eqref{estar:PTPT25}})\  & = \frac{\overline{\nu_0(a+ i b)}}{q^2-1} \left(\sum_{\substack{\alpha+ i \beta \in \FF_{q^2}^*: \\ \beta \neq 0}} \overline{\nu_0(u+iv)}\overline{\psi_1(a)}\nu_0(\alpha+ i \beta)\right. \\
&  \left. \ \ \ \  \ \ \ \  \ \ \ \  \ \ \ \  \ \ \ \  \ \ \ \  \ \ \ \   \ \ \ \  \ \ \ \  \ \ \ \ +  \sum_{\alpha  \in \FF_{q}^*} \overline{\nu_0(\alpha)} \overline{\psi_1(x)}\nu_0(\alpha)\right)\\
(\mbox{by \eqref{estar2:PTPT6} with $\gamma= \alpha\beta^{-1}$})\  & =  \frac{\overline{\nu_0(a+ i b)}}{q+1}
\left(
\sum_{\gamma \in \FF_{q}}\nu_0(\gamma+ i)\overline{\nu_0(x\gamma+ y +i)} \right.\\
& \left.\ \ \ \ \  \cdot \overline{\nu_0\left(\frac{x(\gamma^2-\eta)}{(x\gamma+ y)^2-\eta}\right)} \ \overline{\psi_1\left(\frac{(x\gamma+ y)^2-\eta}{x(\gamma^2-\eta)}\right)}
+ \overline{\psi_1(x)}\right)\\
\mbox{($\nu^\sharp = \psi_1\psi_2$)} \ & = \frac{\overline{\nu_0(a+ ib)}}{q+1} \left(\sum_{\gamma\in \FF_q}\nu_0(\gamma+i)\overline{\nu_0(x\gamma+ y +i)} \cdot\right.\\
& \left. \ \ \ \ \ \  \ \ \ \ \ \  \ \ \ \ \ \  \ \ \ \ \   \ \ \ \ \ \ \cdot \psi_2\left(\frac{(x\gamma+ y)^2-\eta}{x(\gamma^2-\eta)}\right)  + \overline{\psi_1(x)} \right).
\end{split}
\]
\end{proof}

\subsection{Spherical functions for $(\GL(2,\FF_q), C, \nu_0)$: the cuspidal case}
\label{ss:cuspidal case}
We now examine, more closely, the restriction from $G = \GL(2,\FF_q)$ to $C$ of a cuspidal representation. In comparison with the parabolic case examined above,
we follow a slightly different approach, because cuspidal representations are quite intractable. 
We use the standard notation ($\chi,j,\ldots$) as in Section \ref{ss:cuspidal}.
\begin{lemma}\label{l10:PTPT39}
Let $\nu, \nu_0 \in \widehat{\FF_{q^2}^*}$ and suppose that $\nu^\sharp = \nu_0^\sharp$. Then 
\[
\sum_{\gamma \in \FF_q}\nu(\gamma+ i)\overline{\nu_0(\gamma + i)}= (q+1)\delta_{\nu, \nu_0}-1.
\]
\end{lemma}
\begin{proof}
\[
\begin{split}
\sum_{\gamma \in \FF_q} \nu(\gamma +i) \overline{\nu_0(\gamma+ i)} & = \frac{1}{q-1}\sum_{\substack{\alpha \in \FF_q\\ \beta \in \FF_q^*}}\nu(\alpha\beta^{-1}+ i)\overline{\nu_0(\alpha\beta^{-1}+i)}\\
(\nu^\sharp = \nu_0^\sharp) \ \  & = \frac{1}{q-1}\sum_{\substack{a+ i \beta \in \FF_{q^2}^*:\\ \beta \neq 0}}\nu(\alpha+ i \beta) \overline{\nu_0(\alpha+ i \beta)}\\
& = \frac{1}{q-1}\sum_{w \in \FF_{q^2}^*} \nu(w) \overline{\nu_0(w)} - \frac{1}{q-1}\sum_{\alpha \in \FF_q^*}\nu(\alpha) \overline{\nu_0(\alpha)}\\
(\nu^\sharp = \nu_0^\sharp) \ \  & =  (q+1)\delta_{\nu, \nu_0}-1
\end{split}
\]
\end{proof}

\begin{theorem}\label{t9:PTPT34}
Let $\nu, \nu_0 \in \widehat{\FF_{q^2}^*}$ and suppose that $\nu$ is indecomposable. Then the orthogonal projection $E_{\nu_0}$ onto the $\nu_0$-isotypic component of $\Res^G_C\rho_\nu$ is given by:
\[
[E_{\nu_0}f](y)= \sum_{x \in \FF_q^*} f(x) F_0(x,y)
\]
for all $y \in \FF_q^*$ and $f\in L(F_q^*)$, where 
\begin{equation}\label{estar:PTPT34}
F_0(x,y) = \frac{\delta_{\nu_0^\sharp, \nu^\sharp}}{q+1}\left[-\nu(-x)\sum_{\gamma \in \FF_q}\overline{\nu_0(\gamma+ i)} \chi(\gamma(x^{-1}+ y^{-1}) j(x^{-1}y^{-1}(\gamma^2-\eta))+ \delta_x(y)\right].
\end{equation}
Moreover: 
\begin{itemize}
\item if $\nu_0^\sharp \neq \nu^\sharp$ then $F_0\equiv 0$ and $E_{\nu_0} \equiv 0$;
\item  if $\nu_0^\sharp = \nu^\sharp$,  but $\nu_0 = \nu$ or  $\nu_0= \overline{\nu}$, then again $E_{\nu_0}= 0$;
\item  if  $\nu_0^\sharp = \nu^\sharp$ 
 and $\nu_0 \neq \nu, \overline{\nu}$ then $E_{\nu_0}$ is a one-dimensional  projection; in particular,  $\tr(E_{\nu_0}) = 1$, that is,  
$\sum_{x \in \FF_q^*}F_0(x,x) = 1$.
\end{itemize}
\end{theorem}

\begin{proof}
Note that  all itemized statements may be deduced from \eqref{estar:PTPT23} (clearly, not the formula for $F_0$), but we prefer to give another proof of these basic facts in order to check the validity of the formula for $E_{\nu_0}$ (which is quite cumbersone).
From the projection formula \eqref{estar:PTPT1} (with $G$ therein now equal to $C$) and the explicit expressions of $\rho_\nu$
(cf.\ \eqref{cusp1} and \eqref{cusp2}), we deduce that, for $f \in L(\FF_q^*)$ and $y \in \FF_q^*$,
\begin{equation}
\label{e:GCS30}
\begin{split}
[E_{\nu_0}f](y) & = \frac{1}{|C|}\sum_{g \in C}\overline{\nu_0(g)}[\rho_\nu(g) f](y) \\
& = \frac{1}{q^2-1}\left[-\sum_{\substack{\alpha+ i \beta \in \FF_{q^2}^*:\\ \beta \neq 0}}\overline{\nu_0(\alpha+ i \beta)}  \cdot  \sum_{x \in \FF_q^*}\nu(-\beta x) \chi(\alpha \beta^{-1}y^{-1}+ \beta^{-1}\alpha x^{-1})\right. \\
&\ \ \ \ \   \ \ \ \ \ \ \ \ \ \ \ \ \   \ \ \ \ \  \ \ \ \ \  \cdot j(\beta^{-2}y^{-1}x^{-1}(\alpha^2-\eta \beta^2))f(x)    \left.  + \sum_{\alpha \in \FF_q^*}\overline{\nu_0(\alpha)}\nu(\alpha) f(y)\right]\\ 
( \gamma = \alpha \beta^{-1}) \ \  & = \frac{1}{q^2-1}\left\{-\sum_{\gamma\in \FF_q}\overline{\nu_0(\gamma+ i)}\left[\sum_{\beta\in \FF_q^*}\overline{\nu_0(\beta)}\nu(\beta)\right]\right.\\
& \ \ \ \ \   \ \ \ \ \  \ \ \ \ \  \ \ \ \ \   \ \ \ \ \  \ \ \ \ \  \cdot \sum_{x \in \FF_q^*}\nu(-x)\chi(\gamma(x^{-1}+ y^{-1}))j(x^{-1}y^{-1}(\gamma^2-\eta))f(x)\\
& \ \ \ \ \  \ \ \ \ \  \ \ \ \ \ \ \ \ \ \   \ \ \ \ \  \ \ \ \ \ \left.+ \delta_{\nu_0^\sharp, \nu^\sharp}(q-1)f(y)\right\}\\
& = \frac{\delta_{\nu_0^\sharp, \nu^\sharp}}{q+1}\sum_{x \in \FF_q^*}\left[-\nu(-x)\sum_{\gamma\in \FF_q}\overline{\nu_0(\gamma+ i)}\right.\\
&  \ \ \ \ \  \ \ \ \ \  \ \ \ \ \  \ \  \left. \cdot \chi(\gamma(x^{-1}+ y^{-1}))j(x^{-1}y^{-1}(\gamma^2-\eta)) + \delta_x(y)\right]f(x),
\end{split}
\end{equation}
where in the last identity we have written $f(y) = \sum_{x \in \FF_q^*}\delta_x(y)f(x)$. 
This gives the expression for $E_{\nu_0}$. In particular, $E_{\nu_0} = 0$ if $\nu_0^\sharp \neq \nu^\sharp$.
We now compute the trace of $E_{\nu_0}$ (assuming $\nu_0^\sharp = \nu^\sharp$):
\[
\begin{split}
\tr(E_{\nu_0}) & = \sum_{x \in \FF_q^*}F_0(x,x)\\
& = \frac{1}{q+1}\left[\sum_{x \in \FF_q^*}-\nu(-x)\sum_{\gamma \in \FF_q}\overline{\nu_0(\gamma+ i)}\chi(2\gamma x^{-1})j(x^{-2}(\gamma^2-\eta))+ (q-1)\right]\\
& = \frac{1}{q+1}\left[\sum_{\gamma \in \FF_q}-\overline{\nu_0(\gamma+ i)}\sum_{x \in \FF_q^*}\nu(-x)\chi(2\gamma x^{-1})j(x^{-2}(\gamma^2-\eta))+ (q-1)\right]\\
\mbox{(by \eqref{estar3:PTP16})}\ \ & = \frac{1}{q+1}\left[\sum_{\gamma \in \FF_q}-\overline{\nu_0(\gamma+ i)}\left[\nu(\gamma+ i)+ \overline{\nu}(\gamma+ i)\right]+ (q-1)\right]\\
\mbox{(by Lemma \ref{l10:PTPT39}})\ \ & = \frac{1}{q+1}\left[-(q+1)\delta_{\nu, \nu_0}- (q+1)\delta_{\overline{\nu}, \nu_0}+ 2 + (q-1)\right]\\
& = \begin{cases} 0 & \mbox{  if $\nu_0= \nu$ or $\nu_0 = \overline{\nu}$}\\
1 & \mbox{  if  $\nu_0 \neq \nu, \overline{\nu}$}.
\end{cases}
\end{split}
\]
\end{proof}
From now on, we assume $\nu_0^\sharp = \nu^\sharp$ and $\nu_0 \neq \nu, \overline{\nu}$.
\begin{remark}\label{r11:PTPT41}{\rm 
Since $E_{\nu_0}$ is the projection onto a {\it one-dimensional} subspace, there exists $f_0\in L(\FF_q^*)$ satisfying $\|f_0\| = 1$ such  that  
\[
E_{\nu_0} f = \langle f, f_0\rangle f_0,
\]
that is, 
\begin{equation}\label{estar:PTPT41}
F_0(x,y) = \overline{f_0(x)}f_0(y)
\end{equation}
for all $x,y \in \FF_q^*$, and
\begin{equation}\label{estar:PTPT43}
\rho_{\nu}(g)f_0 =  \nu_0(g) f_0
\end{equation}
for all $g \in C$.

Note that, by virtue of \eqref{estar:PTPT41}, we have $f_0(y) = ({\overline{f_0(1)}})^{-1}F_0(1,y)$, where 
\[
F_0(1,y) = \frac{1}{q+1}\left[-\nu(-1)\sum_{\gamma\in \FF_q}\overline{\nu_0(\gamma+ i)}\chi(\gamma(1+ y^{-1}))j(y^{-1}(\gamma^2-\eta))+ \delta_1(y)\right], 
\]
and, moreover,
\[
|f_0(1)|^2\equiv F(1,1) = \frac{1}{q+1}\left[-\nu(-1)\sum_{\gamma\in \FF_q}\overline{\nu_0(\gamma+ i)}\chi(2\gamma)j(\gamma^2-\eta)+ 1\right].
\]
We were unable to find a simple expression for $f_0$ satisfying \eqref{estar:PTPT41} (resp.\ for the norm of $F_0(1,y)$ and for the value $|f_0(1)|$). We leave it as an {\it open problem} to find such possible simple expressions.
Fortunately, this does not constitute an obstruction for our subsequent computation of the spherical functions. 
\par
Finally, note that the computation $\sum_{x\in \FF_q^*}F_0(x,x)= 1$ in Theorem \ref{t9:PTPT34} is equivalent to $\sum_{x \in \FF_q^*}|f_0(x)|^2 = 1$.}
\end{remark}

A group theoretical proof of \eqref{estar:PTPT43} is trivial: from the identity in the first line of
\eqref{e:GCS30} we get, for $g\in C$ and $f \in L(\FF_q^*)$,
\begin{equation}\label{estar:castor}
\begin{split}
\rho_\nu(g)E_{\nu_0} f  & = \frac{1}{|C|}\sum_{g' \in C}\overline{\nu_0(g')}\rho_\nu(gg') f\\
\mbox{(set $h = gg'$)}\ \ & =\frac{1}{|C|}\sum_{h \in C}\overline{\nu_0(g^{-1}h)}\rho_\nu(h) f\\ 
& = \nu_0(g) E_{\nu_0} f.
\end{split}
\end{equation}
Then \eqref{estar:PTPT43} follows from \eqref{estar:castor} after observing that  $f_0 = E_{\nu_0} f_0$.
Note also that  \eqref{estar:castor} relies on the identity  $\rho_\nu(g)\rho_\nu(g') = \rho_\nu(gg')$, a  quite nontrivial fact that 
has been proved analytically  in \cite[Section 14.6]{book4}. 

In the following,  we give a direct analytic proof of \eqref{estar:PTPT43}, in order to also check the validity of our formulas
(and in view of the open problem in Remark \ref{r11:PTPT41} and the intractability of the expression for cuspidal representations).  
We use the notation in \eqref{estar:PTPT41} (but we always use $F_0(x,y)$ in the computation).

We need a preliminary, quite useful and powerful lemma, which is the core of our analytical computations.
For all $x,y \in \FF_q^*$ it is convenient to set
\begin{equation}\label{estar:pontecorvo}
\widetilde{F_0}(x,y) = -\nu(-x)\sum_{\gamma\in \FF_q}\overline{\nu_0(\gamma+ i)}\chi(\gamma(x^{-1}+ y^{-1}))j(x^{-1}y^{-1}(\gamma^2- \eta))
\end{equation}
so that 
\begin{equation}\label{estar:PTPT45}
F_0(x,y)= \frac{1}{q+1}[\widetilde{F_0}(x,y)+ \delta_x(y)].
\end{equation}

\begin{lemma}\label{l12:PTPT45}
With the notation and the assumptions in Theorem \ref{t9:PTPT34}, we have 
\begin{multline}
\label{e:MULT-cast}
\sum_{z\in \FF_q^*} \widetilde{F_0}(x,z) \nu(-z)\chi(\delta(y^{-1} + z^{-1}))j(y^{-1}z^{-1}(\delta^2-\eta))\\
= -\nu_0(\delta+ i)[\widetilde{F_0}(x,y)+ \delta_x(y)]- \nu(-x)\chi(\delta(x^{-1}+ y^{-1}))j(x^{-1}y^{-1}(\delta^2-\eta))
\end{multline}
and 
\begin{equation}
\label{e:MULT-cast2}
\sum_{z \in \FF_q^*}
\widetilde{F_0}(x,z)\widetilde{F_0}(z,y) = (q-1) \widetilde{F_0}(x,y)+ q\delta_x(y),
\end{equation}
for all $\delta \in \FF_q$ and $x,y\in \FF_q^*$.
\end{lemma}
\begin{proof}
The left hand side of \eqref{e:MULT-cast} is equal to 
\[
\begin{split}
\sum_{z \in \FF_q^*}& \left[-\nu(-x)\sum_{\gamma \in \FF_q}\overline{\nu_0(\gamma+ i)}\chi(\gamma(x^{-1}+ z^{-1}))j(x^{-1}z^{-1}(\gamma^2-\eta))\right] \\ 
&  \hspace{6cm}  \cdot \left[\nu(-z)\chi(\delta(y^{-1}+ z^{-1}))j(y^{-1}z^{-1}(\delta^2-\eta))\right]\\
& = -\nu(x) \nu_0(\delta + i)\sum_{\gamma \in \FF_q}\overline{\nu_0((\gamma+ i)(\delta+ i))}\chi(\gamma x^{-1}+ \delta y^{-1})\\
&  \ \ \ \ \ \ \ \ \ \ \ \ \ \ \ \ \ \ \ \   \cdot \sum_{z \in \FF_q^*}\chi((\gamma + \delta)z^{-1})\nu(z)j(z^{-1}x^{-1}(\gamma^2-\eta))j(z^{-1}y^{-1}(\delta^2-\eta))\\
& =_* -\nu(x) \nu_0(\delta + i)\sum_{\substack{\gamma \in \FF_q:\\ \gamma \neq -\delta}} \overline{\nu_0((\gamma\delta + \eta) + i(\gamma + \delta))}\chi(\gamma x^{-1}+ \delta y^{-1}) \cdot\nu(\gamma + \delta)\\
&  \ \ \ \ \ \ \ \ \ \ \ \ \ \ \ \ \ \ \ \   \cdot \sum_{t \in \FF_q^*}\chi(t)\nu(t^{-1})j(t(\gamma+ \delta)^{-1} x^{-1}(\gamma^2-\eta))j(t(\gamma+ \delta)^{-1}y^{-1}(\delta^2-\eta))\\
& \hspace{1cm} -\nu(x) \nu_0(\delta + i)\overline{\nu_0(-\delta^2+ \eta)}\chi(\delta y^{-1} - \delta x^{-1})\\
&   \hspace{6cm} \cdot \sum_{z \in \FF_q^*}\nu(z)j(z^{-1}x^{-1}(\delta^2-\eta))j(z^{-1}y^{-1}(\delta^2-\eta))\\
& =_{**}  \nu(x)\nu_0(\delta+i)\sum_{\substack{\gamma \in \FF_q:\\ \gamma \neq -\delta}} \overline{\nu_0((\gamma\delta + \eta) + i(\gamma + \delta))} \cdot \nu(\gamma + \delta) \cdot \chi(\gamma x^{-1}+ \delta y^{-1})\\
&    \hspace{4cm}   \cdot \chi(-(\gamma + \delta)^{-1}[x^{-1}(\gamma^2-\eta)+ y^{-1}(\delta^2-\eta)])\\
&   \hspace{5.5cm}  \cdot \nu(-1)j((\gamma + \delta)^{-2}x^{-1}y^{-1}(\delta^2-\eta)(\gamma^2-\eta))\\
&  \hspace{3cm}  -\nu(x) \nu_0(\delta +i)\overline{\nu_0(-\delta^2+ \eta)}\chi(\delta y^{-1}-\delta x^{-1}) \delta_x(y)\nu(-x^{-1}(\delta^2-\eta)),
\end{split}
\]
where $=_*$ follows after taking $t = (\gamma + \delta)z^{-1}$ and $=_{**}$ follows from \eqref{estar2:PTP16} and \eqref{estar:PTP16}.
To complete our calculations will use the following  elementary algebraic identities. Let us first set, for $\gamma \neq -\delta$,
$$\epsilon = \frac{\gamma \delta +\eta}{\gamma + \delta}.$$ 
Then, we have:
\begin{itemize}
\item 
$\overline{\nu_0((\gamma \delta+ \eta)+ i(\gamma +\delta))} \cdot \nu(\gamma+ \delta) = \overline{\nu_0(\epsilon + i)}$ (recall that
$\nu^\sharp = \nu_0^\sharp$);
\item 
$\gamma = \frac{\delta \epsilon-\eta}{\delta -\epsilon}$ so that $\sum_{\gamma \neq -\delta}$ is equivalent to $ \sum_{\epsilon \neq \delta}$;
\item $\gamma - (\gamma+ \delta)^{-1}(\gamma^2-\eta) = \frac{\gamma \delta+ \eta}{\gamma + \delta} = \epsilon = \delta - (\gamma + \delta)^{-1}(\delta^2-\eta)$\\  (this is the coefficient of both $x^{-1}$ and $y^{-1}$ in the grouped argument of $\chi$);
\item in the argument of $j$: 
\[
\begin{split}
(\gamma+ \delta)^{-2}(\gamma^2-\eta)(\delta^2-\eta) & = (\gamma+ \delta)^{-2}(\gamma + i) (\delta+i)(\gamma-i)(\delta-i)\\
& = (\gamma+\delta)^{-2}[(\gamma \delta + \eta)+ (\gamma +\delta)i]\\
& \ \ \ \ \ \ \ \ \ \cdot [(\gamma \delta+ \eta)-(\gamma + \delta)i)]\\
& = \left(\frac{\gamma\delta + \eta}{\gamma + \delta}\right)^2-\eta = \epsilon^2-\eta;
\end{split}
\]
\item  $\nu(x)\overline{\nu_0(-\delta^2+ \eta)}\nu(-x^{-1}(\delta^2-\eta)) = 1$;
\item $\chi(\delta y^{-1}-\delta x^{-1})\delta_x(y) = \delta_x(y)$.
\end{itemize}

Therefore, continuing the above calculations, the left hand side of \eqref{e:MULT-cast} equals 
\[
\begin{split}
&  \nu(-x)\nu_0(\delta+i)\sum_{\substack{\epsilon\in \FF_q:\\ \epsilon \neq \delta}}\overline{\nu_0(\epsilon+ i)}\chi(\epsilon(x^{-1}+ y^{-1}))j(x^{-1}y^{-1}(\epsilon^2-\eta))-\nu_0(\delta+ i)\delta_x(y)\\
& = -\nu_0(\delta+ i)\left[\widetilde{F_0}(x,y)+ \delta_x(y)+ \nu(-x)\overline{\nu_0(\delta+ i)}\chi(\delta(x^{-1}+ y^{-1}))j(x^{-1}y^{-1}(\delta^2-\eta))\right]\\
& = -\nu_0(\delta+ i)\left[\widetilde{F_0}(x,y)+ \delta_x(y)\right]-\nu(-x)\chi(\delta(x^{-1}+ y^{-1}))j(x^{-1}y^{-1}(\delta^2-\eta)).
\end{split}
\]

This completes the proof of \eqref{e:MULT-cast}.
Finally, \eqref{e:MULT-cast2} follows from
\[
\begin{split}
\sum_{z \in \FF_q^*}\widetilde{F_0}(x,z)\widetilde{F_0}(z,y)& = -\sum_{\delta\in \FF_q}\overline{\nu_0(\delta+ i)}\sum_{z \in \FF_q^*}\widetilde{F_0}(x,z)\nu(-z)\chi(\delta(y^{-1}+ z^{-1}))j(y^{-1}z^{-1}(\delta^2-\eta))\\
& = (q-1)\widetilde{F_0}(x,y)+ q\delta_x(y),
\end{split}
\]
where the last equality follows from \eqref{e:MULT-cast} and \eqref{estar:pontecorvo}.
\end{proof}

\begin{proof}[Analytic proof of \eqref{estar:PTPT43}]
Let $g = \begin{pmatrix} \alpha & \eta \beta \\ \beta & \alpha\end{pmatrix} \in C$ and $x,z \in \FF_q^*$.
First of all, note that if $\beta=0$ there is nothing to prove: from \eqref{cusp1} it follows that 
\[
\rho_\nu\begin{pmatrix} \alpha & 0 \\ 0 & \alpha\end{pmatrix} f = \nu(\alpha) f = \nu_0(\alpha) f
\] 
for all $f \in L(\FF_q^*)$, where the last equality follows from our assumption $\nu_0^\sharp = \nu^\sharp$.

Suppose now that $\beta \neq 0$. Since $f_0 \not\equiv 0$, we can find $x \in \FF_q^*$ such that $f_0(x) \neq 0$.
Taking into account \eqref{cusp2}, we then have
\[
\begin{split}
\overline{f_0(x)}[\rho_\nu(g)f_0](y) & = -\overline{f_0(x)}\sum_{z \in \FF_q^*}\nu(-\beta z)\chi(\alpha\beta^{-1}y^{-1}+ \alpha \beta^{-1}z^{-1})\\
& \hspace{4cm} \cdot j(\beta^{-2}y^{-1}z^{-1}(\alpha^2-\eta\beta^2))f_0(z)\\
\mbox{(by \eqref{estar:PTPT41})} \  & = -\sum_{z \in \FF_q^*}\nu(-\beta z)\chi(\alpha\beta^{-1}y^{-1}+ \alpha \beta^{-1}z^{-1})\\
& \hspace{4cm} \cdot j(\beta^{-2}y^{-1}z^{-1}(\alpha^2-\eta\beta^2))F_0(x,z)\\
\mbox{(by \eqref{estar:PTPT45} and $\delta = \alpha\beta^{-1}$)}\ \ & =- \frac{\nu(\beta)}{q+1} \sum_{z \in \FF_q^*}\nu(- z)\chi(\delta(y^{-1}+ z^{-1}))j(y^{-1}z^{-1}(\delta^2-\eta)) \widetilde{F_0}(x,z)\\
& \hspace{0.9cm}  - \frac{\nu(\beta)}{q+1} \sum_{z \in \FF_q^*}\nu(- z)\chi(\delta(y^{-1}+ z^{-1}))j(y^{-1}z^{-1}(\delta^2-\eta))\delta_x(z)\\
\mbox{(by \eqref{e:MULT-cast})}\ \ & = \frac{\nu(\beta)}{q+1}\nu_0(\alpha\beta^{-1}+ i)[\widetilde{F_0}(x,y)+ \delta_x(y)]\\
& \ \ \ \ \ \ \ \ \ \ + \frac{\nu(\beta)}{q+1}\nu(-x)\chi(\delta(x^{-1}+ y^{-1}))j(x^{-1}y^{-1}(\delta^2-\eta))\\
& \ \ \ \ \ \ \ \ \ \  -\frac{\nu(\beta)}{q+1}\nu(-x) \chi(\delta(y^{-1}+ x^{-1}))j(y^{-1}x^{-1}(\delta^2-\eta))\\
\mbox{(by \eqref{estar:PTPT45} and $\nu_0^\sharp = \nu^\sharp$)}\ \ & = \nu_0(\alpha+ i \beta)F_0(x,y)\\
\mbox{(by \eqref{estar:PTPT41})} \ \ & = \overline{f_0(x)}\nu_0(\alpha+ i \beta) f_0(y).
\end{split}
\]
After simplifying (recall that $f_0(x) \neq 0$), one immediately deduces \eqref{estar:PTPT43}.
\end{proof}

We now want to  show how Lemma \ref{l12:PTPT45} can be used to derive by means of purely {\it analytical} methods the other basic properties of the matrix $F_0(x,y)$ (recall that we have already proved, analytically, that its trace is equal to $1$ (cf.\  in Theorem \ref{t9:PTPT34})).

\begin{theorem}
\label{t:castorina}
The function $F_0(x,y) = \overline{f_0(x)}f_0(y)$, $x, y \in \FF_q^*$ (cf.\ \eqref{estar:PTPT41}), satisfies the following identities:
\begin{equation}
\label{e:castorina2}
\sum_{z \in \FF_q^*} F_0(x,z)F_0(z,y) = F_0(x,y), \quad \mbox{(idempotence)}
\end{equation}
\begin{equation}
\label{e:castorina3}
\overline{F_0(x,y)} = F_0(y,x), \quad \mbox{(self-adjointness)}
\end{equation}
for all $x,y,z \in \FF_q^*$.
\end{theorem}
 
\begin{proof}
By \eqref{estar:PTPT45}, we have 
\[
\begin{split}
\sum_{z \in \FF_q^*}F_0(x,z)F_0(z,y) & = \frac{1}{(q+1)^2}\sum_{z \in \FF_q^*}[\widetilde{F_0}(x,z)+ \delta_x(z)]\cdot[\widetilde{F_0}(z,y)+ \delta_z(y)]\\
& = \frac{1}{(q+1)^2}\left[\sum_{z \in \FF_q^*}\widetilde{F_0}(x,z)\widetilde{F_0}(z,y)+ 2 \widetilde{F_0}(x,y)+ \delta_x(y)\right]\\
\mbox{(by \eqref{e:MULT-cast2} )}\ & = \frac{1}{(q+1)^2}\left[(q+1)\widetilde{F_0}(x,y) + (q+1)\delta_x(y)\right]\\
& = F_0(x,y)
\end{split}
\]
proving \eqref{e:castorina2}.

From $\overline{\nu(w)} = \nu(w^{-1})$, $\overline{\chi(z)}= \chi(-z)$, \eqref{estar4:PTP16}, and \eqref{estar:PTPT34} we deduce
\[
\begin{split}
\overline{F_0(x,y)} & = \frac{1}{q+1}\left[-\nu(-x^{-1})\sum_{\gamma\in \FF_q}\nu_0(\gamma + i)\chi(-\gamma(x^{-1}+ y^{-1}))j(x^{-1}y^{-1}(\gamma^2-\eta)) \right. \\
& \ \ \ \ \ \ \ \ \ \ \ \ \ \ \ \ \ \ \ \ \ \ \ \ \ \ \ \ \ \ \ \ \ \ \ \ \ \  \left. \cdot\overline{\nu(-x^{-1}y^{-1}(\gamma^2-\eta))}+ \delta_x(y)\right]\\
\mbox{($\nu_0^\sharp = \nu^\sharp$)}\ \ & = \frac{1}{q+1}\left[-\nu(-y)\sum_{\gamma \in \FF_q}\nu_0(\gamma + i)\overline{\nu_0(\gamma + i)}\overline{\nu_0(\gamma -i)}\nu_0(-1)\chi(-\gamma(x^{-1}+y^{-1}))\right. \\
& \ \ \ \ \ \ \ \ \ \  \ \ \ \ \ \ \ \ \ \ \ \ \ \ \ \ \ \ \ \ \ \ \ \ \ \ \  \left. \cdot j(x^{-1}y^{-1}(\gamma^2-\eta)) + \delta_x(y)\right]\\
(\gamma \mapsto -\gamma) \ \ & =\frac{1}{q+1}\left[-\nu(-y)\sum_{\gamma \in \FF_q}\overline{\nu_0(\gamma + i)}\chi(\gamma(x^{-1}+ y^{-1}))\right.\\
& \ \ \ \ \ \ \ \ \ \ \ \ \ \ \ \ \ \ \ \ \ \ \ \ \ \ \ \ \ \ \ \ \ \ \ \ \  \left. \cdot  j(x^{-1}y^{-1}(\gamma^2-\eta))+ \delta_x(y)\right]\\
\mbox{(by \eqref{estar:PTPT34})} \ \ & = F_0(y,x).
\end{split}
\]
\end{proof}

We end our analytic verifications by proving the orthogonality relations for different projections (corresponding to different indecomposable characters of $\FF_{q^2}^*$). Thus let $\mu_0\in \widehat{\FF_{q^2}^*}$ and suppose that $\mu_0^\sharp = \nu^\sharp$, but $\mu_0\neq \nu_0, \nu, \overline{\nu}$. 
Define $G_0(x,y)$ and $\widetilde{G_0}(x,y)$ as $F_0(x,y)$ and $G_0(x,y)$ in Theorem \ref{t9:PTPT34} and \eqref{estar:pontecorvo}, respectively, but with $\mu_0$ in place of $\nu_0$, so that, as in \eqref{estar:PTPT45},
\[
G_0(x,y) = \frac{1}{q+1}\left[\widetilde{G_0}(x,y) + \delta_x(y)\right].
\] 

\begin{proposition}
With the above notation, we have
\begin{equation}
\label{e:ponte-schia}
\sum_{z \in \FF_q^*}F_0(x,z)G_0(z,y) = 0
\end{equation}
for all $x,y \in \FF_q^*$.
\end{proposition}
\begin{proof}
Arguing as in the proof of the second identity in  Lemma  \ref{l12:PTPT45} we have
\begin{equation*}\label{estar:PTPT58}
\begin{split}
\sum_{z \in \FF_q^*} \widetilde{F_0}(x,z)\widetilde{G_0}(z,y) & = -\sum_{\delta\in \FF_q}\overline{\mu_0(\delta+ i)}\sum_{z \in \FF_q^*}\widetilde{F_0}(x,z)\nu(-z)\\
&  \ \ \ \ \ \ \ \ \ \  \ \ \ \ \ \ \ \ \ \  \cdot \chi(\delta(y^{-1}+ z^{-1}))j(y^{-1}z^{-1}(\delta^2-\eta))\\
\mbox{(by \eqref{e:MULT-cast} and \eqref{estar:pontecorvo} for $\widetilde{G_0}$)} \ \ & = \sum_{\delta \in \FF_q}\nu_0(\delta+ i)\overline{\mu_0(\delta + i)}[\widetilde{F_0}(x,y)+ \delta_x(y)] -\widetilde{G_0}(x,y)\\
\mbox{(by Lemma  \ref{l10:PTPT39})}\ \ & =  -\widetilde{F_0}(x,y)-\widetilde{G_0}(x,y)-\delta_x(y).
\end{split}
\end{equation*}
Therefore, 
\[
\begin{split}
\sum_{z \in \FF_q^*}F_0(x,z)G_0(z,y) & = \frac{1}{(q+1)^2}\sum_{z \in \FF_q^*}\left([\widetilde{F_0}(x,z)+ \delta_x(z)]\cdot[\widetilde{G_0}(z,y)+ \delta_y(z)]\right)\\
& = \frac{1}{(q+1)^2}\left(\sum_{z \in \FF_q^*}\widetilde{F_0}(x,z)\widetilde{G_0}(z,y)+ \widetilde{F_0}(x,y)+\widetilde{G_0}(x,y)+
\delta_x(y) \right)\\
 \ & = \frac{-\widetilde{F_0}(x,y) -\widetilde{G_0}(x,y)-\delta_x(y)+\widetilde{F_0}(x,y) + \widetilde{G_0}(x,y)+ \delta_x(y)}{(q+1)^2}\\
& = 0,
\end{split}
\]
where the last but one equality follows from the previous computations.
\end{proof}

We can now state and prove the analogue of Theorem \ref{t7:PTPT28} for cuspidal representations, completing the computation of the corresponding spherical functions.

\begin{theorem}\label{t13:PTPT60}
Let $\nu_0, \nu \in \widehat{\FF_{q^2}^*}$ indecomposable  and suppose that $\nu^\sharp= \nu_0^\sharp$, but $\nu, \overline{\nu}\neq \nu_0$. Then the spherical function of the multiplicity-free triple $(G,C,\nu_0)$ associated with the cuspidal representation $\rho_\nu$
is given by (cf.\ Lemma \ref{l1:PTPT5}.(1)):
\[
\begin{split}
\phi^\nu\left[\begin{pmatrix} x & y \\ 0 & 1\end{pmatrix}\begin{pmatrix} a & \eta b \\ b & a \end{pmatrix}\right]  & = 
-\frac{\overline{\nu_0(a+ i b)}}{q+1}\sum_{z \in \FF_q^*}\nu(-x^{-1}z)\chi(-yz^{-1})\\
& \ \ \ \ \ \ \ \ \ \ \cdot \sum_{\gamma\in \FF_q}\overline{\nu_0(\gamma + i)}\chi(\gamma z^{-1}(x+1))j(xz^{-2}(\gamma^2-\eta))\\
&  \ \ \ \ \ \ \ \ \ \ \ \ \ \ \ \ \ \ \ \  + \frac{\overline{\nu_0(a+ ib)}}{q+1}\delta_{x,1}(q\delta_{y,0}-1).
\end{split}
\]
\end{theorem}
\begin{proof}
For all $g \in G$, taking into account  \eqref{defphisigma}, we have 
\[
\phi^\nu(g)  = \langle f_0, \rho_\nu(g)f_0\rangle_{L(\FF_q^*)} = \sum_{z \in \FF_q^*}f_0(z)\overline{[\rho_\nu(g)f_0](z)},
\]
so that, writing
$g = \begin{pmatrix} x & y \\  0 & 1 \end{pmatrix}\begin{pmatrix} a & \eta b \\ b & a \end{pmatrix}$, we have
\[
\begin{split}
\phi^\nu(g) & = \sum_{z \in \FF_q^*}f_0(z)\overline{\left[\rho_\nu\begin{pmatrix} x & y \\ 0 & 1\end{pmatrix}\rho_\nu\begin{pmatrix} a & \eta b \\ b & a \end{pmatrix}f_0\right](z)}\\
(\mbox{by \eqref{estar:PTPT43}}) \ \ & = \overline{\nu_0(a+ib)}\sum_{z \in \FF_q^*}f_0(z)\overline{\left[\rho_\nu\begin{pmatrix} x & y \\ 0 & 1\end{pmatrix}f_0\right](z)}\\
(\mbox{by \eqref{cusp1}}) & =  \overline{\nu_0(a+ib)}\sum_{z \in \FF_q^*}\chi(-yz^{-1})f_0(z) \overline{f_0(x^{-1}z)}\\
(\mbox{by \eqref{estar:PTPT34} and \eqref{estar:PTPT41})} & =  \overline{\nu_0(a+ib)}\sum_{z \in \FF_q^*}\chi(-yz^{-1})
\frac{1}{q+1}\left[ - \nu(-x^{-1}z) \right.\\
& \ \ \ \ \ \ \ \ \ \ \ \ \ \ \ \ \ \  \cdot \sum_{\gamma \in \FF_q} \overline{\nu_0(\gamma+i)} \chi(\gamma z^{-1} (1+x)) \cdot
\left. j(xz^{-2}(\gamma^2 - \eta)) +  \delta_{x^{-1}z}(z) \right].
\end{split}
\]
Then we end the proof just by noticing that 
\[
\sum_{z \in \FF_q^*}\chi(-yz^{-1}) \delta_{x^{-1}z}(z) = \delta_{x,1} \sum_{z \in \FF_q^*}\chi(-yz^{-1}) = \delta_{x,1}
(q \delta_{y,0} -1).
\]
\end{proof}

\section{Harmonic analysis of the multiplicity-free triple $(\GL(2,\FF_{q^2}), \GL(2,\FF_{q}), \rho_\nu)$}
\label{s:IItrippa}

In this section we study an example of a  multiplicity-free triple where the representation that we induce
has dimension greater than one.

Let $q = p^h$ with $p$ an odd prime and $h \geq 1$. Set
\[
G_1 = \GL(2,\FF_{q}) \ \mbox{ and } \ G_2 = \GL(2,\FF_{q^2}).
\]
Moreover (cf.\ Section \ref{ss:GGrep}), we denote by $B_j$ (resp.\ $U_j$, resp.\ $C_j$) the Borel (resp.\ the unipotent, resp.\ the Cartan)
subgroup of $G_j$, for $j=1,2$.
Throughout this section, with the notation as in Section \ref{ss:cuspidal}, we let $\nu \in \widehat{\FF_{q^2}^*}$ be a fixed indecomposable character. We assume that $\nu^\sharp = \Res^{\FF_{q^2}^*}_{\FF_{q}^*} \nu$ 
is not a square: this slightly simplifies the decomposition into irreducibles. 
Finally, $\rho_\nu$ denotes the cuspidal representation of $G_1$ associated with $\nu$.

\begin{proposition} 
\label{p1:STPT2}
$(G_2,G_1, \rho_\nu)$ is a multiplicity-free triple and
\begin{equation}
\label{estar:STPT2}
\Ind_{G_1}^{G_2} \rho_\nu \sim \left(\bigoplus \widehat{\chi}_{\xi_1, \xi_2}\right) \oplus \left(\bigoplus \rho_\mu\right)
\end{equation}
where 
\begin{itemize}
\item the first sum runs over all unordered pairs of distinct characters $\xi_1, \xi_2 \in \widehat{\FF_{q^2}^*}$ such that
$\Res^{\FF_{q^2}^*}_{\FF_{q}^*} \xi_1\xi_2 = \Res^{\FF_{q^2}^*}_{\FF_{q}^*} \nu$ but $\overline{\xi_1}\xi_2 \neq \nu, \overline{\nu}$;
\item the second sum runs over all characters $\mu \in \FF_{q^4}^*$ indecomposable over $\FF_{q^2}^*$ such that
$\Res^{\FF_{q^4}^*}_{\FF_{q}^*} \mu = \Res^{\FF_{q^2}^*}_{\FF_{q}^*} \nu$.
\end{itemize}
\end{proposition}
\begin{proof} This is just an easy exercise. For instance, if $\chi^{\rho_\nu}$ and $\chi^{\rho_\mu}$ are the characters of
$\rho_\nu$ and $\rho_\mu$ (with $\mu \in \widehat{\FF_{q^4}^*}$ a generic indecomposable character), respectively, then by
means of the character table of $G_1$ (cf.\ \cite[Table 14.2]{book4}), the character table of the restrictions from $G_2$ to
$G_1$ (cf.\ \cite[Table 14.3]{book4}), and the table of conjugacy classes of $G_1$ (\cite[Table 14.1]{book4}), we find:
\[
\begin{split}
\frac{1}{|G_1|} \langle \chi^{\rho_\nu}, \Res^{G_2}_{G_1} \chi^{\rho_\mu} \rangle_{L(G_1)} & = \frac{(q^2-1)(q-1)}{|G_1|} 
\sum_{x \in \FF_q^*} \nu(x)\overline{\mu(x)} + \frac{q^2-1}{|G_1|} \sum_{x \in \FF_q^*} \nu(x)\overline{\mu (x)}\\
& = \frac{(q^2-1)q}{|G_1|} \sum_{x \in \FF_q^*} \nu(x)\overline{\mu (x)}\\
& = \begin{cases} 1 & \mbox{ if } \Res^{\FF_{q^4}^*}_{\FF_{q}^*} \mu = \Res^{\FF_{q^2}^*}_{\FF_{q}^*} \nu\\
0 & \mbox{ otherwise,}
\end{cases}
\end{split}
\]
where the last equality follows from the orthogonality relations of characters. This yields the multiplicity of $\rho_\nu$ in
$\Res^{G_2}_{G_1} \rho_\mu$ (by \cite[Formula (10.17)]{book4}) and, in turn, the multiplicity of $\rho_\mu$
in $\Ind_{G_1}^{G_2} \rho_\nu$, by Frobenius reciprocity. See \cite[Section 14.10]{book4} for more computations of this kind. 
\end{proof}

The result in the above proposition will complemented in Section \ref{s:munemasa} where we shall study the induction of the trivial and the parabolic representations.

\subsection{Spherical functions for $(\GL(2,\FF_{q^2}), \GL(2,\FF_q), \rho_\nu)$: the parabolic case}
We now compute the spherical functions associated with the parabolic representations in \eqref{estar:STPT2}. 
In order to apply Mackey's lemma (cf.\ Remark \ref{remark:12-7}),
we need a preliminary result. Recall the definition of $i \in \FF_{q^2}$ (cf.\ \eqref{estar:PTPT3}).

\begin{lemma}
\label{l2:STPT5}
Let $W = \begin{pmatrix}  i & 1\\ 1& 0 \end{pmatrix}$. Then
\begin{equation}
\label{estar:STPT5}
G_2 = G_1B_2 \bigsqcup G_1 W B_2
\end{equation}
is the decomposition of $G_2$ into $G_1-B_2$ double cosets. Moreover,
\begin{equation}
\label{estar2:STPT5}
G_1 W B_2 = \left\{\begin{pmatrix}  a & b\\ c& d \end{pmatrix} \in G_2: c \neq 0, ac^{-1} \notin \FF_q\right\}.
\end{equation}
\end{lemma}
\begin{proof} We prove the following facts for $\begin{pmatrix}  a & b\\ c& d \end{pmatrix} \in G_2$:
\begin{enumerate}[{\rm (i)}]
\item $\begin{pmatrix}  a & b\\ c& d \end{pmatrix} \in G_1B_2 \Leftrightarrow c = 0 \mbox{ or } (c \neq 0 \mbox{ and } ac^{-1} \in \FF_q)$;
\item $\begin{pmatrix}  a & b\\ c& d \end{pmatrix} \in G_1WB_2 \Leftrightarrow c \neq 0 \mbox{ and } ac^{-1} \notin \FF_q$.
\end{enumerate}
Let $\begin{pmatrix}  \alpha & \beta\\ \gamma& \delta \end{pmatrix} \in G_1$ and $\begin{pmatrix}  x & y\\ 0& z \end{pmatrix} \in B_2$.
Then $\begin{pmatrix}  \alpha & \beta\\ \gamma& \delta \end{pmatrix}\begin{pmatrix}  x & y\\ 0& z \end{pmatrix} =
\begin{pmatrix}  \alpha x & \alpha y + \beta z\\ \gamma x & \gamma y + \delta z \end{pmatrix}$ with either $\gamma x = 0$ or
($\gamma x \neq 0$ and $\alpha x(\gamma x)^{-1} = \alpha \gamma^{-1} \in \FF_q$).
Conversely, if $c \neq 0$ and $ac^{-1} \in \FF_q$, then
$\begin{pmatrix}  a & b\\ c& d \end{pmatrix} = \begin{pmatrix}  ac^{-1} & 1\\ 1& 0 \end{pmatrix} \begin{pmatrix}  c & d\\ 0& b-ad/c 
\end{pmatrix} \in G_1 B_2$. The case $c=0$ is trivial: indeed $B_2 \subseteq G_1B_2$. This shows (i).

We now consider
\[
\begin{pmatrix}  \alpha & \beta \\ \gamma & \delta \end{pmatrix} \begin{pmatrix}  i & 1\\ 1& 0 \end{pmatrix}
\begin{pmatrix}  x & y\\ 0& z \end{pmatrix} = \begin{pmatrix}  (\alpha i + \beta)x & (\alpha i + \beta)y + \alpha z\\ (\gamma i + \delta)x & (\gamma i + \delta)y + \gamma z \end{pmatrix} \in G_1 W B_2,
\]
and $(\gamma, \delta) \neq (0,0), x \neq 0$ imply $(\gamma i + \delta)x \neq 0$, while
$\det\begin{pmatrix}  \alpha & \beta \\ \gamma & \delta \end{pmatrix} \neq 0$ implies 
$(\alpha i + \beta)(\gamma i + \delta)^{-1} \notin \FF_q$.
Conversely, if $c \neq 0$ and $ac^{-1} \notin \FF_q$, then setting $ac^{-1} = \alpha i + \beta$, with $\alpha, \beta \in \FF_q$, $\alpha \neq 0$, we have

\begin{equation}
\label{esquare:STPT8}
\begin{pmatrix}  a & b\\ c& d \end{pmatrix} = \begin{pmatrix}  ac^{-1} & \alpha\\ 1& 0\end{pmatrix} \begin{pmatrix}  c & d\\ 0& (bc-ad)/(\alpha c)\end{pmatrix} = \begin{pmatrix}  \alpha & \beta \\ 0 & 1 \end{pmatrix} \begin{pmatrix}  i & 1\\ 1& 0 \end{pmatrix} \begin{pmatrix}  c & d\\ 0& (bc-ad)/(\alpha c)\end{pmatrix}.
\end{equation}
Thus (ii) follows as well.
\end{proof}
\begin{remark}{\rm Actually, we have a stronger result, which will be useful in the sequel, namely: any
$\begin{pmatrix}  a & b\\ c& d \end{pmatrix} \in G_2$ with $c \neq 0$ and $ac^{-1} \notin \FF_q$ may be uniquely expressed as in 
\eqref{esquare:STPT8}.}
\end{remark}
\begin{lemma}
\label{l3:STPT9}
In the decomposition \eqref{estar:STPT5} we have $G_1 \cap WB_2W^{-1} = C_1$. Moreover, for $x_1, x_2 \in \FF_q$ we have:
\begin{equation}
\label{estar:STPT9}
W^{-1}\begin{pmatrix}  x_1 & \eta x_2\\ x_2& x_1 \end{pmatrix} W =  \begin{pmatrix}  x_1 + i x_2 & x_2\\ 0& x_1 - i x_2\end{pmatrix}.
\end{equation}
\end{lemma}
\begin{proof} Taking $x = x_1 + i x_2$, $z = z_1 + i z_2$ in $\FF_{q^2}^*$ and $y = y_1 + i y_2 \in \FF_{q^2}$, then
\[
\begin{split}
W\begin{pmatrix}  x & y\\ 0& z \end{pmatrix} W^{-1} & = \begin{pmatrix}  i & 1\\ 1& 0 \end{pmatrix}
\begin{pmatrix}  x_1 + i x_2 & y_1 + i y_2\\ 0& z_1 + i z_2\end{pmatrix} \begin{pmatrix}  0 & 1\\ 1& -i \end{pmatrix}\\
& = \begin{pmatrix}  \eta y_2 + z_1 + i(y_1 + z_2) & \eta x_2 - \eta y_1 - \eta z_2 + i (x_1 - \eta y_2 - z_1)\\ y_1 + i y_2 & x_1 - \eta y_2 + i(x_2 - y_1)\end{pmatrix}
\end{split}
\]
belongs to $G_1$ if and only if
\[
\begin{matrix}  y_1 + z_2 = 0 &  x_1 - \eta y_2 - z_1 = 0\\ y_2 = 0 & x_2 - y_1 = 0 \end{matrix}
\]
equivalently,
\[
y_2 = 0, \ z_2 = -y_1, \ y_1 = x_2, \ \mbox{ and } z_1 = x_1.
\]
This is clearly the same as
\[
W\begin{pmatrix}  x & y\\ 0& z \end{pmatrix} W^{-1} = \begin{pmatrix}  x_1 & \eta x_2\\ x_2& x_1 \end{pmatrix} \in C_1.
\]
The proof of \eqref{estar:STPT9} follows immediately from the above computations.
\end{proof}

In the following, as in Section \ref{ss:cuspidal}, for $\xi \in \widehat{\FF_{q^2}^*}$ we set $\xi^\sharp = \Res^{G_2}_{G_1} \xi$.

\begin{theorem}
\label{t4:STPT10} 
Let $\xi_1, \xi_2 \in \widehat{\FF_{q^2}^*}$ with $\xi_1 \neq \xi_2, \overline{\xi_2}$. Then
\begin{equation}
\label{estar:STPT10}
\Res^{G_2}_{G_1} \widehat{\chi}_{\xi_1, \xi_2} \sim \widehat{\chi}_{\xi_1^\sharp, \xi_2^\sharp} \oplus \Ind^{G_1}_{C_1} \xi_1 \overline{\xi_2}.
\end{equation}
\end{theorem}
\begin{proof} First of all, note that from \eqref{estar:STPT9} we have
\[
\begin{split}
\chi_{\xi_1, \xi_2}\left(W^{-1} \begin{pmatrix}  x_1 & \eta x_2\\ x_2& x_1 \end{pmatrix} W\right) & =  \chi_{\xi_1, \xi_2} 
\begin{pmatrix}  x_1 + i x_2 & x_2\\ 0& x_1 - i x_2\end{pmatrix}\\
\mbox{(by \eqref{e:pag53})} \ & = \xi_1(x_1 + i x_2) \xi_2(x_1 - i x_2)\\
& = (\xi_1 \overline{\xi_2})(x_1+ix_2).
\end{split}
\]
In other words, $\Res^{B_2}_{W^{-1}C_1 W} \chi_{\xi_1, \xi_2} \sim \xi_1 \overline{\xi_2}$.

By Mackey's lemma and Lemma \ref{l3:STPT9} (and $G_1 \cap B_2 = B_1$), we then have
\[
\begin{split}
\Res^{G_2}_{G_1} \widehat{\chi}_{\xi_1, \xi_2} & \sim \Res^{G_2}_{G_1} \Ind^{G_2}_{B_2}\chi_{\xi_1, \xi_2}\\
& \sim \Ind^{G_1}_{B_1}\chi_{\xi_1^\sharp, \xi_2^\sharp} \oplus \Ind^{G_1}_{C_1} \xi_1 \overline{\xi_2}\\
& \sim \widehat{\chi}_{\xi_1^\sharp, \xi_2^\sharp} \oplus \Ind^{G_1}_{C_1} \xi_1 \overline{\xi_2}.
\end{split}
\]
\end{proof}

\begin{remark}
{\rm The above is a quite a surprisingly and unexpected fact: in the study of the parabolic representation involved in the triple $(G_2,G_1, \rho_\nu)$ we must use $\Ind_{C_1}^{G_1}\xi_1\overline{\xi_2}$, that is, the induced representation studied in the triple $(G_1,C_1, \nu)$ (cf.\ Section \ref{s:HAMFT1}).}
\end{remark}

We now analyze, more closely, the $\Ind_{C_1}^{G_1}\xi_1\overline{\xi_2}$-component in \eqref{estar:STPT10}. First of all, note that the representation space of $\Ind_{B_2}^{G_2}\chi_{\xi_1, \xi_2}$ is made up of all functions $F\colon G_2 \to \CC$ such that 
\begin{equation}\label{estar:STPT13}
F\begin{pmatrix} x & y \\ 0 & z\end{pmatrix} = \overline{\xi_1(x)}\overline{\xi_2(z)} F \begin{pmatrix} 1 & 0 \\ 0 & 1\end{pmatrix}
\end{equation}
and, if $w \neq 0$, 
\begin{equation}\label{estar2:STPT13}
F\begin{pmatrix} x & y \\ w & z\end{pmatrix} = \overline{\xi_1(w)}\overline{\xi_2(y -xzw^{-1})} F \begin{pmatrix} xw^{-1} & 1 \\ 1 & 0\end{pmatrix},
\end{equation}
where in the last identity we have used the Bruhat decomposition (cf.\ \eqref{e:bruhat}): $$ \begin{pmatrix} x & y \\ w & z\end{pmatrix}  =  \begin{pmatrix}  1 & xw^{-1} \\ 0 & 1\end{pmatrix} \begin{pmatrix} 0 & 1 \\ 1 & 0\end{pmatrix} \begin{pmatrix} w & z \\ 0 &  y - xzw^{-1}\end{pmatrix}.$$
On the other hand, the \emph{natural} representation space of $\Ind_{C_1}^{G_1}\xi_1\overline{\xi_2}$  is made up of all functions $f\colon  G_1 \to \CC$ such that:
\begin{equation}\label{esquare:STPT14}
f\!\left(\begin{pmatrix} \alpha  & \beta  \\ \gamma  & \delta \end{pmatrix} \begin{pmatrix} x_1 & \eta x_2 \\ x_2 & x_1\end{pmatrix}\right) = 
\overline{\xi_1(x_1+ i x_2)} \cdot \overline{\xi_2(x_1-i x_2)} \cdot f \!\begin{pmatrix} \alpha  & \beta  \\ \gamma  & \delta \end{pmatrix}.
\end{equation}

\begin{proposition}\label{p6:STPT14}
The $\Ind_{C_1}^{G_1}\xi_1\overline{\xi_2}$-component of \eqref{estar:STPT10} is made up of all $F$ satisfying  \eqref{estar:STPT13} and  \eqref{estar2:STPT13}, supported in $G_1WB_2$ (cf.\ \eqref{estar:STPT5}). Moreover,  the map  $F \mapsto f$ given by 
\begin{equation}\label{estar:STPT14}
f \!\begin{pmatrix} \alpha  & \beta  \\ \gamma  & \delta \end{pmatrix} = F \!\left(\begin{pmatrix} \alpha  & \beta  \\ \gamma  & \delta \end{pmatrix} W\right)
\end{equation}
for all $\begin{pmatrix} \alpha  & \beta  \\ \gamma  & \delta \end{pmatrix} \in G_1$, with $F$ satisfying  
\eqref{estar2:STPT13}, is a linear bijection between the $\Ind_{C_1}^{G_1}\xi_1\overline{\xi_2}$-component of \eqref{estar:STPT10} and the natural realization of $\Ind_{C_1}^{G_1}\xi_1\overline{\xi_2}$ (see  \eqref{esquare:STPT14}). The inverse map is given by: $f \mapsto F$, with $f$ satisfying  \eqref{esquare:STPT14} and $F$, supported in $G_1WB_2$, is given by
\begin{equation} \label{estar2:STPT15}
F\!\left( \begin{pmatrix} \alpha  & \beta  \\ \gamma  & \delta \end{pmatrix} W  \begin{pmatrix} x & y \\ 0 & z\end{pmatrix} \right) = 
\overline{\xi_1(x)} \cdot \overline{\xi_2(z) } \cdot f \!\begin{pmatrix} \alpha  & \beta  \\ \gamma  & \delta \end{pmatrix}
\end{equation}
for all $\begin{pmatrix} \alpha  & \beta  \\ \gamma  & \delta \end{pmatrix} \in G_1$, $\begin{pmatrix} x & y \\ 0 & z\end{pmatrix}  
\in B_2$.
\end{proposition}

\begin{proof}
From the proof of Theorem \ref{t4:STPT10} and Mackey's lemma (see the formulation in \cite[Section 11.5]{book4})
it follows that the $\Ind_{C_1}^{G_1}\xi_1\overline{\xi_2}$-component of $\Res^{G_2}_{G_1}\widehat{\chi}_{\xi_1,\xi_2}$ is made up of all functions in the representation space of $\widehat{\chi}_{\xi_1,\xi_2}$ that are supported in $G_1WB_2$. If $F$ satisfies 
\eqref{estar2:STPT13} and $f$ is given by \eqref{estar:STPT14} then for all
$\begin{pmatrix} \alpha  & \beta  \\ \gamma  & \delta \end{pmatrix} \in G_1$,  $\begin{pmatrix} x_1  & \eta x_2  \\ x_2  & x_1 \end{pmatrix} \in C_1$ we have:
\[
\begin{split}
f\left( \begin{pmatrix} \alpha  & \beta  \\ \gamma  & \delta \end{pmatrix} \begin{pmatrix} x_1  & \eta x_2  \\ x_2  & x_1 \end{pmatrix}\right) & = F\left( \begin{pmatrix} \alpha  & \beta  \\ \gamma  & \delta \end{pmatrix} \begin{pmatrix} x_1  & \eta x_2  \\ x_2  & x_1 \end{pmatrix} W\right)\\
\mbox{(by \eqref{estar:STPT9})} \ & = F\left(\begin{pmatrix} \alpha  & \beta  \\ \gamma  & \delta \end{pmatrix} W \begin{pmatrix} x_1 + i x_2 & x_2  \\ 0   & x_1- i x_2 \end{pmatrix}\right)\\
\mbox{(by 
\eqref{estar2:STPT13})} \ & = \overline{\xi_1(x_1+ i x_2)} \cdot \overline{\xi_2(x_1- i x_2)} \cdot f \!\begin{pmatrix} \alpha  & \beta  \\ \gamma  & \delta \end{pmatrix},
\end{split}
\]
so that $f$ satisfies \eqref{esquare:STPT14}. On the other hand, if $f$ satisfies \eqref{esquare:STPT14} and $F$ is given by \eqref{estar2:STPT15}, then, for all $\begin{pmatrix} \alpha  & \beta  \\ \gamma  & \delta \end{pmatrix} \in G_1$ and $\begin{pmatrix} x & y \\ 0 & z\end{pmatrix} \in B_2$, we have:
\[
\begin{split}
F\left( \begin{pmatrix} \alpha  & \beta  \\ \gamma  & \delta \end{pmatrix} W \begin{pmatrix} x & y \\ 0 & z\end{pmatrix}\right) & = 
 \overline{\xi_1(x)} \cdot \overline{\xi_2(z)} \cdot f \!\begin{pmatrix} \alpha  & \beta  \\ \gamma  & \delta \end{pmatrix}\\
& = \overline{\xi_1(x)} \cdot \overline{\xi_2(z)} \cdot F \!\left( \begin{pmatrix} \alpha  & \beta  \\ \gamma  & \delta \end{pmatrix} W\right)
\end{split}
\]
so that $F$ satisfies  
\eqref{estar2:STPT13}. Clearly, the map $f\mapsto F$ is the inverse of the map $F \mapsto f$.
\end{proof}

\begin{remark}
\label{r6bis:STPT17}
Let us take stock of the situation: in order to study $\Ind^{G_2}_{G_1} \rho_\nu$ we first have to describe the $\rho_\nu$-component
of $\Res^{G_2}_{G_1} \widehat{\chi}_{\xi_1, \xi_2}$ (cf.\ \eqref{estar:STPT2}). Then we use the machinery developed in Section
\ref{ss:cuspidal case} with $\nu_0 = \xi_1 \overline{\xi_2}$: the $\rho_\nu$-component of 
$\Res^{G_2}_{G_1} \widehat{\chi}_{\xi_1, \xi_2}$ is clearly contained in the $\Ind^{G_1}_{C_1} \xi_1 \overline{\xi_2}$-component in
\eqref{estar:STPT10}. As in the beginning of Section \ref{s:HaR}, we fix, once and for all, the vector $\delta_1 \in V_{\rho_\nu} = L(\FF_q^*)$ satisfying that (cf.\ \eqref{cusp1})
\begin{equation}
\label{e:tullio-pipi}
\rho_\nu \begin{pmatrix} 1 & y \\ 0 & 1 \end{pmatrix} \delta_1 = \chi(y) \delta_1.
\end{equation}
In other words, $\delta_1$ spans the $\chi$-component of $\Res^{G_1}_{U_1} \rho_\nu$. Note that vector $\delta_1$ is also used in the computation of the spherical function of the Gelfand-Graev representation (see Section \ref{ss:GGrep} and, for an explicit computation, \cite[Section 14.7]{book4}). Therefore, setting $\nu_0 = \xi_1 \overline{\xi_2}$, taking $f_0 \in L(\FF_q^*)$ as in Remark \ref{r11:PTPT41},
we denote by $L \colon L(\FF_q^*) \to \Ind^{G_1}_{C_1} \nu_0$ the associated intertwining operator (see \cite[Equation (13.31)]{book4}
which is just the particular case of Proposition \ref{p:7.1} corresponding to the case $d_\theta = 1$ therein).
\end{remark}

\begin{lemma}
\label{l7:STPT19}
With the above notation, the function
\[
L\delta_1 \in \Ind^{G_1}_{C_1} \nu_0 \subseteq \Res^{G_2}_{G_1} \widehat{\chi}_{\xi_1, \xi_2}
\]
is given by (modulo the identification $L\delta_1 = f_1 \mapsto F_1$ as in Proposition \ref{p6:STPT14})
\begin{equation}
\label{estar:STPT20}
F_1 \left(\begin{pmatrix} \alpha & \beta\\0 & 1 \end{pmatrix} \begin{pmatrix} a & \eta b\\b & a \end{pmatrix} W
\begin{pmatrix} x & y\\0 & z \end{pmatrix} \right) = \frac{1}{\sqrt{q}}
\overline{\xi_1(x(a+ib))} \cdot \overline{\xi_2(z(a-ib))} \cdot \chi(-y)\overline{f_0(x^{-1})}
\end{equation}
for all $\begin{pmatrix} \alpha & \beta\\0 & 1 \end{pmatrix} \begin{pmatrix} a & \eta b\\b & a \end{pmatrix} \in G_1$
(cf.\ Lemma \ref{l1:PTPT5}) and $\begin{pmatrix} x & y\\0 & z \end{pmatrix} \in B_2$.
\end{lemma}
\begin{proof} First of all, we compute $L\delta_1$: arguing as in the proof of Theorem \ref{t13:PTPT60}, we have, for $g =
\begin{pmatrix} x & y\\0 & 1 \end{pmatrix}\begin{pmatrix} a & \eta b\\b & a \end{pmatrix} \in G_1$,
\begin{equation}
\label{estar:STPT21}
\begin{split}
[L \delta_1](g) & = \frac{1}{\sqrt{q}} \langle \delta_1, \rho_\nu(g) f_0 \rangle_{L(\FF_q^*)}\\
& = \frac{1}{\sqrt{q}} \sum_{z \in \FF_q^*} \delta_1(z)
\overline{\left[\rho_\nu \begin{pmatrix} x & y\\0 & 1 \end{pmatrix} \rho_\nu \begin{pmatrix} a & \eta b\\b & a \end{pmatrix}f_0\right](z)}\\
& = \frac{1}{\sqrt{q}} \overline{\nu_0(a+ib)}\chi(-y) \overline{f_0(x^{-1})},
\end{split}
\end{equation}
where the first equality follows from \cite[Equation (13.31)]{book4} (here $\sqrt{\frac{d_\theta}{|G/K|}} = \frac{1}{\sqrt{q}}$),
and the last equality follows from \eqref{cusp1}, since $\rho_\nu \begin{pmatrix} a & \eta b\\b & a \end{pmatrix}f_0 =
\nu_0(a+ib) f_0$ (cf.\ \eqref{estar:PTPT43}).

We now apply to the function $L \delta_1 \in \Ind^{G_1}_{C_1} \nu_0$ the map $f \mapsto F$ in Proposition \ref{p6:STPT14}, so that
$L \delta_1 \mapsto F_1$ where
\[
\begin{split}
F_1 \left(\!\!\begin{pmatrix} \alpha & \beta\\0 & 1 \end{pmatrix}\!\! \begin{pmatrix} a & \eta b\\b & a \end{pmatrix}\! W\!
\begin{pmatrix} x & y\\0 & z \end{pmatrix}\!\! \right) & = \overline{\xi_1(x)} \cdot \overline{\xi_2(z)} [L\delta_1]
\left(\!\!\begin{pmatrix} \alpha & \beta\\0 & 1 \end{pmatrix} \begin{pmatrix} a & \eta b\\b & a \end{pmatrix}\!\!\right)\\
& = \frac{1}{\sqrt{q}} \overline{\xi_1(x)} \cdot \overline{\xi_2(z)} \cdot \overline{\xi_1(a+ib)} \cdot \overline{\xi_2(a-ib)} \cdot \chi(-y)
\overline{f_0(x^{-1})},
\end{split}
\]
where the first equality follows from \eqref{estar2:STPT15} and the second one follows from \eqref{estar:STPT21}.
This proves \eqref{estar:STPT20}.
\end{proof}

\begin{remark}
With respect to the decomposition \eqref{esquare:STPT8} we have
\begin{multline}
\label{estar:STPT22}
F_1\begin{pmatrix} a & b\\c & d \end{pmatrix} = F_1 \left(\!\!\begin{pmatrix} \alpha & \beta\\0 & 1 \end{pmatrix} W 
\begin{pmatrix} c & d\\0 & (bc-ad)/(\alpha c) \end{pmatrix}\!\!\right)\\ = \frac{1}{\sqrt{q}} \overline{\xi_1(c)}
\cdot \overline{\xi_2((bc-ad)/(\alpha c))} \cdot \chi(-d) \overline{f_0(c^{-1})},
\end{multline}
where $\begin{pmatrix} a & b\\c & d \end{pmatrix} \in G_2$, $c \neq 0$, and $\alpha i + \beta = ac^{-1} \notin \FF_q$ (note that here
$a=1$ and $b=0$ in \eqref{estar:STPT20}).
\end{remark}

\begin{theorem}
Let $\xi_1, \xi_2 \in \widehat{\FF_{q^2}^*}$ and suppose that they satisfy the hypotheses in Proposition \ref{p1:STPT2}.
Then the spherical function of the triple $(G_2, G_1, \rho_\nu)$ associated with the parabolic representation
$\widehat{\chi}_{\xi_1,\xi_2}$ and with the choice of the vector $\delta_1 \in L(\FF_q^*)$ (cf.\ Remark \ref{r6bis:STPT17})
is given by
\[
\begin{split}
\phi^{\xi_1,\xi_2}\begin{pmatrix} a & b\\c & d \end{pmatrix} & = \frac{1}{q(q+1)} \!\!\!\!\!\!\!\! \sum_{\substack{u = \alpha i + \beta \in \FF_{q^2} \setminus \FF_q:\\
cu+d \neq 0\\
\alpha_1 i + \beta_1 = (au+b)/(cu+d) \notin \FF_q}} \!\!\!\!\!\!\!\! \overline{\xi_1(cu+d)} \cdot \overline{\xi_2\left(\frac{\alpha(ad-bc)}{\alpha_1(cu+d)}\right)} \cdot \\
& \hspace{2cm} \cdot [\chi(-c) \cdot(-\nu(-(cu+d)^{-1})) \cdot \\
& \cdot \sum_{\gamma \in \FF_q} \overline{\xi_1(\gamma + i)} \cdot \overline{\xi_2(\gamma - i)} \chi(\gamma (cu+d+1)) \cdot j\left((cu+d)(\gamma^2 - \eta)\right) + \delta_{(cu+d)^{-1}}(1)]
\end{split}
\]
\end{theorem}

\begin{proof}
We use \eqref{defphisigma}, the 
 decomposition for $G_2$ (cf.\ \eqref{e:bruhat}), and the last expression in \eqref{H33} to compute, for $g \in G_2$
\begin{equation}
\label{ediamond:STPT25}
\begin{split}
\phi^{\xi_1,\xi_2}(g) & = \langle F_1, \lambda(g) F_1 \rangle_{V_{\widehat{\chi}_{\xi_1, \xi_2}}}\\
& = \langle \lambda(g^{-1}) F_1,  F_1 \rangle_{V_{\widehat{\chi}_{\xi_1, \xi_2}}}\\
& = \sum_{u \in U_2} F_1(guw) \overline{F_1(uw)} + F_1(g) \overline{F_1(1_{G_2})}\\
& = \sum_{u \in U_2} F_1(guw) \overline{F_1(uw)},
\end{split}
\end{equation}
where the last equality follows from the fact that $F_1$ is supported on $G_1WB_2 \not\ni 1_{G_2}$. Similarly, we need to determine
when $uw$ and $guw$ belong to $G_1WB_2$. Now, if $u = \begin{pmatrix} 1 & u\\0 & 1 \end{pmatrix}$ with $u \in \FF_{q^2}$, by
\eqref{estar2:STPT5} we have that
\[
uw = \begin{pmatrix} 1 & u\\0 & 1 \end{pmatrix}\begin{pmatrix} 0 & 1\\1 & 0 \end{pmatrix} = \begin{pmatrix} u & 1\\1 & 0 \end{pmatrix}
\] 
belongs to $G_1 W B_2$ if and only if $u \in \FF_{q^2} \setminus \FF_q$, and, if this is the case,  its decomposition (cf.\ \eqref{esquare:STPT8}) is given by
\begin{equation}
\label{estar:STPT26}
\begin{pmatrix} u & 1\\1 & 0 \end{pmatrix} = \begin{pmatrix} \alpha & \beta\\0 & 1 \end{pmatrix} \begin{pmatrix} i & 1\\1 & 0 \end{pmatrix} \begin{pmatrix} 1 & 0\\0 & \alpha^{-1} \end{pmatrix},
\end{equation}
where $u = \alpha i + \beta$, with $\alpha, \beta \in \FF_q$ and $\alpha \neq 0$.

If $g = \begin{pmatrix} a & b\\c & d \end{pmatrix} \in G_2$, then 
\[
guw = \begin{pmatrix} a & b\\c & d \end{pmatrix} \begin{pmatrix} u & 1\\1 & 0 \end{pmatrix} = \begin{pmatrix} au+b & a\\cu+d & c \end{pmatrix}
\]
so that it belongs to $G_1WB_2$ if an only if $cu+d \neq 0$ and $(au+b)/(cu+d) \notin \FF_q$, and, if this is the case, its decomposition (cf.\ \eqref{esquare:STPT8}) is given by
\begin{equation}
\label{esquare:STPT27}
\begin{pmatrix} au+b & a\\cu+d & c \end{pmatrix} = \begin{pmatrix} \alpha_1 & \beta_1\\0 & 1 \end{pmatrix} W \begin{pmatrix} cu+d & c\\0 & (ad-bc)(\alpha_1(cu+d))^{-1} \end{pmatrix},
\end{equation}
where $\alpha_1 i + \beta_1 = (au+b)/(cu+d)$, with $\alpha_1, \beta_1 \in \FF_q$ and $\alpha_1 \neq 0$.

Then, by virtue of \eqref{estar:STPT26}, \eqref{esquare:STPT27}, \eqref{estar:STPT20}, and \eqref{estar:STPT22}, formula
\eqref{ediamond:STPT25} becomes
\[
\begin{split}
\phi^{\xi_1,\xi_2} \begin{pmatrix} a & b\\c & d \end{pmatrix} & = \frac{1}{q} \sum_{\substack{u \in \FF_{q^2} \setminus \FF_q:\\ 
cu+d \neq 0\\ (au+b)/(cu+d) \notin \FF_q}} \overline{\xi_1(cu+d)} \cdot \overline{\xi_2((ad-bc)(\alpha_1(cu+d))^{-1})} \cdot\\
& \ \ \ \ \ \ \ \ \  \ \ \ \ \ \ \ \  \ \ \ \ \  \ \ \ \ \ \ \ \ \ \ \ \ \ \ \  \ \ \ \ \ \ \ \ \ \cdot \chi(-c) \overline{f_0((cu+d)^{-1})} \cdot \xi_2(\alpha^{-1}) f_0(1)\\
\mbox{(by \eqref{estar:PTPT41})} \ & = \frac{1}{q} \sum_{\substack{u \in \FF_{q^2} \setminus \FF_q:\\ 
cu+d \neq 0\\ (au+b)/(cu+d) \notin \FF_q}} \overline{\xi_1(cu+d)} \cdot \overline{\xi_2(\alpha(ad-bc)(\alpha_1(cu+d))^{-1})} \cdot\\
& \ \ \ \ \ \ \ \ \ \ \ \ \ \ \ \ \ \ \ \ \ \ \ \ \ \ \ \ \ \ \ \ \ \ \ \ \ \ \ \ \  \ \ \ \ \ \ \ \  \ \ \ \ \ \ \ \ \ \cdot \chi(-c) F_0((cu+d)^{-1},1).
\end{split}
\]
We then end the proof by invoking the explicit expression of $F_0(x,y)$ given by \eqref{estar:PTPT34}.
\end{proof}

\subsection{Spherical functions for $(\GL(2,\FF_{q^2}), \GL(2,\FF_q), \rho_\nu)$: the cuspidal case}
We now examine the cuspidal representations in \eqref{estar:STPT2}. We fix an indecomposable character $\mu \in \widehat{\FF_{q^4}^*}$
such that $\mu^\sharp = \nu^\sharp$, where $\mu^\sharp = \Res^{\FF_{q^4}^*}_{\FF_{q}^*} \mu$ and 
$\nu^\sharp = \Res^{\FF_{q^2}^*}_{\FF_{q}^*} \nu$, as in Section \ref{ss:cuspidal}.

We also define $\widetilde{\chi} \in \widehat{\FF_{q^2}}$ by setting
\begin{equation}
\label{estar:STPT29}
\widetilde{\chi}(x+iy) = \chi(x) \ \mbox{for all } x+iy \in \FF_{q^2},
\end{equation}
where $\chi$ is the same nontrivial additive character of $\FF_q$ fixed in Section \ref{ss:cuspidal} in order to define the generalized Kloosterman sum $j = j_{\chi, \nu}$. Then $\widetilde{\chi}$ is a nontrivial additive character of $\FF_{q^2}$ and we can use it to
define the corresponding Kloosterman sum $J = J_{\widetilde{\chi}, \mu}$ over $\FF_{q^2}^*$:
\[
J(x) = \frac{1}{q^2} \sum_{\substack{w \in \FF_{q^4}^*:\\ w \overline{w} = x}} \widetilde{\chi}(w + \overline{w}) \mu(w)
\]
for all $x \in \FF_{q^2}^*$ (here, $\overline{w}$ indicates the conjugate of $w \in \FF_{q^4}$: we regard $\FF_{q^4}$ as a quadratic extension of
$\FF_{q^2}$). The function $J$ is then used in \eqref{cusp2} in order to define the cuspidal representation $\rho_\mu$ of
$G_2$ (whose representation space is now $L(\FF_{q^2}^*)$).

Our main goal is to study the restriction $\Res^{G_2}_{G_1} \rho_\mu$. We take a preliminary step which leads to some operators that
simplify the calculations. More precisely, we start from the following fundamental fact: $\Res^{G_1}_{B_1} \rho_\nu$ is an irreducible
$B_1$-representation (see \cite[Section 14.6, in particular Theorem 14.6.9]{book4} for a more precise description of this representation. Observe that, in fact, it is given by \eqref{cusp1}).
Then we want to study the restriction $\Res^{G_2}_{B_1} \rho_\mu$.

Set $\Gamma = \{x+ iy \in \FF_{q^2}^*: x = 0\} \equiv \{iy: y\in \FF_{q}^*\}$ and,
for $\theta \in \FF_{q}$, define $\Omega_\theta = \{x+ iy \in \FF_{q^2}^*: x \neq 0, y/x = \theta\} \equiv \{x+i\theta x: x\in \FF_{q}^*\}$.
Clearly,
\[
\FF_{q^2}^* = \Gamma \sqcup \left( \sqcup_{\theta \in \FF_q} \Omega_\theta \right).
\]
We also set $W = L(\Gamma)$ and, for $\theta \in \FF_q$, define $V_\theta = L(\Omega_{\theta})$.

\begin{lemma}
\label{l10:STPT32}
\[
L(\FF_{q^2}^*) = W \bigoplus \left(\bigoplus_{\theta \in \FF_q} V_\theta\right)
\]
is the decomposition of $L(\FF_{q^2}^*)$ into $\Res^{G_2}_{B_1} \rho_\mu$-invariant subspaces. Moreover, each $V_\theta$, $\theta
\in \FF_q$, is $B_1$-irreducible and isomorphic to $\Res^{G_1}_{B_1} \rho_\nu$, while $W = \bigoplus_{\substack{\psi_1, \psi_2 \in 
\widehat{\FF_q^*}:\\ \psi_1 \psi_2 = \mu^\sharp}} V_{\chi_{\psi_1,\psi_2}}$ is the sum of one-dimensional $B_1$-representations
$($cf.\ \eqref{e:pag53}$)$. 
Finally, for each $\theta \in \FF_q$, the operator
\[
L_\theta \colon L(\FF_q^*) \to L(\FF_{q^2}^*)
\]
defined by setting
\begin{equation}
\label{estar:STPT33}
[L_\theta f](x+iy) = \begin{cases} f((1 - \theta^2 \eta)x) & \mbox{ if } x+iy \in \Omega_\theta\\
0 & \mbox{ otherwise}
\end{cases}
\end{equation}
for all $f \in L(\FF_q^*)$ and $x+iy \in \FF_{q^2}^*$, intertwines $\Res^{G_1}_{B_1} \rho_\nu$ with $\Res^{G_2}_{B_1} \rho_\mu$
(and, clearly, the image of $L_\theta$ is precisely $V_\theta$).
\end{lemma}
\begin{proof}
First of all, from \eqref{cusp1} we deduce that for $\begin{pmatrix} \alpha & \beta\\0 & \delta \end{pmatrix} \in B_1$,
$f \in L(\FF_{q^2}^*)$, and $x+iy \in \FF_{q^2}^*$
\begin{equation}
\label{estar:STPT34}
\begin{split}
\left[\rho_\mu \begin{pmatrix} \alpha & \beta\\0 & \delta \end{pmatrix} f \right](x+iy) & = \mu(\delta) \widetilde{\chi}
(\delta^{-1}\beta(x+iy)^{-1}) f(\delta\alpha^{-1}(x+iy))\\
\mbox{(since $\mu^\sharp = \nu^\sharp$)} \ & = \nu(\delta) \widetilde{\chi}
(\delta^{-1}\beta(x+iy)^{-1}) f(\delta\alpha^{-1}(x+iy)).
\end{split}
\end{equation}
Now, if we choose $\psi \in \widehat{\FF_q^*}$ and we define $f \in W$ by setting
\[
f(iy) = \psi(y) \ \ \ \mbox{ for all } y \in \FF_q^*
\]
then
\[
\begin{split}
\left[\rho_\mu \begin{pmatrix} \alpha & \beta\\0 & \delta \end{pmatrix} f \right](x+iy) & = \mu(\delta) \widetilde{\chi}
(\delta^{-1}\beta(x+iy)^{-1}) f(\delta\alpha^{-1}(x+iy))\\
\mbox{(since $\widetilde{\chi}(iy)=1$)} \ 
& = \begin{cases}
0 & \mbox{ if } x \neq 0\\
\mu(\delta) \psi(\delta\alpha^{-1}y) & \mbox{ otherwise}
\end{cases}\\
& = \begin{cases}
0 & \mbox{ if } x \neq 0\\
\mu(\delta) \psi(\delta)\psi(\alpha^{-1}) \psi(y) & \mbox{ otherwise}
\end{cases}\\
\mbox{(setting $\psi_1 = \overline{\psi}$ and $\psi_2 = \mu^\sharp \psi$)}\
& = \begin{cases}
0 & \mbox{ if } x \neq 0\\
 \psi_1(\alpha)\psi_2(\delta)\psi(y) &  \mbox{ otherwise.}
\end{cases}
\end{split}
\]
In other words, $\rho_\mu \begin{pmatrix} \alpha & \beta\\0 & \delta \end{pmatrix}f = \psi_1(\alpha)\psi_2(\delta)f$, and this
proves the statement relative to $W$. On the other hand, invariance of $V_\theta$, $\theta \in \FF_q$, follows
from \eqref{estar:STPT34} since if $f \in V_\theta$ then the function $f_1(x+iy) = f(\delta\alpha^{-1}(x+iy))$
still belongs to $V_\theta$. Indeed $f_1 = 0$ unless $x \neq 0$, 
and $y/x \equiv (\delta \alpha^{-1}y)/(\delta \alpha^{-1}x) = \theta$.

It remains to show that $L_\theta$ intertwines $\Res^{G_1}_{B_1} \rho_\nu$ with the restriction of $\Res^{G_2}_{B_1} \rho_\mu$
on $V_\theta$. We shall use the elementary identity
\begin{equation}
\label{estar:STPT35} 
x+ i \theta x = x(1 - i \theta)^{-1}(1 - \eta \theta^2).
\end{equation}
For all $\begin{pmatrix} \alpha & \beta\\0 & \delta \end{pmatrix} \in B_1$, $f \in L(\FF_q^*)$, $x \in \FF_q^*$, and $\theta \in \FF_q$
we have, using \eqref{estar:STPT34},
\[
\begin{split}
\left[\rho_\mu \begin{pmatrix} \alpha & \beta\\0 & \delta \end{pmatrix} L_\theta f \right](x+i \theta x) & =
\nu(\delta) \widetilde{\chi}
(\delta^{-1}\beta(x+i \theta x)^{-1}) [L_\theta f](\delta\alpha^{-1}(x+i \theta x))\\
\mbox{(by \eqref{estar:STPT33}, \eqref{estar:STPT35}, and \eqref{estar:STPT29})} \ & 
= \nu(\delta) {\chi}(\delta^{-1}\beta x^{-1}(1 - \eta \theta^2)^{-1}) f(\delta\alpha^{-1} x(1 - \eta \theta^2)).
\end{split}
\]
On the other hand,
\[
\begin{split}
\left[L_\theta \rho_\nu \begin{pmatrix} \alpha & \beta\\0 & \delta \end{pmatrix} f \right](x+i \theta x) & =
\left[\rho_\nu \begin{pmatrix} \alpha & \beta\\0 & \delta \end{pmatrix} f \right]\left(x(1 - \eta \theta^2)\right)\\
& = \nu(\delta) {\chi}(\delta^{-1}\beta x^{-1}(1 - \eta \theta^2)^{-1}) f(\delta\alpha^{-1} x(1 - \eta \theta^2)).
\end{split}
\]
That is, $\rho_\mu \begin{pmatrix} \alpha & \beta\\0 & \delta \end{pmatrix} L_\theta f = L_\theta \rho_\nu \begin{pmatrix} \alpha & \beta\\0 & \delta \end{pmatrix} f$, and this ends the proof.
\end{proof}

\begin{corollary}
\label{c11:STPT37}
The operators $L_\theta$, $\theta \in \FF_q$, form a basis for $\Hom_{B_1}(\Res^{G_1}_{B_1} \rho_\nu, \Res^{G_2}_{B_1} \rho_\mu)$.
\hfill $\Box$
\end{corollary}

In the following lemma we collect three basic, elementary, but quite useful identities.

\begin{lemma}
\label{l12:STPT38}
For all $\begin{pmatrix} x & y\\0 & z \end{pmatrix} \in B_1$, $u \in \FF_q^*$, and $\theta \in \FF_q$, we have
\begin{equation}
\label{estar:STPT38}
L_\theta \delta_u = \delta_{u(1-i \theta)^{-1}}
\end{equation}
\begin{equation}
\label{estar2:STPT38}
\rho_\nu\begin{pmatrix} x & y\\0 & z \end{pmatrix} \delta_u = \nu(z) \chi(yz^{-1}u^{-1}) \delta_{u xz^{-1}}
\end{equation}
\begin{equation}
\label{estar3:STPT38}
L_\theta \rho_\nu\begin{pmatrix} x & y\\0 & z \end{pmatrix} \delta_u = \nu(z) \chi(yz^{-1}u^{-1}) \delta_{uxz^{-1}(1-i \theta)^{-1}}
\end{equation}
\end{lemma}
\begin{proof}
For instance, $[L_\theta \delta_u](x+ i\theta x) = \delta_u(x(1-\theta^2 \eta))$ and  $x(1-\theta^2 \eta)= u$ yields
\[
x+ i \theta x = \frac{u}{1 -\theta^2 \eta}(1+ i\theta) = u(1-i\theta)^{-1}.
\]
\end{proof}

We now introduce three operators. First of all, we define $Q_1 \colon L(\FF_{q^2}^*) \to L(\FF_{q^2}^*)$ by setting
\begin{equation}
\label{estar:STPT39}
Q_1 f = \sum_{\theta \in \FF_q}\langle f , L_\theta \delta_1 \rangle_{L(\FF_{q^2}^*)} L_\theta \delta_1
\end{equation}
for all $f \in L(\FF_{q^2}^*)$, where $\delta_1 \in L(\FF_q^*)$, as in \eqref{e:tullio-pipi}. Clearly, $Q_1$ is an orthogonal projection. Then we set 
\begin{equation}
\label{estar2:STPT39}
P = \frac{q-1}{|G_1|}\sum_{g \in G_1}\overline{\chi^{\rho_\nu}(g)}\rho_\mu(g),
\end{equation}
where  $\chi^{\rho_\nu}$ is the character of  $\rho_\nu$. By \eqref{estar:PTPT1} this is the projection of $ L(\FF_{q^2}^*)$ onto its subspace, in  the restriction  $\Res^{G_2}_{G_1} \rho_\mu$,  isomorphic to $\rho_\nu$.

Finally, we assume that  $L \colon  L(\FF_{q}^*) \to  L(\FF_{q^2}^*)$ is an {\it isometric} immersion such that 
\begin{equation}
\label{estar:STPT40}
L \rho_\nu(g) = \rho_\mu(g) L \quad \mbox{ for all  $g \in G_1$}
\end{equation}
(note that $L$ coincides with the operator $L_\sigma$ used in Section \ref{section 7}). Actually, we are only able to find an explicit formula far  an operator $L$ which is {\it not} isometric (see Section \ref{sezionefinale}).  However, an explicit formula for $L$ (isometric or not) is not necessary in order to compute the spherical functions.

We just note two basic facts: on  the one hand, since $L$ spans $\Hom_{G_1}(\rho_\nu, \Res^{G_2}_{G_1} \rho_\mu)$ it also belongs to 
$\Hom_{B_1}(\Res^{G_1}_{B_1} \rho_\nu, \Res^{G_2}_{B_1} \rho_\mu)$ and therefore, by Corollary \ref{c11:STPT37}, there exist coefficients $\varphi (\theta)$, $\theta \in \FF_q$, such that 
\begin{equation}
\label{estar:STPT41}
L = \sum_{\theta \in \FF_q}\varphi(\theta)L_\theta.
\end{equation}
On the other hand,  $\{\delta_x: x \in \FF_q^*\}$ is an orthonormal basis  of $L(\FF_q^*)$  and therefore $\{L\delta_x: x \in \FF_q^*\}$  is an orthonormal basis for the subspace of $L(\FF_{q^2}^*)$ which is $\Res^{G_2}_{G_1} \rho_\mu$-isomorphic to $V_{\rho_\nu}$, so that
\begin{equation}
\label{estar:STPT42}
Pf = \sum_{x \in \FF_q^*}\langle f, L \delta_x \rangle_{L(\FF_{q^2}^*)}L\delta_x
\end{equation}
for all $f\in L(\FF_{q^2}^*)$.

\begin{theorem}
\label{t13:STPT42}
The operators $Q_1$ and $P$ commute. Moreover, setting $P_1 = P Q_1$, we have 
\begin{equation}
\label{estar2:STPT42}
P_1f = \langle f, L \delta_1 \rangle_{L(\FF_{q^2}^*)}L\delta_1
\end{equation}
for all $f\in L(\FF_{q^2}^*)$.

Finally, setting 
\begin{equation}
\label{estar:STPT43}
 F_1((1-i \theta)^{-1}, (1-i\sigma)^{-1})= [PL_\theta \delta_1]((1-i\sigma)^{-1}) = [P \delta_{(1-i \theta)^{-1}}]((1-i\sigma)^{-1})
\end{equation}
for all $\theta, \sigma \in \FF_q$, we have, for $f\in L(\FF_{q^2}^*)$:
\begin{equation}
\label{estar2:STPT43}
[P_1f ]((1-i\sigma)^{-1})= \sum_{\theta \in \FF_q} F_1((1-i \theta)^{-1}, (1-i\sigma)^{-1})f((1-i \theta)^{-1})
\end{equation}
for all $\sigma \in \FF_q$,  while $[P_1f ](u + i v) = 0$ if $u+iv$ is not of the form  $(1-i\sigma)^{-1}$ for some $\sigma \in \FF_q$.
\end{theorem}

\begin{proof}
First of all, we prove that 
\begin{equation}
\label{estar3:STPT43}
\langle L_\theta \delta_x, L\delta_z\rangle = 0  \ \mbox{ if } \ x \neq z.
\end{equation}
Indeed, since $L_\theta$, $\theta \in \FF_q$, and $L$ belong to $\Hom_{B_1}(\Res^{G_1}_{B_1} \rho_\nu, \Res^{G_2}_{B_1} \rho_\mu)$ (cf.\
Corollary \ref{c11:STPT37} and \eqref{estar:STPT41}) and $U_1 \subseteq B_1$, we have, for all $x \in \FF_q^*$
\[
\rho_\mu \begin{pmatrix} 1 & y\\0 & 1 \end{pmatrix} L_\theta \delta_x = L_\theta \rho_\nu \begin{pmatrix} 1 & y\\0 & 1 \end{pmatrix} \delta_x
= \chi(x^{-1}y) L_\theta \delta_x,
\]
where the last equality follows from \eqref{estar2:STPT38}, and, similarly,
\[
\rho_\mu \begin{pmatrix} 1 & y\\0 & 1 \end{pmatrix} L \delta_x = \chi(x^{-1}y) L \delta_x.
\]
In other words, setting $\chi_x(y) = \chi(xy)$ (cf.\ \cite[Proposition 7.1.1]{book4}), then $L_\theta \delta_x$ and $L \delta_x$ belong to the
$(\chi_{x^{-1}})$-isotypic component of $\Res^{G_2}_{U_1} \rho_\mu$. Then, \eqref{estar3:STPT43} just expresses the orthogonality between
distinct $U_1$-isotypic components.

Note that we may also deduce that $Q_1$, defined by \eqref{estar:STPT39}, is just the orthogonal projection onto the $\chi$-isotypic
component of $\Res^{G_2}_{U_1} \rho_\mu$. In particular
\begin{equation}
\label{esquare:STPT45}
Q_1 L \delta_1 = L \delta_1.
\end{equation}
Alternatively, it is easy to deduce \eqref{esquare:STPT45} directly from \eqref{estar:STPT41} and \eqref{estar:STPT39}.

Let $f\in L(\FF_{q^2}^*)$. Then, from \eqref{estar:STPT39} and \eqref{estar:STPT42} we deduce that
\[
\begin{split}
Q_1 P f & = \sum_{\theta \in \FF_q} \sum_{x \in \FF_q^*} \langle f, L \delta_x \rangle \cdot \langle L \delta_x, L_\theta \delta_1 \rangle L_\theta \delta_1\\
\mbox{(by \eqref{estar3:STPT43})} \ & = \langle f, L \delta_1 \rangle \sum_{\theta \in \FF_q} \langle L \delta_1, L_\theta \delta_1 \rangle L_\theta \delta_1\\
\mbox{(by \eqref{estar:STPT39})} \ & = \langle f, L \delta_1 \rangle Q_1 L \delta_1\\
\mbox{(by \eqref{esquare:STPT45})} \ & = \langle f, L \delta_1 \rangle L \delta_1.
\end{split}
\]
Similarly,
\[
\begin{split}
P Q_1 f & = \sum_{\theta \in \FF_q} \sum_{x \in \FF_q^*} \langle f, L_\theta \delta_1 \rangle \cdot \langle L_\theta \delta_1, L \delta_x \rangle L \delta_x\\
\mbox{(by \eqref{estar3:STPT43})} \ & = \sum_{\theta \in \FF_q} \langle f, L_\theta \delta_1 \rangle \cdot \langle L_\theta \delta_1, L \delta_1 \rangle L \delta_1\\
& = \langle f, \sum_{\theta \in \FF_q} \langle L \delta_1, L_\theta \delta_1\rangle L_\theta \delta_1 \rangle L \delta_1\\
\mbox{(by \eqref{estar:STPT39})} \ & =  \langle f, Q_1 L \delta_1 \rangle L \delta_1\\
\mbox{(by \eqref{esquare:STPT45})} \ & = \langle f, L \delta_1 \rangle L \delta_1.
\end{split}
\]
In conclusion, we have proved the equlity $P Q_1 = Q_1 P$ and \eqref{estar2:STPT42}.

Moreover, as $P_1 = Q_1 P Q_1$, for $f\in L(\FF_{q^2}^*)$ we have
\[
\begin{split}
P_1 f & = Q_1 P Q_1 f\\
\mbox{(by \eqref{estar:STPT39})} \ & = \sum_{\sigma \in \FF_q} \langle P Q_1 f, L_\sigma \delta_1 \rangle L_\sigma \delta_1\\
\mbox{(again by \eqref{estar:STPT39})} \ & = \sum_{\theta,\sigma \in \FF_q} \langle f, L_\theta \delta_1 \rangle \cdot
\langle P L_\theta \delta_1, L_\sigma \delta_1 \rangle L_\sigma \delta_1\\
\mbox{(by \eqref{estar:STPT38})} \ & = \sum_{\theta,\sigma \in \FF_q} f((1-i \theta)^{-1}) \cdot 
\langle P \delta_{(1-i \theta)^{-1}}, \delta_{(1-i \sigma)^{-1}} \rangle \delta_{(1-i \sigma)^{-1}}
\end{split}
\]
and \eqref{estar2:STPT43} follows as well.
\end{proof}

\begin{corollary}
\label{c13bis:STPT65}
For $u,v \in \FF_q$ one has $[L \delta_1](u+iv) = 0$ if $u+iv$ is not of the form $(1-i \sigma)^{-1}$ for some $\sigma \in \FF_q$.
\end{corollary}
\begin{proof} 
This follows from the fact that for all $f \in L(\FF_{q^2}^*)$ one has 
$[P_1 f](u+iv) = 0$ if $u+iv$ is not of the form $(1-i \sigma)^{-1}$ for some $\sigma \in \FF_q$, and that, after taking
$f = L \delta_1$, one has $[L \delta_1](u+iv) = [P_1 f](u+iv)$, by \eqref{estar2:STPT42}.
\end{proof}

\begin{corollary}
\label{c13tris:STPT66}
For $\sigma,\theta \in \FF_q$ one has $($cf.\ \eqref{estar:STPT43}$)$
\[
\overline{[L \delta_1]((1-i \theta)^{-1})} \cdot [L \delta_1]((1-i \sigma)^{-1}) = F_1((1-i \theta)^{-1}, (1-i \sigma)^{-1}).
\]
\end{corollary}
\begin{proof} 
This is an immediate consequence of \eqref{estar2:STPT42} (with $f = L \delta_1$) and \eqref{estar2:STPT43}.
\end{proof}

Note that the expression in the above corollary is the analogue, in the present setting, of \eqref{estar:PTPT41}.

The following lemmas lead to a considerable simplification in the application of \eqref{estar2:STPT39}.

\begin{lemma}
\label{l14:STPT49}
Let $\begin{pmatrix} a & b\\c & d \end{pmatrix} \in G_1$ with $c \neq 0$. Then we have a unique factorization
\[
\begin{pmatrix} a & b\\c & d \end{pmatrix} = \begin{pmatrix} \gamma & \eta\\1 & \gamma \end{pmatrix} \begin{pmatrix} x & y\\0 & z \end{pmatrix}
\]
with $\gamma, y \in \FF_q$ and $x,z \in \FF_q^*$.
\end{lemma}
\begin{proof}
The equality is equivalent to the system
\[
\begin{cases}
\gamma x = a\\
\gamma y + \eta z = b\\
x = c\\
y + \gamma z = d
\end{cases}
\]
which admits a unique solution, namely: $x = c$, $\gamma = ax^{-1}$, and $y, z$ may be computed by means of the Cramer rule, taking into
account that $\det \begin{pmatrix} \gamma & \eta\\1 & \gamma \end{pmatrix} = \gamma^2 - \eta \neq 0$, since $\eta$ is not  a square.
\end{proof}

\begin{lemma}\label{l15:STPT50}
Let $\begin{pmatrix} a & b\\c & d \end{pmatrix} \in G_1$ with $c \neq 0$. Then we have
\[
\sum_{y \in \FF_q}\! \overline{\chi^{\rho_\nu}\!\!\left(\!\!\begin{pmatrix} a & b\\c & d \end{pmatrix} \begin{pmatrix} 1 & y\\0 & 1 \end{pmatrix} \!\!\right)} \rho_\nu \begin{pmatrix} 1 & y\\0 & 1 \end{pmatrix} \!\delta_1 = -q \nu(c(ad-bc)^{-1}) \chi(-ac^{-1} - dc^{-1}) j(c^{-2}(ad-bc)) \delta_1.
\]
\end{lemma}
\begin{proof} Using the orthonormal basis $\{\delta_u: u \in \FF_q^*\}$ in $L(\FF_q^*)$ to compute the character $\chi^{\rho_\nu}$, we find:
\[
\begin{split}
\sum_{y \in \FF_q}\! \overline{\chi^{\rho_\nu}\!\!\left(\!\!\begin{pmatrix} a & b\\c & d \end{pmatrix}\!\! \begin{pmatrix} 1 & y\\0 & 1 
\end{pmatrix} \!\!\right)} \rho_\nu \!\!\begin{pmatrix} 1 & y\\0 & 1 \end{pmatrix} \!\delta_1 & =
\sum_{y \in \FF_q}\! \sum_{u \in \FF_q^*} \left\langle \delta_u, \rho_\nu \!\begin{pmatrix} a & b\\c & d \end{pmatrix} \rho_\nu\! \begin{pmatrix} 1 & y\\0 & 1 \end{pmatrix} \delta_u \right\rangle \cdot \rho_\nu\! \begin{pmatrix} 1 & y\\0 & 1 \end{pmatrix}\!\! \delta_1\\
\mbox{(by \eqref{estar2:STPT38})} \ & = \sum_{u \in \FF_q^*} \left[\sum_{y \in \FF_q} \chi(y(1-u^{-1})) \right]
\left\langle \delta_u, \rho_\nu \begin{pmatrix} a & b\\c & d \end{pmatrix} \delta_u \right\rangle \delta_1\\
\mbox{(orthogonality relations infer $u =1$)} \ & = q \left\langle \delta_1, \rho_\nu \begin{pmatrix} a & b\\c & d \end{pmatrix} \delta_1 \right\rangle \delta_1\\
\mbox{(setting $x=y=1$ in \eqref{cusp2})} \ & = -q \nu(-c^{-1}) \chi(-ac^{-1} - dc^{-1}) \overline{j(c^{-2}(ad-bc))} \delta_1
\end{split}
\]
and, by means of the identity \eqref{estar4:PTP16}, one immediately obtains the desired expression in the statement.
\end{proof}

\begin{remark}
Note that the coefficient of $\delta_1$ in the formula of the above lemma is $q$-times the spherical function of the Gelfand-Graev representation
(see Theorem \ref{t:cuspidal}) associated with the cuspidal representation $\rho_\nu$; see \cite[Proposition 14.7.9.(iii)]{book4}.
\end{remark}

We are now in a position to face up to the most difficult computation of this section. Recall that $j$ (resp.\ $J$) is the generalized
Kloosterman sum over $\FF_q^*$ (resp.\ $\FF_{q^2}^*$).

\begin{theorem}
\label{t17:STPT53}
The function $F_1$ in Theorem \ref{t13:STPT42} has the following expression:
\[
F_1((1-i\theta)^{-1}, (1-i\sigma)^{-1}) = \frac{1}{q+1} \delta_{\theta, \sigma} + \frac{q}{q+1} \sum_{t \in \FF_q^*} \mu(-t^{-1}(1-i\theta)^{-1}) j(t) J(t(1-i\sigma)(1-i\theta)),
\]
for all $\theta, \sigma \in \FF_q$.
\end{theorem}
\begin{proof}
We make use of the following facts that lead to a considerable simplification.
First of all, we have (cf.\ Corollary \ref{c11:STPT37})
\begin{equation}
\label{esquare:STPT54}
\rho_\mu \begin{pmatrix} x & y\\0 & z \end{pmatrix} L_\theta = L_\theta \rho_\nu \begin{pmatrix} x & y\\0 & z \end{pmatrix}
\end{equation}
for all $\theta \in \FF_q$, $x,z \in \FF_q^*$, and $y \in \FF_q$.

Moreover, since $\Res^{G_1}_{B_1} \rho_\nu$ is irreducible (cf.\ \cite[Section 14.6]{book4}), then the orthogonal projection formula \eqref{estar:PTPT1} gives:
\begin{equation}
\label{esquare2:STPT54}
\frac{q-1}{|B_1|} \sum_{\begin{psmallmatrix} x & y\\0 & z \end{psmallmatrix} \in B_1} \overline{\chi^{\rho_\nu} \begin{pmatrix} x & y\\0 & z \end{pmatrix}} \rho_\nu \begin{pmatrix} x & y\\0 & z \end{pmatrix} = I_{L(\FF_q^*)},
\end{equation}
where $I_{L(\FF_q^*)}$ is the identity operator on $L(\FF_q^*)$.

Then, starting from \eqref{estar:STPT43}, we get

\[
\begin{split}
F_1((1-i\theta)^{-1}, (1-i\sigma)^{-1}) & = [P L_\theta \delta_1]((1-i\sigma)^{-1})\\
\mbox{(by \eqref{estar2:STPT39})} \ & = \frac{1}{q(q^2-1)} \sum_{g \in G_1} \overline{\chi^{\rho_\nu}(g)} [\rho_\mu(g) L_\theta \delta_1]((1-i\sigma)^{-1})\\
\mbox{(by \eqref{e:bruhat}, Lemma \ref{l14:STPT49})} \ & = \frac{1}{q(q^2-1)} \sum_{\begin{psmallmatrix} x & y\\0 & z \end{psmallmatrix} \in B_1} \overline{\chi^{\rho_\nu}\begin{pmatrix} x & y\\0 & z \end{pmatrix}} [\rho_\mu\begin{pmatrix} x & y\\0 & z \end{pmatrix} L_\theta \delta_1](1-i\sigma)^{-1})\\
& \ \ \ \ \ \ \ \ \ \ \ \ \ \ \ \ \  + \frac{1}{q(q^2-1)} \sum_{\gamma \in \FF_q} \sum_{\substack{x,z \in \FF_q^*\\ y \in \FF_q}} \overline{\chi^{\rho_\nu}\left(\begin{pmatrix} \gamma & \eta\\1 & \gamma \end{pmatrix}\begin{pmatrix} x & y\\0 & z \end{pmatrix}\right)} \cdot\\
& \ \ \ \ \ \ \ \ \ \ \ \ \ \ \ \ \ \ \ \ \ \ \ \ \ \ \ \ \cdot [\rho_\mu \begin{pmatrix} \gamma & \eta\\1 & \gamma \end{pmatrix} \rho_\mu\begin{pmatrix} x & y\\0 & z \end{pmatrix} L_\theta \delta_1](1-i\sigma)^{-1})\\
\mbox{(by \eqref{esquare:STPT54} and   \eqref{esquare2:STPT54})} \ & = \frac{1}{q+1}[L_\theta  I_{L(\FF_q^*)} \delta_1]((1-i\sigma)^{-1})\\
& \ \ \ \ \ \ \ \ \ \ \ \ \ \ \ \ \  + \frac{1}{q(q^2-1)} \sum_{\gamma \in \FF_q} \sum_{\substack{x,z \in \FF_q^*\\ y \in \FF_q}} \overline{\chi^{\rho_\nu}\left(\begin{pmatrix} \gamma & \eta\\ 1 & \gamma \end{pmatrix}\begin{pmatrix} x & y\\ 0 & z \end{pmatrix}\right)} \cdot\\
& \ \ \ \ \ \ \ \ \ \ \ \ \ \ \ \ \ \ \ \ \ \ \ \ \ \ \ \ \cdot [\rho_\mu \begin{pmatrix} \gamma & \eta\\ 1 & \gamma \end{pmatrix}L_\theta \rho_\nu\begin{pmatrix} x & y\\ 0 & z \end{pmatrix} \delta_1](1-i\sigma)^{-1})\\
\mbox{(by \eqref{estar:STPT38})} \ & = \frac{1}{q+1}\delta_{\theta, \sigma}+ \frac{1}{q(q^2-1)} 
\sum_{\gamma \in \FF_q} \sum_{x,z \in \FF_q^*}
 [\rho_\mu \begin{pmatrix} \gamma & \eta\\ 1 & \gamma \end{pmatrix}L_\theta \rho_\nu\begin{pmatrix} x & 0\\ 0 & z \end{pmatrix} \cdot\\
& \ \ \ \ \ \ \ \ \ \ \ \ \ \ \ \ \  \cdot \sum_{y \in \FF_q} \overline{\chi^{\rho_\nu}\left(\begin{pmatrix} \gamma & \eta\\ 1 & \gamma \end{pmatrix}\begin{pmatrix} x & 0 \\ 0 & z \end{pmatrix}\begin{pmatrix} 1 & x^{-1}y\\ 0 & 1\end{pmatrix}\right)} \cdot \\
&  \ \ \ \ \ \ \ \ \ \ \ \ \ \ \ \ \ \ \ \ \ \ \ \ \ \ \ \ \ \cdot \rho_\nu\begin{pmatrix} 1 & x^{-1}y \\0 & 1 \end{pmatrix} \delta_1]((1-i\sigma)^{-1})\\
\mbox{(by Lemma \ref{l15:STPT50})} \ & = \frac{1}{q+1} \delta_{\theta, \sigma}- \frac{1}{(q^2-1)}\sum_{\gamma \in \FF_q} \sum_{x,z \in \FF_q^*}\nu(z^{-1}(\gamma^2-\eta)^{-1})  \cdot \\
& \ \ \ \ \ \ \ \ \ \ \ \ \ \ \ \ \  \cdot  \chi(-\gamma(1 + zx^{-1})) j(x^{-1}z(\gamma^2-\eta))[\rho_\mu \begin{pmatrix} \gamma & \eta\\1 & \gamma \end{pmatrix} \cdot \\
&  \ \ \ \ \ \ \ \ \ \ \ \ \ \ \ \ \ \ \ \ \ \ \ \ \ \ \ \ \ \cdot   L_\theta \rho_\nu\begin{pmatrix} x & 0\\0 & z \end{pmatrix} \delta_1]((1-i\sigma)^{-1})\\
\mbox{(by \eqref{estar3:STPT38})} \ & = \frac{1}{q+1} \delta_{\theta, \sigma}- \frac{1}{(q^2-1)}\sum_{\gamma \in \FF_q} \sum_{x,z \in \FF_q^*}\nu((\gamma^2-\eta)^{-1})  \cdot \\
& \ \ \ \ \ \ \ \ \ \ \ \ \ \ \ \ \  \cdot  \chi(-\gamma(1 + zx^{-1})) j(x^{-1}z(\gamma^2-\eta)) \cdot \\
&  \ \ \ \ \ \ \ \ \ \ \ \ \ \ \ \ \ \ \ \ \ \ \ \ \ \ \ \ \ \cdot  [\rho_\mu \begin{pmatrix} \gamma & \eta\\1 & \gamma \end{pmatrix} \delta_{xz^{-1}(1-i\theta)^{-1}}]((1-i\sigma)^{-1})
\end{split}
\]
Actually, only  the expression  $x^{-1}z$ appears and therefore  the sum over $x$ may be omitted  and a factor $q-1$ must be added.
\[
\begin{split}
\mbox{(by \eqref{cusp2})} \ & = \frac{1}{q+1} \delta_{\theta, \sigma}+\frac{1}{q+1}\sum_{\gamma \in \FF_q} \sum_{z \in \FF_q^*}\nu((\gamma^2-\eta)^{-1})  \cdot \\
& \ \ \ \ \ \ \ \ \ \ \ \ \ \ \ \ \  \cdot  \chi(-\gamma(1 + z)) j(z(\gamma^2-\eta)) \sum_{v+ iw \in \FF_{q^2}^*}[\mu(-(v+iw))\cdot \\
&  \ \ \ \ \ \ \ \ \ \ \ \ \ \ \ \ \ \ \ \  \cdot \widetilde{\chi}(\gamma(1-i \sigma)+ \gamma(v+iw)^{-1}) \cdot \\
&  \ \ \ \ \ \ \ \ \ \ \ \ \ \ \ \ \  \ \ \ \ \ \ \cdot  J((1-i\sigma)(v+iw)^{-1}(\gamma^2-\eta))\\
&  \ \ \ \ \ \ \ \ \ \ \ \ \ \ \ \ \ \ \ \ \ \ \ \ \ \ \ \ \ \cdot  \delta_{z^{-1}(1-i\theta)^{-1}}(v+iw)]\\
& = \frac{1}{q+1}\delta_{\theta, \sigma} + \frac{1}{q+1}\sum_{\gamma \in \FF_q}\sum_{z \in \FF_q^*}\nu((\gamma^2-\eta)^{-1}) \chi(-\gamma(1 + z))\cdot \\
& \ \ \ \ \ \ \ \ \ \ \ \ \ \ \ \ \  \cdot   j(z(\gamma^2-\eta))[\mu (-z^{-1}(1-i\theta)^{-1})\chi(\gamma+ \gamma z)\cdot \\
&  \ \ \ \ \ \ \ \ \ \ \ \ \ \ \ \ \ \ \ \  \cdot J((1-i\sigma)(1-i \theta)z(\gamma^2-\eta))],\\
\end{split}
\]
where in the last equality we have  used the fact that   $v+ i w = z^{-1}(1-i\theta)^{-1}$  and the equality  $\widetilde{\chi}(\alpha+ i \beta) = \chi(\alpha)$ which follows from \eqref{estar:STPT29}.
The two $\chi$ factors simplify and, introducing  the new variable $t = z(\gamma^2-\eta)$ in place of $z$, also the variable $\gamma$ disappears. Moreover, the sum over $\gamma$ yields a factor $q$ and,  recalling that $\mu^\sharp = \nu^\sharp$, one gets the  expression in the statement.
\end{proof}
Arguing as in the proof of Theorem  \ref{t13:PTPT60}, we use the projection formula \eqref{estar2:STPT43}, with the explicit expression of $F_1$ given in Theorem \ref{t17:STPT53}, to compute the spherical function associated with the cuspidal representation $\rho_\mu$.
We need an elementary, preliminary result.

\begin{proposition}
\label{p18:STPT61}
Each number $\alpha + i \beta \in \FF_{q^2}$ with $\beta \neq 0$ may be uniquely represented in the form $\alpha + i \beta = \frac{1-i\theta}{1-i \sigma}$, with $\theta, \sigma \in \FF_q$. Moreover,
$1$ may be represented in $q$ different ways (namely, setting $\theta = \sigma \in \FF_q$), while $\alpha \in \FF_q$ , $\alpha \neq  1$, is not representable in this way.
\end{proposition}
\begin{proof}
The equation $\alpha + i \beta = \frac{1-i\theta}{1-i \sigma}$ leads  to the system
\[
\begin{cases} 
\alpha -\beta \sigma \eta & = 1\\
\beta- \sigma \alpha  & = -\theta
\end{cases}
\]
which,  for $\beta  \neq 0$, has  the unique solution
\[
\sigma = \frac{\alpha -1}{\beta \eta}, \quad \theta = \sigma \alpha -\beta.
\]
\end{proof}

Following the preceding proposition,
we set $\mathfrak{S}_0= \{\alpha + i \beta \in \FF_q: \beta \neq 0\}$.

\begin{theorem} \label{t19:STPT62}
The spherical function $\phi^{\rho_\mu}$  of the multiplicity-free triple $(G_2, G_1, \rho_\nu)$ associated with the cuspidal representation $\rho_\mu$ (where  $\mu \in \widehat{\FF_{q^4}^*}$ is indecomposable over $\FF_{q^2}^*$ and 
$\mu^\sharp = \nu^\sharp$)  has the following expression: for $\begin{pmatrix} a & b \\ c & d \end{pmatrix} \in G_2$, 
\begin{equation}\label{esquare:STPT63}
\phi^{\rho_\mu}\begin{pmatrix} a & b \\ 0 & d \end{pmatrix} = 0  \quad  \quad  \quad  \mbox{if } a \neq d \ \mbox{ and } \   da^{-1} \notin \mathfrak{S}_0;
\end{equation}
\begin{equation}\label{esquare2:STPT63}
\phi^{\rho_\mu}\begin{pmatrix} a & b \\ 0 & d \end{pmatrix} =  \mu(d^{-1}) \widetilde{\chi}(-d^{-1}b(1-i\theta)) F_1((1-i\sigma)^{-1}, (1-i\theta)^{-1})
\end{equation}
if $a\neq  d$, $da^{-1} \in  \mathfrak{S}_0$, and  $da^{-1} = \frac{1-i\theta}{1-i\sigma}$;
\begin{equation}\label{esquare3:STPT63}
\phi^{\rho_\mu}\begin{pmatrix} a & b \\ 0 & a \end{pmatrix} =\sum_{\theta \in \FF_q}  \mu(d^{-1}) \widetilde{\chi}(-a^{-1}b(1-i\theta)) F_1((1-i\theta)^{-1}, (1-i\theta)^{-1})
\end{equation}
if $d = a$;
\begin{equation}\label{esquare4:STPT63}
\begin{split}
\phi^{\rho_\mu}\begin{pmatrix} a & b \\ c & d \end{pmatrix} &  = - \sum_{\theta, \sigma \in \FF_q}  \mu(c(1-i\sigma)^{-1}(ad-bc)^{-1})) \widetilde{\chi}(-ac^{-1}(1-i \sigma) -dc^{-1}(1-i\theta))\cdot \\
& \ \ \ \ \ \ \ \ \ \ \ \  \cdot J(c^{-2}(1-i \sigma)(1-i \theta)(ad-bc)) F_1((1-i\theta)^{-1}, (1-i\sigma)^{-1})
\end{split}
\end{equation}
if $c \neq 0$.
\end{theorem}
\begin{proof}
From  \eqref{defphisigma}, with $L\delta_1$ in place  of $w^\sigma$, we get the general expression for $g \in G_2$:
\begin{equation}\label{ediamond:STPT64}
\phi^{\rho_\mu}(g) = \langle L\delta_1, \rho_\mu(g)L\delta_1 \rangle_{L(\FF_{q^2}^*)} = \sum_{\theta \in \FF_q}[L\delta_1]((1-i\theta)^{-1})\overline{[\rho_\mu(g)L\delta_1]((1-i\theta)^{-1})},
\end{equation}
where the last equality folows from Corollary \ref{c13bis:STPT65}.
For $g = \begin{pmatrix} a & b \\ 0 & d \end{pmatrix} \in B_2$, applying \eqref{cusp1} to  \eqref{ediamond:STPT64} we have 
\[
\begin{split}
\phi^{\rho_\mu}(g) & = \sum_{\theta \in \FF_q}[L\delta_1]((1-i \theta)^{-1}) \mu(d^{-1}) \widetilde{\chi}(-d^{-1}b(1-i \theta)) \cdot\\
& \ \ \ \ \ \ \ \ \ \ \ \ \ \ \ \ \ \ \ \ \ \ \ \ \ \ \ \ \ \ \ \ \ \ \ \ \ \ \ \ \ \ \ \ \ \  \cdot \overline{[L \delta_1](da^{-1}(1-i \theta)^{-1})}\\
\mbox{(by Corollary \ref{c13tris:STPT66})} \ & = \sum_{\theta \in \FF_q}\mu(d^{-1}) \widetilde{\chi}(-d^{-1}b(1-i \theta)) F_1(da^{-1}(1-i \theta)^{-1}, (1-i \theta)^{-1}).
\end{split}
\]
We now impose  the condition in  Corollary \ref{c13bis:STPT65}: $da^{-1}(1-i \theta)^{-1}$ must be of the form $(1-i\sigma)^{-1}$ for some $\sigma \in\FF_q$.
If $a = d$ then $\sigma = \theta$ and we get  \eqref{esquare3:STPT63}.

By Proposition \ref{p18:STPT61}, if $a \neq d$ then $da^{-1}(1-i\theta)^{-1} = (1-i\sigma)^{-1}$  for some  $\sigma \in \FF_q$ if and only if  $da^{-1} \in \mathfrak{S}_0$ and  we deduce  \eqref{esquare:STPT63} and  \eqref{esquare2:STPT63}.

Finally we examine  the case $g = \begin{pmatrix} a & b \\  c & d \end{pmatrix} \in G_2$ with $c \neq 0$.  Now by \eqref{cusp2} applied to \eqref{ediamond:STPT64} we get
\[
\begin{split}
\phi^{\rho_\mu} (g) & = -\sum_{\sigma\in \FF_q}[L\delta_1]((1-i\sigma)^{-1}) \sum_{\theta \in \FF_q} \overline{\mu(-(1-i\theta)^{-1}c)} \cdot\\
&  \ \ \ \ \ \ \ \ \ \ \ \ \ \ \ \ \ \cdot \overline{\widetilde{\chi}(ac^{-1}(1-i\sigma) + dc^{-1}(1-i\theta))} \cdot \\
& \ \ \ \ \  \ \ \ \ \ \ \ \ \ \ \ \ \ \ \ \ \ \cdot \overline{J(c^{-2}(1-i\sigma)(1-i\theta)(ad-bc))} \cdot\\
& \ \ \ \ \  \ \ \ \ \ \ \ \ \ \ \ \ \ \ \ \ \  \ \ \ \cdot\overline{[L\delta_1]((1-i\theta)^{-1})}\\
\mbox{(by \eqref{estar4:PTP16} and Corollary \ref{c13tris:STPT66})} \ & = -\sum_{\sigma, \theta\in \FF_q}\mu(c(ad-bc)^{-1}(1-i\sigma)^{-1}) \cdot \\
&   \ \ \ \ \ \ \ \ \ \ \ \ \ \ \ \ \ \cdot \widetilde{\chi}(-ac^{-1}(1-i\sigma)-dc^{-1}(1-i \theta))\cdot \\ 
& \ \ \ \ \  \ \ \ \ \ \ \ \ \ \ \ \ \ \ \ \ \ \cdot J(c^{-2}(1-i\sigma)(1-i \theta)(ad-bc))\cdot \\ 
& \ \ \ \ \  \ \ \ \ \ \ \ \ \ \ \ \ \ \ \ \ \ \ \ \cdot F_1((1-i\theta)^{-1}, (1-i\sigma)^{-1}).
\end{split}
\]
\end{proof}

\begin{remark}{\rm By setting  $\theta = 0$ in $F_1$ (cf.\ Theorem \ref{t17:STPT53}) we get the vector 
\begin{equation}\label{estar:STPT69}
f_1(\sigma) = \frac{1}{q+1}\delta_{0,\sigma}+ \frac{q}{q+1}\sum_{t \in \FF_q^*} \mu(-t^{-1}) j(t)J(t(1-i \sigma))
\end{equation}
which, by virtue of Corollary \ref{c13tris:STPT66}, is a non-normalized multiple of $L\delta_1$. Actually, it is not easy to compute the norm of $f_1$: we discuss this problem in the next section.}
\end{remark}

\subsection{A non-normalized $\tilde{L} \in \Hom_{G_1}(\rho_\nu, \Res^{G_2}_{G_1}\rho_\mu)$}
\label{sezionefinale}
In this section, we want to describe an alternative approach to \eqref{estar:STPT69} by deriving a non-normalized multiple $\widetilde{L}$ of $L$ in \eqref{estar:STPT40}. This is also a way to revisit the results and the calculations in the preceding section. The idea is simple: we set
\begin{equation}\label{estar:STPT70}
\widetilde{L} = \frac{1}{|G_1|}\sum_{g \in G_1}\rho_\mu(g) L_0\rho_\nu(g^{-1})
\end{equation}
where $L_0$ is as in \eqref{estar:STPT33}. Then, for all $h\in G_1$,
\[
\begin{split}
\rho_\mu(h)\widetilde{L} & = \frac{1}{|G_1|}\sum_{g \in G_1}\rho_\mu(hg)L_0\rho_\nu(g^{-1})\\
\mbox{(setting  $r = hg$)} \ & =  \frac{1}{|G_1|}\sum_{r \in G_1}\rho_\mu(r)L_0\rho_\nu(r^{-1})\rho_\nu(h)\\
 & = \widetilde{L}\rho_\nu(h),
\end{split}
\]
that is, $ \widetilde{L} \in \Hom_{G_1}(\rho_\nu, \Res^{G_2}_{G_1} \rho_\mu)$.
The operator $L_0$ may be replaced by any linear operator $T \colon L(\FF_q^*) \to L(\FF_{q^2}^*)$ (but checking that, eventually, 
$\widetilde{L} \neq 0$), for instance by taking $T = L_\theta$, with $\theta \in \FF_q$. The choice of $L_0$ greatly simplifies the calculation and the final formula.
We split the explicit computation of \eqref{estar:STPT70} in several preliminary results.

\begin{lemma}\label{l21:STPT72}
Let $T\colon L(\FF_q^*) \to L(\FF_{q^2}^*)$  be a linear operator with matrix representation 
\[
[Tf](x+iy) = \sum_{v \in \FF_q^*}a(x+iy, v)f(v)
\]
for all $f \in L(\FF_q^*)$ and  $x+ iy \in \FF_{q^2}^*$. Then 
\begin{equation}\label{estar:STPT72}
\left[\sum_{u\in \FF_q}\rho_\mu\begin{pmatrix}1 & u \\ 0 &1\end{pmatrix} T\rho_\nu\begin{pmatrix}1 & -u \\ 0 &1\end{pmatrix} f\right](x+ iy) = 
\begin{cases} 0 & \mbox{if $x = 0$}\\
q a(x+ iy,\frac{x^2-\eta y^2}{x}) f(\frac{x^2-\eta y^2}{x})& \mbox{if $x\neq 0$}.
\end{cases}
\end{equation}
\end{lemma}
\begin{proof}
By \eqref{cusp1} 
\[
\begin{split}
\left[\sum_{u\in \FF_q}\rho_\mu\begin{pmatrix}1 & u \\ 0 &1\end{pmatrix} T\rho_\nu\begin{pmatrix}1 & -u \\ 0 &1\end{pmatrix} f\right](x+ iy)  & = \sum_{u \in \FF_q}\widetilde{\chi}((x+iy)^{-1}u)[T\rho_\nu \begin{pmatrix}1 & -u \\ 0 &1\end{pmatrix} f](x+iy)\\
& = \sum_{u\in \FF_q}\sum_{v \in \FF_q^*}\widetilde{\chi}((x+iy)^{-1}u) a(x+iy,v)\cdot \\
& \ \ \ \ \ \ \ \ \ \ \ \ \ \ \ \ \cdot [\rho_\nu \begin{pmatrix}1 & -u \\ 0 &1\end{pmatrix} f](v)\\
& = \sum_{u\in \FF_q}\sum_{v \in \FF_q^*}\widetilde{\chi}(u[(x+iy)^{-1}-v^{-1}]) a(x+iy,v)f(v).
\end{split}
\]
But from 
\[
\frac{1}{x+iy} = \frac{x}{x^2-\eta y^2} - i \frac{y}{x^2-\eta y^2},
\]
definition \eqref{estar:STPT29}, and the orthogonal relations in $\widehat{\FF_q}$ it follows that, 
for $x \neq 0$, 
\[
\sum_{u \in \FF_q}\widetilde{\chi}(u[(x+iy)^{-1}-v^{-1}]) = \sum_{u \in \FF_q}\chi\left(u\left(\frac{x}{x^2-\eta y^2}- \frac{1}{v}\right)\right) = 
q \delta_{v, \frac{x}{x^2 - \eta  y^2}},
\]
while this sum, for $x = 0$, is equal to $\sum_{u \in \FF_q} \chi(-\frac{u}{v}) = 0$. 
Then formula \eqref{estar:STPT72} follows immediately. 
\end{proof}
Recall that $w =\begin{pmatrix}0 & 1\\ 1 &0\end{pmatrix}$ (see the Bruhat decomposition  \eqref{e:bruhat}).

\begin{lemma}\label{l22:STPT75}
The matrix representating the operator 
$\rho_\mu(w)L_0\rho_\nu(w)$ is given by 
\[
a(x+iy,z) = \sum_{s \in \FF_q^*} \nu(sz)J(-s^{-1}(x+iy)^{-1})j(-s^{-1}z^{-1}).
\]
\end{lemma}

\begin{proof}
By \eqref{cusp2}, for $f \in L(\FF_q^*)$ we have that 
\[
\begin{split}
[\rho_\mu(w)L_0\rho_\nu(w)f](x+iy) & = - \sum_{s+it\in \FF_{q^2}^*}\mu(-(s+it))J(-(s+it)^{-1}(x+iy)^{-1})\cdot\\
& \  \ \ \ \ \ \ \ \ \ \ \ \ \cdot [L_0\rho_\nu(w)f](s+it)\\
\mbox{(by \eqref{estar:STPT33} and $\theta = 0$)}\ & = \sum_{z,s \in \FF_q^*}\nu(sz)J(-s^{-1}(x+iy)^{-1})j(-s^{-1}z^{-1})f(z).
\end{split}
\]
\end{proof}

\begin{corollary}\label{c23:STPT76}
\[
\sum_{u \in U_1}\rho_\mu(uw)L_0\rho_\nu(wu^{-1})= q\sum_{\theta \in \FF_q}\left[\sum_{t \in \FF_q^*}\nu(-t^{-1})J(t(1-i\theta))j(t)\right] \!\!L_\theta.
\]
\end{corollary}
\begin{proof}
By Lemma \ref{l21:STPT72} and with $a$ as in Lemma \ref{l22:STPT75}, we have, for $f \in L(\FF_q^*)$ and $x \neq 0$,
\[
\begin{split}
\left[\sum_{\alpha \in \FF_q}\rho_\mu\begin{pmatrix}1 & \alpha\\ 0 &1\end{pmatrix}\rho_\mu(w)L_0 \right.& \left. \rho_\nu(w)\rho_\nu\begin{pmatrix}1 & -\alpha\\ 0 &1\end{pmatrix}f\right]\!\!(x+iy)  = q a(x+iy, \frac{x^2-\eta y^2}{x})f(\frac{x^2-\eta y^2}{x})\\
& = q \sum_{s \in  \FF_q^*}\nu(sx^{-1}(x^2-\eta y^2))J(-s^{-1}(x+iy)^{-1})\cdot \\
&  \ \ \ \ \ \ \ \  \cdot j(-s^{-1}x(x^2-\eta y^2)^{-1})f((x^2-\eta y^2)x^{-1})\\
\mbox{(by \eqref{estar:STPT33} with $\sigma = y/x$)}\ &  = q \sum_{s \in  \FF_q^*}\nu(sx(1-\eta \sigma^2))J(-s^{-1}x^{-1}(1+ i \sigma)^{-1})\cdot\\
& = \ \ \ \ \ \ \ \  \cdot j(-s^{-1}x^{-1}(1-\sigma^2 \eta)^{-1}) [L_\sigma f](x+i y)\\
& = q \sum_{\theta \in \FF_q}\left [\sum_{t \in \FF_q^*}\nu(-t^{-1})J(t(1-i \theta))j(t)\right][L_\theta f](x+iy),
\end{split}
\]
where, in the last equality, we have set $t = -s^{-1}x^{-1}(1-\eta\theta^2)^{-1}$ so that  $-s^{-1}x^{-1}(1+i \theta)^{-1} = t(1-i\theta)$ and we have used, once again, \eqref{estar:STPT33}.
\end{proof}

\begin{theorem}\label{t24:STPT78}
The operator $\widetilde{L}$ in \eqref{estar:STPT70} is equal to 
\[
\frac{1}{q+1}L_0+\frac{q}{q+1}\sum_{\theta \in \FF_q}\left[\sum_{t\in \FF_q^*}\nu(-t^{-1})J(t(1-i\theta))j(t)\right]\!\!L_\theta.
\]
\end{theorem}
\begin{proof}
Using the Bruhat decomposition  (cf.\ \eqref{e:bruhat}) $G_1 = B_1 \bigsqcup (U_1wB_1)$ and  the fact that 
\begin{equation}\label{sopra}
\rho_\mu(b)L_0 = L_0\rho_\nu(b) \quad \quad  \mbox{for all $b \in B_1$}
\end{equation}
(see Corollary \ref{c11:STPT37}) we get 
\[
\begin{split}
\frac{1}{|G_1|}\sum_{g \in G_1}\rho_\mu(g)L_0\rho_\nu(g^{-1})  & = \frac{1}{|G_1|}\left[\sum_{b \in B_1}\rho_\mu(b)L_0\rho_\nu(b^{-1})+ \right.\\
& \ \ \  \ \  \left. + \sum_{u \in U_1}\sum_{b \in B_1}\rho_\mu(u)\rho_\mu(w)\rho_\mu(b)L_0\rho_\nu(b^{-1})\rho_\nu(w)\rho_\nu(u^{-1})\right]\\
\mbox{(by  \eqref{sopra})}\ &  =  \frac{|B_1|}{|G_1|} L_0 + \frac{|B_1|}{|G_1|}\sum_{u \in U_1} \rho_\mu(u)\rho_\mu(w)L_0\rho_\nu(w)\rho_\nu(u^{-1})
\end{split}
\]
so that the result follows immediately from Corollary \ref{c23:STPT76}.
\end{proof}

\begin{remark}{\rm First of all, from \eqref{estar:STPT38} we deduce that 
\[
\widetilde{L} \delta_1 =\frac{1}{q+1}\delta_1+ \frac{q}{q+1}\sum_{\theta\in \FF_q}\left[\sum_{t \in \FF_q^*}\nu(-t^{-1})J(t(1-i\theta))j(t)\right]\!\!\delta_{(1-i\theta)^{-1}}
\]
so that 
\[
[\widetilde{L} \delta_1]((1-i\sigma)^{-1}) = \frac{1}{q+1}\delta_{0, \sigma} + \frac{q}{q+1}\sum_{t \in \FF_q^*}\nu(-t^{-1})J(t(1-i\sigma))j(t).
\]
This coincides with \eqref{estar:STPT69} and yields a revisitation of the calculations in the preceding section. However, we are not able to compute the norm of $\widetilde{L} \delta_1$ or, more generally, of $\widetilde{L} f$, in order to normalize $\widetilde{L}$ and thus
obtaining an {\it isometry}.}
\end{remark}

\begin{problem} Compute the norm of $\widetilde{L} \delta_1$ or, more generally, of $\widetilde{L} f$, for $f \in L(\FF_q^*)$. 
\end{problem}
A solution to the above problem would lead to a different approach to the computations in the preceding section. 
See also Remark \ref{r11:PTPT41} for a very similar problem.

\section{Final remarks}
\subsection{On a question of Ricci and Samanta}
\label{s:ricci}
Recently, Ricci and Samanta \cite[Corollary 3.3]{Ricci-Samanta} proved the following result: let $G$ be a locally compact Lie group, let $K$ be  a compact subgroup with $G/K$ connected, and let $\tau$ be an irreducible $K$-representation such that the triple $(G,K,\tau)$ is multiplicity-free ({\it commutative} in their terminology); then $(G,K)$ is a Gelfand pair. Then the authors pose the following natural question: is the same statement true outside of the realm of Lie groups?  It turns out that, in the setting of finite groups, we are able to answer their question in the following striking manner:

\begin{theorem}
\label{t:ricci}
$(\GL(2,\mathbb{F}_q),U,\chi)$ is a multiplicity-free triple for every {\em non-trivial} additive character $\chi$ of the subgroup 
$U\cong\mathbb{F}_q$ of unipotent matrices {\em but} $(\GL(2,\mathbb{F}_q),U)$ is not a Gelfand pair, that is, 
$(\GL(2,\mathbb{F}_q),U,\chi_0)$ is not multiplicity-free if $\chi_0$ is the trivial character. 
\end{theorem}

\begin{proof}
The first fact is proved in Theorem \ref{t:cuspidal}; the corresponding decomposition into irreducibles and the computation of the spherical functions are in \cite[Sections 14.6 and 14.7]{book4}. The second fact is in \cite[Exercise 14.5.10.(2)]{book4}. 
More precisely, we have (with $G = \GL(2,\mathbb{F}_q)$):
\begin{equation}\label{Ricciquest}
\Ind_U^G\chi_0=\left(\bigoplus_{\psi\in\widehat{F^*_q}}\widehat{\chi}^0_\psi\right)\bigoplus\left(\bigoplus_{\psi\in\widehat{F^*_q}}\widehat{\chi}^1_\psi\right)\bigoplus 2\left(\bigoplus_{\{\psi_1,\psi_2\}}\widehat{\chi}_{\psi_1,\psi_2}\right),
\end{equation}
where the last sum runs over all two-subsets of $\widehat{F^*_q}$. 
In order to prove the decomposition \eqref{Ricciquest}, we do not follow the proof indicated in the exercise in our monograph, because it is based on a detailed analysis of the induced representations involving also those of the affine group, but
we follow the lines of the proof of Theorem \ref{t4:PTPT19}. To this end, we compute the restriction to $U$ of the irreducible representations of $\GL(2,\mathbb{F}_q)$. First of all, note that if $b\neq 0$ then
\[
\begin{pmatrix}
b^{-1} &0 \\
0 & 1
\end{pmatrix} 
\begin{pmatrix}
1 &b \\
0 & 1
\end{pmatrix}
\begin{pmatrix}
b &0 \\
0 & 1
\end{pmatrix} 
 =\begin{pmatrix}
1 &1  \\
0 & 1
\end{pmatrix},
\]
that is, the unipotent elements
\[
\begin{pmatrix}
1 &b  \\
0 & 1
\end{pmatrix}\qquad\text{ and }
 \qquad\begin{pmatrix}
1 &1  \\
0 & 1
\end{pmatrix},
\]
are conjugate, Therefore, from the the character table of $\GL(2,\mathbb{F}_q)$  (cf.\ \cite[Table 14.2, Section 4.9]{book4}) we get,
for all $b\in\mathbb{F}_q$:
\begin{itemize}
\item  $\widehat{\chi}^0_\psi
\begin{pmatrix}
1 &b  \\
0 & 1
\end{pmatrix}\equiv \widehat{\chi}^0_\psi
\begin{pmatrix}
1 &1  \\
0 & 1
\end{pmatrix} = \psi(1)=1;$

\item $\chi^{\widehat{\chi}^1_\psi}
\begin{pmatrix}
1 &b  \\
0 & 1
\end{pmatrix} = 
\begin{cases}
0 &  \mbox{ if } b\neq 0\\
q &   \mbox{ if } b  = 0; 
\end{cases}$\\

\item $\chi^{\widehat{\chi}_{\psi_1, \psi_2}}
\begin{pmatrix} 
1 &b  \\
0 & 1
\end{pmatrix} =
\begin{cases}
\psi_1(1)\psi_2(1) = 1  &\mbox{ if } b \neq 0\\
(q+1)\psi_1(1)\psi_2(1) = (q+1) &   \mbox{ if } b  = 0; 
\end{cases}$\\

\item $\chi^{\rho_\nu}
\begin{pmatrix} 
1 &b  \\
0 & 1
\end{pmatrix} =
\begin{cases}
- \nu(1)=-1 & \mbox{ if } b\neq 0\\
(q-1)\nu(1) = (q-1) &   \mbox{ if } b  = 0.
\end{cases}$
\end{itemize}

By Frobenius reciprocity, the multiplicity of $\widehat{\chi}_{\psi}^0$ in $\Ind_U^G\chi_0$ is equal to the multiplicity of $\chi_0$ in $\Res^G_U\widehat{\chi}_\psi^0$, that is to
\[
\frac{1}{q}\sum_{b\in\mathbb{F}_q}1=1.
\]
This proves the first block in \eqref{Ricciquest}. Similarly, for $\widehat{\chi}_\psi^1$ we get:
\[
\frac{1}{q}\left(\sum_{b\in\mathbb{F}_q^*}0+ q\right)=1,
\]
and this proves the second block. For $\widehat{\chi}_{\psi_1,\psi_2}$ we get:
\[
\frac{1}{q} \left(\sum_{b\in\mathbb{F}_q^*}1 + (q+1) \right) = \frac{q-1+q+1}{q}=2,
\]
so that the third block is proved, showing, in particular, that multiplicities occur. 
Finally, for $\rho_\nu$ we get:
\[
\frac{1}{q}\left(\sum_{b\in\mathbb{F}_q^*}(-1) + (q-1)\right)=\frac{-(q-1)+(q-1)}{q}=0,
\]
showing that no cuspidal representations appear in the decomposition \eqref{Ricciquest}.
\end{proof}

\subsection{The Gelfand pair $(\GL(2,\mathbb{F}_{q^2}), \GL(2,\mathbb{F}_q))$}
\label{s:munemasa}
Akihiro Munemasa \cite{Munemasa}, after reading a preliminary version of the present paper, pointed out to us that 
$(\GL(2,\mathbb{F}_{q^2}),\GL(2,\mathbb{F}_q))$ is a Gelfand pair. This is due to Gow \cite[Theorem 3.6]{Gow} 
who proved a more general result, namely, that $(\GL(n,\mathbb{F}_{q^2}), \GL(n,\mathbb{F}_q))$ is a Gelfand pair, 
for any $n \geq 1$. Gow's result was generalized by Henderson \cite{Henderson} who, using Lusztig's crucial work on character sheaves, showed that $(G(\FF_{q^2}), G(\FF_q))$ is a Gelfand pair for any connected reductive algebraic group $G$, and found an effective algorithm for computing the corresponding spherical functions. 
This shares a strong similarity with the results of Bannai, Kawanaka, and Song \cite{BKS},
where the Gelfand pair $(\GL_{2n}(\FF_q), {\rm Sp}_{2n}(\FF_q))$ is analyzed (here ${\rm Sp}$ stands for the {\it symplectic group}).

A proof of Gow's theorem (for $n=2$) can be directly deduced from our computations in our monograph \cite[Section 14.11]{book4}, where we studied induction from $\GL(2,\mathbb{F}_q)$ to $\GL(2,\mathbb{F}_{q^m})$ for  $m \geq 2$. 
Since the relative decompositions are left as a terrific set of exercises and the case $m=2$ it is not well specified, we now complement Proposition \ref{p1:STPT2} by presenting the formulas for the induction of the other representations (the one-dimensional and the parabolic representations). 
We set $G = \GL(2,\mathbb{F}_q)$ and $G_m = \GL(2,\mathbb{F}_{q^m})$. Moreover, from  Section \ref{ss:cuspidal} we use the notation 
\eqref{estar:PTPT14} and for $\xi\in\widehat{\mathbb{F}_{q^2}^*}$ we set $\xi^\sharp=\Res^{\mathbb{F}_{q^2}^*}_{\mathbb{F}_q^*}\xi$. Then the first formula on page 537 of our monograph (namely, the decomposition of the induced representation $\Ind_G^{G_m} \widehat{\chi}^0_\psi$) for $m=2$ becomes:
\[
\Ind_G^{G_2}\widehat{\chi}_\psi^0=\left(\bigoplus_{\xi^\sharp=\psi}\widehat{\chi}^0_\xi\right)\bigoplus\left(\bigoplus_{\xi^\sharp=\psi}\widehat{\chi}_\xi^1\right)\bigoplus\left(\bigoplus_{\xi_1^\sharp=\xi_2^\sharp=\psi}\widehat{\chi}_{\xi_1,\xi_2}\right)\bigoplus\left(\bigoplus_{\overline{\xi}_1\xi_2=\Psi}\widehat{\chi}_{\xi_1,\xi_2}\right).
\]
This is {\em multiplicitiy free}. To show this, it suffices to prove that the third and the fourth block have no common summands. 
Otherwise, if $\widehat{\chi}_{\xi_1,\xi_2}$ were in both blocks, we would have $\overline{\xi}_1\xi_2=\Psi$ and $\xi_1^\sharp=\psi$. Then
\begin{equation}\label{Munemasa}
\xi_1(z)\overline{\xi}_1(z)=\xi_1(z)\xi_1(\overline{z})=\xi_1(z\overline{z})=\psi(z\overline{z})=\Psi(z)=\overline{\xi}_1(z)\xi_2(z)
\end{equation}
for all $z\in\mathbb{F}_q^*$, and therefore $\xi_1=\xi_2$, a contradiction. In particular, if $\psi$ is the trivial (multiplicative) character of $\mathbb{F}_q$, so that $\widehat{\chi}^0_\psi$ equals $\iota_G$, the trivial representation of $G$, we have that 
$\Ind_G^{G_2}\chi_\psi^0$ decomposes without multiplicities. We thus obtain

\begin{theorem}[Gow]
\label{t:munemasa}
$(\GL(2,\mathbb{F}_{q^2}), \GL(2,\mathbb{F}_q))$ is a Gelfand pair.
\end{theorem}

Returning back to our computations, the second formula on page 537 of our monograph (namely, the decomposition of the induced representation $\Ind_G^{G_m} \widehat{\chi}^1_\psi$) for $m=2$ becomes:
\[
\Ind_G^{G_2}\widehat{\chi}_\psi^1=\left(\bigoplus_{(\xi^\sharp)^2=\psi^2}\widehat{\chi}^1_\xi\right)\bigoplus\left(\bigoplus_{\substack{(\xi_1\xi_2)^\sharp=\psi^2\\\overline{\xi}_1\xi_2\neq\Psi}}\widehat{\chi}_{\xi_1,\xi_2}\right)\bigoplus\left(\bigoplus_{\xi_1^\sharp=\xi_2^\sharp=\psi}\widehat{\chi}_{\xi_1\xi_2}\right)\bigoplus\left(\bigoplus_{\nu^\sharp=\psi^2}\rho_\nu\right).
\]
Now multiplicities do appear! Indeed, each representation in the third block is also in the second block. For, if $\xi_1^\sharp=\xi_2^\sharp=\psi$ then also $(\xi_1\xi_2)^\sharp=\psi^2$ (and $\overline{\xi}_1\xi_2\neq\Psi$, as proved above, cf.\ \eqref{Munemasa}).

Finally, Formula (14.64) in \cite{book4} for $m=2$ becomes:
\[
\begin{split}
\Ind_G^{G_2}\widehat{\chi}_{\psi_1,\psi_2}=&\left(\bigoplus_{(\xi^\sharp)^2=\psi_1\psi_2}\widehat{\chi}^1_\xi\right)\bigoplus\left(\bigoplus_{(\xi_1\xi_2)^\sharp=\psi_1\psi_2}\widehat{\chi}_{\xi_1,\xi_2}\right)\bigoplus\\
&\qquad\qquad\bigoplus\left(\bigoplus_{\substack{\xi_1^\sharp=\psi_1\\\xi_2^\sharp=\psi_2}}\widehat{\chi}_{\xi_1\xi_2}\right)\bigoplus\left(\bigoplus_{\nu^\sharp=\psi_1\psi_2}\rho_\nu\right).
\end{split}
\]
Once more, multiplicities do occur, since the third block is clearly contained in the second. In conclusion, we have the following
strengthening of Theorem \ref{t:munemasa}:

\begin{theorem}
The decomposition of the induced representation $\Ind_G^{G_2}\theta$ is multiplicity-free if and only if $\theta$ is cuspidal or 
one-dimensional. If $\theta$ is parabolic, then some subrepresentations in $\Ind_G^{G_2}\theta$ appear with multiplicity $2$.
\end{theorem}

\subsection{On some questions of Charles F.\ Dunkl}
Charles Dunkl \cite{dunkl0}, after reading a preliminary version of the manuscript, pointed out to us interesting connections
with other similar constructions and asked a few questions in relation with these.

\noindent
{\bf 1.\ Algebras of conjugacy-invariant functions.} The first one is related to the work of I.I.\ Hirschman \cite{hirschman} which involves two subgroups $K \leq H \leq G$ (with $H$ 
contained in the normalizer of $K$) of a finite group $G$, and yields a subalgebra $L(G,H,K)$ of $L(G)$.
We thus recall from \cite[Chapter 2]{book2} a generalization of the theory of subgroup-conjugacy-invariant algebras due to A. Greenhalgh
\cite{Green0} and that, following Diaconis \cite{diaconisPAT}, we called Greenhalgebras. These were also considered by Brender \cite{BrenderII} (but Greenhalgh considered a more general case).
\par
Let $G$ be a finite group and suppose that $K$ and $H$ are two subgroups of $G$, with 
$K \trianglelefteq H \leq G$ (thus, in particular, $H$ is contained in the normalizer 
of $K$, as in \cite{hirschman}). 
Then, the {\it Greenhalgebra} associated with $G$, $H$, and $K$ is the subalgebra $\mathcal{G}(G,H,K)$ of $L(G)$ consisting of
all functions that are both $H$-conjugacy invariant and bi-$K$-invariant (cf.\ \cite[Section 2.1.3]{book2}):
\[
\mathcal{G}(G,H,K) \!=\! \{f\!\!\in\!\! L(G):  
f(h^{-1}gh)\! =\! f(g) \mbox{ and } 
 f(k_1gk_2)\! = \!f(g),
\forall g\! \in \!G,\  h\!\in\! H,\  k_1,k_2\!\in \!K\}.
\]

Set $\widetilde{H} = \{(h,h): h \in H\} \leq G \times G$ and
\[
B = \left(K \times \{1_G\}\right)\!\widetilde{H} = \{(kh,h): k \in K \mbox{ and } h \in H\} \leq H \times H.
\]
Given an irreducible representation $\theta \in \widehat{B}$ we say that the quadruple $(G,H,K;\theta)$ is
\emph{multiplicity-free} provided the induced representation $\Ind_B^{G \times H} \theta$ decomposes without multiplicities.
Let us also set $\widehat{H}_K = \{\rho \in \widehat{H}: \Res^H_K \rho = (\dim \rho) \iota_K\}$
(where $\iota_K$ is the trivial representation of $K$). 
We note (cf.\ \cite[Lemma 2.1.15]{book2}) that $\mathcal{G}(G,H,K)$ is isomorphic to the algebra 
\[
^B\!L(G \times H)^B
=\{f \in L(G \times H): f(b_1(g,h)b_2) = f(g,h) \mbox{ for all } b_1,b_2 \in B \mbox{ and } g \in G, h \in H\}
\] of all bi-$B$-invariant functions on $G \times H$.
\par
When $\theta = \iota_B$ is the trivial representation of $B$ we have (cf.\ \cite[Theorem 2.1.19]{book2}):
\begin{theorem}
\label{t:greenhalg}
With the above notation, the following conditions are equivalent:
\begin{enumerate}[{\rm (a)}]
\item the quadruple $(G,H,K;\iota_B)$ is multiplicity-free;
\item the Greenhalgebra $\mathcal{G}(G,H,K)$ is commutative;
\item $(G \times H,B)$ is a Gelfand pair;
\item{for every $\sigma \in \widehat{G}$ and $\rho \in \widehat{H}_K$, the multiplicity
of $\sigma$ in $\Ind_H^G\rho$ is $\leq 1$.}
\end{enumerate}
\end{theorem}

Note that when $K = \{1_G\}$, then $B = \widetilde{H}$ and
${\mathcal{G}}(G,H,K)$ is simply the subalgebra $\mathcal{C}(G,H) = \{f \in L(G):  \  f(h^{-1}gh) = f(g) \mbox{ for all } g \in G \mbox{ and }  h\in H\}$ of all $H$-conjugacy invariant functions on $G$ (cf.\ \cite[Section 2.1.1]{book2}). Moreover,
Theorem \ref{t:greenhalg} yields (cf.\ \cite[Theorem 2.1.10]{book2}):

\begin{theorem}\label{t;cit5} 
The following conditions are equivalent:
\begin{enumerate}[{\rm (a)}]
\item the quadruple $(G,H,\{1_G\};\iota_{\widetilde{H}})$ is multiplicity-free;
\item{the algebra ${\mathcal C}(G,H)$ is commutative;}
\item{$(G\times H, \widetilde{H})$ is a Gelfand pair;}
\item{$H$ is a multiplicity-free subgroup of $G$.}
\end{enumerate}
\end{theorem}

Returning back to a general irreducible representation $\theta \in \widehat{B}$, motivated by Dunkl's question, we pose the following:
\begin{problem}
Given a quadruple $(G,H,K;\theta)$ as above, define a Hecke-type Greenahlgebra ${\mathcal {HG}}(G,H,K;\theta)$
in a such a way that (i) ${\mathcal {HG}}(G,H,K;\iota_B) = {\mathcal G}(G,H,K)$ (that is, when $\theta = \iota_B$ is the
trivial representation, then the Hecke-type Greenahlgebra coincides with the Greenhalgebra of the triple $(G,H,K)$) and
(ii) ${\mathcal {HG}}(G,H,K;\theta)$ is commutative if and only if the quadruple $(G,H,K;\theta)$ is
multiplicity-free.
\end{problem}
We shall try to address this problem in a future paper.\\

\noindent
{\bf 2.\ Positive-definite functions.}
Let $G$ be a finite group. A function $\phi \colon G \to \C$ is said to be \emph{positive-definite }(or \emph{of positive type}) if equivalently (see \cite[Section 3.3]{folland}, \cite[Capitolo VII, Sezione 7]{ricci}):
\begin{enumerate}[{\rm (a)}]
\item $\sum_{g,f \in G}\phi(h^{-1}g)f(g)\overline{f(g)} \geq 0$ for all $f \in L(G)$;
\item $\sum_{i,j=1}^n c_i\overline{c_j}\phi(g_j^{-1}g_i) \geq 0$ for all $c_1,c_2,\ldots, c_n \in \C$, $g_1,g_2, \ldots, g_n \in G$, and $n \geq 1$;
\item there exists a (unitary) representation $(\sigma_\phi, V_\phi)$ of $G$ and a (cyclic) vector $v_\phi \in V_\phi$ such that
$\phi(g) = \langle \sigma_\phi(g)v_\phi,v_\phi \rangle_{V_\phi}$ for all $g \in G$.
\end{enumerate}

Suppose that $(G,K)$ is a Gelfand pair. Let $\phi \in L(G)$ be a function of positive type. Then $\phi$ is bi-$K$-invariant if and only
if $\sigma_\phi(k)v_\phi = v_\phi$ for all $k \in K$; additionally, $\phi$ is spherical if and only if $\sigma_\phi$ is irreducible.
Moreover, the spherical functions of positive type together with the zero function are exactly the extremal points of the (convex) set
of all functions of positive type. 

In our setting, the spherical functions $\phi^\sigma$, $\sigma \in J$ as in (4.2.4)
are positive-definite.

Recall that the Hecke algebra ${\mathcal H}(G,K,\psi) \leq L(G)$ is isomorphic, via the isomorphism $S_v^{-1}$ (cf.\ Theorem 3.9)  
to the Hecke algebra $\widetilde{{\mathcal H}}(G,K,\theta)$ which consists of $\End_G(V_\theta)$-valued functions on $G$.
Then the matrix-valued functions $S_v^{-1}(\phi^\sigma) \in \widetilde{{\mathcal H}}(G,K,\theta)$ are positive-definite in the sense of Dunkl 
(cf.\ \cite[Definition 5.1]{dunkl3}: with $H = K$).

A comment about positive-definite functions on finite groups: we thank Ch.F.\ Dunkl for this interesting historical information. 
In the mid-Seventies,  L.L.\ Scott, while involved in the search for new simple finite groups, computed possible value tables for the spherical functions of homogeneous spaces of rank $3$. 
He showed his computations to his colleague Ch.F.\ Dunkl at the University of Virginia who checked them and found 
that, as this positivity condition for the functions involved therein was missing (cf.\ \cite[Section 5]{dunkl3}), 
the corresponding tables could not be valid (so that this possibility had to be eliminated). 
Scott wrote this up and dubbed it the \emph{Krein condition} (this positivity condition in finite permutation groups 
is perfectly analogous to a condition obtained in harmonic analysis by the Soviet mathematician M.G.\ Krein) 
\cite{scott, scott2}.\\

\noindent
{\bf 3.\ Finite hypergroups.}
A finite \emph{(algebraic) hypergroup} (see, e.g., \cite{corsini}) is a pair $(X,*)$, where $X$ is a nonempty finite set equipped with a multi-valued map, called \emph{hyperoperation} and denoted $*$, from $X \times X$ to ${\mathcal P}^*(X)$, the set of all nonempty subsets of $X$, satisfying  the following
properties:
\begin{enumerate}[{\rm (i)}]
\item $(x * y) * z = x * (y * z)$ for all $x,y,z \in X$ (\emph{associative property});
\item $x * X = X * x = X$ for all $x \in X$ (\emph{reproduction property}),
\end{enumerate}
where for subsets $Y,Z \subset X$ one defines $Y * Z = \{y*z: y \in Y, z \in Z\} \subset X$.

If, in addition one has 
\begin{enumerate}
\item[{\rm (iii)}] $x * y = y * x$ for all $x,y \in X$ (\emph{commutative property)}
\end{enumerate}
one says that $(X,*)$ is commutative. Also, an element $e \in X$ is called a \emph{unit} provided
\begin{enumerate}
\item[{\rm (iv)}] $x \in (e * x) \cap (x * e)$ for all $x \in X$.
\end{enumerate}

Given a  finite set $X$, we denote by $L^1(X)$ (resp.\ $L^1_+(X)$) the space of all function $f \in L(X)$ such that
$\sum_{x \in X} f(x) = 1$ (resp.\ $f \in L^1(X)$ such that $f \geq 0$).  

A  finite \emph{functional hypergroup} (cf.\ \cite{dunkl1, dunkl2, jewett}) is a pair $(X,\lambda)$ where $X$ is a nonempty finite set and 
$\lambda \colon X \times X \to L^1_+(X)$ satisfies the following properties:
\begin{enumerate}[{\rm (i')}]
\item $(\mu *_\lambda \nu) *_\lambda \xi = \mu *_\lambda (\nu *_\lambda \xi)$ for all $\mu, \nu, \xi \in L^1_+(X)$ (\emph{associative property}),
\item $\cup_{y \in Y} \supp(\lambda(x,y)) = \cup_{y \in Y} \supp(\lambda(y,x)) = X$ for all $x \in X$ (\emph{reproduction property}),
\end{enumerate}
where, $\delta_x \in L^1_+(X)$ is the Dirac delta at $x \in X$, $\supp(f) = \{x \in X: f(x) \neq 0\}$ is the \emph{support} of $f
\in L^1(X)$, and $*_\lambda \colon L^1_+(X) \times L^1_+(X) \to L^1_+(X)$ is the (bilinear) product defined by
\begin{equation}
\label{e:mu-star-lambda-nu}
\mu *_\lambda \nu = \sum_{x,y \in X} \mu(x)\mu(y)\lambda(x,y)
\end{equation}
for all $\mu = \sum_{x \in X} \mu(x)\delta_x$ and 
$\nu = \sum_{y \in X} \nu(y)\delta_y$ in $L^1_+(X)$.
Note that \eqref{e:mu-star-lambda-nu} is equivalent to
\begin{equation}
\label{e:delta-dirac}
\delta_x *_\lambda \delta_y = \lambda(x,y)
\end{equation}
for all $x,y \in X$.

If, in addition, one has
\begin{enumerate}
\item[{\rm (iii')}] $\lambda(x,y) = \lambda(y,x)$ for all $x,y \in X$ (\emph{commutative property)}
\end{enumerate}
one says that $(X,\lambda)$ is commutative.
Finally, an element $e \in X$ such that
\begin{enumerate}
\item[{\rm (iv')}] $\lambda(x,e) = \lambda(e,x) = \delta_x$ 
\end{enumerate}
for all $x \in X$, is called a unit of the functional hypergroup $X$.

A (finite) \emph{weak functional hypergroup} is defined verbatim except that one replaces $L^1_+(X)$ by $L^1(X)$, that is, one drops the condition $\lambda(x,y) \geq 0$ for all $x,y \in X$. Clearly, every functional hypergroup is a weak functional hypergroup.

\begin{example} (1) Every finite algebraic hypergroup is a finite functional hypergroup. Indeed, given a finite algebraic hypergroup $(X,*)$ one may
set $\lambda(x,y) = \frac{1}{|x*y|}\sum_{z \in x * y} \delta_z$ for all $x,y \in X$. Clearly, an algebraic hypergroup is
commutative (resp.\ has a unit) if and only if the associated functional hypergroup is (resp.\ has). 
Vice versa, with any functional hypergroup $(X,\lambda)$ one
may associate an algebraic hypergroup by setting $x * y = \supp(\lambda(x,y))$ for all $x, y \in X$ (cf.\ \cite[Proposition 1.4]{dunkl1}).

(2) Every finite group (resp.\ abelian group) $(G,\cdot)$ is an algebraic hypergroup (resp.\ an abelian hypergroup). This follows immediately
after defining $x * y = \{x \cdot y\}$. Indeed, (i) follows from associativity of the group operation, while (ii) is a consequence of the fact that left and right multiplication by a fixed group element constitutes a permutation of the group. Moreover the identity element of $G$
serves as a unit for the hypergoup.

(3) Let $G$ be a finite group. Then $X = \widehat{G}$ is an algebraic hypergroup after setting $x * y = \{z \in X: z \preceq x \otimes y\}$
for all $x,y \in X$. Equivalently, $X$ is a functional hypergroup by setting $\lambda(x,y)(z) = \frac{\dim V_z}{\dim V_x \dim V_y} \dim \Hom_G(z, x \otimes y)$
(in other words, $\lambda(x,y)(z) = \frac{\dim V_z}{\dim V_x \dim V_y}\mu_{xyz}$, where $x \otimes y = \sum_{z \in X} \mu_{xyz} z$) 
for all $x,y,z \in X$. Moreover, the trivial representation $\iota_G \in X$ serves as a unit for the hypergoup.

(4) Let $G$ be a finite group and let $X$ denote the set of all conjugacy classes of $G$. Given $x \in X$ we denote by $f_x = \frac{1}{|x|}\sum_{g \in x} \delta_g \in L^1_+(X)$ the normalized characteristic function of $x \subseteq G$. It is well known that $\{f_x: x \in X\}$
constitutes a base for the subspace ${\mathcal C}(G) = \{f \in L(G): f(h^{-1}gh) = f(g)$ for all $g,h \in G\}$ of \emph{conjugacy-invariant}
functions on $G$ and that, indeed, ${\mathcal C}(G)$ is a subalgebra of $L(G)$. We may thus define a functional hypergroup $(X,\lambda)$
by setting $\lambda(x,y) = f_x * f_y \in L^1_+(G)$ (where $*$ is the usual convolution) for all $x, y \in X$. Moreover, the conjugacy class of the identity element of $G$ serves as a unit for the hypergoup.

(5) Let $G$ be a finite group and let $K \leq G$ be a subgroup. Denote by $X = K \backslash G / K = \{KgK: g \in G\}$ the set of all double $K$-cosets of $G$. For $x \in X$ denote by $f_x = \frac{1}{|x|} \sum_{g \in x} \delta_g \in L^1_+(X)$ the normalized characteristic function of 
$x \in X$. It is well known that $\{f_x: x \in X\}$ constitutes a basis for the subspace $\!^K\!L(G)^K = \{f \in L(G): f(k_1^{-1}gk_2) = f(g)$ for all $g \in G$ and $k_1,k_2 \in K\}$ of \emph{bi-$K$-invariant}
functions on $G$ and that, indeed, $\!^K\!L(G)^K$ is a subalgebra of $L(G)$. We may thus define a functional hypergroup $(X,\lambda)$
by setting $\lambda(x,y) = f_x * f_y \in L^1_+(G)$ (where $*$ is the usual convolution) for all $x, y \in X$. 
The double coset $K = K \{1_G\} K$ serves as a unit for the hypergoup.
Finally, $(X,\lambda)$ is commutative if and only if $\!^K\!L(G)^K$ is commutative, equivalenlty, if and only if $(G,K)$ is a Gelfand pair.

(6) In the setting of the present paper (see also \cite[Section 13.2]{book4}), let $G$ be a finite group, $K \leq G$ a subgroup, and suppose that $\chi \in \widehat{K}$ is one-dimensional. Consider the Hecke algebra 
\[
\mathcal{H}(G,K,\chi)\!=\!\left\{f\!\in\! L(G):f(k_1gk_2)=\overline{\chi(k_1k_2)}f(g), \text{ for all }g\!\in\! G,k_1,k_2\!\in\! K\right\}
\]
(when $\chi = \iota_K$, the trivial representation of $K$, this is nothing but $\!^K\!L(G)^K$, the algebra of bi-$K$-invariant, discussed in (4) above). Let ${\mathcal S}$ be a set of representatives for the set $K \backslash G / K$ of double $K$-cosets and set
\[
X = {\mathcal S}_0 = \{s\in\mathcal{S}:\chi(x)=\chi(s^{-1}xs) \mbox{ for all } x\in K \cap sKs^{-1}\}.
\] 
Then the functions $a_x \in L(G)$ defined by
\[
a_x(g) = \begin{cases}
\frac{1}{|K|} \overline{\chi(k_1)}\overline{\chi(k_2)}& \mbox{if } g = k_1 x k_2 \ \mbox{ for some }k_1,k_2 \in K\\
0 & \mbox{if } g \notin K x K
\end{cases}
\]
for all $g \in G$ and $x \in X$, constitute a basis (called the \emph{Curtis and Fossum} basis) for $\mathcal{H}(G,K,\chi)$ and 
the numbers $\mu_{{x y z}}$ such that 
 \begin{equation}\label{fosbas2}
 a_x*a_y=  \sum_{z \in X}\mu_{{x y z}}a_z
 \end{equation}
for all $x,y,z \in X$, are called the {\it structure constants} of $\mathcal{H}(G,K,\chi)$.
Denoting by $f_x = \frac{1}{|K x K|} a_x \in L^1(G)$ the corresponding normalization, we have, for all $x,y \in X$,
\[
f_x * f_y = \sum_{z \in X} \mu'_{xyz} f_z,
\]
where $\mu'_{xyz} = \frac{|K z K|}{|K x K| \cdot |K y K|} \mu_{{x y z}} \in \CC$. We may thus define a weak-hypergroup $(X,\lambda)$
by setting $\lambda(x,y) = \sum_{z \in X} \mu'_{xyz} f_z$ for all $x, y \in X$.
\end{example}


\begin{thebibliography}{99}



\bibitem{BannaiIto} E.\ Bannai and T.\ Ito, {\it Algebraic Combinatorics}, Benjamin, Menlo Park, CA, 1984.

\bibitem{BKS} E.\ Bannai, N.\ Kawanaka, and S.-Y.\ Song, The character table of the Hecke algebra
${\mathcal H}(\GL_{2n}(\FF_q); {\rm Sp}_{2n}(\FF_q))$, {\it J.\ Algebra}, {\bf 129} (1990), 320--366.

\bibitem{BT1} E.\ Bannai and H.\ Tanaka, The decomposition of the permutation character $1^{{\tiny\GL}(2n,q)}_{{\tiny\GL}(n,q^2)}$, 
{\it J.\ Algebra} {\bf 265} (2003), no.\ 2, 496--512.

\bibitem{BT2} E.\ Bannai and H.\ Tanaka, Appendix: On some Gelfand pairs and commutative association schemes, {\it Jpn.J.\ Math.} 
{\bf 10} (2015), no.\ 1, 97--104. 

\bibitem{BZ} Ya.G.\ Berkovich, and E.M.\ Zhmudʹ, {\it Characters of finite groups}. Part 1. Translations of Mathematical Monographs, 172. American Mathematical Society, Providence, RI, 1998.

\bibitem{BrenderII} M. Brender, A class of Schur algebras, {\it Trans.\ Amer.\ Math.\ Soc.} {\bf 248} (1979), no. 2, 435--444. 


\bibitem{Bump} D. Bump, {\it Lie groups}. Graduate Texts in Mathematics, 225. Springer-Verlag, New York, 2004.

\bibitem{BumpGinz} D.\ Bump and D.\ Ginzburg, Generalized Frobenius-Schur numbers, {\it J.\ Algebra} {\bf 278} (2004), no.\ 1, 294--313. 

\bibitem{corsini}
P.\ Corsini and V.\ Leoreanu, {\it Applications of Hyperstructure Theory}, Springer, 2003.


\bibitem{GelGeorg} T. Ceccherini-Silberstein,  F. Scarabotti and F. Tolli, Finite Gelfand pairs and their applications to probability and statistics, {\it  J. Math. Sci. (N. Y.)} {\bf 141} (2007), no. 2, 1182--1229.

\bibitem{book} T. Ceccherini-Silberstein, F. Scarabotti and F. Tolli, {\it Harmonic analysis on finite groups:
representation theory, Gelfand pairs and Markov chains.}  
Cambridge Studies in Advanced Mathematics {108}, Cambridge University Press 2008.

\bibitem{CSTind} T. Ceccherini-Silberstein, A. Mach\`\i, F. Scarabotti and F. Tolli, Induced representations and Mackey theory. {\it J. Math. Sci. (N.Y.)} {\bf 156} (2009), no. 1, 11--28.

\bibitem{CSTCli} T. Ceccherini-Silberstein, F. Scarabotti and F. Tolli, Clifford theory and applications. {\it J. Math. Sci. (N.Y.)} {\bf 156} (2009), no. 1, 29--43. 


\bibitem{book2} T. Ceccherini-Silberstein, F. Scarabotti and F. Tolli, {\it Representation theory of the symmetric groups: the Okounkov-Vershik approach, character formulas, and partition algebras.} Cambridge Studies in Advanced Mathematics {121}, Cambridge University Press 2010.

\bibitem{book3} T. Ceccherini-Silberstein, F.Scarabotti and F.Tolli:  {\it Representation Theory and Harmonic Analysis of wreath products of finite groups}. London Mathematical Society Lecture Note Series {\bf 410},  Cambridge University Press, 2014.

\bibitem{AM} T. Ceccherini-Silberstein, F. Scarabotti and F. Tolli, Mackey's theory of $\tau$-conjugate representations for finite groups, {\it  Jpn.\ J.\ Math.} {\bf 10} (2015), no. 1, 43--96.

\bibitem{CSTTohoku} T.\ Ceccherini-Silberstein, F.\ Scarabotti, and F.\ Tolli, Mackey's criterion for subgroup restriction of Kronecker products and harmonic analysis on Clifford groups, {\it Tohoku Math.\ J.} (2) {\bf 67} (2015), no.\ 4, 553--571.

\bibitem{book4} T. Ceccherini-Silberstein, F.Scarabotti and F.Tolli:  {\it Discrete Harmonic Analysis: Representations, Number Theory, Expanders, and the Fourier Transform}. Cambridge Studies in Advanced Mathematics {172}, Cambridge University Press, 2018.

\bibitem{CST4} T. Ceccherini-Silberstein, F. Scarabotti and F. Tolli, Clifford theory, Mackey obstruction and applications.  
Work in progress.

\bibitem{CURFOS} C.~W. Curtis, T.~V Fossum, On centralizer rings and characters of
representations of finite groups. {\it Math. Z.} {\bf 107} (1968) 402--406.
\bibitem{CR1} Ch. W. Curtis and I. Reiner, {\it Representation theory of finite
groups and associative algebras}. Reprint of the 1962 original. Wiley Classics
Library. A Wiley-Interscience Publication. John Wiley \& Sons, Inc., New York, 1988.
\bibitem{CR2} Ch.W.\ Curtis and I. Reiner,{ \it  Methods of Representation Theory.
With Applications to Finite Groups and Orders}, Vols. I, Pure Appl. Math., John Wiley \& Sons, New York 1981.

\bibitem{Diaconis} P.\ Diaconis, {\it Groups Representations in Probability and Statistics.} IMS Hayward, CA, 1988.

\bibitem{diaconisPAT} P. Diaconis, Patterned matrices. 
Matrix theory and applications (Phoenix, AZ, 1989), 37--58, 
{\it Proc. Sympos. Appl. Math.}, {\bf 40}, 
Amer. Math. Soc., Providence, RI, 1990.


\bibitem{dunkl0}
Ch.F.\ Dunkl, {\it private communication}.

\bibitem{dunkl1} 
Ch.F.\ Dunkl,
The measure algera of a locally compact hypergroup, {\it Trans.\ Amer.\ Math.\ Soc.} {\bf 179} (1973), 331--348.

\bibitem{dunkl2} 
Ch.F.\ Dunkl,  Structure hypergroups for measure algebras, {\it Pacific J.\ Math.} {\bf 47} (1973), 413--425.

\bibitem{dunkl3}
Ch.F.\ Dunkl, Spherical functions on compact groups and applications to special functions. Symposia Mathematica, Vol. XXII (Convegno sull'Analisi Armonica e Spazi di Funzioni su Gruppi Localmente Compatti, INDAM, Rome, 1976), pp.~145--161. Academic Press, London, 1977. 

\bibitem{Dunkl} Ch.F.\ Dunkl, Orthogonal functions on some permutation groups, {\it Proc. Symp. Pure Math.} {\bf 34}, Amer. Math. Soc., Providence, RI, (1979), 129--147.

\bibitem{FellDoran} J.M.G.\ Fell, R.S.\ Doran, {\it Representations of *-algebras, locally compact groups, and Banach *-algebraic bundles}. Vol. 2. Banach *-algebraic bundles, induced representations, and the generalized Mackey analysis. Pure and Applied Mathematics, 126. Academic Press, Inc., Boston, MA, 1988.

\bibitem{folland}
G.B.\ Folland {\it A course in abstract harmonic analysis}, Studies in Advanced Mathematics. CRC Press, Boca Raton, FL, 1995.

\bibitem{GG}
I.M.\ Gelfand and M.|I.\ Graev, Construction of irreducible representations of simple algebraic groups over a finite field, 
{\it Dokl.\ Akad. Nauk SSSR}, {\bf 147} (1962), 529--532.

\bibitem{Gow}
R.\ Gow, Two multiplicity-free permutations of the general linear group $\GL(n; q^2)$, {\it Math. Z.},
{\bf 188} (1984), 45--54.

\bibitem{Green}
J.A.\ Green, The characters of the finite general linear groups, {\it Trans. Amer. Math. Soc.},
{\bf 80} (1955), 402--447.

\bibitem{Green0} A.S. Greenhalgh, Measure on groups with subgroups invariance properties, {\it Technical report No. 321}, Department of Statistics, Stanford University, 1989.


\bibitem{Harpe}
P.\ de la Harpe, private communication.

\bibitem{Henderson}
A.\ Henderson, Spherical functions of the symmetric space $G(\FF_{q^2})/G(\FF_q)$, {\it Represent.\ Theory} {\bf 5} (2001), 581--614.

\bibitem{hirschman} I.I.\ Hirschman Jr, Integral equations on certain compact homogeneous spaces, {\it SIAM J.\ Math.\ Anal.} {\bf 3} 
(1972), 314--343.

 

\bibitem{Hu} B.\ Huppert, {\it Character Theory of Finite Groups}, De Gruyter Expositions in Mathematics, 25, Walter de Gruyter, 1998.
 
\bibitem{Isaacs} I.M. Isaacs, {\it Character theory of finite groups}, Corrected reprint of the 1976 original [Academic Press, New York]. Dover Publications, Inc., New York, 1994.

\bibitem{jewett} Robert J.\ Jewett, Spaces with an abstract convolution of measures, {\it Advances in Math.} {\bf 18} (1975), no. 1, 1--101.

\bibitem{Macdonald} I.G.\ Macdonald, {\it Symmetric functions and Hall polynomials.} Second edition. 
With contributions by A. Zelevinsky. Oxford Mathematical Monographs. Oxford Science Publications.
The Clarendon Press, Oxford University Press, New York, 1995.
 
\bibitem{Mihailovs} A.\ Mihailovs, The orbit method for finite groups of nilpontency class two of odd order, Preprint; arXiv.org: math.RT/0001092. 

\bibitem{Mizukawa1}
H.\ Mizukawa,
Twisted Gelfand pairs of complex reflection groups and $r$-congruence properties of Schur functions.
{\it Ann.\ Comb.} {\bf 15} (2011), no. 1, 119--125.

\bibitem{Mizukawa2}
H.\ Mizukawa, Wreath product generalizations of the triple $(S_{2n},H_n,\phi)$ and their spherical functions. 
{\it J.\ Algebra} {\bf 334} (2011), 31--53.

\bibitem{Munemasa}
A.\ Munemasa, {\it private communication}.

\bibitem{NS} M.A. Naimark and A.I. Stern, {\it Theory of Group Representations.} Springer-Verlag, New York, 1982.
\bibitem{PS} I. Piatetski-Shapiro, {\it Complex representations of GL(2,K)  for
finite fields K.} Contemporary Mathematics, 16. American Mathematical Society,
Providence, R.I., 1983.

\bibitem{OV} A.\ Okounkov and A.M.\ Vershik, A new approach to representation theory of symmetric groups. 
{\it Selecta Math.\ (N.S.)} {\bf 2} (1996), no. 4, 581--605.

\bibitem{ricci} F.\ Ricci {\it Analisi di Fourier non commutativa}, Class Notes SNS, 2018.

\bibitem{Ricci-Samanta} F.\ Ricci and A.\ Samanta,  Spherical analysis on homogeneous vector bundles, 
{\it Adv.\ Math.} {\bf 338} (2018), 953--990.

\bibitem{Saxl} J. Saxl, On multiplicity-free permutation representations, in {\it Finite Geometries and Designs}, p. 337--353, 
London Math. Soc. Lecture Notes Series, {\bf 48}, Cambridge University Press, 1981.
\bibitem{st1} F. Scarabotti and F. Tolli, Harmonic analysis on a finite homogeneous space, {\it Proc. Lond. Math. Soc.}(3) {\bf{100}} (2010),  no. 2, 348--376.
\bibitem{st3}  F. Scarabotti and F. Tolli, Fourier analysis of subgroup-conjugacy invariant functions on finite groups, {\it Monatsh. Math.} {\bf 170} (2013) 465--479.
\bibitem{st4}  F. Scarabotti and F. Tolli, Hecke algebras and harmonic analysis on  finite groups,  {\it Rend. Mat. Appl.} (7) {\bf 33} (2013), no. 1-2, 27--51.
\bibitem{st5} F. Scarabotti and F. Tolli, Induced representations and harmonic analysis on finite groups, {\it  Monatsh.\ Math.} {\bf 181} (2016), no. 4, 937--965.

\bibitem{scott} L.L.\ Scott, 
Some properties of character products, {\it J.\ Algebra} {\bf 45} (1977), no.\ 2, 259--265. 

\bibitem{scott2}
http://people.virginia.edu/~lls2l/collaborators.htm

\bibitem{Simon} B. Simon, {\it Representations of finite and compact groups}, American Math. Soc., 1996. 

\bibitem{Sternberg} S. Sternberg, {\it Group theory and physics.} Cambridge University Press, Cambridge, 1994.

\bibitem{Stanton} D. Stanton, An introduction to group representations and orthogonal polynomials, 
in {\it Orthogonal Polynomials} (P. Nevai Ed.), 419--433, Kluwer Academic Dordrecht, 1990.

\bibitem{Stembridge} J.R.\ Stembridge, On Schur’s Q-functions and the primitive idempotents of a commutative
Hecke algebra, {\it J. Algebraic Combin.} {\bf 1} (1992), no. 1, 71--95.

\bibitem{Terras} A. Terras, {\it Fourier analysis on finite groups and applications}. London Mathematical Society Student Texts, 43. Cambridge University Press, Cambridge, 1999.
\end{thebibliography}
\end{document}